\documentclass[english,leqno]{article}

\usepackage{latexsym}
\usepackage{linearb}
\usepackage[LGR, OT1]{fontenc}
\usepackage{amsmath,amsthm}
\usepackage{amssymb}
\usepackage{MnSymbol}
\usepackage{graphicx}
\usepackage{marvosym}
\usepackage{eufrak}
\usepackage[margin=1.25in,dvips]{geometry}
\usepackage{color}
\usepackage{comment}
\usepackage{mathrsfs}\usepackage[LGR,T1]{fontenc}
\usepackage[latin9]{inputenc}
\usepackage{geometry}
\usepackage{babel}
\usepackage{enumitem}
\usepackage{amsthm}
\usepackage{amstext}
\usepackage{amssymb}
\usepackage{esint}
\usepackage[unicode=true,pdfusetitle,
 bookmarks=true,bookmarksnumbered=false,bookmarksopen=false,
 breaklinks=false,pdfborder={0 0 1},backref=false,colorlinks=false]
 {hyperref}
\makeatletter

\DeclareRobustCommand{\greektext}{%
  \fontencoding{LGR}\selectfont\def\encodingdefault{LGR}}
\DeclareRobustCommand{\textgreek}[1]{\leavevmode{\greektext #1}}
\DeclareFontEncoding{LGR}{}{}
\DeclareTextSymbol{\~}{LGR}{126}
\newcommand{\lyxmathsym}[1]{\ifmmode\begingroup\def\b@ld{bold}
  \text{\ifx\math@version\b@ld\bfseries\fi#1}\endgroup\else#1\fi}

\usepackage{enumitem}		
\numberwithin{equation}{section}
\numberwithin{figure}{section}
\newcommand{\lyxaddress}[1]{
\par {\raggedright #1
\vspace{1.4em}
\noindent\par}
}
  \theoremstyle{plain}
  \newtheorem{thm}{\protect\theoremname}[section]
  \theoremstyle{remark}
  \newtheorem*{rem*}{\protect\remarkname}
  \theoremstyle{plain}
  \newtheorem{cor}{\protect\corollaryname}[section]
  \theoremstyle{plain}
  \newtheorem{lem}{\protect\lemmaname}[section]
  \theoremstyle{plain}
  \newtheorem{prop}{\protect\propositionname}[section]

\makeatother

\usepackage{babel}
  \providecommand{\lemmaname}{Lemma}
  \providecommand{\propositionname}{Proposition}
  \providecommand{\remarkname}{Remark}
\providecommand{\corollaryname}{Corollary}
\providecommand{\theoremname}{Theorem}

\begin{document}

\title{A proof of Friedman's ergosphere instability for scalar waves}

\author{Georgios Moschidis}

\maketitle

\lyxaddress{Princeton University, Department of Mathematics, Fine Hall, Washington
Road, Princeton, NJ 08544, United States, \tt gm6@math.princeton.edu}
\begin{abstract}
Let $(\mathcal{M}^{3+1},g)$ be a real analytic, stationary and asymptotically
flat spacetime with a non-empty ergoregion $\mathscr{E}$ and no future
event horizon $\mathcal{H}^{+}$. In \cite{Friedman1978}, Friedman
observed that, on such spacetimes, there exist solutions $\text{\textgreek{f}}$
to the wave equation $\square_{g}\text{\textgreek{f}}=0$ such that
their local energy does not decay to $0$ as time increases. In addition,
Friedman provided a heuristic argument that the energy of such solutions
actually grows to $+\infty$. In this paper, we provide a rigorous
proof of Friedman's instability. Our setting is, in fact, more general.
We consider smooth spacetimes $(\mathcal{M}^{d+1},g)$, for any $d\ge2$,
not necessarily globally real analytic. We impose only a unique continuation
condition for the wave equation across the boundary $\partial\mathscr{E}$
of $\mathscr{E}$ on a small neighborhood of a point $p\in\partial\mathscr{E}$.
This condition always holds if $(\mathcal{M},g)$ is analytic in that
neighborhood of $p$, but it can also be inferred in the case when
$(\mathcal{M},g)$ possesses a second Killing field $\text{\textgreek{F}}$
such that the span of $\text{\textgreek{F}}$ and the stationary Killing
field $T$ is timelike on $\partial\mathscr{E}$. We also allow the
spacetimes $(\mathcal{M},g)$ under consideration to possess a (possibly
empty) future event horizon $\mathcal{H}^{+}$, such that, however,
$\mathcal{H}^{+}\cap\mathscr{E}=\emptyset$ (excluding, thus, the
Kerr exterior family). As an application of our theorem, we infer
an instability result for the acoustical wave equation on the hydrodynamic
vortex, a phenomenon first investigated numerically by Oliveira, Cardoso
and Crispino in \cite{Oliveira2014}. Furthermore, as a side benefit
of our proof, we provide a derivation, based entirely on the vector
field method, of a Carleman-type estimate on the exterior of the ergoregion
for a general class of stationary and asymptotically flat spacetimes.
Applications of this estimate include a Morawetz-type bound for solutions
$\text{\textgreek{f}}$ of $\square_{g}\text{\textgreek{f}}=0$ with
frequency support bounded away from $\text{\textgreek{w}}=0$ and
$\text{\textgreek{w}}=\pm\infty$. 
\end{abstract}
\tableofcontents{}

\section{Introduction}

In the field of general relativity, stationary and asymptotically
flat spacetimes $(\mathcal{M},g)$ arise naturally as models of the
asymptotic state of isolated self-gravitating systems. In this context,
questions on the stability properties of such spacetimes as solutions
to the initial value problem for the Einstein equations 
\begin{equation}
Ric_{\text{\textgreek{m}\textgreek{n}}}(g)-\frac{1}{2}R(g)g_{\text{\textgreek{m}\textgreek{n}}}=8\text{\textgreek{p}}T_{\text{\textgreek{m}\textgreek{n}}}\label{eq:EinsteinEquations}
\end{equation}
(where $T_{\text{\textgreek{m}\textgreek{n}}}$ is the stress-energy
tensor associated to the matter fields, with $T_{\text{\textgreek{m}\textgreek{n}}}=0$
in the vacuum case) are of particular importance, being directly related
to the physical relevance of the spacetimes themselves. 

The stability of Minkowski spacetime $(\mathbb{R}^{3+1},\text{\textgreek{h}})$
as a solution to the vacuum Einstein equations was established in
the monumental work of Christodoulou--Klainerman \cite{Christodoulou1993}.
Until today, Minkowski spacetime is the only stationary and asymptotically
flat vacuum spacetime which is known to be non-linearly stable. A
more complicated example of a family of stationary and asymptotically
flat spacetimes expected to be stable are the subextremal Kerr exterior
spacetimes $(\mathcal{M}_{M,a},g_{M,a})$, with mass $M$ and angular
momentum $a$ satisfying $0\le|a|<M$ (for a detailed formulation
of the Kerr stability conjecture, see \cite{Dafermos2016}). While
the non-linear stability of the family $(\mathcal{M}_{M,a},g_{M,a})$
has not been established so far, the \emph{linear }stability of the
Schwarzschild exterior (i.\,e.~$(\mathcal{M}_{M,a},g_{M,a})$ for
$a=0$) was recently obtained by Dafermos--Holzegel--Rodnianski (see
\cite{Dafermos2016}). 

Owing to the fact that the wave equation 
\begin{equation}
\square_{g}\text{\textgreek{f}}=0\label{eq:WaveEquation}
\end{equation}
can be viewed as a simple model of the linearised vacuum Einstein
equations (\ref{eq:EinsteinEquations}) around $(\mathcal{M}_{M,a},g_{M,a})$,
the stability properties of equation (\ref{eq:WaveEquation}) in the
case $0\le|a|<M$ had been extensively studied in the years preceding
\cite{Dafermos2016}, culminating in the proof of polynomial decay
estimates for solutions $\text{\textgreek{f}}$ to (\ref{eq:WaveEquation})
on $(\mathcal{M}_{M,a},g_{M,a})$ in the full subextremal case $0\le|a|<M$
in \cite{DafRodSchlap,Shlap}. For earlier results in the Schwarzschild
case $a=0$ and the very slowly rotating case $|a|\ll M$, see \cite{KayWald,DafRod1,DafRod2,DafRod4,BlueSof1,BlueSterb}
and \cite{DafRod5,DafRod6,DafRod9,TatToh1,AndBlue1} respectively. 

One important aspect of the geometry of $(\mathcal{M}_{M,a},g_{M,a})$
in the case $a\neq0$ is the existence of an \emph{ergoregion} (or
``ergosphere'') $\mathscr{E}$; recall that $\mathscr{E}\subset\mathcal{M}_{M,a}$
is defined as 
\begin{equation}
\mathscr{E}\doteq\overline{\{p\in\mathcal{M}_{M,a}|\, g(T_{p},T_{p})>0\}},
\end{equation}
where $T$ is the stationary Killing vector field on $(\mathcal{M}_{M,a},g_{M,a})$.
The fact that $\mathscr{E}$ is non-empty when $a\neq0$ gives rise
to the phenomenon of \emph{superradiance} for solutions to (\ref{eq:WaveEquation})
on $(\mathcal{M}_{M,a},g_{M,a})$, $a\neq0$: there exist solutions
$\text{\textgreek{f}}$ to (\ref{eq:WaveEquation}) such that their
$T$-energy flux through future null infinity $\mathcal{I}^{+}$ is
\underline{greater} than their $T$-energy flux initially. In general,
superradiance poses a serious difficulty in obtaining stability results
for equation (\ref{eq:WaveEquation}). In the case of $(\mathcal{M}_{M,a},g_{M,a})$,
superradiance does not eventually render equation (\ref{eq:WaveEquation})
unstable, owing, partly, to the presence of the future \emph{event
horizon} $\mathcal{H}^{+}$, allowing for part of the energy of solutions
of (\ref{eq:WaveEquation}) to ``leave'' the black hole exterior.
Notice, however, that superradiance-related mode instabilities do
appear on $(\mathcal{M}_{M,a},g_{M,a})$ for the Klein--Gordon equation
(see \cite{Shlapentokh-Rothman2013}), or even for the wave equation
with a (well-chosen) short-range non-negative potential (see \cite{Moschidisa}). 

Stationary and asymptotically flat spacetimes $(\mathcal{M},g)$ with
a non-empty ergoregion $\mathscr{E}$ but lacking a future event horizon
$\mathcal{H}^{+}$ appear in the literature as models for rapidly
rotating self-gravitating objects, for instance, as models of self-gravitating
dense rotating fluids (see \cite{Butterworth1976}). In \cite{Friedman1978},
Friedman studied the instability properties of equation (\ref{eq:WaveEquation})
on such spacetimes, making the following observation: There exist
smooth solutions $\text{\textgreek{f}}$ to (\ref{eq:WaveEquation})
with negative $T$-energy flux initially, i.\,e. 
\begin{equation}
\int_{\text{\textgreek{S}}}J_{\text{\textgreek{m}}}^{T}(\text{\textgreek{f}})n_{\text{\textgreek{S}}}^{\text{\textgreek{m}}}<0
\end{equation}
on a Cauchy hypersurface $\text{\textgreek{S}}$ of $(\mathcal{M},g)$
(see Section \ref{sec:Notational-conventions} for our notations on
vector field currents), and, in view of the conservation of the $T$-energy
flux, the absence of a future event horizon $\mathcal{H}^{+}$ and
the non-negativity of $J_{\text{\textgreek{m}}}^{T}(\cdot)n_{\text{\textgreek{S}}}^{\text{\textgreek{m}}}$outside
$\mathscr{E}$, any such function $\text{\textgreek{f}}$ satisfies
for all $\text{\textgreek{t}}\ge0$: 
\begin{equation}
\int_{\text{\textgreek{S}}_{\text{\textgreek{t}}}\cap\mathscr{E}}J_{\text{\textgreek{m}}}^{T}(\text{\textgreek{f}})n_{\text{\textgreek{S}}_{\text{\textgreek{t}}}}^{\text{\textgreek{m}}}\le\int_{\text{\textgreek{S}}}J_{\text{\textgreek{m}}}^{T}(\text{\textgreek{f}})n_{\text{\textgreek{S}}}^{\text{\textgreek{m}}}<0
\end{equation}
(where $\text{\textgreek{S}}_{\text{\textgreek{t}}}$ denotes the
image of $\text{\textgreek{S}}$ under the flow of $T$ for time $\text{\textgreek{t}}$).
Therefore, the local energy of $\text{\textgreek{f}}$ can not decay
to $0$ with time. 

Based on the above observation, Friedman provided a heuristic argument
suggesting that, under the additional assumption that the spacetime
$(\mathcal{M},g)$ is real analytic, any such solution $\text{\textgreek{f}}$
satisfies 
\begin{equation}
\limsup_{\text{\textgreek{t}}\rightarrow+\infty}\int_{\text{\textgreek{S}}_{\text{\textgreek{t}}}}J_{\text{\textgreek{m}}}^{N}(\text{\textgreek{f}})n_{\text{\textgreek{S}}_{\text{\textgreek{t}}}}^{\text{\textgreek{m}}}=+\infty
\end{equation}
 for a globally timelike $T$-invariant vector field $N$. In view
of the aforementioned connection between equation (\ref{eq:WaveEquation})
and the Einstein equations (\ref{eq:EinsteinEquations}), Friedman
suggested that such spacetimes can not appear as the final state of
the evolution of a self-gravitating system. See \cite{Friedman1978}
for more details. For a numerical investigation of Friedman's instability,
see \cite{Comins1978,Yoshida1996,Cardoso2008}.

In this paper, we will provide a rigorous proof of Friedman's instability
for equation (\ref{eq:WaveEquation}). Our proof will in fact not
require that $(\mathcal{M},g)$ is real analytic, but we will assume,
instead, a substantially weaker unique continuation condition for
equation (\ref{eq:WaveEquation}) through a subset of the boundary
$\partial\mathscr{E}_{ext}$ of the ``extended'' ergoregion $\mathscr{E}_{ext}$,
where we define $\mathscr{E}_{ext}$ to be equal to the union of the
ergoregion $\mathscr{E}$ with the connected components of $\mathcal{M}\backslash\mathscr{E}$
which intersect neither $\mathcal{H}^{+}$ nor the asymptotically
flat region of $\mathcal{M}$.%
\footnote{Notice that $\partial\mathscr{E}_{ext}\subset\partial\mathscr{E}$.%
} Note that, in the case when $\mathcal{M}\backslash\mathscr{E}$ is
connected, $\mathscr{E}_{ext}$ coincides with $\mathscr{E}$. In
particular, we will establish the following result:
\begin{thm}
\label{thm:TheoremIntroduction}Let $(\mathcal{M}^{d+1},g)$, $d\ge2$,
be a smooth, globally hyperbolic, stationary and asymptotically flat
spacetime with a non-empty ergoregion $\mathscr{E}$ and a future
event horizon $\mathcal{H}^{+}$ which is either empty or satisfies
$\mathscr{E}\cap\mathcal{H}^{+}=\emptyset$. Assume, in addition,
that the following unique continuation condition through the boundary
$\partial\mathscr{E}_{ext}$ of the ``extended'' ergoregion $\mathscr{E}_{ext}$
holds:

\medskip{}

\noindent \textbf{\emph{Unique continuation condition:}}\textbf{ }There
exists a point $p\in\partial\mathscr{E}_{ext}$ and an open neighborhood
$\mathcal{U}$ of $p$ in $\mathcal{M}$ such that, for any solution
$\text{\textgreek{f}}$ to equation (\ref{eq:WaveEquation}) on $\mathcal{M}$
with $\text{\textgreek{f}}\equiv0$ on $\mathcal{M}\backslash\mathscr{E}_{ext}$,
we have $\text{\textgreek{f}}=0$ also on $\mathscr{E}_{ext}\cap\mathcal{U}$.

\medskip{}

Then, there exists a smooth solution $\text{\textgreek{f}}$ to (\ref{eq:WaveEquation})
with compactly supported initial data on a Cauchy hypersurface $\text{\textgreek{S}}$
of $(\mathcal{M},g)$, such that 
\begin{equation}
\limsup_{\text{\textgreek{t}}\rightarrow+\infty}\int_{\text{\textgreek{S}}_{\text{\textgreek{t}}}}J_{\text{\textgreek{m}}}^{N}(\text{\textgreek{f}})n^{\text{\textgreek{m}}}=+\infty,
\end{equation}
where $T$ is the stationary Killing field of $(\mathcal{M},g)$,
$N$ is a globally timelike and $T$-invariant vector field on $\mathcal{M}$,
coinciding with $T$ in the asymptotically flat region of $\mathcal{M}$,
and $\text{\textgreek{S}}_{\text{\textgreek{t}}}$ is the image of
$\text{\textgreek{S}}$ under the flow of $T$ for time $\text{\textgreek{t}}$. \end{thm}
\begin{rem*}
Note that the assumption $\mathscr{E}\cap\mathcal{H}^{+}=\emptyset$
excludes the Kerr exterior family with angular momentum $a\neq0$.
\end{rem*}
For a more detailed statement of Theorem \ref{thm:TheoremIntroduction}
and the assumptions on the spacetimes under consideration, see Section
\ref{sec:StatementAssumptionsResults}. In the detailed statement
of Theorem \ref{thm:TheoremIntroduction}, we will introduce an additional
restricion on the class of spacetimes $(\mathcal{M},g)$ under consideration,
namely the condition that every connected component of $\mathcal{M}\backslash\mathscr{E}$
intersecting $\mathcal{H}^{+}$ also intersects the asymptotically
flat region of $(\mathcal{M},g)$. However, our proof of Theorem \ref{thm:TheoremIntroduction}
can be adapted to the case when this condition does not hold. For
a comparison between the heuristics of Friedman in \cite{Friedman1978}
and the results of this paper, see Section \ref{sec:DiscussionFriedmanHeuristics}.

We should remark that the unique continuation condition through an
open subset of $\partial\mathscr{E}_{ext}$, appearing in the statement
of Theorem \ref{thm:TheoremIntroduction}, is always satisfied in
the case when $(\mathcal{M},g)$ possesses an axisymmetric Killing
field $\text{\textgreek{F}}$ such that the span of $T,\text{\textgreek{F}}$
on $\partial\mathscr{E}_{ext}$ contains a timelike direction, or
in the case when the spacetime $(\mathcal{M},g)$ is real analytic
in an open subset $\mathcal{U}\subset\mathcal{M}$ such that $\mathcal{U}\cap\partial\mathscr{E}_{ext}\neq\emptyset$;
see the discussion in Section \ref{sub:Discussion On the unique continuation assumption}.
It would be natural to expect that this condition can be completely
removed from the statement of Theorem \ref{thm:TheoremIntroduction},
but we have not succeeded so far in doing so.

The proof of Theorem \ref{thm:TheoremIntroduction}, presented in
Section \ref{sec:Proof-of-Theorem}, proceeds by contradiction. In
particular, assuming that every smooth solution $\text{\textgreek{f}}$
of equation (\ref{eq:WaveEquation}) on $(\mathcal{M},g)$ with compactly
supported initial data satisfies 
\begin{equation}
\limsup_{\text{\textgreek{t}}\rightarrow+\infty}\int_{\text{\textgreek{S}}_{\text{\textgreek{t}}}}J_{\text{\textgreek{m}}}^{N}(\text{\textgreek{f}})n^{\text{\textgreek{m}}}<+\infty,\label{eq:EnergyBoundContradictionIntroduction}
\end{equation}
it is shown that $\text{\textgreek{f}}$ decays in time on $\mathcal{M}\backslash\mathscr{E}$.
This fact is then shown to lead to a contradiction after a suitable
choice of the initial data for $\text{\textgreek{f}}$, combined with
the unique continuation assumption of Theorem \ref{thm:TheoremIntroduction}.
See Section \ref{sec:Proof-of-Theorem} for more details. The decay
of $\text{\textgreek{f}}$ on $\mathcal{M}\backslash\mathscr{E}$
is established through some suitable Carleman-type estimates, derived
in Section \ref{sec:Carleman}. These estimates could have been obtained
by methods similar to the ones implemented in \cite{Moschidisb},
but we chose instead to provide an alternative proof, based entirely
on the method of first order multipliers for equation (\ref{eq:WaveEquation}).
For more details on this, see Section \ref{sub:CarlemanEstimatesIntroduction}.

The instability mechanism proposed by Friedman is of interest not
only in general relativity, but also in all areas of mathematical
physics where stationary and asymptotically flat Lorentzian manifolds
$(\mathcal{M},g)$, and the associated wave equation (\ref{eq:WaveEquation}),
arise. For instance, in the field of fluid mechanics, the steady flow
of a (locally) irrotational, inviscid and barotropic fluid on an open
subset $\mathcal{V}$ of $\mathbb{R}^{3}$ gives rise to a stationary
Lorentzian metric $g$ on $\mathcal{M}=\mathbb{R}\times\mathcal{V}$,
the so called \emph{acoustical} metric, and the wave equation (\ref{eq:WaveEquation})
associated to $g$ governs the evolution of small perturbations of
the flow. In \cite{Oliveira2014}, the authors investigate numerically
the Friedman instability for the acoustic wave equation on the hydrodynamic
vortex $(\mathbb{R}\times\mathcal{V}_{hyd,\text{\textgreek{d}}},g_{hyd})$,
where $\mathcal{V}_{hyd,\text{\textgreek{d}}}=\mathbb{R}^{3}\backslash\{\bar{r}\le\text{\textgreek{d}}\}$
for some $\text{\textgreek{d}}\ll1$ (in the cylindrical $(\bar{r},\text{\textgreek{j}},z)$
coordinate system) and 
\begin{equation}
g_{hyd}=-\big(1-\frac{C^{2}}{\bar{r}^{2}}\big)dt^{2}+d\bar{r}^{2}-2Cdtd\text{\textgreek{j}}+\bar{r}^{2}d\text{\textgreek{j}}^{2}+dz^{2},\label{eq:HydrodynamicVortex}
\end{equation}
with suitable boundary conditions imposed for (\ref{eq:WaveEquation})
at $\bar{r}=\text{\textgreek{d}}$. Note that the quotient of $(\mathbb{R}\times\mathcal{V}_{hyd},g_{hyd})$
by the group of translations in the $z$ direction is asymptotically
flat, possesses a non-empty ergoregion $\mathscr{E}=\{\text{\textgreek{d}}<\bar{r}\le C\}$
(corresponding to the region where the fluid velocity exceeds the
speed of sound) and has no event horizon.

As a straightforward application of Theorem \ref{thm:TheoremIntroduction},
we will establish a Friedman-type instability for the acoustical wave
equation on the hydronamic vortex:
\begin{cor}
\label{cor:VortexIntroduction}For any $\text{\textgreek{d}}<1$,
there exist smooth and $z$-invariant solutions $\text{\textgreek{f}}_{D},\text{\textgreek{f}}_{N}$
to the acoustical wave equation (\ref{eq:WaveEquation}) on $(\mathbb{R}\times\mathcal{V}_{hyd,\text{\textgreek{d}}},g_{hyd})$,
satisfying Dirichlet and Neumann boundary conditions, respectively,
on $\{\bar{r}=\text{\textgreek{d}}\}$, with smooth initial data at
time $t=0$ which are compactly supported when restricted on $\{z=0\}$,
such that (in the $(t,\bar{r},\text{\textgreek{j}},z)$ coordinate
chart on $\mathbb{R}\times\mathcal{V}_{hyd,\text{\textgreek{d}}}$):
\begin{equation}
\limsup_{\text{\textgreek{t}}\rightarrow+\infty}\int_{\{t=\text{\textgreek{t}}\}\cap\{z=0\}\cap\{\bar{r}\ge\text{\textgreek{d}}\}}\big(|\partial_{t}\text{\textgreek{f}}_{D}|^{2}+|\nabla_{\mathbb{R}^{3}}\text{\textgreek{f}}_{D}|^{2}\big)\,\bar{r}d\bar{r}d\text{\textgreek{j}}=+\infty
\end{equation}
and 
\begin{equation}
\limsup_{\text{\textgreek{t}}\rightarrow+\infty}\int_{\{t=\text{\textgreek{t}}\}\cap\{z=0\}\cap\{\bar{r}\ge\text{\textgreek{d}}\}}\big(|\partial_{t}\text{\textgreek{f}}_{N}|^{2}+|\nabla_{\mathbb{R}^{3}}\text{\textgreek{f}}_{N}|^{2}\big)\,\bar{r}d\bar{r}d\text{\textgreek{j}}=+\infty.
\end{equation}

\end{cor}
For a more detailed statement of Corollary \ref{cor:VortexIntroduction},
see Section \ref{sub:The-main-theorem}.

\section{\label{sec:StatementAssumptionsResults}Statement of the main results}

In this section, we will outline in detail the assumptions on the
spacetimes $(\mathcal{M},g)$ under consideration, and we will state
the main results of this paper.

\subsection{\label{sub:Assumptions}Assumptions on the spacetimes under consideration}

Let $(\mathcal{M}^{d+1},g)$, $d\ge2$, be a smooth, globally hyperbolic
Lorentzian manifold with piecewise smooth boundary $\partial\mathcal{M}$
(allowed to be empty). Before stating our main results, we will need
to introduce a number of assumptions on the structure of $(\mathcal{M},g)$.
In Section \ref{sub:ExampleSpacetime}, we will present some explicit
examples of spacetims $(\mathcal{M},g)$ satisfying all the assumptions
that will be introduced in this section.

\subsubsection{\emph{Assumpti}on \emph{G1} \label{Assumption 1} \emph{(Asymptotic
flatness and stationarity).}}

We will assume that $(\mathcal{M},g)$ satisfies the following conditions:
\begin{itemize}
\item There exists a Killing field $T$ on $(\mathcal{M},g)$ with complete
orbits which is tangential to $\partial\mathcal{M}$, as well as a
smooth Cauchy hypersurface $\tilde{\text{\textgreek{S}}}\subset\mathcal{M}$,
such that $T|_{\tilde{\text{\textgreek{S}}}}$ is everywhere transversal
to $\tilde{\text{\textgreek{S}}}\backslash\partial\mathcal{M}$ and
timelike outside a compact subset of $\tilde{\text{\textgreek{S}}}$.
\item The triad $(\tilde{\text{\textgreek{S}}},g_{\tilde{\text{\textgreek{S}}}},k_{\tilde{\text{\textgreek{S}}}})$,
where $g_{\tilde{\text{\textgreek{S}}}}$ is the induced (Riemannian)
metric on $\tilde{\text{\textgreek{S}}}$ and $k_{\tilde{\text{\textgreek{S}}}}$
its second fundamental form, defines an asymptotically flat Riemannian
manifold (possibly with boundary $\tilde{\text{\textgreek{S}}}\cap\partial\mathcal{M}$
), with a finite number of asymptotically flat ends (possibly more
than one); see also the definition in Section 2.1.1 of \cite{Moschidisb}.
Let $\mathcal{I}_{as}$ be the asymptotically flat region of $\mathcal{M}$
(see \cite{Moschidisb} for the relevant definition). Expressed in
a polar coordinate chart of the form $(t,r,\text{\textgreek{sv}})$
in each conected component of $\mathcal{I}_{as}$, $g$ has the following
form: 
\begin{equation}
g=-\Big(1+O_{4}(r^{-1})\Big)dt^{2}+\Big(1+O_{4}(r^{-1})\Big)dr^{2}+r^{2}\cdot\Big(g_{\mathbb{S}^{d-1}}+O_{4}^{\mathbb{S}^{d-1}}(r^{-1})\Big)+O_{4}(1)dtd\text{\textgreek{sv}},\label{eq:metric}
\end{equation}
where $O_{4}^{\mathbb{S}^{d-1}}(\text{\textgreek{r}}^{-1})$ is a
symmetric $(0,2)$-tensor field on the coordinate sphere $\{r=\text{\textgreek{r}}\}\simeq\mathbb{S}^{d-1}$
with $O_{4}(\text{\textgreek{r}}^{-1})$ asymptotics as $\text{\textgreek{r}}\rightarrow+\infty$.
See Section \ref{sub:BigONotation} for the $O_{k}(\cdot),O_{k}^{\mathbb{S}^{d-1}}(\cdot)$
notation and Section \ref{sub:DerivativesOntheSphere} for the $\text{\textgreek{sv}}$
notation on the angular variables of a polar coordinate chart.
\item Let $\mathcal{H}=\partial\big(J^{+}(\mathcal{I}_{as})\cap J^{-}(\mathcal{I}_{as})\big)$
be the horizon of $\mathcal{M}$, split as $\mathcal{H}=\mathcal{H}^{+}\cup\mathcal{H}^{-}$,
with\emph{ }$\mathcal{H}^{+}=\overline{J^{+}(\mathcal{I}_{as})}\cap\partial J^{-}(\mathcal{I}_{as})$
and $\mathcal{H}^{-}=\overline{J^{-}(\mathcal{I}_{as})}\cap\partial J^{+}(\mathcal{I}_{as})$.
Then $\mathcal{H}$ coincides with $\partial\mathcal{M}$, and $\mathcal{H}^{+}$
and $\mathcal{H}^{-}$ are smooth null hypersurfaces with smooth boundary
$\mathcal{H}^{+}\cap\mathcal{H}^{-}$, with $T\neq0$ on $\mathcal{H}\backslash\mathcal{H}^{-}$
(the case $\mathcal{H}^{+}=\emptyset$ or $\mathcal{H}^{-}=\emptyset$
is also trivially included in this condition).
\end{itemize}
See Assumption $1$ in Section 2.1.1 of \cite{Moschidisb} for a detailed
statement of these conditions and their related geometric constructions,
as well as their implications on the geometry of $\mathcal{M}$. Notice
that the domain of outer communications of the asymptotically flat
region $\mathcal{I}_{as}$ of $\mathcal{M}$ is the whole of $\mathcal{M}\backslash\mathcal{H}$.
In view of the remarks in Section 2.1.1 of \cite{Moschidisb}, $\mathcal{H}^{+}$
and $\mathcal{H}^{-}$ are invariant under the flow of the stationary
Killing field $T$.

\begin{figure}[t] 
\centering 
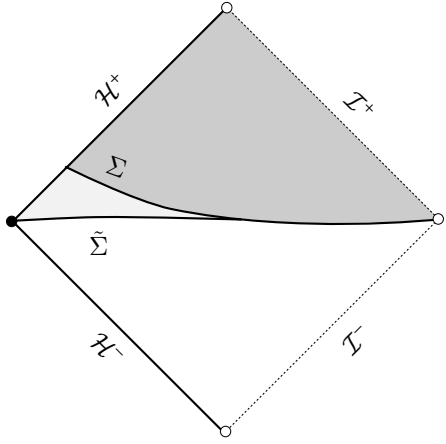 
\caption{The subextremal Kerr exterior spacetime $(\mathcal{M}_{M,a},g_{M,a})$ satisfies Assumptions G1 and G2, but not Assumption G3. In the case of $(\mathcal{M}_{M,a},g_{M,a})$, the intersection of the hypersurfaces $\tilde{\text{\textgreek{S}}}, \text{\textgreek{S}}$, defined in Assumption G1,  with the $1+1$ dimensional slice $\{\theta=\pi/2,\phi=0\}\subset\mathcal{M}_{M,a}$ are schematically as depicted above. }
\end{figure}

Let $\text{\textgreek{S}}$ be a spacelike hypersurface intersecting
$\mathcal{H}^{+}$ transversally (if $\mathcal{H}^{+}\neq\emptyset$)
and satisfying $\text{\textgreek{S}}\cap\mathcal{H}^{-}=\emptyset$,
such that $\text{\textgreek{S}}$ coincides with $\tilde{\text{\textgreek{S}}}$
outside a small neighborhood of $\mathcal{H}^{+}$. In view of the
remarks in Section 2.1.1 of \cite{Moschidisb}, our assumption on
the structure of $(\mathcal{M},g)$ implies that $\text{\textgreek{S}}\cap\mathcal{H}^{+}$
is compact. Notice that, in case $\mathcal{H}^{+}\neq\emptyset$,
$\text{\textgreek{S}}$ will not be a Cauchy hypersurface of $\mathcal{M}$. 

We will also fix a smooth spacelike hyperboloidal hypersurface $\mathcal{S}\subset\mathcal{M}$
terminating at future null infinity $\mathcal{I}^{+}$ as in Section
2.1.1 of \cite{Moschidisb}, such that $\mathcal{S}|_{\{r\le R_{1}\}}\equiv\text{\textgreek{S}}|_{\{r\le R_{1}\}}$
for some fixed constant $R_{1}\gg1$. 
\begin{rem*}
In view of the remarks of Section 2.1.1 of \cite{Moschidisb}, the
causal future sets $J^{+}(\text{\textgreek{S}}),J^{+}(\mathcal{S})$
of $\text{\textgreek{S}},\mathcal{S}$, respectively, in $\mathcal{M}$
coincide with the future domains of dependence $\mathcal{D}^{+}(\text{\textgreek{S}}),\mathcal{D}^{+}(\mathcal{S})$
of $\text{\textgreek{S}},\mathcal{S}$. Furthermore, the images of
$\text{\textgreek{S}},\mathcal{S}$ under the flow of $T$ covers
the whole of $\mathcal{M}\backslash\mathcal{H}^{-}$. 
\end{rem*}
As in Section 2.1.1 of \cite{Moschidisb}, we will extend the polar
radius coordinate function $r:\mathcal{I}_{as}\rightarrow(0,+\infty)$
as a non-negative, smooth and $T$-invariant function on the whole
of $\mathcal{M}\backslash\mathcal{H}^{-}$, such that $r>0$ on $\mathcal{M}\backslash\mathcal{H}$
and $r|_{\mathcal{H}^{+}\backslash\mathcal{H}^{-}}=0$, $dr|_{\mathcal{H}^{+}\backslash\mathcal{H}^{-}}\neq0$.
We will also define the function $t:\mathcal{M}\backslash\mathcal{H}^{-}\rightarrow\mathbb{R}$
by the relations 
\begin{equation}
t|_{\text{\textgreek{S}}}=0\mbox{ and }T(t)=1,
\end{equation}
as well as the function $\bar{t}:\mathcal{M}\backslash\mathcal{H}^{-}\rightarrow\mathbb{R}$
by 
\begin{equation}
\bar{t}|_{\mathcal{S}}=0\mbox{ and }T(\bar{t})=1.
\end{equation}
Note that $t=\bar{t}$ on $\{r\le R_{1}\}$. 

We will introduce the reference Riemannian metric 
\begin{equation}
g_{ref}\doteq dt^{2}+g_{\text{\textgreek{S}}}\label{eq:ReferenceRiemannianMetric}
\end{equation}
 on $\mathcal{M}\backslash\mathcal{H}^{-}\simeq\mathbb{R}\times\text{\textgreek{S}}$.
We will denote the natural extension of $g_{ref}$ on $\oplus_{l_{1},l_{2}\in\mathbb{N}}\big(\otimes^{l_{1}}T(\mathcal{M}\backslash\mathcal{H}^{-})\otimes^{l_{2}}T^{*}(\mathcal{M}\backslash\mathcal{H}^{-})\big)$
also as $g_{ref}$.

\subsubsection{\emph{Assumption G2}\label{Assumption 2} \emph{(Killing horizon
with positive surface gravity)}.}

In the case $\mathcal{H}^{+}\neq\emptyset$, we will assume that the
Killing field $T$, when restricted to $\mathcal{H}^{+}\backslash\mathcal{H}^{-}$,
is parallel to the null generators of $\mathcal{H}^{+}\backslash\mathcal{H}^{-}$.
Furthermore, we will assume that there exists a $T$-invariant strictly
timelike vector field $N$ on $J^{+}(\text{\textgreek{S}})$, which,
when restricted on $J^{+}(\text{\textgreek{S}})\cap\mathcal{H}$,
satisfies 
\begin{equation}
K^{N}(\text{\textgreek{y}})\ge cJ_{\text{\textgreek{m}}}^{N}(\text{\textgreek{y}})N^{\text{\textgreek{m}}}\label{eq:PositiveKN}
\end{equation}
 for some $c>0$ and any $\text{\textgreek{y}}\in C^{1}(\mathcal{M})$
(see Section \ref{sec:Notational-conventions} for the notation on
vector field currents). We will extend $N$ on the whole of $\mathcal{M}\backslash\mathcal{H}^{-}$
by the condition $[T,N]=0$.

We will call the vector field $N$ the \emph{red shift} vector field.
The reason for this name is that a vector field of that form was shown
to exist for a general class of Killing horizons with positive surface
gravity by Dafermos and Rodnianski in \cite{DafRod2}. However, here
we will just assume the existence of such a vector field without specifying
the geometric origin of it. 

Note that we can modify the vector field $N$ away from the horizon
$\mathcal{H}$, so that in the asymptotically flat region $\{r\gg1\}$
(i.\,e.~$\mathcal{I}_{as}$) it coincides with $T$, and still retain
the bound (\ref{eq:PositiveKN}) on $J^{+}(\text{\textgreek{S}})\cap\mathcal{H}$.%
\footnote{The convexity of the cone of the future timelike vectors over each
point of $\mathcal{M}$ is used in this argument.%
} We will hence assume without loss of generality that $N$ has been
chosen so that $N\equiv T$ in the region $\{r\gg1\}$.

Due to the smoothness of $N$, there exists an $r_{0}>0$, such that
(\ref{eq:PositiveKN}) also holds (possibly with a smaller constant
$c$ on the right hand side) in a neighborhood of $\mathcal{H}^{+}\backslash(\mathcal{H}^{+}\cap\mathcal{H}^{-})$
in $\mathcal{M}$ of the form $\{r\le r_{0}\}$. For $r\gg1$, since
$N\equiv T$ there, we have $K^{T}(\text{\textgreek{y}})\equiv0$.
Hence, due to the $T-$invariance of $N$ and the compactness of the
sets of the form $\{r\le R\}\cap\text{\textgreek{S}}$, there exists
a (possibly large) constant $C>0$ such that 
\begin{equation}
|K^{N}(\text{\textgreek{y}})|\le C\cdot J_{\text{\textgreek{m}}}^{N}(\text{\textgreek{y}})N^{\text{\textgreek{m}}}\label{eq:UpperBoundKN}
\end{equation}
 everywhere on $\mathcal{M}$ for any $\text{\textgreek{y}}\in C^{\infty}(\mathcal{M})$.

Without loss of generality, we will also assume that $r_{0}$ is sufficiently
small so that 
\begin{equation}
dr\neq0\label{eq:NonVanishingNormalDerivativeR}
\end{equation}
 on $\{r\le3r_{0}\}$. This is possible, since $dr|_{\mathcal{H}^{+}\backslash\mathcal{H}^{+}\cap\mathcal{H}^{-}}\neq0$
and $\text{\textgreek{S}}\cap\mathcal{H}^{+}$ is compact.

In the case $\mathcal{H}^{+}=\emptyset$, we will fix $N$ to be an
arbitrary $T$-invariant timelike vector field on $\mathcal{M}\backslash\mathcal{H}^{-}$,
such that $N\equiv T$ for $r\gg1$, and we will set $r_{0}=\frac{1}{4}\inf_{\text{\textgreek{S}}}r$
(which is possible since $r>0$ on $\text{\textgreek{S}}$ when $\mathcal{H}^{+}=\emptyset$),
so that $\{r\le3r_{0}\}=\emptyset$. In this case, (\ref{eq:PositiveKN}),
(\ref{eq:UpperBoundKN}) and (\ref{eq:NonVanishingNormalDerivativeR})
are trivially satisfied.

\subsubsection{\emph{Assumption G3\label{Assumption 3} (Non-empty ergoregion avoiding
the future event horizon).}}

We will assume that the ergoregion of $(\mathcal{M},g)$ is non-empty,
i.\,e. 
\begin{equation}
\mathscr{E}\doteq\overline{\{g(T,T)>0\}}\neq\emptyset,\label{eq:Ergoregion}
\end{equation}
and furthermore
\begin{equation}
\mathscr{E}\cap\mathcal{H}^{+}=\emptyset.\label{eq:NonIntersectionErgoregionHorizon}
\end{equation}
Notice that the condition (\ref{eq:NonIntersectionErgoregionHorizon})
is trivially satisfied when $\mathcal{H}^{+}=\emptyset$. Note also
that the subextremal Kerr exterior family with $a\neq0$ has a non-empty
ergoregion, but does not satisfy (\ref{eq:NonIntersectionErgoregionHorizon}).

In the case when $\mathcal{H}^{+}\neq\emptyset$, we will also assume
that every connected component of $\mathcal{M}\backslash\mathscr{E}$
that intersects $\mathcal{H}^{+}$ also intersects the asymptotically
flat region $\mathcal{I}_{as}$ of $(\mathcal{M},g)$.%
\footnote{Note that $\mathcal{I}_{as}$ and $\mathcal{H}^{+}$ might have several
components.%
}
\begin{rem*}
The assumption that every component of $\mathcal{M}\backslash\mathscr{E}$
intersecting $\mathcal{H}^{+}$ also intersects $\mathcal{I}_{as}$
is not necessary for the results of this paper, which can also be
established without this condition. The reason for adopting this assumption
is that it leads to considerable simplifications in the proof of the
Carleman-type estimates in Section \ref{sec:Carleman}.
\end{rem*}
We will assume that $T$ is strictly timelike on the complement of
$\mathcal{H}\cup\mathscr{E}$, i.\,e.: 
\begin{equation}
g(T,T)<0\mbox{ on }\mathcal{M}\backslash(\mathscr{E}\cup\mathcal{H}).\label{eq:NondegenrateTOutsideErgoregion}
\end{equation}
Furthermore, we will assume that the boundary $\partial\mathscr{E}$
of $\mathscr{E}$ is a smooth hypersurface of $\mathcal{M}$. 

The complement $\mathcal{M}\backslash\mathscr{E}$ of $\mathscr{E}$
might consist of more than one components. In view of our assumption
that every component of $\mathcal{M}\backslash\mathscr{E}$ intersecting
$\mathcal{H}^{+}$ also intersects $\mathcal{I}_{as}$, the connected
components of $\mathcal{M}\backslash\mathscr{E}$ fall into two disjoint
categories: The ones that intersect the asymptotically flat region
$\mathcal{I}_{as}$ and the future event horizon $\mathcal{H}^{+}$,
and the ones that intersect neither $\mathcal{I}_{as}$ nor $\mathcal{H}^{+}$.
Let us call the union of the components of $\mathcal{M}\backslash\mathscr{E}$
falling into the last category the \emph{enclosed} region of $\mathcal{M}$,
and denote it by $\mathcal{M}_{enc}$.%
\footnote{The reason for calling $\mathcal{M}_{enc}$ enclosed is because $\partial\mathcal{M}_{enc}\subset\partial\mathscr{E}$,
i.\,e.~$\mathcal{M}_{enc}$ is enclosed by the ergoregion.%
} We will also introduce the notion of the \emph{extended }ergoregion
of $\mathcal{M}$ defined by 
\begin{equation}
\mathscr{E}_{ext}\doteq\mathscr{E}\cup\mathcal{M}_{enc}.\label{eq:ExtendedErgoregion}
\end{equation}

Note that, since $\mathscr{E}_{ext}\cap\mathcal{H}^{+}=\emptyset$,
we have $r>0$ on $\mathscr{E}_{ext}$. Thus, in view of the $T$-invariance
of $\mathscr{E}_{ext}$, we can assume without loss of generality
that $r_{0}$ has been fixed sufficiently small so that $\{r\le r_{0}\}\cap\mathscr{E}_{ext}=\emptyset$.
Note also that $\partial\mathscr{E}_{ext}\subseteq\partial\mathscr{E}$.

\subsubsection{\emph{Assumption A1\label{Assumption 4} (Unique continuation around
$p\in\partial\mathscr{E}_{ext}$).}}

We will assume that there exists a point $p$ on the boundary $\partial\mathscr{E}_{ext}$
of $\mathscr{E}_{ext}$ and an open neighborhood $\mathcal{U}$ of
$p$ in $\mathcal{M}$, such that for any $\text{\textgreek{y}}\in H_{loc}^{1}\big(\mathcal{M}\backslash\mathcal{H}^{-}\big)$
solving the wave equation (\ref{eq:WaveEquation}) and satisfying
$\text{\textgreek{y}}\equiv0$ on $\mathcal{M}\backslash\mathscr{E}_{ext}$,
we also have $\text{\textgreek{y}}\equiv0$ on $\mathcal{U}$. Since
$T$ is a Killing field of $\mathcal{M}$, the same result also holds
on any $T$-translate of $\mathcal{U}$, and, for this reason, we
will assume without loss of generality that $\mathcal{U}$ is $T$-invariant.
Furthermore, since $\partial\mathscr{E}_{ext}\subset\partial\mathscr{E}$,
we will assume without loss of generality that $\mathcal{U}$ is small
enough so that $\mathcal{U}\cap\mathscr{E}_{ext}\subset\mathscr{E}$.
\begin{rem*}
Assumption \hyperref[Assumption 4]{A1} is satisfied in the case when
$\mathcal{M}$ is axisymmetric with axisymmetric Killing field $\text{\textgreek{F}}$,
such that $[\text{\textgreek{F}},T]=0$ and the span of $\{\text{\textgreek{F}},T\}$
is timelike, or in the case when there exists a point $p\in\partial\mathscr{E}$
such that $g$ is real analytic on an open neighborhhod of $p$ in
$\mathcal{M}$. See Section \ref{sub:Discussion On the unique continuation assumption}. 
\end{rem*}

\subsection{\label{sub:The-main-theorem}The main results}

The main result of this paper is the following:
\begin{thm}
\label{thm:FriedmanInstability}Let $(\mathcal{M}^{d+1},g)$, $d\ge2$,
be a globally hyperbolic Lorentzian manifold satisfying Assumptions
\hyperref[Assumption 1]{G1}, \hyperref[Assumption 2]{G2}, \hyperref[Assumption 3]{G3}
and \hyperref[Assumption 4]{A1}, and let the vector field $T$, $N$
and the spacelike hypersurface $\text{\textgreek{S}}$ be as described
in Assumptions \hyperref[Assumption 1]{G1}-\hyperref[Assumption 2]{G2}.
Then, there exists a smooth function $\text{\textgreek{f}}:J^{+}(\text{\textgreek{S}})\rightarrow\mathbb{C}$
solving the wave equation (\ref{eq:WaveEquation}) on $J^{+}(\text{\textgreek{S}})$
with compactly supported initial data on $\text{\textgreek{S}}$,
such that 
\begin{equation}
\limsup_{\text{\textgreek{t}}\rightarrow+\infty}\int_{\text{\textgreek{S}}_{\text{\textgreek{t}}}}J_{\text{\textgreek{m}}}^{N}(\text{\textgreek{f}})n^{\text{\textgreek{m}}}=+\infty.\label{eq:EnergyGoingToInfinity}
\end{equation}

\end{thm}
The proof of Theorem \ref{thm:FriedmanInstability} will be presented
in Section \ref{sec:Proof-of-Theorem}.
\begin{rem*}
The proof of Theorem \ref{thm:FriedmanInstability} immediately generalises
to the case when the boundary of the spacetime $(\mathcal{M},g)$
has a smooth, timelike and $T$-invariant component $\partial_{tim}\mathcal{M}$,
such that $\text{\textgreek{S}}\cap\partial_{tim}\mathcal{M}$ is
compact and $\partial_{tim}\mathcal{M}\cap\mathcal{H}=\emptyset$,
assuming that Dirichlet or Neumann boundary conditions are imposed
for equation (\ref{eq:WaveEquation}) on $\partial_{tim}\mathcal{M}$.
In this case, we have to assume that the double $(\widetilde{\mathcal{M}},\tilde{g})$
of $(\mathcal{M},g)$ across $\partial_{tim}\mathcal{M}$ is a globally
hyperbolic spacetime satisfying Assumptions \hyperref[Assumption 1]{G1},
\hyperref[Assumption 2]{G2}, \hyperref[Assumption 3]{G3} and \hyperref[Assumption 4]{A1}
(see Section \ref{sub:ProofWithBoundaryConditions} for the relevant
constructions).

Let us also note that we can readily replace the qualitative instability
statement (\ref{eq:EnergyGoingToInfinity}) with the following quantitative
statement: For any $C>0$, there exists a solution $\text{\textgreek{f}}$
to equation (\ref{eq:WaveEquation}) as in the statement of Theorem
\ref{thm:FriedmanInstability}, such that 
\begin{equation}
\limsup_{\text{\textgreek{t}}\rightarrow+\infty}\Big(\big(\log(2+\text{\textgreek{t}})\big)^{-C}\int_{\text{\textgreek{S}}_{\text{\textgreek{t}}}}J_{\text{\textgreek{m}}}^{N}(\text{\textgreek{f}})n^{\text{\textgreek{m}}}\Big)=+\infty.\label{eq:LogInstability}
\end{equation}
See the remark at the beginning of Section \ref{sec:Proof-of-Theorem}.
However, we do not expect the logarithmic rate of growth in (\ref{eq:LogInstability})
to be sharp.
\end{rem*}
As a straightforward application of Theorem \ref{thm:FriedmanInstability},
we will obtain the following instability estimate for solutions to
the acoustical wave equation on the hydronamic vortex $(\mathbb{R}\times\mathcal{V}_{hyd,\text{\textgreek{d}}},g_{hyd})$,
where $\mathcal{V}_{hyd,\text{\textgreek{d}}}\subset\mathbb{R}^{3}$
is the set $\{\bar{r}\ge\text{\textgreek{d}}\}$ (in the cylindrical
$(\bar{r},\text{\textgreek{j}},z)$ coordinate system on $\mathbb{R}^{3}$)
and $g_{hyd}$ is given by the expression (\ref{eq:HydrodynamicVortex}):
\begin{cor}
\label{cor:VortexMain}For any $\text{\textgreek{d}}<1$, there exist
 smooth and $z$-invariant solutions $\text{\textgreek{f}}_{D},\text{\textgreek{f}}_{N}$
to (\ref{eq:WaveEquation}) on $(\mathbb{R}\times\mathcal{V}_{hyd,\text{\textgreek{d}}},g_{hyd})$,
satisfying the boundary conditions 
\begin{equation}
\text{\textgreek{f}}_{D}|_{\{\bar{r}=\text{\textgreek{d}}\}}=0
\end{equation}
and 
\begin{equation}
\partial_{\bar{r}}\text{\textgreek{f}}_{N}|_{\{\bar{r}=\text{\textgreek{d}}\}}=0
\end{equation}
 and having smooth initial data at time $t=0$ which are compactly
supported when restricted on $\{z=0\}$, such that (in the $(t,\bar{r},\text{\textgreek{j}},z)$
coordinate chart on $\mathbb{R}\times\mathcal{V}_{hyd,\text{\textgreek{d}}}$):
\begin{equation}
\limsup_{\text{\textgreek{t}}\rightarrow+\infty}\int_{\{t=\text{\textgreek{t}}\}\cap\{z=0\}\cap\{\bar{r}\ge\text{\textgreek{d}}\}}\big(|\partial_{t}\text{\textgreek{f}}_{D}|^{2}+|\nabla_{\mathbb{R}^{3}}\text{\textgreek{f}}_{D}|^{2}\big)\,\bar{r}d\bar{r}d\text{\textgreek{j}}=+\infty\label{eq:VortexFriedmanDirichlet}
\end{equation}
and 
\begin{equation}
\limsup_{\text{\textgreek{t}}\rightarrow+\infty}\int_{\{t=\text{\textgreek{t}}\}\cap\{z=0\}\cap\{\bar{r}\ge\text{\textgreek{d}}\}}\big(|\partial_{t}\text{\textgreek{f}}_{N}|^{2}+|\nabla_{\mathbb{R}^{3}}\text{\textgreek{f}}_{N}|^{2}\big)\,\bar{r}d\bar{r}d\text{\textgreek{j}}=+\infty.\label{eq:VortexFriedmanNeumann}
\end{equation}

\end{cor}
For the proof of Corollary \ref{cor:VortexMain}, see Section \ref{sec:Proof-of-Corollary}.

\subsection{\label{sub:Discussion On the unique continuation assumption}Discussion
on Assumption \hyperref[Assumption 4]{A1}}

There exists a class of natural geometric conditions, such that spacetimes
$(\mathcal{M},g)$ satisfying these conditions (in addition to Assumptions
\hyperref[Assumption 1]{G1}--\hyperref[Assumption 3]{G3}) automatically
satisfy Assumption \hyperref[Assumption 4]{A1}. Examples of such
conditions are the following:

\begin{itemize}

\item{ Assumption \hyperref[Assumption 4]{A1} is always satisfied
on spacetimes $(\mathcal{M},g)$ having an axisymmetric Killing field
$\text{\textgreek{F}}$, such that $[\text{\textgreek{F}},T]=0$ and
the span of $\text{\textgreek{F}},T$ on $\partial\mathscr{E}_{ext}$
contains a timelike direction. This is a consequence of Lemma \ref{lem:UniqueContinuationAxisymmetry}
at the end of this section (choosing $\mathcal{U}$ to be a suitable
small neighborhood of a point $p\in\partial\mathscr{E}_{ext}\backslash\{\text{\textgreek{F}}=0\}$
and $\mathcal{S}=\partial\mathscr{E}_{ext}\cap\mathcal{U}$ in the
statement of Lemma \ref{lem:UniqueContinuationAxisymmetry}).}

\item{ Assumption \hyperref[Assumption 4]{A1} is always satisfied
on spacetimes $(\mathcal{M},g)$ on which there exists a point $p\in\partial\mathscr{E}_{ext}$
and an open neighborhood $\mathcal{U}$ of $p$ such that $(\mathcal{U},g)$
is a real analytic Lorentzian manifold and $\partial\mathscr{E}_{ext}\cap\mathcal{U}$
is a real analytic hypersurface. This is a consequence of Holmgren's
uniqueness theorem (see \cite{Holmgren1901}).}

\end{itemize}

On the other hand, we believe that Assumption \hyperref[Assumption 4]{A1}
does not hold on all spacetimes satisfying Assumptions \hyperref[Assumption 1]{G1}--\hyperref[Assumption 3]{G3}.
In particular, by adjusting the arguments of \cite{Alinhac1995},
we were able to construct a $3+1$-dimensional spacetime $(\mathcal{M},g)$,
satisfying Assumptions \hyperref[Assumption 1]{G1}--\hyperref[Assumption 3]{G3},
as well as a suitable $T$-invariant smooth potential $V:\mathcal{M}\rightarrow\mathbb{C}$,
so that Assumption \hyperref[Assumption 4]{A1} for equation 
\begin{equation}
\square_{g}\text{\textgreek{f}}-V\text{\textgreek{f}}=0\label{eq:WaveEquationPotential}
\end{equation}
in place of (\ref{eq:WaveEquation}) is not satisfied. Note that such
a construction is non-trivial, in view of the requirement that $T(V)=0$;
for instance, Assumption \hyperref[Assumption 4]{A1} always holds
for equation (\ref{eq:WaveEquationPotential}) on stationary spacetimes
without an ergoregion (see \cite{Tataru1995}). We will not pursue
this issue any further in this paper.

Although we believe that Assumption \hyperref[Assumption 4]{A1} can
be removed from the statement of Theorem \ref{thm:FriedmanInstability},
we were not able to do so. 

The following lemma can be used to establish that Assumption \hyperref[Assumption 4]{A1}
always holds in the presence of a second Killing field \textgreek{F}
on $\mathcal{M}$ such that the span of $T,\text{\textgreek{F}}$
is timelike:
\begin{lem}
\label{lem:UniqueContinuationAxisymmetry}Let $\mathcal{U}$ be an
open subset of a smooth spacetime $(\mathcal{M},g)$ with two Killing
fields $T,\text{\textgreek{F}}$ such that $[T,\text{\textgreek{F}}]=0$,
$\text{\textgreek{F}}\neq0$ on $\mathcal{U}$ and the span of $T,\text{\textgreek{F}}$
contains a timelike direction everywhere on $\mathcal{U}$. Let also
$\mathcal{S}\subset\mathcal{U}$ be a $T,\text{\textgreek{F}}$-invariant
smooth hypersurface, separating $\mathcal{U}$ into two connected
components $\mathcal{U}_{1},\mathcal{U}_{2}$, so that $\mathcal{U}_{2}$
lies in the domain of dependence of $\mathcal{U}_{1}$. Then, any
$\text{\textgreek{y}}\in H_{loc}^{2}(\mathcal{U})$ solving (\ref{eq:WaveEquation})
on $(\mathcal{U},g)$ such that $\text{\textgreek{y}}\equiv0$ on
$\mathcal{U}_{1}$ must vanish everywhere on $\mathcal{U}$. \end{lem}
\begin{proof}
Since $[T,\text{\textgreek{F}}]=0$ and $\text{\textgreek{F}}\neq0$
on $\mathcal{U}$, we can assume without loss of generality (by shrinking
$\mathcal{U}$ if necessary) that $\mathcal{U}$ is covered by a coordinate
chart $(t,\text{\textgreek{f}},x^{2},\ldots,x^{d})$ such that:

\begin{enumerate}

\item $T(\text{\textgreek{f}}),T(x^{2}),\ldots T(x^{d})=0$,

\item $\text{\textgreek{F}}(t),\text{\textgreek{F}}(x^{2}),\ldots,\text{\textgreek{F}}(x^{d})=0$,

\item $T(t)=\text{\textgreek{F}}(\text{\textgreek{f}})=1$,

\item $\mathcal{S}=\{x^{1}=0\}$.

\end{enumerate}

In view of the fact that the span of $\text{\textgreek{F}},T$ contains
a timelike direction everywhere on $\mathcal{U}$, the wave operator
(\ref{eq:WaveEquation}) in the $(t,\text{\textgreek{f}},x^{2},\ldots,x^{d})$
coordinate system takes the form (using the shorthand notation $x=(x^{2},\ldots,x^{d})$):
\begin{align}
\square_{g}\text{\textgreek{y}}=\text{\textgreek{D}}_{x}\text{\textgreek{y}} & +\sum_{j=2}^{d}\big(a_{tj}^{(2)}(x)\partial_{x^{j}}\partial_{t}\text{\textgreek{y}}+a_{\text{\textgreek{f}}j}^{(2)}(x)\partial_{x^{j}}\partial_{\text{\textgreek{f}}}\text{\textgreek{y}}+a_{tj}^{(1)}(x)\partial_{t}\text{\textgreek{y}}+a_{\text{\textgreek{f}}j}^{(1)}(x)\partial_{\text{\textgreek{f}}}\text{\textgreek{y}}\big)+\label{eq:WaveOperator}\\
 & +a_{tt}(x)\partial_{t}^{2}\text{\textgreek{y}}+a_{t\text{\textgreek{f}}}(x)\partial_{t}\partial_{\text{\textgreek{f}}}\text{\textgreek{y}}+a_{\text{\textgreek{f}\textgreek{f}}}(x)\partial_{\text{\textgreek{f}}}^{2}\text{\textgreek{y}},\nonumber 
\end{align}
 where the operator $\text{\textgreek{D}}_{x}$ in the right hand
side of (\ref{eq:WaveOperator}) is a $t,\text{\textgreek{f}}$-invariant
second order elliptic operator in the $x^{2},\ldots,x^{d}$ variables.
Since the coefficients of (\ref{eq:WaveOperator}) are independent
of $t,\text{\textgreek{f}}$ and $\text{\textgreek{D}}_{x}$ is elliptic,
the proof of the Lemma follows readily by the unique continuation
result of Tataru \cite{Tataru1995}.
\end{proof}

\subsection{\label{sub:ExampleSpacetime}Examples of spacetimes satisfying Assumptions
\hyperref[Assumption 1]{G1}--\hyperref[Assumption 3]{G3} and \hyperref[Assumption 4]{A1}}

In this section, we will examine some explicit examples of spacetimes
satisfying all of the Assumptions \hyperref[Assumption 1]{G1}--\hyperref[Assumption 3]{G3}
and \hyperref[Assumption 4]{A1}.

\paragraph*{An example with $\mathcal{H}^{+}=\emptyset$.}

Our first example will be a simple spacetime with no event horizon.
Let $\mathcal{M}=\mathbb{R}^{3+1}$, and let us fix fix two smooth
functions $\text{\textgreek{q}}_{\bar{r}}:[0,+\infty)\rightarrow[0,1]$
and $\text{\textgreek{q}}_{\text{\textgreek{j}}}:[0,\text{\textgreek{p}}]\rightarrow[0,1]$,
satisfying $\text{\textgreek{q}}_{\bar{r}}\equiv0$ on $[0,3]\cup[6,+\infty)$,
$\text{\textgreek{q}}_{\bar{r}}\equiv1$ on $[4,5]$, $\text{\textgreek{q}}_{\text{\textgreek{j}}}\equiv0$
on $[0,\frac{\text{\textgreek{p}}}{6}]\cup[\frac{5\text{\textgreek{p}}}{6},\text{\textgreek{p}}]$
and $\text{\textgreek{q}}_{\text{\textgreek{j}}}\equiv1$ on $[\frac{\text{\textgreek{p}}}{4},\frac{3\text{\textgreek{p}}}{4}]$.
We will also assume that $\text{\textgreek{q}}_{\bar{r}},\text{\textgreek{q}}_{\text{\textgreek{j}}}$
have been chosen so that the set of zeros of the function 
\begin{equation}
f(\bar{r},\text{\textgreek{j}})\doteq1-(\text{\textgreek{q}}_{\bar{r}}(\bar{r})\text{\textgreek{q}}_{\text{\textgreek{j}}}(\text{\textgreek{j}}))^{2}\label{eq:FunctionForErgoregion-1}
\end{equation}
is a smooth curve without self-intersections in the open rectangle
$(3,6)\times(\frac{\text{\textgreek{p}}}{6},\frac{5\text{\textgreek{p}}}{\text{6}})$,
and the region $\{f\le0\}\subset(3,6)\times(\frac{\text{\textgreek{p}}}{6},\frac{5\text{\textgreek{p}}}{\text{6}})$
is simply connected.

We will consider the following metric on $\mathcal{M}$ in the usual
time-polar coordinate chart $(t,\bar{r},\text{\textgreek{j}},\text{\textgreek{f}})$
on $\mathbb{R}^{3+1}$: 
\begin{equation}
g=-\big(1-(\text{\textgreek{q}}_{\bar{r}}(\bar{r})\text{\textgreek{q}}_{\text{\textgreek{j}}}(\text{\textgreek{j}}))^{2}\big)dt^{2}-1000\text{\textgreek{q}}_{\bar{r}}(\bar{r})\text{\textgreek{q}}_{\text{\textgreek{j}}}(\text{\textgreek{j}})dtd\text{\textgreek{f}}+d\bar{r}^{2}+\bar{r}^{2}\big(d\text{\textgreek{j}}^{2}+\sin^{2}\text{\textgreek{j}}d\text{\textgreek{f}}^{2}\big).\label{eq:AlmostFlatExample}
\end{equation}
Note that $g$ is everywhere non-degenerate, and has Lorentzian signature.
Furthermore, $(\mathcal{M},g)$ is a globally hyperbolic spacetime,
with Cauchy hypersurface $\{t=0\}$, satisfying the following properties:

\begin{enumerate}

\item The vector field $T=\partial_{t}$ is a Killing field of $(\mathcal{M},g)$.
Furthermore, $(\mathcal{M},g)$ is asymptotically flat and satisfies
Assumption \hyperref[Assumption 1]{G1}. Notice that $(\mathcal{M},g)$
has no event horizon, since every point in $\mathcal{M}$ can be connected
with the asymptotically flat region $\mathcal{I}_{as}=\{\bar{r}\ge R_{0}\gg1\}$
through both a future directed and a past directed timelike curve,
by following the flow of the timelike vector fields $\partial_{\bar{r}}+C\big(\partial_{t}+\frac{1}{10}\text{\textgreek{q}}_{\bar{r}}(\bar{r})\text{\textgreek{q}}_{\text{\textgreek{j}}}(\text{\textgreek{j}})\partial_{\text{\textgreek{f}}}\big)$
and $\partial_{\bar{r}}-C\big(\partial_{t}+\frac{1}{10}\text{\textgreek{q}}_{\bar{r}}(\bar{r})\text{\textgreek{q}}_{\text{\textgreek{j}}}(\text{\textgreek{j}})\partial_{\text{\textgreek{f}}}\big)$,
respectively (for some fixed $C\gg1$). The function $r:\mathcal{M}\rightarrow[0,+\infty)$,
introduced in Assumption \hyperref[Assumption 1]{G1}, can be chosen
to be equal to $\big(1+\bar{r}^{2}\big)^{1/2}$.

\item The spacetime $(\mathcal{M},g)$ has no event horizon, and,
thus, it trivially satisfies Assumption \hyperref[Assumption 2]{G2}.

\item The ergoregion $\mathscr{E}=\{g(T,T)>0\}$ of $(\mathcal{M},g)$
is non-empty, and satisfies 
\begin{equation}
\{4\le\bar{r}\le5\}\cap\{\frac{\text{\textgreek{p}}}{4}\le\text{\textgreek{j}}\le\frac{3\text{\textgreek{p}}}{4}\}\subset\mathscr{E}\subset\{3\le\bar{r}\le6\}\cap\{\frac{\text{\textgreek{p}}}{6}\le\text{\textgreek{j}}\le\frac{5\text{\textgreek{p}}}{6}\}.\label{eq:ErgoregionExample-1}
\end{equation}
Since $\mathcal{H}^{+}=\emptyset$, we have $\mathscr{E}\cap\mathcal{H}^{+}=\emptyset$.
Furthermore, since we assumed that the region $\{f\le0\}\subset(3,6)\times(\frac{\text{\textgreek{p}}}{6},\frac{5\text{\textgreek{p}}}{\text{6}})$
for the function (\ref{eq:FunctionForErgoregion-1}) is simply connected,
we can readily infer that that $\mathcal{M}\backslash\mathscr{E}$
is connected and, thus, $\mathscr{E}_{ext}=\mathscr{E}$. Furthermore,
the boundary $\partial\mathscr{E}$ of $\mathscr{E}$ is a smooth
hypersurface of $\mathcal{M}$ and (\ref{eq:NondegenrateTOutsideErgoregion})
is satisfied, in view of our assumption on the set of zeros of the
function (\ref{eq:FunctionForErgoregion-1}). Therefore, $(\mathcal{M},g)$
satisfies Assumption \hyperref[Assumption 3]{G3}.

\item The spacetime $(\mathcal{M},g)$ possesses an additional Killing
field, i.\,e.~$\text{\textgreek{F}}=\partial_{\text{\textgreek{f}}}$.
The span of $T,\text{\textgreek{F}}$ contains the everywhere timelike
vector field $T+\frac{1}{10}\text{\textgreek{q}}_{\bar{r}}(\bar{r})\text{\textgreek{q}}_{\text{\textgreek{j}}}(\text{\textgreek{j}})\text{\textgreek{F}}$
and, thus, Lemma \ref{lem:UniqueContinuationAxisymmetry} implies
that $(\mathcal{M},g)$ satisfies Assumption \hyperref[Assumption 4]{A1}.
In particular, any point $p\in\partial\mathscr{E}_{ext}\subseteq\partial\mathscr{E}$
and any open neighborhood $\mathcal{U}$ of $p$ in $\mathcal{M}$
satisfy the unique continuation property of Assumption \hyperref[Assumption 4]{A1}.

\end{enumerate}

Therefore, $(\mathcal{M},g)$ satisfies all of the Asumptions \hyperref[Assumption 1]{G1}--\hyperref[Assumption 3]{G3}
and \hyperref[Assumption 4]{A1}. 
\begin{rem*}
The hydrodynamic vortex $(\mathbb{R}\times\mathcal{V}_{hyd,\text{\textgreek{d}}},g_{hyd})$
of Corollary \ref{cor:VortexMain} is not a globally hyperbolic spacetime,
since its boundary $\partial\big(\mathbb{R}\times\mathcal{V}_{hyd,\text{\textgreek{d}}}\big)=\{\bar{r}=\text{\textgreek{d}}\}$
is a timelike hypersurface. However, as we will show in the proof
of Corollary \ref{cor:VortexMain}, the double of $(\mathbb{R}\times\mathcal{V}_{hyd,\text{\textgreek{d}}},g_{hyd})$
across $\partial\big(\mathbb{R}\times\mathcal{V}_{hyd,\text{\textgreek{d}}}\big)$
is a globally hyperbolic spacetime without an event horizon, satisfying
Assumptions \hyperref[Assumption 1]{G1}--\hyperref[Assumption 3]{G3}
and \hyperref[Assumption 4]{A1}. In addition, the double of $(\mathbb{R}\times\mathcal{V}_{hyd,\text{\textgreek{d}}},g_{hyd})$
is an example of a spacetime having two asymptotically flat ends,
with $(\mathbb{R}\times\mathcal{V}_{hyd,\text{\textgreek{d}}},g_{hyd})\backslash\mathscr{E}$
having two connected components.
\end{rem*}

\paragraph*{An example with $\mathcal{H}^{+}\neq\emptyset$.}

We will now proceed to construct a slightly more complicated example
of a spacetime satisfying Assumptions \hyperref[Assumption 1]{G1}--\hyperref[Assumption 3]{G3}
and \hyperref[Assumption 4]{A1}, possessing in addition a non-empty
event horizon. Note that, as we mentioned in Section \ref{sub:Assumptions},
the subextremal Kerr exterior family $(\mathcal{M}_{M,a},g_{M,a})$
does not satisfy \hyperref[Assumption 3]{G3}, since the future event
horizon $\mathcal{H}^{+}$ and the ergoregion $\mathscr{E}$ of $(\mathcal{M}_{M,a},g_{M,a})$
have a non-empty intersection. 

For any $M>0$, let $\mathcal{M}_{M}$ be diffeomorphic to $\mathbb{R}\times(2M,+\infty)\times\mathbb{S}^{2}$.
Let $\text{\textgreek{q}}_{\bar{r}},\text{\textgreek{q}}_{\text{\textgreek{j}}}$
be as before, assuming, in addition, that they have been chosen so
that the set of zeros of the function 
\begin{equation}
f_{M}(\bar{r},\text{\textgreek{j}})\doteq\big(1-\frac{2M}{\bar{r}}-(\text{\textgreek{q}}_{\bar{r}}(M^{-1}\bar{r})\text{\textgreek{q}}_{\text{\textgreek{j}}}(\text{\textgreek{j}}))^{2}\big)\label{eq:FunctionForErgoregion}
\end{equation}
is a smooth curve without self-intersections in the open rectangle
$(3M,6M)\times(\frac{\text{\textgreek{p}}}{6},\frac{5\text{\textgreek{p}}}{\text{6}})$
and the region $\{f_{M}\le0\}\subset(3,6)\times(\frac{\text{\textgreek{p}}}{6},\frac{5\text{\textgreek{p}}}{\text{6}})$
is simply connected.

Let us consider the following metric in the $(t,\bar{r},\text{\textgreek{j}},\text{\textgreek{f}})$
coordinate chart on $\mathcal{M}_{M}$: 
\begin{equation}
g_{M}=-\big(1-\frac{2M}{\bar{r}}-(\text{\textgreek{q}}_{\bar{r}}(M^{-1}\bar{r})\text{\textgreek{q}}_{\text{\textgreek{j}}}(\text{\textgreek{j}}))^{2}\big)dt^{2}-1000M\text{\textgreek{q}}_{\bar{r}}(M^{-1}\bar{r})\text{\textgreek{q}}_{\text{\textgreek{j}}}(\text{\textgreek{j}})dtd\text{\textgreek{f}}+\big(1-\frac{2M}{\bar{r}}\big)^{-1}d\bar{r}^{2}+\bar{r}^{2}\big(d\text{\textgreek{j}}^{2}+\sin^{2}\text{\textgreek{j}}d\text{\textgreek{f}}^{2}\big).\label{eq:AlmostSchwarzschildExample}
\end{equation}
The metric $g_{M}$ is everywhere non-degenerate, and has Lorentzian
signature.

The spacetime $(\mathcal{M}_{M},g_{M})$ is isometric to the Schwarzschild
exterior spacetime $(\mathcal{M}_{M,Sch},g_{M,Sch})$ outside the
region $\{3M\le\bar{r}\le6M\}\cap\{\frac{\text{\textgreek{p}}}{6}\le\text{\textgreek{j}}\le\frac{5\text{\textgreek{p}}}{6}\}$
and, thus, it can be extended into a larger globally hyperbolic spacetime
$(\widetilde{\mathcal{M}}_{M},\tilde{g}_{M})$.%
\footnote{Of course, the coordinate chart $(t,\bar{r},\text{\textgreek{j}},\text{\textgreek{f}})$
will not be regular up to the boundary of $\mathcal{M}_{M}$ in this
extension.%
} This extension can be chosen to be the Schwarzschild maximal extension
across $\bar{r}=2M$ (see e.\,g.~Section 2 \cite{DafRod6}).%
\footnote{Note that the $\bar{r}$ coordinate function on $\mathcal{M}_{M}$
corresponds to the usual $r$ coordinate function on Schwarzschild
exterior $\mathcal{M}_{M,Sch}$%
} Let us denote 
\begin{equation}
\overline{\mathcal{M}}_{M}=i(\mathcal{M}_{M})\cup\partial\mathcal{M}_{M},
\end{equation}
where $i:\mathcal{M}_{M}\rightarrow\widetilde{\mathcal{M}}_{M}$ is
the natural inclusion of $(\mathcal{M}_{M},g_{M})$ into its extension
and $\partial\mathcal{M}_{M}$ is the boundary of $i(\mathcal{M}_{M})$
inside $\widetilde{\mathcal{M}}_{M}$. Note that, in view of the properties
of the maximally extended Schwarzschild spacetime, $(\overline{\mathcal{M}}_{M},g_{M})$
is a smooth Lorentzian manifold with piecewise smooth boundary $\partial\mathcal{M}_{M}$,
consisting of two intersecting smooth null hypersurfaces. The functions
$\bar{r},\text{\textgreek{j}},\text{\textgreek{f}}$ can be smothly
extended on $\partial\mathcal{M}_{M}$, with $\bar{r}|_{\partial\mathcal{M}_{M}}=2M$.

The spacetime $(\overline{\mathcal{M}}_{M},g_{M})$ is globally hyperbolic,
with $\text{\textgreek{S}}=\overline{\{t=0\}}$ being a smooth Cauchy
hypersurface, and satisfies the following properties: 

\begin{enumerate}

\item The vector field $T=\partial_{t}$ on $\mathcal{M}_{M}$ extends
smoothly on $\partial\mathcal{M}_{M}$ and is a Killing vector field
of $(\overline{\mathcal{M}}_{M},g_{M})$. Furthermore, the spacetime
$(\overline{\mathcal{M}}_{M},g_{M})$ is asymptotically flat and satisfies
Assumption \hyperref[Assumption 1]{G1}. Note that the event horizon
$\mathcal{H}$ of $(\overline{\mathcal{M}}_{M},g_{M})$ coincides
with $\partial\mathcal{M}_{M}$, since all the points in $\mathcal{M}_{M}$
can be joined with the asymptotically flat region $\mathcal{I}_{as}=\{\bar{r}\ge R_{0}\gg1\}$
through both a future directed and a past directed timelike curve,
by following the flow of the timelike vector fields $\partial_{\bar{r}}+C\big(\partial_{t}+\frac{1}{10M}\text{\textgreek{q}}_{\bar{r}}(M^{-1}\bar{r})\text{\textgreek{q}}_{\text{\textgreek{j}}}(\text{\textgreek{j}})\partial_{\text{\textgreek{f}}}\big)$
and $\partial_{\bar{r}}-C\big(\partial_{t}+\frac{1}{10M}\text{\textgreek{q}}_{\bar{r}}(M^{-1}\bar{r})\text{\textgreek{q}}_{\text{\textgreek{j}}}(\text{\textgreek{j}})\partial_{\text{\textgreek{f}}}\big)$,
respectively (for some fixed $C\gg1$). The function $r:\overline{\mathcal{M}}_{M}\rightarrow[0,+\infty)$,
introduced in Assumption \hyperref[Assumption 1]{G1}, can be chosen
to be equal to $\bar{r}-2M$.

\item There exists a $T$-invariant neighborhood $\mathcal{V}$ of
$\mathcal{H}=\partial\mathcal{M}_{M}$ in $\overline{\mathcal{M}}_{M}$,
so that $(\mathcal{V},g_{M})$ is isometric to a neighborhood of the
event horizon $\mathcal{H}_{M,Sch}$ of Schwarzschild exterior spacetime.
In particular, $T$ is parallel to the null generators of $\mathcal{H}^{+}\backslash\mathcal{H}^{-}$
and there exists a $T$-invariant timelike vector field $N$ on $\overline{\mathcal{M}}_{M}$
as in Assumption \hyperref[Assumption 2]{G2}, satisfying (\ref{eq:PositiveKN})
(see \cite{DafRod2,DafRod6}). In particular, $(\overline{\mathcal{M}}_{M},g_{M})$
satisfies Assumption \hyperref[Assumption 2]{G2}.

\item The ergoregion $\mathscr{E}=\{g(T,T)>0\}$ of $(\overline{\mathcal{M}}_{M},g_{M})$
is non-empty, and satisfies 
\begin{equation}
\{4M\le\bar{r}\le5M\}\cap\{\frac{\text{\textgreek{p}}}{4}\le\text{\textgreek{j}}\le\frac{3\text{\textgreek{p}}}{4}\}\subset\mathscr{E}\subset\{3M\le\bar{r}\le6M\}\cap\{\frac{\text{\textgreek{p}}}{6}\le\text{\textgreek{j}}\le\frac{5\text{\textgreek{p}}}{6}\}.\label{eq:ErgoregionExample}
\end{equation}
Thus, $\mathscr{E}\cap\mathcal{H}^{+}=\emptyset$, since $\bar{r}>2M$
on $\mathscr{E}$. Furthermore, since the function (\ref{eq:FunctionForErgoregion})
was assumed to have the property that the set $\{f_{M}\le0\}$ is
simply connected, we can readily infer that $\overline{\mathcal{M}}_{M}\backslash\mathscr{E}$
is connected. Thus, $\mathscr{E}_{ext}=\mathscr{E}$ and $\mathcal{H}^{+}$
lies in the same connected component of $\overline{\mathcal{M}}_{M}\backslash\mathscr{E}$
as the asymptotically flat region $\mathcal{I}_{as}$. The boundary
$\partial\mathscr{E}$ of $\mathscr{E}$ is a smooth hypersurface
of $\overline{\mathcal{M}}_{M}$ and (\ref{eq:NondegenrateTOutsideErgoregion})
is satisfied, in view of our assumption on the set of zeros of the
function (\ref{eq:FunctionForErgoregion}). Therefore, $(\overline{\mathcal{M}}_{M},g_{M})$
satisfies Assumption \hyperref[Assumption 3]{G3}.

\item In view of the fact that $(\mathcal{M}_{M},g_{M})$ possesses
an additional Killing field, namely $\text{\textgreek{F}}=\partial_{\text{\textgreek{f}}}$,
and the span of $T,\text{\textgreek{F}}$ contains the vector field
$T+\frac{1}{10M}\text{\textgreek{q}}_{\bar{r}}(M^{-1}\bar{r})\text{\textgreek{q}}_{\text{\textgreek{j}}}(\text{\textgreek{j}})\text{\textgreek{F}}$
which is everywhere timelike on $\mathcal{M}_{M}$, Lemma \ref{lem:UniqueContinuationAxisymmetry}
implies that $(\overline{\mathcal{M}}_{M},g_{M})$ satisfies Assumption
\hyperref[Assumption 4]{A1}. In particular, any point $p\in\partial\mathscr{E}_{ext}=\partial\mathscr{E}$
and any open neighborhood $\mathcal{U}$ of $p$ satisfy the unique
continuation property of Assumption \hyperref[Assumption 4]{A1}.

\end{enumerate}

Thus, $(\overline{\mathcal{M}}_{M},g_{M})$ satisfies Assumptions
\hyperref[Assumption 1]{G1}--\hyperref[Assumption 3]{G3} and \hyperref[Assumption 4]{A1},
and, in addition, $(\overline{\mathcal{M}}_{M},g_{M})$ has a non-empty
future event horizon.

\subsection{\label{sub:CarlemanEstimatesIntroduction}A remark on the Carleman-type
estimates in the proof of Theorem \ref{thm:FriedmanInstability}}

As we discussed in the introduction, a crucial step in the proof of
Theorem \ref{thm:FriedmanInstability} consists of showing that, under
the assumption that 
\begin{equation}
\limsup_{\text{\textgreek{t}}\rightarrow+\infty}\int_{\text{\textgreek{S}}_{\text{\textgreek{t}}}}J_{\text{\textgreek{m}}}^{N}(\text{\textgreek{f}})n^{\text{\textgreek{m}}}<+\infty\label{eq:EnergyGoingToInfinity-1}
\end{equation}
 holds for every smooth function $\text{\textgreek{f}}:J^{+}(\text{\textgreek{S}})\rightarrow\mathbb{C}$
solving the wave equation (\ref{eq:WaveEquation}) on $J^{+}(\text{\textgreek{S}})$
with compactly supported initial data on $\text{\textgreek{S}}$,
we also have that $\text{\textgreek{f}}$ decays on $\mathcal{M}\backslash\mathscr{E}$;
see Section \ref{sec:Proof-of-Theorem} for more details. This fact
is inferred using some suitable Carleman-type estimates on $(\mathcal{M}\backslash\mathscr{E},g)$
for \textgreek{f} which are particularly useful when $\text{\textgreek{f}}$
has localised frequency support in time (see Proposition \ref{prop:GeneralCarlemanEstimate}
in Section \ref{sec:Carleman}; for the technical details related
to the frequency decomposition of $\text{\textgreek{f}}$, see Section
\ref{sec:FrequencyDecomposition}). 

The aforementioned estimates could have been established using the
techniques of our previous \cite{Moschidisb}. However, we chose,
instead, to provide an alternative proof, based entirely on the use
of first order multipliers for equation (\ref{eq:WaveEquation}).
As a consequence, we obtain an alternative proof for the estimates
of Section 7.1 of \cite{Moschidisb}, as well as for the Carleman-type
estimates established in \cite{Rodnianski2011} for the inhomogeneous
Helmholtz equation 
\begin{equation}
\text{\textgreek{D}}_{\bar{g}}u+\text{\textgreek{w}}^{2}u-Vu=G,
\end{equation}
$0<Im(\text{\textgreek{w}})\ll1$, $Re(\text{\textgreek{w}})\neq0$,
on an asymptotically conic Riemannian manifold $(\text{\textgreek{S}},\bar{g})$,
where the potential $V:\text{\textgreek{S}}\rightarrow\mathbb{R}$
satisfies some suitable decay conditions on the asymptotically conic
end of $\text{\textgreek{S}}$. For a more detailed statement of these
results, see Section \ref{sec:Carleman}.

\section{\label{sec:Notational-conventions}Notational conventions and Hardy
inequalities}

In this section, we will introduce some conventions on denoting constants
and parameters that will appear throughout this paper. We will adopt
similar conventions as in \cite{Moschidisb}.

\subsection{Constants and dependence on parameters}

We will adopt the following convention for denoting constants appearing
in inequalities: We will use capital letters (e.\,g.~$C$) to denote
``large'' constants, typically appearing on the right hand side
of inequalities. (Such constants can be ``freely'' replaced by larger
ones without rendering the inequality invalid.) Lower case letters
(e.\,g. $c$) will be used to denote ``small'' constants (which
can similarly freely be replaced by smaller ones). The same characters
will be frequently used to denote different constants, even in adjacent
lines. 

We will assume that all non-explicit constants will depend on the
specific geometric aspects of $(\mathcal{M},g)$ and we will not keep
track of this dependence, except for some very specific cases. However,
since we will introduce a plethora of parameters throughout this paper,
we will always keep track of the dependence of all constants on each
of these parameters. Once a parameter is fixed (which will be clearly
stated in the text), the dependence of constants on it will be dropped.

\subsection{Inequality symbols}

We will use the notation $f_{1}\lesssim f_{2}$ for two real functions
$f_{1},f_{2}$ as usual to imply that there exists some $C>0$, such
that $f_{1}\le Cf_{2}$. This constant $C$ might depend on free parameters,
and these parameters will be stated clearly in each case. If nothing
is stated regarding the dependence of this constant on parameters,
it should be assumed that it only depends on the geometry of the spacetime
$(\mathcal{M},g)$ under consideration.

We will denote $f_{1}\sim f_{2}$ when we can bound $f_{1}\lesssim f_{2}$
and $f_{2}\lesssim f_{1}$. The notation $f_{1}\ll f_{2}$ will be
equivalent to the statement that $\frac{|f_{1}|}{|f_{2}|}$ can be
bounded by some sufficiently small positive constant, the magnitude
and the dependence of which on variable parameters will be clear in
each case from the context. For any function $f:\mathcal{M}\rightarrow[0,+\infty)$,
$\{f\gg1\}$ will denote the subset $\{f\ge C\}$ of $\mathcal{M}$
for some constant $C\gg1$.

For functions $f_{1},f_{2}:[x_{0},+\infty)\rightarrow\mathbb{R}$,
the notation $f_{1}=o(f_{2})$ will mean that $\frac{|f_{1}|}{|f_{2}|}$
can be bounded by some continuous function $h:[x_{0},+\infty)\rightarrow(0,+\infty)$
such that $h(x)\rightarrow0$ as $x\rightarrow+\infty$. This bound
$h$ might deppend on free parameters, and this fact will be clear
in each case from the context.

\subsection{\label{sub:SpecialSubsets}Some special subsets of $\mathcal{M}$}

The future event horizon of $\mathcal{M}$ will be denoted by $\mathcal{H}^{+}$,
and the past event horizon by $\mathcal{H}^{-}$, i.\,e.
\begin{gather*}
\mathcal{H}^{+}=\overline{J^{+}(\mathcal{I}_{as})}\cap\partial J^{-}(\mathcal{I}_{as}),\\
\mathcal{H}^{-}=\overline{J^{-}(\mathcal{I}_{as})}\cap\partial J^{+}(\mathcal{I}_{as}).
\end{gather*}

For any $\text{\textgreek{t}}_{1}\le\text{\textgreek{t}}_{2}$, we
will denote 
\begin{equation}
\mathcal{R}(\text{\textgreek{t}}_{1},\text{\textgreek{t}}_{2})\doteq\{\text{\textgreek{t}}_{1}\le t\le\text{\textgreek{t}}_{2}\}\subset\mathcal{M}\backslash\mathcal{H}^{-}
\end{equation}
and 
\begin{equation}
\text{\textgreek{S}}_{\text{\textgreek{t}}}\doteq\{t=\text{\textgreek{t}}\},
\end{equation}
where the function $t:\mathcal{M}\backslash\mathcal{H}^{-}\rightarrow\mathbb{R}$
is defined in Assumption \hyperref[Assumption 1]{G1}.

The ergoregion of $\mathcal{M}$, defined by (\ref{eq:Ergoregion}),
will be denoted by $\mathscr{E}$. The boundary of $\mathscr{E}$
(which is smooth, according to Assumption \hyperref[Assumption 3]{G3})
will be denoted by $\partial\mathscr{E}$ .We will fix a smooth $T$-invariant
spacelike vector field $n_{\partial\mathscr{E}}$ in a small $T$-invariant
neighborhood of $\partial\mathscr{E}$, such that $n_{\partial\mathscr{E}}|_{\partial\mathscr{E}}$
is the unit normal of $\partial\mathscr{E}$. We will denote with
$\mathscr{E}_{ext}$ the extended ergoregion of $(\mathcal{M},g)$,
defined by (\ref{eq:ExtendedErgoregion}). Notice that $\mathscr{E}\subseteq\mathscr{E}_{ext}$,
but $\partial\mathscr{E}_{ext}\subseteq\partial\mathscr{E}$. 

For any $\text{\textgreek{d}}>0$, we will denote 
\begin{equation}
\mathscr{E}_{\text{\textgreek{d}}}=\{x\in\mathcal{M}\backslash\mathcal{H}^{-}\,|\, dist_{g_{ref}}(x,\mathscr{E}_{ext})\le\text{\textgreek{d}}\}.
\end{equation}
Note that $\mathscr{E}_{ext}\subset\mathscr{E}_{\text{\textgreek{d}}}$
for any $\text{\textgreek{d}}>0$, and $\cap_{\text{\textgreek{d}}>0}\mathscr{E}_{\text{\textgreek{d}}}=\mathscr{E}_{ext}$.

\subsection{\label{sub:IntegrationConventions}Notations on metrics, connections
and integration}

For any pseudo-Riemannian manifold $(\mathcal{N},h_{\mathcal{N}})$
appearing in this paper, we will denote with $dh_{\mathcal{N}}$ the
natural volume form associated with $h_{\mathcal{N}}$. Recall that
in any local coordinate chart $(x^{1},x^{2},\ldots x^{k})$ on $\mathcal{N}$,
$dh_{\mathcal{N}}$ is expressed as 
\[
dh_{\mathcal{N}}=\sqrt{|det(h_{\mathcal{N}})|}dx^{1}\cdots dx^{k}.
\]
 We will also denote with $\nabla_{h_{\mathcal{N}}}$ the natural
connection associated to $h_{\mathcal{N}}$. When $(\mathcal{N},h_{\mathcal{N}})=(\mathcal{M},g)$,
we will denote $\nabla_{h_{\mathcal{N}}}$ simply as $\nabla$. If
$h_{\mathcal{N}}$ is Riemannian, $\big|\cdot\big|_{h_{\mathcal{N}}}$
will denote the associated norm on the tensor bundle of $\mathcal{N}$. 

For any integer $l\ge0$, we will denote with $\big(\nabla^{h_{\mathcal{N}}}\big)^{l}$
or $\nabla_{h_{\mathcal{N}}}^{l}$ the higher order operator 
\begin{equation}
\underbrace{\nabla_{h_{\mathcal{N}}}\cdots\nabla_{h_{\mathcal{N}}}}_{l\mbox{ times}}.\label{eq:ProductDerivatives}
\end{equation}
 Note that the product (\ref{eq:ProductDerivatives}) is not symmetrised.
We will also adopt the convention that we will always use Latin characters
to denote such powers of covariant derivative operators. On the other
hand, Greek characters will be used for the indices of a tensor in
an abstract index notation. 

For any smooth and spacelike hypersurface $\mathcal{S}\subset\mathcal{M}$,
$g_{\mathcal{S}}$ will denote the induced (Riemannian) metric on
$\mathcal{S}$, and $n_{\mathcal{S}}$ the future directed unit normal
to $\mathcal{S}$.

Some examples of pseudo-Riemannian manifolds that will appear throughout
this paper are $(\mathcal{M},g)$, $(\mathcal{M},g_{ref})$ and $(\text{\textgreek{S}}_{\text{\textgreek{t}}},g_{\text{\textgreek{S}}_{\text{\textgreek{t}}}})$,
where $g_{ref}$ is the reference Riemannian metric (\ref{eq:ReferenceRiemannianMetric}).
We will raise and lower indices of tensors on $\mathcal{M}$ \underline{only}
with the use of $g$.

In some cases, we will omit the volume form $dg$ or $dg_{\text{\textgreek{S}}_{\text{\textgreek{t}}}}$
when integrating over domains in $\mathcal{M}$ or the hypersurfaces
$\text{\textgreek{S}}_{\text{\textgreek{t}}}$, respectively.

In the case of a smooth null hypersurface $\mathscr{H}\subset\mathcal{M}$,
the volume form with which integration will be considered will as
usual depend on the choice of a future directed null generator $n_{\mathcal{\mathscr{H}}}$
for $\mathcal{\mathscr{H}}$. For any such choice of $n_{\mathcal{\mathscr{H}}}$,
selecting an arbitrary vecor field $X$ on $T_{\mathcal{\mathscr{H}}}M$
such that $g(X,n_{\mathscr{H}})=-1$ enables the construction of a
non-degenerate top dimensional form on $\mathcal{\mathscr{H}}$: $dvol_{\mathscr{H}}\doteq i_{X}dg$,
which depends on the on the precise choice of $n_{\mathcal{\mathscr{H}}}$,
but not on the choice for $X$. In that case, $dvol_{\mathscr{H}}$
(or $dvol_{n_{\mathscr{H}}}$) will be the volume form on $\mathscr{H}$
associated with $n_{\mathcal{\mathscr{H}}}$.

\subsection{\label{sub:CoordinateCharts}Coordinate charts on $\mathcal{M}\backslash\mathcal{H}^{-}$}

Using the function $t$ as a projection, we can identify $\mathcal{M}\backslash\mathcal{H}^{-}$
with $\mathbb{R}\times\text{\textgreek{S}}$. Under this identification,
any local coordinate chart $(x^{1},\ldots,x^{d})$ on a subset $\mathcal{V}$
of $\text{\textgreek{S}}$ can be extended to a coordinate chart $(t,x^{1},\ldots,x^{d})$
on $\mathbb{R}\times\mathcal{V}\subset\mathbb{R}\times\text{\textgreek{S}}$,
and in this chart, we have $\partial_{t}=T$. We will usually work
in such coordinate charts throughout this paper.

In view of the the flat asymptotics of $(\mathcal{M},g)$ and the
fact that $\text{\textgreek{S}}$ intersects $\mathcal{H}^{+}$ transversally,
the coarea formula yields that in the region $J^{+}(\text{\textgreek{S}})$,
the volume forms $dg$ and $dt\wedge dg_{\text{\textgreek{S}}}$ are
equivalent, i.\,e.~there exists a $C>0$, such that for any integrable
$\text{\textgreek{f}}:\mathcal{M}\rightarrow[0,+\infty)$ and any
$0\le\text{\textgreek{t}}_{1}\le\text{\textgreek{t}}_{2}$ (identifying
$\mathcal{M}\backslash\mathcal{H}^{-}$ with $\mathbb{R}\times\text{\textgreek{S}}$):
\begin{equation}
C^{-1}\int_{\mathcal{R}(\text{\textgreek{t}}_{1},\text{\textgreek{t}}_{2})}\text{\textgreek{f}}\, dg\le\int_{\text{\textgreek{t}}_{1}}^{\text{\textgreek{t}}_{2}}\Big(\int_{\text{\textgreek{S}}}\text{\textgreek{f}}(t,x)\, dg_{\text{\textgreek{S}}}\Big)\, dt\le C\int_{\mathcal{R}(\text{\textgreek{t}}_{1},\text{\textgreek{t}}_{2})}\text{\textgreek{f}}\, dg.
\end{equation}
Similarly, for any $\text{\textgreek{d}}>0$, there exists a $C_{\text{\textgreek{d}}}>0$
so that for any integrable $\text{\textgreek{f}}:\mathcal{M}\rightarrow[0,+\infty)$
and any $\text{\textgreek{t}}_{1}\le\text{\textgreek{t}}_{2}$ (not
necessarily non-negative): 
\begin{equation}
C_{\text{\textgreek{d}}}^{-1}\int_{\mathcal{R}(\text{\textgreek{t}}_{1},\text{\textgreek{t}}_{2})\cap\{r\ge\text{\textgreek{d}}\}}\text{\textgreek{f}}\, dg\le\int_{\text{\textgreek{t}}_{1}}^{\text{\textgreek{t}}_{2}}\Big(\int_{\text{\textgreek{S}}\cap\{r\ge\text{\textgreek{d}}\}}\text{\textgreek{f}}(t,x)\, dg_{\text{\textgreek{S}}}\Big)\, dt\le C_{\text{\textgreek{d}}}\int_{\mathcal{R}(\text{\textgreek{t}}_{1},\text{\textgreek{t}}_{2})\cap\{r\ge\text{\textgreek{d}}\}}\text{\textgreek{f}}\, dg.\label{eq:DegeneracyCoareaFormula}
\end{equation}

\subsection{\label{sub:DerivativesOntheSphere}Notations for derivatives on $\mathbb{S}^{d-1}$}

In this paper, we will frequently work in polar coordinates in the
asymptotically flat region of $(\mathcal{M},g)$ or $(\text{\textgreek{S}},g_{\text{\textgreek{S}}})$.
For this reason, we will adopt the same shorthand $\text{\textgreek{sv}}$-notation
for the angular variables in such a polar coordinate, as we did in
\cite{Moschidisb,Moschidisc}. See Section 3.6 of \cite{Moschidisb}
for a detailed statement of this convention.

As an example of this convention, on subset $\mathcal{U}$ of a spacetime
$\mathcal{M}$ covered by a polar coordinate chart $(u_{1},u_{2},\text{\textgreek{sv}}):\mathcal{U}\rightarrow\mathbb{R}_{+}\times\mathbb{R}_{+}\times\mathbb{S}^{d-1}$,
for any function $h:\mathcal{U}\rightarrow\mathbb{C}$ and any \underline{symmetric}
$(l,0)$-tensor $b$ on $\mathbb{S}^{d-1}$, the following schematic
notation for the contraction of the tensor $\big(\nabla_{g_{\mathbb{S}^{d-1}}}\big)^{l}h(u_{1},u_{2},\cdot)$
with $b$ will be frequently used:
\begin{equation}
b\cdot\partial_{\text{\textgreek{sv}}}^{l}h(u_{1},u_{2},\cdot)\doteq b^{\text{\textgreek{i}}_{1}\ldots\text{\textgreek{i}}_{l}}(\nabla_{g_{\mathbb{S}^{d-1}}}^{l})_{\text{\textgreek{i}}_{1}\ldots\text{\textgreek{i}}_{l}}h(u_{1},u_{2},\cdot),\label{eq:ContractionOneFunction-1}
\end{equation}
where $g_{\mathbb{S}^{d-1}}$ is the standard metric on the unit sphere
$\mathbb{S}^{d-1}$. Furthermore, we will also denote in this case
\begin{equation}
|\partial_{\text{\textgreek{sv}}}^{l}h(u_{1},u_{2},\cdot)|\doteq\big|\nabla_{g_{\mathbb{S}^{d-1}}}^{l}h(u_{1},u_{2},\cdot)\big|_{g_{\mathbb{S}^{d-1}}}.
\end{equation}
Notice, also, the following commutation relation holds: 
\begin{equation}
[\mathcal{L}_{\partial_{u_{i}}},\nabla^{\mathbb{S}^{d-1}}]=0,
\end{equation}
where $\partial_{u_{i}}$ is the coordinate vector field associated
to the coordinate function $u_{i}$, $i=1,2$. Therefore, we will
frequently denote for any function $h:\mathcal{U}\rightarrow\mathbb{C}$:
\begin{equation}
\mathcal{L}_{\partial_{u_{i}}}\nabla^{\mathbb{S}^{d-1}}h\doteq\partial_{u_{i}}\partial_{\text{\textgreek{sv}}}h,
\end{equation}
and, in this notation, we will be allowed to commute $\partial_{u_{i}}$
with $\partial_{\text{\textgreek{sv}}}$, as if $\partial_{\text{\textgreek{sv}}}$
was a regular coordinate vector field. See Section 3.6 of \cite{Moschidisb}
for more details.

\subsection{\label{sub:BigONotation}The $O_{k}(\cdot)$ notation}

For any integer $k\ge0$ and any $b\in\mathbb{R}$, the notation $h=O_{k}(r^{b})$
for some smooth function $h:\mathcal{M}\rightarrow\mathbb{C}$ will
be used to denote that, in the $(t,r,\text{\textgreek{sv}})$ polar
coordinate chart on each connected component of the region $\{r\gg1\}$
of $\mathcal{M}$ (see Assumption \hyperref[Assumption 1]{1}): 
\begin{equation}
\sum_{j=0}^{k}\sum_{j_{1}+j_{2}+j_{3}=j}r^{j_{2}+j_{3}}|\partial_{\text{\textgreek{sv}}}^{j_{1}}\partial_{t}^{j_{2}}\partial_{r}^{j_{3}}h|\le C\cdot r^{b}
\end{equation}
for some constant $C>0$ dependng on $k$ and $h$. The same notation
(omitting the $\partial_{t}$ derivatives) will also be used for functions
on regions of manifolds cover by an $(r,\text{\textgreek{sv}})$ polar
coordinate chart.

Similarly, the notation $h=O_{k}^{\mathbb{S}^{d-1}}(r^{b})$ will
be used to denote a smooth tensor field $h$ on $\mathcal{M}$ such
that, in the $(t,r,\text{\textgreek{sv}})$ polar coordinate chart
on each connected component of the region $\{r\gg1\}$ of $\mathcal{M}$,
$h$ is tangential to the $\{r=const\}$ coordinate spheres (i.~e.~$h$
contracted with $\partial_{r},\partial_{t}$ or $dr,d\text{\textgreek{sv}}$,
depending on its type, yields zero), and satisfies $|h|_{g_{\mathbb{S}^{d-1}}}=O_{k}(r^{b})$.
The type of the tensor $h$ will always be clear from the context.

\subsection{\noindent \label{sub:Currents}Vector field multipliers and currents}

In this paper, we will frequently use the language of Lagrangean currents
and vector field multipliers for equation (\ref{eq:WaveEquation}):
On any Lorentzian manifold $(\mathcal{M},g)$, associated to the wave
operator $\square_{g}=\frac{1}{\sqrt{-det(g)}}\partial_{\text{\textgreek{m}}}\Big(\sqrt{-det(g)}g^{\text{\textgreek{m}\textgreek{n}}}\partial_{\text{\textgreek{n}}}\Big)$
is a symmetric $(0,2)$-tensor called the \emph{energy momentum tensor}
$Q$. For any smooth function $\text{\textgreek{y}}:\mathcal{M}\rightarrow\mathbb{C}$,
the energy momentum tensor takes the form

\begin{equation}
Q_{\text{\textgreek{m}\textgreek{n}}}(\text{\textgreek{y}})=\frac{1}{2}\Big(\partial_{\text{\textgreek{m}}}\text{\textgreek{y}}\cdot\partial_{\text{\textgreek{n}}}\bar{\text{\textgreek{y}}}+\partial_{\text{\textgreek{m}}}\bar{\text{\textgreek{y}}}\cdot\partial_{\text{\textgreek{n}}}\text{\textgreek{y}}\Big)-\frac{1}{2}\big(\partial^{\text{\textgreek{l}}}\text{\textgreek{y}}\cdot\partial_{\text{\textgreek{l}}}\bar{\text{\textgreek{y}}}\big)g_{\text{\textgreek{m}\textgreek{n}}}.
\end{equation}

For any continuous and piecewise $C^{1}$ vector field $X$ on $\mathcal{M}$,
the following associated currents can be defined almost everywhere:

\begin{equation}
J_{\text{\textgreek{m}}}^{X}(\text{\textgreek{y}})=Q_{\text{\textgreek{m}\textgreek{n}}}(\text{\textgreek{y}})X^{\text{\textgreek{n}}},
\end{equation}
\begin{equation}
K^{X}(\text{\textgreek{y}})=Q_{\text{\textgreek{m}\textgreek{n}}}(\text{\textgreek{y}})\nabla^{\text{\textgreek{m}}}X^{\text{\textgreek{n}}}.
\end{equation}
 The following divergence identity then holds almost everywhere on
$\mathcal{M}$:

\begin{equation}
\nabla^{\text{\textgreek{m}}}J_{\text{\textgreek{m}}}^{X}(\text{\textgreek{y}})=K^{X}(\text{\textgreek{y}})+Re\Big\{(\square_{g}\text{\textgreek{y}})\cdot X\bar{\text{\textgreek{y}}}\Big\}.
\end{equation}

\subsection{\label{sub:HardyInequalities}Hardy-type inequalities}

Frequently throughout this paper, we will need to control the weighted
$L^{2}$ norm of some function $u$ by some weighted $L^{2}$ norm
of its derivative $\nabla u$. This will always be accomplished with
the use of some variant of the following Hardy-type inequality on
$\mathbb{R}^{d}$ (which is true for $d\ge1$, although we will only
need it for $d\ge2$):
\begin{lem}
\label{lem:HardyInequalities}For any $a>0$, there exists some $C_{a}>0$
such that for any smooth and compactly supported function $u:\mathbb{R}^{d}\rightarrow\mathbb{C}$
and any $0<R_{1}<R_{2}$ we can bound
\begin{align}
\int_{\mathbb{R}^{d}\cap\{R_{1}\le r\le R_{2}\}}r^{-d+a}|u|^{2}\, dx+ & \int_{\{r=R_{1}\}}R_{1}^{-(d-1)+a}|u|^{2}\, dg_{\{r=R_{1}\}}\le\label{eq:GeneralHardyPolynomial}\\
 & \le C_{a}\int_{\mathbb{R}^{d}\cap\{R_{1}\le r\le R_{2}\}}r^{-(d-2)+a}|\partial_{r}u|^{2}\, dx+\int_{\{r=R_{2}\}}R_{2}^{-(d-1)+a}|u|^{2}\, dg_{\{r=R_{2}\}}\nonumber 
\end{align}
and 
\begin{align}
\int_{\mathbb{R}^{d}\cap\{R_{1}\le r\le R_{2}\}}r^{-d}|u|^{2}\, dx+ & \int_{\{r=R_{1}\}}R_{1}^{-(d-1)}\log(R_{1})|u|^{2}\, dg_{\{r=R_{1}\}}\le\label{eq:GeneralHardyLogarithmic}\\
 & \le C\int_{\mathbb{R}^{d}\cap\{R_{1}\le r\le R_{2}\}}r^{-(d-2)}\big(\log(r)\big)^{2}|\partial_{r}u|^{2}\, dx+\int_{\{r=R_{2}\}}R_{2}^{-(d-1)}\log(R_{2})|u|^{2}\, dg_{\{r=R\}}.\nonumber 
\end{align}
In the above, $r$ is the polar distance on $\mathbb{R}^{d}$, $dx$
is the usual volume form on $\mathbb{R}^{d}$ and $dg_{\{r=R\}}$
is the volume form of the induced metric on the sphere $\{r=R\}\subset\mathbb{R}^{d}$.
\end{lem}
The proof of Lemma \ref{lem:HardyInequalities} is straightforward
(see also Section 3.9 of \cite{Moschidisb}).

\section{\label{sec:Proof-of-Theorem}Proof of Theorem \ref{thm:FriedmanInstability} }

The proof of Theorem \ref{thm:FriedmanInstability} will proceed by
contradiction: We will assume that all smooth solutions $\text{\textgreek{f}}$
to (\ref{eq:WaveEquation}) on $\mathcal{D}(\text{\textgreek{S}})$
with compactly supported initial data on $\text{\textgreek{S}}$ satisfy
\begin{equation}
\mathcal{E}[\text{\textgreek{f}}]\doteq\sup_{\text{\textgreek{t}}\ge0}\int_{\text{\textgreek{S}}_{\text{\textgreek{t}}}}J_{\text{\textgreek{m}}}^{N}(\text{\textgreek{f}})n^{\text{\textgreek{m}}}<+\infty,\label{eq:BoundForContradiction}
\end{equation}
and we will reach a contradiction after choosing $\text{\textgreek{f}}$
appropriately. To this end, we will need to establish a decay without
a rate result outside the extended ergoregion $\mathscr{E}_{ext}$
for solutions $\text{\textgreek{f}}$ to (\ref{eq:WaveEquation}),
given the bound (\ref{eq:BoundForContradiction}); see Proposition
\ref{prop:QuantitativeDecayOutsideErgoregion} in Section \ref{sub:Decay-outside-the-ergoregion}.
This result is highly non-trivial and actually lies at the heart of
the proof of Theorem \ref{thm:FriedmanInstability}, with Sections
\ref{sec:FrequencyDecomposition}--\ref{sec:Carleman} being devoted
to the development of the necessary technical machinery for the proof
of Proposition \ref{prop:QuantitativeDecayOutsideErgoregion}. In
fact, the proof of Proposition \ref{prop:QuantitativeDecayOutsideErgoregion}
will be postponed until Section \ref{sec:Proof-of-Proposition}. 
\begin{rem*}
Instead of assuming (\ref{eq:BoundForContradiction}), our proof of
Theorem \ref{thm:FriedmanInstability} also applies under the weaker
assumption: 
\begin{equation}
\sup_{\text{\textgreek{t}}\ge0}\Big(\big(\log(2+\text{\textgreek{t}})\big)^{-C}\int_{\text{\textgreek{S}}_{\text{\textgreek{t}}}}J_{\text{\textgreek{m}}}^{N}(\text{\textgreek{f}})n^{\text{\textgreek{m}}}\Big)<+\infty\label{eq:LogBoundContradiction}
\end{equation}
for an arbitrary $C>0$. Furthermore, as a consequence of the discussion
in Section \ref{sub:ProofWithBoundaryConditions} (see also the remark
below Proposition \ref{prop:QuantitativeDecayOutsideErgoregion}),
the proof of Theorem \ref{thm:FriedmanInstability} also applies without
any significant change in the case when $(\mathcal{M},g)$ has a $T$-invariant
timelike boundary component $\partial_{tim}\mathcal{M}$, with $\partial_{tim}\mathcal{M}\cap\text{\textgreek{S}}$
compact and $\partial_{tim}\mathcal{M}\cap\mathcal{H}=\emptyset$,
and $\text{\textgreek{f}}$ is assumed to satisfy either Dirichlet
or Neumann boundary conditions on $\partial_{tim}\mathcal{M}$ (see
Section \ref{sub:ProofWithBoundaryConditions} for more details on
the assumptions on the geometry of $(\mathcal{M},g)$ in this case).
\end{rem*}
In Sections \ref{sub:InitialDataNegativeHigherOrderEnergy}--\ref{sub:LimitingBehaviour},
we will establish some auxiliary results concerning the behaviour
of solutions $\text{\textgreek{f}}$ to (\ref{eq:WaveEquation}),
that will be used in the Section \ref{sub:Finishing-the-proof} to
complete the proof of Theorem \ref{thm:FriedmanInstability}.

\subsection{\label{sub:InitialDataNegativeHigherOrderEnergy}Construction of
initial data on $\text{\textgreek{S}}$ with negative higher order
energy}

In this section, we will establish the following result:
\begin{lem}
\label{lem:InitialDataNegativeTEenergy}There exists a smooth initial
data set $(\text{\textgreek{f}}^{(0)},\text{\textgreek{f}}^{(1)}):\text{\textgreek{S}}\rightarrow\mathbb{C}^{2}$
supported in $\text{\textgreek{S}}\cap\mathcal{U}$ (where $\mathcal{U}\subset\mathcal{M}$
is the set described in Assumption \hyperref[Assumption 4]{A1}) so
that the function $\text{\textgreek{f}}:\mathcal{D}(\text{\textgreek{S}})\rightarrow\mathbb{C}$,
defined by solving 
\begin{equation}
\begin{cases}
\square_{g}\text{\textgreek{f}}=0 & \mbox{on }\mathcal{D}(\text{\textgreek{S}}),\\
(\text{\textgreek{f}}|_{\text{\textgreek{S}}},T\text{\textgreek{f}}|_{\text{\textgreek{S}}})=(\text{\textgreek{f}}^{(0)},\text{\textgreek{f}}^{(1)}),
\end{cases}\label{eq:InitialValueProblem}
\end{equation}
satisfies 
\begin{equation}
\int_{\text{\textgreek{S}}}J_{\text{\textgreek{m}}}^{T}(T\text{\textgreek{f}})n^{\text{\textgreek{m}}}=-1.\label{eq:NegativeInitialEnergy}
\end{equation}
\end{lem}
\begin{rem*}
Notice that the initial value problem (\ref{eq:InitialValueProblem})
is well posed, since the vector field $T$, although not everywhere
timelike, is everywhere tranversal to $\text{\textgreek{S}}$. \end{rem*}
\begin{proof}
Since $\mathcal{U}$ is an open subset of $\mathcal{M}$ intersecting
$\mathscr{E}$ (according to Assumption \hyperref[Assumption 4]{A1}),
in view of the definition (\ref{eq:Ergoregion}) of $\mathscr{E}$
we infer that there exists a point $q\in\mathcal{U}\cap\text{\textgreek{S}}$
and a contractible open neighborhood $\mathcal{V}$ of $q$ in $\mathcal{M}$
such that $T$ is strictly spacelike on $\mathcal{V}$. Therefore,
provided $\mathcal{V}$ is sufficiently small, there exists a vector
field $L$ on $\mathcal{V}$ satisfying 
\begin{equation}
g(L,L)=0,\mbox{ }g(L,T)>0\mbox{ and }\nabla L=0.\label{eq:PropertiesL}
\end{equation}
The condition $\nabla L=0$ on $\mathcal{V}$ implies that there exists
a function $w:\mathcal{V}\rightarrow\mathbb{R}$ such that 
\begin{equation}
\nabla w=L.\label{eq:GradientW}
\end{equation}

Let us fix a smooth cut-off function $\text{\textgreek{q}}:\mathcal{M}\rightarrow[0,1]$
suported in $\mathcal{V}$ such that $\text{\textgreek{q}}(q)=1$.
Then, for any $l\gg1$, the function 
\begin{equation}
\tilde{\text{\textgreek{f}}}_{l}\doteq\text{\textgreek{q}}e^{ilw}
\end{equation}
on $\mathcal{M}$ is supported in $\mathcal{V}$ and satisfies (in
view of (\ref{eq:PropertiesL}) and (\ref{eq:GradientW})) 
\begin{equation}
\square_{g}\tilde{\text{\textgreek{f}}}_{l}=\text{\textgreek{q}}l^{2}\partial_{\text{\textgreek{m}}}w\partial^{\text{\textgreek{m}}}we^{ilw}+O(l)=O(l).\label{eq:InhomogeneousApproxiamtion}
\end{equation}
Furthermore, we compute: 
\begin{align}
\int_{\text{\textgreek{S}}}J^{T}(T\tilde{\text{\textgreek{f}}}_{l})n^{\text{\textgreek{m}}} & =\int_{\text{\textgreek{S}}}\big(n(T\tilde{\text{\textgreek{f}}}_{l})\cdot T^{2}\tilde{\text{\textgreek{f}}}_{l}-\frac{1}{2}g(n,T)\partial_{\text{\textgreek{m}}}T\tilde{\text{\textgreek{f}}}_{l}\partial^{\text{\textgreek{m}}}T\tilde{\text{\textgreek{f}}}_{l}\big)\, dg_{\text{\textgreek{S}}}=\\
 & =\int_{\text{\textgreek{S}}\cap\mathcal{V}}\Big(\big(\text{\textgreek{q}}l^{2}(nw)(Tw)e^{ilw}+O(l)\big)\big(\text{\textgreek{q}}l^{2}(Tw)^{2}e^{ilw}+O(l)\big)\nonumber \\
 & \hphantom{=\int_{\text{\textgreek{S}}\cap\mathcal{V}}\Big(\big(\text{\textgreek{q}}l^{2}(nw)(Tw)}-\frac{1}{2}g(n,T)\big(\text{\textgreek{q}}l^{2}(\partial_{\text{\textgreek{m}}}w)(Tw)e^{ilw}+O(l)\big)\big(\text{\textgreek{q}}l^{2}(\partial^{\text{\textgreek{m}}}w)(Tw)e^{ilw}+O(l)\big)\Big)\, dg_{\text{\textgreek{S}}}=\nonumber \\
 & =\int_{\text{\textgreek{S}}\cap\mathcal{V}}\Big(\text{\textgreek{q}}^{2}l^{4}\big(g(n,L)(g(T,L))^{3}-\frac{1}{2}g(n,T)g(L,L)(g(T,L))^{2}\big)+O(l^{3})\Big)\, dg_{\text{\textgreek{S}}},\nonumber 
\end{align}
which, in view of (\ref{eq:PropertiesL}) (and the fact that $g(n,L)<0$),
yields:
\begin{equation}
\int_{\text{\textgreek{S}}}J^{T}(T\tilde{\text{\textgreek{f}}}_{l})n^{\text{\textgreek{m}}}=-c_{0}l^{4}+O(l^{3})\label{eq:NegativeEnergyApproximation}
\end{equation}
for some $c_{0}>0$.

Let us set 
\begin{equation}
(\text{\textgreek{f}}^{(0)},\text{\textgreek{f}}^{(1)})\doteq(\tilde{\text{\textgreek{f}}}_{l}|_{\text{\textgreek{S}}},T\tilde{\text{\textgreek{f}}}_{l}|_{\text{\textgreek{S}}}).\label{eq:SameInitialData}
\end{equation}
Note that $(\text{\textgreek{f}}^{(0)},\text{\textgreek{f}}^{(1)})$
is supported in $\mathcal{V}\cap\text{\textgreek{S}}\subset\mathcal{U}\cap\text{\textgreek{S}}$.
Then, the function 
\begin{equation}
\check{\text{\textgreek{f}}}\doteq\text{\textgreek{f}}-\tilde{\text{\textgreek{f}}}_{l},\label{eq:DefinitionPhiCheck}
\end{equation}
where $\text{\textgreek{f}}$ is defined by (\ref{eq:InitialValueProblem}),
satisfies (in view of (\ref{eq:InitialValueProblem}), (\ref{eq:InhomogeneousApproxiamtion})
and (\ref{eq:SameInitialData})): 
\begin{equation}
\begin{cases}
\square_{g}\check{\text{\textgreek{f}}}=O(l) & \mbox{on }\mathcal{D}(\text{\textgreek{S}}),\\
(\check{\text{\textgreek{f}}}|_{\text{\textgreek{S}}},T\check{\text{\textgreek{f}}}|_{\text{\textgreek{S}}})=(0,0).
\end{cases}\label{eq:InitialValueProblem-1}
\end{equation}
In view of the fact that $\check{\text{\textgreek{f}}}|_{\text{\textgreek{S}}}=0$
and $\nabla\check{\text{\textgreek{f}}}|_{\text{\textgreek{S}}}=0$
(implying also that $\nabla_{g_{\text{\textgreek{S}}}}^{2}\check{\text{\textgreek{f}}}|_{\text{\textgreek{S}}}=0$
and $\nabla_{g_{\text{\textgreek{S}}}}T\check{\text{\textgreek{f}}}|_{\text{\textgreek{S}}}=0$),
the expression of the wave operator in a coordinate chart of the form
$(t,x)$ on $\mathcal{V}$ readily yields
\begin{equation}
(\square_{g}\check{\text{\textgreek{f}}})|_{\text{\textgreek{S}}}=\big(g^{00}T^{2}\check{\text{\textgreek{f}}}\big)|_{\text{\textgreek{S}}}=\big(g^{0\text{\textgreek{m}}}\partial_{\text{\textgreek{m}}}(T\check{\text{\textgreek{f}}})\big)|_{\text{\textgreek{S}}}=\big(\frac{1}{g(n,T)}n(T\check{\text{\textgreek{f}}})\big)|_{\text{\textgreek{S}}}.\label{eq:ExpressionFromWaveEquation}
\end{equation}
Thus, from (\ref{eq:InitialValueProblem-1}), (\ref{eq:ExpressionFromWaveEquation})
and the fact that $\square_{g}\check{\text{\textgreek{f}}}$ is supported
in $\mathcal{V}$, we can readily bound 
\begin{equation}
\int_{\text{\textgreek{S}}}J_{\text{\textgreek{m}}}^{N}(T\check{\text{\textgreek{f}}})n^{\text{\textgreek{m}}}=O(l^{2}).\label{eq:UpperBoundEnergyCheck}
\end{equation}

From (\ref{eq:DefinitionPhiCheck}), a Cauchy--Schwarz inequality
implies: 
\begin{equation}
\Big|\int_{\text{\textgreek{S}}}J_{\text{\textgreek{m}}}^{T}(T\text{\textgreek{f}})n^{\text{\textgreek{m}}}-\int_{\text{\textgreek{S}}}J_{\text{\textgreek{m}}}^{T}(T\tilde{\text{\textgreek{f}}}_{l})n^{\text{\textgreek{m}}}\Big|\le C\int_{\text{\textgreek{S}}}J_{\text{\textgreek{m}}}^{N}(T\check{\text{\textgreek{f}}})n^{\text{\textgreek{m}}},
\end{equation}
and thus, in view also of (\ref{eq:NegativeEnergyApproximation})
and (\ref{eq:UpperBoundEnergyCheck}): 
\begin{equation}
\int_{\text{\textgreek{S}}}J_{\text{\textgreek{m}}}^{T}(T\text{\textgreek{f}})n^{\text{\textgreek{m}}}=-c_{0}l^{4}+O(l^{3})<0,
\end{equation}
provided $l\gg1$. Multiplying $(\text{\textgreek{f}}^{(0)},\text{\textgreek{f}}^{(1)})$
with a suitable non-zero constant, we can therefore achieve (\ref{eq:NegativeInitialEnergy}),
and therefore the proof of the Lemma is complete. 
\end{proof}

\subsection{\label{sub:Decay-outside-the-ergoregion}Decay outside the extended
ergoregion}

The following proposition, establishing decay without a rate outside
the ergoregion for solutions to equation (\ref{eq:WaveEquation}),
lies at the heart of the proof of Theorem \ref{thm:FriedmanInstability}:
\begin{prop}
\label{prop:QuantitativeDecayOutsideErgoregion}Let $\text{\textgreek{f}}:\mathcal{D}(\text{\textgreek{S}})\rightarrow\mathbb{C}$
be a smooth function satisfying (\ref{eq:WaveEquation}) with compactly
supported initial data on $\text{\textgreek{S}}$, and let us set
$\text{\textgreek{y}}=T\text{\textgreek{f}}$. Assume that the energy
bound (\ref{eq:BoundForContradiction}) holds for $\text{\textgreek{f}}$,
$\text{\textgreek{y}}$, $T\text{\textgreek{y}}$ and $T^{2}\text{\textgreek{y}}$,%
\footnote{Note that, since $T$ is a Killing field of $(\mathcal{M},g)$, the
functions $\text{\textgreek{y}}$, $T\text{\textgreek{y}}$ and $T^{2}\text{\textgreek{y}}$
also solve (\ref{eq:WaveEquation}) with compactly supported initial
data on $\text{\textgreek{S}}$. %
} i.\,e.:
\begin{equation}
\mathcal{E}[\text{\textgreek{f}}]+\mathcal{E}[\text{\textgreek{y}}]+\mathcal{E}[T\text{\textgreek{y}}]+\mathcal{E}[T^{2}\text{\textgreek{y}}]<+\infty.\label{BoundednessEnergyPsi}
\end{equation}
Then for any $0<\text{\textgreek{e}}<1$, any $\text{\textgreek{d}}_{1}>0$,
any $R,\text{\textgreek{t}}_{*}\gg1$ and any $\bar{\text{\textgreek{t}}}_{0}\gg1$,
there exists a $\text{\textgreek{t}}_{\natural}\ge\bar{\text{\textgreek{t}}}_{0}+\text{\textgreek{t}}_{*}$
depending on $\text{\textgreek{e}},\text{\textgreek{d}}_{1},R,\text{\textgreek{t}}_{*},\bar{\text{\textgreek{t}}}_{0},\mathcal{E}_{log}[\text{\textgreek{f}}],\mathcal{E}_{log}[\text{\textgreek{y}}],\mathcal{E}_{log}[T\text{\textgreek{y}}]$
and $\mathcal{E}[T^{2}\text{\textgreek{y}}]$ such that 
\begin{equation}
\sum_{j=0}^{1}\int_{(\mathcal{R}(\text{\textgreek{t}}_{\natural}-\text{\textgreek{t}}_{*},\text{\textgreek{t}}_{\natural}+\text{\textgreek{t}}_{*})\backslash\mathscr{E}_{\text{\textgreek{d}}_{1}})\cap\{r\le R\}}\big(J_{\text{\textgreek{m}}}^{N}(T^{j}\text{\textgreek{y}})N^{\text{\textgreek{m}}}+|T^{j}\text{\textgreek{y}}|^{2}\big)<\text{\textgreek{e}}\label{eq:ForLocalConvergence}
\end{equation}
(see (\ref{eq:WeightedInitialEnergy}) for the definition of the quantity
$\mathcal{E}_{log}[\cdot]$)
\end{prop}
For the proof of Proposition \ref{prop:QuantitativeDecayOutsideErgoregion},
see Section \ref{sec:Proof-of-Proposition}.
\begin{rem*}
The proof of Proposition \ref{prop:QuantitativeDecayOutsideErgoregion}
also applies when $\mathscr{E}=\emptyset$. Furthermore, in view of
the discussion in Section \ref{sub:ProofWithBoundaryConditions},
the proof of Proposition \ref{prop:QuantitativeDecayOutsideErgoregion}
in Section \ref{sec:Proof-of-Proposition} also applies in the case
when $(\mathcal{M},g)$ has a $T$-invariant timelike boundary component
$\partial_{tim}\mathcal{M}$, with $\partial_{tim}\mathcal{M}\cap\text{\textgreek{S}}$
compact and $\partial_{tim}\mathcal{M}\cap\mathcal{H}=\emptyset$,
and $\text{\textgreek{f}}$ is assumed to satisfy either Dirichlet
or Neumann boundary conditions on $\partial_{tim}\mathcal{M}$.%
\footnote{In this case, we have to assume that the double $(\widetilde{\mathcal{M}},\tilde{g})$
of $(\mathcal{M},g)$ across $\partial_{tim}\mathcal{M}$ is a globally
hyperbolic spacetime satisfying Assumptions \hyperref[Assumption 1]{G1}--\hyperref[Assumption 3]{G3}
(note that Assumption \hyperref[Assumption 4]{A1} is not necessary
for the proof of Proposition \ref{prop:QuantitativeDecayOutsideErgoregion}).%
} As a consequence, the proof of Theorem \ref{thm:FriedmanInstability}
will also apply in this case as well (all the other steps in the proof
of Theorem \ref{thm:FriedmanInstability} immediately generalise in
this case without any change).
\end{rem*}

\subsection{\label{sub:LimitingBehaviour}Limiting behaviour for solutions of
(\ref{eq:WaveEquation})}

We will need the following lemma on the behaviour of $\text{\textgreek{y}}$
asymptotically as $t\rightarrow+\infty$, following essentially from
Proposition \ref{prop:QuantitativeDecayOutsideErgoregion}:
\begin{lem}
\label{lem:DecayToATrappedSolution}Let $\text{\textgreek{f}},\text{\textgreek{y}}:\mathcal{D}(\text{\textgreek{S}})\rightarrow\mathbb{C}$
be as in the statement of Proposition \ref{prop:QuantitativeDecayOutsideErgoregion},
and let us define, for any $\text{\textgreek{t}}\ge0$, the function
$\text{\textgreek{y}}_{\text{\textgreek{t}}}:\mathcal{M}\backslash\mathcal{H}^{-}\rightarrow\mathbb{C}$
as follows:
\begin{equation}
\text{\textgreek{y}}_{\text{\textgreek{t}}}(t,x)\doteq\begin{cases}
\text{\textgreek{y}}(t+\text{\textgreek{t}},x), & t\ge-\text{\textgreek{t}},\\
0, & t<-\text{\textgreek{t}}.
\end{cases}\label{eq:TranslatedSolution}
\end{equation}
 Then, there exists an increasing sequence $\{\text{\textgreek{t}}_{n}\}_{n\in\mathbb{N}}$
of non-negative numbers and a function $\tilde{\text{\textgreek{y}}}:\mathcal{M}\backslash\mathcal{H}^{-}\rightarrow\mathbb{C}$
with $\tilde{\text{\textgreek{y}}},T\tilde{\text{\textgreek{y}}}\in H_{loc}^{1}(\mathcal{M}\backslash\mathcal{H}^{-})$,
such that $\tilde{\text{\textgreek{y}}}$ solves (\ref{eq:WaveEquation})
on $\mathcal{M}\backslash\mathcal{H}^{-}$, satisfying in addition

\begin{equation}
\int_{-\text{\textgreek{t}}_{*}}^{+\text{\textgreek{t}}_{*}}\int_{\text{\textgreek{S}}_{\text{\textgreek{t}}}}\big(J_{\text{\textgreek{m}}}^{N}(\tilde{\text{\textgreek{y}}})+J_{\text{\textgreek{m}}}^{N}(T\tilde{\text{\textgreek{y}}})\big)n^{\text{\textgreek{m}}}\Big)\, d\text{\textgreek{t}}<+\infty\mbox{ for any }\text{\textgreek{t}}_{*}>0,\label{eq:UpperBoundEnergy}
\end{equation}

\begin{equation}
\tilde{\text{\textgreek{y}}}\equiv0\mbox{ on }\mathcal{M}\backslash(\mathscr{E}_{ext}\cup\mathcal{H}^{-})\label{eq:VanishingOutOfTheErgoregion}
\end{equation}
and $(\text{\textgreek{y}}_{\text{\textgreek{t}}_{n}},T\text{\textgreek{y}}_{\text{\textgreek{t}}_{n}})\rightarrow(\tilde{\text{\textgreek{y}}},T\tilde{\text{\textgreek{y}}})$
weakly in $H_{loc}^{1}(\mathcal{M}\backslash\mathcal{H}^{-})\times H_{loc}^{1}(\mathcal{M}\backslash\mathcal{H}^{-})$
and strongly in $H_{loc}^{1}(\mathcal{M}\backslash(\mathscr{E}_{ext}\cup\mathcal{H}^{-}))\times H_{loc}^{1}(\mathcal{M}\backslash(\mathscr{E}_{ext}\cup\mathcal{H}^{-}))$
and in $L_{loc}^{2}(\mathcal{M}\backslash\mathcal{H}^{-})\times L_{loc}^{2}(\mathcal{M}\backslash\mathcal{H}^{-})$
in the following sense:

\begin{itemize}

\item{ For any compactly supported test functions $\{\text{\textgreek{z}}_{j}\}_{j=0,1}\in L^{2}(\mathcal{M}\backslash\mathcal{H}^{-})$
and compactly supported vector fields $\{X_{j}\}_{j=0,1}$ on $\mathcal{M}\backslash\mathcal{H}^{-}$
such that $|X_{j}|_{g_{ref}}\in L^{2}(\mathcal{M}\backslash\mathcal{H}^{-})$:
\begin{equation}
\lim_{n\rightarrow+\infty}\sum_{j=0}^{1}\int_{\mathcal{M}\backslash\mathcal{H}^{-}}Re\Big\{ g_{ref}\big(\nabla(T^{j}\text{\textgreek{y}}_{\text{\textgreek{t}}_{n}}-T^{j}\tilde{\text{\textgreek{y}}}),X_{j}\big)+(T^{j}\text{\textgreek{y}}_{\text{\textgreek{t}}_{n}}-T^{j}\tilde{\text{\textgreek{y}}})\text{\textgreek{z}}_{j}\Big\}\, dg=0.\label{eq:WeakConvergence}
\end{equation}
}

\item{For any compact subset $\mathcal{K}\subset\mathcal{M}\backslash\mathcal{H}^{-}$
and any $\text{\textgreek{d}}>0$: 
\begin{equation}
\lim_{n\rightarrow+\infty}\Big(\sum_{j=0}^{1}\int_{\mathcal{K}}|T^{j}\text{\textgreek{y}}_{\text{\textgreek{t}}_{n}}-T^{j}\tilde{\text{\textgreek{y}}}|^{2}\, dg+\sum_{j=0}^{1}\int_{\mathcal{K}\backslash\mathscr{E}_{\text{\textgreek{d}}}}|\nabla(T^{j}\text{\textgreek{y}}_{\text{\textgreek{t}}_{n}})-\nabla(T^{j}\tilde{\text{\textgreek{y}}})|_{g_{ref}}^{2}\, dg\Big)=0.\label{eq:StrongConvergence}
\end{equation}
}

\end{itemize}\end{lem}
\begin{proof}
Let us fix four sequences of positive numbers $\{\text{\textgreek{e}}_{n}\}_{n\in\mathbb{N}}$,
$\{\text{\textgreek{d}}_{n}\}_{n\in\mathbb{N}}$, $\{R_{n}\}_{n\in\mathbb{N}}$
and $\{\text{\textgreek{t}}_{n}^{*}\}_{n\in\mathbb{N}}$ such that
$\text{\textgreek{e}}_{n},\text{\textgreek{d}}_{n}\rightarrow0$ and
$R_{n},\text{\textgreek{t}}_{n}^{*}\rightarrow+\infty$ as $n\rightarrow+\infty$.
We then define the sequence $\{\text{\textgreek{t}}_{n}\}_{n\in\mathbb{N}}$
inductively: Setting $\text{\textgreek{t}}_{0}=0$, $\text{\textgreek{t}}_{n}$
is defined for any $n\ge1$ as the value $\text{\textgreek{t}}_{\natural}>0$
from Proposition \ref{prop:QuantitativeDecayOutsideErgoregion} for
$\text{\textgreek{e}}_{n}$ in place of $\text{\textgreek{e}}$, $\text{\textgreek{d}}_{n}$
in place of $\text{\textgreek{d}}$, $R_{n}$ in place of $R$, $\text{\textgreek{t}}_{n}^{*}$
in place of $\text{\textgreek{t}}_{*}$ and $\text{\textgreek{t}}_{n-1}$
in place of $\bar{\text{\textgreek{t}}}_{0}$ (notice that the last
condition guarantees that $\text{\textgreek{t}}_{n}$ is an increasing
sequence). Then, Proposition \ref{prop:QuantitativeDecayOutsideErgoregion}
applied for the pair $(\text{\textgreek{y}},T\text{\textgreek{y}})$
implies that the pair $(\text{\textgreek{y}}_{\text{\textgreek{t}}_{n}},T\text{\textgreek{y}}_{\text{\textgreek{t}}_{n}})$
(which is merely a $\text{\textgreek{t}}_{n}$-translate of $(\text{\textgreek{y}},T\text{\textgreek{y}})$
in the region $\{t\ge-\text{\textgreek{t}}_{n}\}$) satisfies the
following estimate for any $n\in\mathbb{N}$:
\begin{equation}
\sum_{j=0}^{1}\int_{-\text{\textgreek{t}}_{n}^{*}}^{+\text{\textgreek{t}}_{n}^{*}}\Big(\int_{(\text{\textgreek{S}}_{t}\backslash\mathscr{E}_{\text{\textgreek{d}}_{n}})\cap\{r\le R_{n}\}}\big(J_{\text{\textgreek{m}}}^{N}(T^{j}\text{\textgreek{y}}_{\text{\textgreek{t}}_{n}})n^{\text{\textgreek{m}}}+|T^{j}\text{\textgreek{y}}_{\text{\textgreek{t}}_{n}}|^{2}\big)\Big)\, dt\big)<\text{\textgreek{e}}_{n}.\label{eq:ForLocalConvergence-1}
\end{equation}

In view of the bounds (\ref{BoundednessEnergyPsi}) and \ref{eq:ForLocalConvergence-1},
as well as the Poincare-type inequality 
\begin{equation}
\int_{\mathcal{R}(\bar{\text{\textgreek{t}}}_{1},\bar{\text{\textgreek{t}}}_{2})\cap\{r\le R\}}|\text{\textgreek{y}}_{\text{\textgreek{t}}_{n}}|^{2}\le CR^{2}\int_{\mathcal{R}(\bar{\text{\textgreek{t}}}_{1},\bar{\text{\textgreek{t}}}_{2})\cap\{r\le2R\}}J_{\text{\textgreek{m}}}^{N}(\text{\textgreek{y}}_{\text{\textgreek{t}}_{n}})n^{\text{\textgreek{m}}}+C\int_{\mathcal{R}(\bar{\text{\textgreek{t}}}_{1},\bar{\text{\textgreek{t}}}_{2})\cap\{R\le r\le2R\}}|\text{\textgreek{y}}_{\text{\textgreek{t}}_{n}}|^{2}
\end{equation}
holding for any $\bar{\text{\textgreek{t}}}_{1}\le\bar{\text{\textgreek{t}}}_{2}$,
we infer that, for any compact subset $\mathcal{K}$ of $\mathcal{M}\backslash\mathcal{H}^{-}$,
setting 
\begin{equation}
n_{0}(\mathcal{K})=\min\Big\{ n\in\mathbb{N}:\,\mathcal{K}\mbox{ is contained in the set }\big\{\max\{-\text{\textgreek{t}}_{n},-\text{\textgreek{t}}_{n}^{*}\}<t<\text{\textgreek{t}}_{n}^{*}\}\cap\{r\le R_{n}\}\big\}\Big\},
\end{equation}
there exists a $C=C_{\mathcal{K}}$ such that:
\begin{equation}
\sup_{n\ge n_{0}(\mathcal{K})}\Big(\sum_{j=0}^{1}\int_{\mathcal{K}}\big(J_{\text{\textgreek{m}}}^{N}(T^{j}\text{\textgreek{y}}_{\text{\textgreek{t}}_{n}})N^{\text{\textgreek{m}}}+|T^{j}\text{\textgreek{y}}_{\text{\textgreek{t}}_{n}}|^{2}\big)\le C_{\mathcal{K}}(\mathcal{E}[\text{\textgreek{y}}]+\mathcal{E}[T\text{\textgreek{y}}])+\sup_{n\ge n_{0}(\mathcal{K})}\text{\textgreek{e}}_{n}<+\infty.\label{eq:UpperBoundEnergies}
\end{equation}

For any compact $\mathcal{K}\subset\mathcal{M}\backslash\mathcal{H}^{-}$,
Rellich--Kondrachov's theorem yields that the embedding $H^{1}(\mathcal{K})\times H^{1}(\mathcal{K})\hookrightarrow L^{2}(\mathcal{K})\times L^{2}(\mathcal{K})$
is compact. Thus, (\ref{eq:UpperBoundEnergies}) implies that for
any compact $\mathcal{K}\subset\mathcal{M}\backslash\mathcal{H}^{-}$
and any infinite subset $A\subseteq\mathbb{N}$, there exists an infinite
subset $B_{\mathcal{K},A}\subseteq A$ such that the subsequence $\{(\text{\textgreek{y}}_{\text{\textgreek{t}}_{n}},T\text{\textgreek{y}}_{\text{\textgreek{t}}_{n}})\}_{n\in B_{\mathcal{K},A}}$
of $\{(\text{\textgreek{y}}_{\text{\textgreek{t}}_{n}},T\text{\textgreek{y}}_{\text{\textgreek{t}}_{n}})\}_{n\in\mathbb{N}}$
converges weakly in the $H^{1}(\mathcal{K})\times H^{1}(\mathcal{K})$
norm and strongly in the $L^{2}(\mathcal{K})\times L^{2}(\mathcal{K})$
to some limit pair $(\tilde{\text{\textgreek{y}}}_{\mathcal{K}},\grave{\text{\textgreek{y}}}_{\mathcal{K}})$
in $H^{1}(\mathcal{K})\times H^{1}(\mathcal{K})$. Note that in this
case, we necessarily have $\grave{\text{\textgreek{y}}}_{\mathcal{K}}=T\tilde{\text{\textgreek{y}}}_{\mathcal{K}}$
in the sense of distributions.

Let $\{\mathcal{K}_{m}\}_{m\in\mathbb{N}}$ be a sequence of compact
subsets of $\mathcal{M}\backslash\mathcal{H}^{-}$ such that $\mathcal{K}_{m}\subset\mathcal{K}_{m+1}$
and $\cup_{m\in\mathbb{N}}\mathcal{K}_{m}=\mathcal{M}\backslash\mathcal{H}^{-}$.
Then, setting $A_{-1}=B_{\mathcal{K}_{0},\mathbb{N}}$, $A_{m}=B_{\mathcal{K}_{m},A_{m-1}}$
for $m\in\mathbb{N}$, and defining recursively 
\begin{equation}
A=\cup_{m\in\mathbb{N}}\big\{\min\big(A_{m}\backslash\{n:\, n<m\}\big)\big\},
\end{equation}
we infer that there exists a pair $(\tilde{\text{\textgreek{y}}},T\tilde{\text{\textgreek{y}}})\in H_{loc}^{1}(\mathcal{M}\backslash\mathcal{H}^{-})\times H_{loc}^{1}(\mathcal{M}\backslash\mathcal{H}^{-})$
such that the subsequence $\{\text{\textgreek{y}}_{\text{\textgreek{t}}_{n}},T\text{\textgreek{y}}_{\text{\textgreek{t}}_{n}}\}_{n\in A}$
satisfies (\ref{eq:WeakConvergence}) and, for any compact $\mathcal{K}\subset\mathcal{M}\backslash\mathcal{H}^{-}$
(after permanently renumbering the indices of $\{\text{\textgreek{y}}_{\text{\textgreek{t}}_{n}}\}_{n\in A}$
through a map $\mathbb{N}\rightarrow A$) : 
\begin{equation}
\lim_{n\rightarrow+\infty}\sum_{j=0}^{1}\int_{\mathcal{K}}|T^{j}\text{\textgreek{y}}_{\text{\textgreek{t}}_{n}}-T^{j}\tilde{\text{\textgreek{y}}}|^{2}\, dg=0.\label{eq:StrongConvergence-1}
\end{equation}

Since the functions $\text{\textgreek{y}}_{\text{\textgreek{t}}_{n}}$
solve (\ref{eq:WaveEquation}) on $\{t>-\text{\textgreek{t}}_{n}\}$,
$\tilde{\text{\textgreek{y}}}$ also solves (\ref{eq:WaveEquation})
on $\mathcal{M}\backslash\mathcal{H}^{-}$ in the sense of distributions,
in view of (\ref{eq:WeakConvergence}). Furthermore, in view of (\ref{BoundednessEnergyPsi}),
we can bound for any $\text{\textgreek{t}}_{*}>0$ 
\begin{equation}
\sup_{n\in\mathbb{N}}\Big\{\sum_{j=0}^{1}\int_{\max\{-\text{\textgreek{t}}_{*},-\text{\textgreek{t}}_{n}\}}^{\text{\textgreek{t}}_{*}}\Big(\int_{\text{\textgreek{S}}_{\text{\textgreek{t}}}}J_{\text{\textgreek{m}}}^{N}(T^{j}\text{\textgreek{y}})n^{\text{\textgreek{m}}}\Big)\, d\text{\textgreek{t}}\Big\}<+\infty
\end{equation}
and, thus, (\ref{eq:UpperBoundEnergy}) holds. The identity (\ref{eq:VanishingOutOfTheErgoregion})
follows by letting $n\rightarrow+\infty$ in (\ref{eq:ForLocalConvergence-1}).
Finally, (\ref{eq:StrongConvergence}) follows from (\ref{eq:VanishingOutOfTheErgoregion}),
(\ref{eq:ForLocalConvergence-1}) and (\ref{eq:StrongConvergence-1}).
\end{proof}

\subsection{\label{sub:Finishing-the-proof}Finishing the proof}

Let us assume, for the sake of contradiction, that any smooth solution
$\text{\textgreek{f}}$ to (\ref{eq:WaveEquation}) on $\mathcal{D}(\text{\textgreek{S}})$
with compactly supported initial data on $\text{\textgreek{S}}$ satisfies
(\ref{eq:BoundForContradiction}). 

Let $\text{\textgreek{f}}:\mathcal{D}(\text{\textgreek{S}})\rightarrow\mathbb{C}$
be as in the statement of Lemma \ref{lem:InitialDataNegativeTEenergy},
and let us set 
\begin{equation}
\text{\textgreek{y}}=T\text{\textgreek{f}}.
\end{equation}
In view of Lemma \ref{lem:InitialDataNegativeTEenergy}, $(\text{\textgreek{y}},T\text{\textgreek{y}})|_{\text{\textgreek{S}}}$
is smooth and compactly supported in $\mathcal{U}\cap\text{\textgreek{S}}$,
and moreover 
\begin{equation}
\int_{\text{\textgreek{S}}}J_{\text{\textgreek{m}}}^{T}(\text{\textgreek{y}})n^{\text{\textgreek{m}}}=-1.\label{eq:NegativeEnergyPsi}
\end{equation}
Let $\{\text{\textgreek{t}}_{n}\}_{n\in\mathbb{N}}$ be the sequence
defined by Lemma \ref{lem:DecayToATrappedSolution}, and let $\text{\textgreek{y}}_{\text{\textgreek{t}}_{n}},\tilde{\text{\textgreek{y}}}:\mathcal{M}\backslash\mathcal{H}^{-}\rightarrow\mathbb{C}$
be the functions defined by Lemma \ref{lem:DecayToATrappedSolution}. 

We will make use of the following identity, appearing also in \cite{Friedman1978},
holding for any acausal, inextendible and piecewise smooth hypersurface
$\mathcal{S}\subset\mathcal{M}\backslash\mathcal{H}^{-}$ such that
$T$ is everywhere transversal to $\mathcal{S}$ and any smooth function
$\text{\textgreek{f}}_{1}:\mathcal{M}\backslash\mathcal{H}^{-}\rightarrow\mathbb{C}$
such that $supp(\text{\textgreek{f}}_{1})\cap\mathcal{S}$ is compact
and $supp(\text{\textgreek{f}}_{1})\cap\mathcal{S}\cap\mathcal{H}^{+}=\emptyset$:
\begin{equation}
\int_{\mathcal{S}}J_{\text{\textgreek{m}}}^{T}(\text{\textgreek{f}}_{1})n_{\mathcal{S}}^{\text{\textgreek{m}}}=\int_{\mathcal{S}}Re\big\{ n_{\mathcal{S}}\text{\textgreek{f}}_{1}\cdot T\bar{\text{\textgreek{f}}}_{1}-\text{\textgreek{f}}_{1}\cdot n_{\mathcal{S}}(T\bar{\text{\textgreek{f}}}_{1})\big\}\, dg_{\mathcal{S}}-\int_{\mathcal{S}}Re\big\{\text{\textgreek{f}}_{1}\square_{g}\bar{\text{\textgreek{f}}}_{1}\big\} g(n_{\mathcal{S}},T)\, dg_{\mathcal{S}},\label{eq:Energy-SymplecticForm}
\end{equation}
where $n_{\mathcal{S}}$ is the future directed unit normal to $\mathcal{S}$. 

\medskip{}

\noindent \emph{Proof of (\ref{eq:Energy-SymplecticForm}). }One way
to obtain (\ref{eq:Energy-SymplecticForm}) is the following: Since
$supp(\text{\textgreek{f}}_{1})\cap\mathcal{S}$ is compact and $supp(\text{\textgreek{f}}_{1})\cap\mathcal{S}\cap\mathcal{H}^{+}=\emptyset$,
we can assume without loss of generality (by changing $\text{\textgreek{f}}_{1}$
away from $\mathcal{S}$ if necessary) that $\text{\textgreek{f}}_{1}$
has compact support in $\mathcal{M}\backslash(\mathcal{H}^{+}\cup\mathcal{H}^{-})$.
Then, integrating the identity 
\begin{equation}
-2Re\big\{ T\text{\textgreek{f}}_{1}\square_{g}\bar{\text{\textgreek{f}}}_{1}\big\}=-Re\big\{ T\text{\textgreek{f}}_{1}\square_{g}\bar{\text{\textgreek{f}}}_{1}-\text{\textgreek{f}}_{1}\square_{g}(T\bar{\text{\textgreek{f}}}_{1})+T(\text{\textgreek{f}}_{1}\square_{g}\bar{\text{\textgreek{f}}}_{1})\big\}\label{eq:BasicIdentity}
\end{equation}
 over $J^{-}(\mathcal{S})$, we readily obtain: 
\begin{equation}
-2\int_{J^{-}(\mathcal{S})}Re\big\{ T\text{\textgreek{f}}_{1}\square_{g}\bar{\text{\textgreek{f}}}_{1}\big\}\, dg=-\int_{J^{-}(\mathcal{S})}Re\big\{ T\text{\textgreek{f}}_{1}\square_{g}\bar{\text{\textgreek{f}}}_{1}-\text{\textgreek{f}}_{1}\square_{g}(T\bar{\text{\textgreek{f}}}_{1})\big\}\, dg+\int_{\mathcal{S}}Re\big\{\text{\textgreek{f}}_{1}\square_{g}\bar{\text{\textgreek{f}}}_{1}\big\} g(n_{\mathcal{S}},T)\, dg_{\mathcal{S}}.\label{eq:IdentityForSymplectic}
\end{equation}
Using the identities 
\begin{equation}
-2\int_{J^{-}(\mathcal{S})}Re\big\{ T\text{\textgreek{f}}_{1}\square_{g}\bar{\text{\textgreek{f}}}_{1}\big\}\, dg=\int_{\mathcal{S}}J_{\text{\textgreek{m}}}^{T}(\text{\textgreek{f}})n_{\mathcal{S}}^{\text{\textgreek{m}}}
\end{equation}
and
\[
-\int_{J^{-}(\mathcal{S})}Re\big\{ T\text{\textgreek{f}}_{1}\square_{g}\bar{\text{\textgreek{f}}}_{1}-\text{\textgreek{f}}_{1}\square_{g}(T\bar{\text{\textgreek{f}}}_{1})\big\}\, dg=\int_{\mathcal{S}}Re\big\{ n_{\mathcal{S}}\text{\textgreek{f}}_{1}\cdot T\bar{\text{\textgreek{f}}}_{1}-\text{\textgreek{f}}_{1}\cdot n_{\mathcal{S}}(T\bar{\text{\textgreek{f}}}_{2})\big\}\, dg_{\mathcal{S}}
\]
(holding because of the assumption that $\text{\textgreek{f}}_{1}$
has compact support in $\mathcal{M}\backslash(\mathcal{H}^{+}\cup\mathcal{H}^{-})$),
we finally obtain (\ref{eq:Energy-SymplecticForm}).

\medskip{}

We will also introduce the following (indefinite) inner product on
the hypersurfaces $\text{\textgreek{S}}_{\text{\textgreek{t}}}$:
For any two functions $\text{\textgreek{f}}_{1},\text{\textgreek{f}}_{2}:\mathcal{M}\backslash\mathcal{H}^{-}\rightarrow\mathbb{C}$
such that for any $\text{\textgreek{t}}_{*}>0$: 
\[
\sup_{\text{\textgreek{t}}\in[-\text{\textgreek{t}}_{*},\text{\textgreek{t}}_{*}]}\sum_{j=1}^{2}\int_{\text{\textgreek{S}}_{\text{\textgreek{t}}}}\big(J_{\text{\textgreek{m}}}^{N}(\text{\textgreek{f}}_{j})+J_{\text{\textgreek{m}}}^{N}(T\text{\textgreek{f}}_{j})\big)n^{\text{\textgreek{m}}}<+\infty
\]
 and at least one of them has compact support in space (i.\,e.~for
any $\text{\textgreek{t}}_{*}>0$, its support in $\{-\text{\textgreek{t}}_{*}\le t\le\text{\textgreek{t}}_{*}\}$
is compact), we will define for any $\text{\textgreek{t}}\in\mathbb{R}$:
\begin{equation}
\left\langle \text{\textgreek{f}}_{1},\text{\textgreek{f}}_{2}\right\rangle _{T,\text{\textgreek{t}}}=\frac{1}{2}\int_{\text{\textgreek{S}}_{\text{\textgreek{t}}}}Re\Big\{\big(n_{\mathcal{S}}\text{\textgreek{f}}_{1}\cdot T\bar{\text{\textgreek{f}}}_{2}+n_{\mathcal{S}}\text{\textgreek{f}}_{2}\cdot T\bar{\text{\textgreek{f}}}_{1}\big)-\big(\text{\textgreek{f}}_{1}\cdot n_{\mathcal{S}}(T\bar{\text{\textgreek{f}}}_{2})+\text{\textgreek{f}}_{2}\cdot n_{\mathcal{S}}(T\bar{\text{\textgreek{f}}}_{1})\big)\Big\}.\label{eq:TInnerProductSymplecic}
\end{equation}
Note that, if both $\text{\textgreek{f}}_{1}$ and $\text{\textgreek{f}}_{2}$
solve equation (\ref{eq:WaveEquation}) and at least one of them is
supported away from $\mathcal{H}^{+}$, then for any $\text{\textgreek{t}}_{1}\le\text{\textgreek{t}}_{2}$
the following identity holds: 
\begin{equation}
\left\langle \text{\textgreek{f}}_{1},\text{\textgreek{f}}_{2}\right\rangle _{T,\text{\textgreek{t}}_{1}}=\left\langle \text{\textgreek{f}}_{1},\text{\textgreek{f}}_{2}\right\rangle _{T,\text{\textgreek{t}}_{2}}.\label{eq:InnerProductIdentity}
\end{equation}
The equality (\ref{eq:TInnerProductSymplecic}) readily follows after
integrating the identity 
\begin{equation}
\frac{1}{2}Re\Big\{\big(\square_{g}\text{\textgreek{f}}_{1}T\bar{\text{\textgreek{f}}}_{2}+\square_{g}\text{\textgreek{f}}_{2}T\bar{\text{\textgreek{f}}}_{1}\big)-\big(\text{\textgreek{f}}_{1}\square_{g}(T\bar{\text{\textgreek{f}}}_{2})+\text{\textgreek{f}}_{2}\square_{g}(T\bar{\text{\textgreek{f}}}_{1})\big)\Big\}=0
\end{equation}
 over $\mathcal{R}(\text{\textgreek{t}}_{1},\text{\textgreek{t}}_{2})$.
\begin{rem*}
Note that, in the case when $\text{\textgreek{f}}_{1}$ and $\text{\textgreek{f}}_{2}$
solve equation (\ref{eq:WaveEquation}) and at least one of them is
supported away from $\mathcal{H}^{+}$, the expression (\ref{eq:TInnerProductSymplecic})
is the inner product of $\text{\textgreek{f}}_{1},\text{\textgreek{f}}_{2}$
associated to the $\int_{\text{\textgreek{S}}_{\text{\textgreek{t}}}}J_{\text{\textgreek{m}}}^{T}(\cdot)n^{\text{\textgreek{m}}}$
``norm'', in view of (\ref{eq:Energy-SymplecticForm}). Thus, (\ref{eq:InnerProductIdentity})
is a consequence of the conservation of the $T$-energy flux.
\end{rem*}
For any $\text{\textgreek{t}}\ge0$, the $T$-energy identity for
$\text{\textgreek{y}}$ in the region $\mathcal{R}(0,\text{\textgreek{t}})$
combined with (\ref{eq:NegativeEnergyPsi}) yields:
\begin{equation}
\int_{\text{\textgreek{S}}_{\text{\textgreek{t}}}}J_{\text{\textgreek{m}}}^{T}(\text{\textgreek{y}})n^{\text{\textgreek{m}}}+\int_{\mathcal{H}^{+}\cap\mathcal{R}(0,\text{\textgreek{t}})}J_{\text{\textgreek{m}}}^{T}(\text{\textgreek{y}})n_{\mathcal{H}^{+}}^{\text{\textgreek{m}}}=-1.\label{eq:EnergyIdentity}
\end{equation}
Since $T$ is causal on $\mathcal{M}\backslash\mathscr{E}$, we can
bound for any $\text{\textgreek{t}}\ge0$ and any $\text{\textgreek{d}}>0$:
\begin{equation}
\int_{\text{\textgreek{S}}_{\text{\textgreek{t}}}\backslash\mathscr{E}_{\text{\textgreek{d}}}}J_{\text{\textgreek{m}}}^{T}(\text{\textgreek{y}})n^{\text{\textgreek{m}}}+\int_{\mathcal{H}^{+}\cap\mathcal{R}(0,\text{\textgreek{t}})}J_{\text{\textgreek{m}}}^{T}(\text{\textgreek{y}})n_{\mathcal{H}^{+}}^{\text{\textgreek{m}}}\ge0.\label{eq:PositivityOfFlux}
\end{equation}
Therefore, (\ref{eq:EnergyIdentity}) and (\ref{eq:PositivityOfFlux})
imply that for any $\text{\textgreek{t}}\ge0$, $\text{\textgreek{d}}>0$:
\begin{equation}
\int_{\text{\textgreek{S}}_{\text{\textgreek{t}}}\cap\mathscr{E}_{\text{\textgreek{d}}}}J_{\text{\textgreek{m}}}^{T}(\text{\textgreek{y}})n^{\text{\textgreek{m}}}\le-1.\label{eq:LowerboundErgoregion}
\end{equation}
Since the functions $\text{\textgreek{y}}_{\text{\textgreek{t}}_{n}}$
satisfy (\ref{eq:TranslatedSolution}), from (\ref{eq:LowerboundErgoregion})
we obtain for any $\text{\textgreek{d}}>0$, any $\text{\textgreek{t}}>-\text{\textgreek{t}}_{n}$,
and any $n\in\mathbb{N}$:
\begin{equation}
\int_{\text{\textgreek{S}}_{\text{\textgreek{t}}}\cap\mathscr{E}_{\text{\textgreek{d}}}}J_{\text{\textgreek{m}}}^{T}(\text{\textgreek{y}}_{\text{\textgreek{t}}_{n}})n^{\text{\textgreek{m}}}\le-1.\label{eq:NegativeEnergySequence}
\end{equation}

Let $\text{\textgreek{q}}:\mathcal{M}\backslash\mathcal{H}^{-}\rightarrow[0,1]$
be a smooth function of compact support such that $\text{\textgreek{q}}\equiv1$
on $\mathcal{R}(-1,2)\cap\mathscr{E}_{\text{\textgreek{d}}_{0}}$
for some $0<\text{\textgreek{d}}_{0}<1$ and $supp(\text{\textgreek{q}})\cap\mathcal{H}^{+}=\emptyset$.
Applying the identity (\ref{eq:Energy-SymplecticForm}) for the function
$\text{\textgreek{q}}\text{\textgreek{y}}_{\text{\textgreek{t}}_{n}}$,
and using the fact that $\text{\textgreek{y}}_{\text{\textgreek{t}}_{n}}$
solves (\ref{eq:WaveEquation}), we obtain for any $n\in\mathbb{N}$
and any $0<\text{\textgreek{t}}_{0}\le1$: 
\begin{align}
\int_{0}^{\text{\textgreek{t}}_{0}}\Big(\int_{\text{\textgreek{S}}_{s}}J_{\text{\textgreek{m}}}^{T}(\text{\textgreek{q}}\text{\textgreek{y}}_{\text{\textgreek{t}}_{n}})n^{\text{\textgreek{m}}}\Big)\, ds=\int_{0}^{\text{\textgreek{t}}_{0}}\Big(\int_{\text{\textgreek{S}}_{s}} & Re\big\{ n(\text{\textgreek{q}}\text{\textgreek{y}}_{\text{\textgreek{t}}_{n}})T(\text{\textgreek{q}}\bar{\text{\textgreek{y}}}_{\text{\textgreek{t}}_{n}})-(\text{\textgreek{q}}\text{\textgreek{y}}_{\text{\textgreek{t}}_{n}})n_{\mathcal{S}}(T(\text{\textgreek{q}}\bar{\text{\textgreek{y}}}_{\text{\textgreek{t}}_{n}}))\big\}\, dg_{\text{\textgreek{S}}}\Big)\, ds-\label{eq:Energy-SymplecticForm-1}\\
 & -\int_{0}^{\text{\textgreek{t}}_{0}}\Big(\int_{\text{\textgreek{S}}_{s}}Re\big\{\text{\textgreek{q}}\text{\textgreek{y}}_{\text{\textgreek{t}}_{n}}(2\nabla^{\text{\textgreek{m}}}\text{\textgreek{q}}\nabla_{\text{\textgreek{m}}}\bar{\text{\textgreek{y}}}_{\text{\textgreek{t}}_{n}}+(\square_{g}\text{\textgreek{q}})\bar{\text{\textgreek{y}}}_{\text{\textgreek{t}}_{n}}\big\} g(n,T)\, dg_{\text{\textgreek{S}}}\Big)\, ds.\nonumber 
\end{align}
In view of (\ref{eq:NegativeEnergySequence}) and the fact that $\text{\textgreek{q}}\equiv1$
on $\mathcal{R}(-1,2)\cap\mathscr{E}_{\text{\textgreek{d}}_{0}}$,
(\ref{eq:Energy-SymplecticForm-1}) yields: 
\begin{equation}
\int_{0}^{\text{\textgreek{t}}_{0}}\Big(\int_{\text{\textgreek{S}}_{s}\cap\mathscr{E}_{\text{\textgreek{d}}_{0}}}Re\big\{ n\text{\textgreek{y}}_{\text{\textgreek{t}}_{n}}T\bar{\text{\textgreek{y}}}_{\text{\textgreek{t}}_{n}}-\text{\textgreek{y}}_{\text{\textgreek{t}}_{n}}n_{\mathcal{S}}(T\bar{\text{\textgreek{y}}}_{\text{\textgreek{t}}_{n}})\big\}\, dg_{\text{\textgreek{S}}}\Big)\, ds\le-\text{\textgreek{t}}_{0}+C\sum_{j=0}^{1}\int_{supp(\text{\textgreek{q}})\backslash\mathscr{E}_{\text{\textgreek{d}}_{0}}}\big(|\nabla T^{j}\text{\textgreek{y}}_{\text{\textgreek{t}}_{n}}|_{g_{ref}}^{2}+|T^{j}\text{\textgreek{y}}_{\text{\textgreek{t}}_{n}}|^{2}\big)\, dg.\label{eq:FornegativeEnergySequenceLimit}
\end{equation}
Let us examine the properties of (\ref{eq:FornegativeEnergySequenceLimit})
as $n\rightarrow+\infty$. 

\begin{enumerate}

\item In view of (\ref{eq:StrongConvergence}) and the fact that
$supp(\text{\textgreek{q}})$ is compact, the right hand side of (\ref{eq:FornegativeEnergySequenceLimit})
converges to $-\text{\textgreek{t}}_{0}$ as $n\rightarrow+\infty$.

\item For any compact subset $\mathcal{K}\subset\mathcal{M}\backslash\mathcal{H}^{+}$
and any pair of sequences $(\text{\textgreek{f}}_{n}^{(1)},\text{\textgreek{f}}_{n}^{(2)})_{n\in\mathbb{N}}\in L^{2}(\mathcal{K})\times L^{2}(\mathcal{K})$
such that $\sup_{n}||\text{\textgreek{f}}_{n}^{(1)}||_{L^{2}(\mathcal{K})}<+\infty$,
$\text{\textgreek{f}}_{n}^{(1)}\rightarrow\text{\textgreek{f}}^{(1)}$
weakly in $L^{2}(\mathcal{K})$ and $\text{\textgreek{f}}_{n}^{(2)}\rightarrow\text{\textgreek{f}}^{(2)}$
strongly in $L^{2}(\mathcal{K})$, one readily obtains that 
\begin{equation}
\lim_{n\rightarrow+\infty}\int_{\mathcal{K}}\text{\textgreek{f}}_{n}^{(1)}\text{\textgreek{f}}_{n}^{(2)}\, dg=\int_{\mathcal{K}}\text{\textgreek{f}}^{(1)}\text{\textgreek{f}}^{(2)}\, dg.
\end{equation}
Therefore, (\ref{BoundednessEnergyPsi}), (\ref{eq:WeakConvergence})
and (\ref{eq:StrongConvergence}) imply that: 
\begin{equation}
\lim_{n\rightarrow+\infty}\int_{0}^{\text{\textgreek{t}}_{0}}\Big(\int_{\text{\textgreek{S}}_{s}\cap\mathscr{E}_{\text{\textgreek{d}}_{0}}}Re\big\{ n\text{\textgreek{y}}_{\text{\textgreek{t}}_{n}}T\bar{\text{\textgreek{y}}}_{\text{\textgreek{t}}_{n}}-\text{\textgreek{y}}_{\text{\textgreek{t}}_{n}}n_{\mathcal{S}}(T\bar{\text{\textgreek{y}}}_{\text{\textgreek{t}}_{n}})\big\}\, dg_{\text{\textgreek{S}}}\Big)\, ds=\int_{0}^{\text{\textgreek{t}}_{0}}\Big(\int_{\text{\textgreek{S}}_{s}\cap\mathscr{E}_{\text{\textgreek{d}}_{0}}}Re\big\{ n\tilde{\text{\textgreek{y}}}T\bar{\tilde{\text{\textgreek{y}}}}-\tilde{\text{\textgreek{y}}}n_{\mathcal{S}}(T\bar{\tilde{\text{\textgreek{y}}}})\big\}\, dg_{\text{\textgreek{S}}}\Big)\, ds.
\end{equation}

\end{enumerate}

Thus, taking the limit $n\rightarrow+\infty$ in (\ref{eq:FornegativeEnergySequenceLimit}),
we obtain for any $0<\text{\textgreek{t}}_{0}\le1$: 
\begin{equation}
\int_{0}^{\text{\textgreek{t}}_{0}}\Big(\int_{\text{\textgreek{S}}_{s}\cap\mathscr{E}_{\text{\textgreek{d}}_{0}}}Re\big\{ n\tilde{\text{\textgreek{y}}}T\bar{\tilde{\text{\textgreek{y}}}}-\tilde{\text{\textgreek{y}}}n_{\mathcal{S}}(T\bar{\tilde{\text{\textgreek{y}}}})\big\}\, dg_{\text{\textgreek{S}}}\Big)\, ds\le-\text{\textgreek{t}}_{0}.\label{eq:NegativeEnergyTrappedSolution}
\end{equation}

According to Lemma \ref{lem:DecayToATrappedSolution}, $\tilde{\text{\textgreek{y}}}$
belongs to $H_{loc}^{1}(\mathcal{M}\backslash\mathcal{H}^{-})$ and
vanishes outside $\mathscr{E}_{ext}$, and, thus, Assumption \hyperref[Assumption 4]{A1}
implies that 
\begin{equation}
\tilde{\text{\textgreek{y}}}\equiv0\mbox{ on }\mathcal{U}.\label{eq:SupportTrappedFunction}
\end{equation}
Since $(\text{\textgreek{y}},T\text{\textgreek{y}})|_{\text{\textgreek{S}}}$
is compactly supported in $\mathcal{U}\cap\text{\textgreek{S}}$ and
$\mathcal{U}$ is open, in view of the finite speed of propagation
property of equation (\ref{eq:WaveEquation}), there exists some $0<\text{\textgreek{t}}_{0}\le1$
(depending on the support of $\text{\textgreek{y}}$ on $\text{\textgreek{S}}\cap\mathcal{U}$),
such that for all $0\le\bar{\text{\textgreek{t}}}\le\text{\textgreek{t}}_{0}$:
\begin{equation}
(\text{\textgreek{y}},T\text{\textgreek{y}})=(0,0)\mbox{ on }\text{\textgreek{S}}_{\bar{\text{\textgreek{t}}}}\backslash\mathcal{U}.\label{eq:FromFiniteSpeedOfPropagation}
\end{equation}
In view of the fact that $\mathcal{U}$ is translation invariant,
(\ref{eq:TInnerProductSymplecic}), (\ref{eq:SupportTrappedFunction})
and (\ref{eq:FromFiniteSpeedOfPropagation}) imply that for any $\text{\textgreek{t}}\in\mathbb{R}$:
\begin{equation}
\int_{0}^{\text{\textgreek{t}}_{0}}\left\langle \text{\textgreek{y}},\mathcal{F}_{\text{\textgreek{t}}}^{*}\tilde{\text{\textgreek{y}}}\right\rangle _{T,\bar{\text{\textgreek{t}}}}\, d\bar{\text{\textgreek{t}}}=0\label{eq:ZeroInnerProduct}
\end{equation}
(the expression (\ref{eq:ZeroInnerProduct}) is well defined, in view
of (\ref{eq:UpperBoundEnergy})), where 
\begin{equation}
\mathcal{F}_{\text{\textgreek{t}}}^{*}\tilde{\text{\textgreek{y}}}(t,x)\doteq\tilde{\text{\textgreek{y}}}(t+\text{\textgreek{t}},x).\label{eq:TranslatedTildePsi}
\end{equation}
In view of Assumption \hyperref[Assumption 3]{G3}, we have $\mathscr{E}_{ext}\cap\mathcal{H}^{+}=\emptyset$.
Thus, since $\tilde{\text{\textgreek{y}}}$ vanishes outside $\mathscr{E}_{ext}$,
we have $\tilde{\text{\textgreek{y}}}\equiv0$ on $\mathcal{H}^{+}$.
This fact, combined with (\ref{eq:ZeroInnerProduct}) and the identity
(\ref{eq:InnerProductIdentity}) (applied to a sequence of smooth
approximations of $\tilde{\text{\textgreek{y}}}$ in the norm defined
by (\ref{eq:UpperBoundEnergy})) yields for any $s,\text{\textgreek{t}}\in\mathbb{R}$:
\begin{equation}
\int_{s}^{s+\text{\textgreek{t}}_{0}}\left\langle \text{\textgreek{y}},\mathcal{F}_{\text{\textgreek{t}}}^{*}\tilde{\text{\textgreek{y}}}\right\rangle _{T,\bar{\text{\textgreek{t}}}}\, d\bar{\text{\textgreek{t}}}=0.\label{eq:ZeroInnerProductEverywhere}
\end{equation}

In view of the definitions (\ref{eq:TranslatedSolution}) and (\ref{eq:TranslatedTildePsi}),
the identity (\ref{eq:ZeroInnerProductEverywhere}) for $s=\text{\textgreek{t}}_{n}$
and $\text{\textgreek{t}}=-s$ yields: 
\begin{equation}
\int_{0}^{\text{\textgreek{t}}_{0}}\left\langle \text{\textgreek{y}}_{\text{\textgreek{t}}_{n}},\tilde{\text{\textgreek{y}}}\right\rangle _{T,\bar{\text{\textgreek{t}}}}\, d\bar{\text{\textgreek{t}}}=0.\label{eq:ZeroInnerProductBeforeContradiction}
\end{equation}
Thus, since $\tilde{\text{\textgreek{y}}}$ is supported in $\mathscr{E}_{ext}$
and $\mathcal{R}(0,1)\cap\mathscr{E}_{ext}$ is compact, (\ref{eq:WeakConvergence})
implies, after letting $n\rightarrow+\infty$ in (\ref{eq:ZeroInnerProductBeforeContradiction}):
\begin{equation}
\int_{0}^{\text{\textgreek{t}}_{0}}\left\langle \tilde{\text{\textgreek{y}}},\tilde{\text{\textgreek{y}}}\right\rangle _{T,\bar{\text{\textgreek{t}}}}\, d\bar{\text{\textgreek{t}}}=0
\end{equation}
or, in view of (\ref{eq:TInnerProductSymplecic}): 
\begin{equation}
\int_{0}^{\text{\textgreek{t}}_{0}}\Big(\int_{\text{\textgreek{S}}_{s}}Re\big\{ n\tilde{\text{\textgreek{y}}}T\bar{\tilde{\text{\textgreek{y}}}}-\tilde{\text{\textgreek{y}}}n_{\mathcal{S}}(T\bar{\tilde{\text{\textgreek{y}}}})\big\}\, dg_{\text{\textgreek{S}}_{s}}\Big)\, ds=0.\label{eq:ZeroForContradiction}
\end{equation}
The contradiction now follows after comparing (\ref{eq:ZeroForContradiction})
with (\ref{eq:NegativeEnergyTrappedSolution}) (using also the fact
that $\tilde{\text{\textgreek{y}}}$ is supported in $\mathscr{E}_{ext}$).
Thus, the proof of Theorem \ref{thm:FriedmanInstability} is complete.
\qed

\section{\label{sec:FrequencyDecomposition}Frequency decomposition}

As we remarked in Section \ref{sec:Proof-of-Theorem}, Sections \ref{sec:FrequencyDecomposition}--\ref{sec:Carleman}
will be devoted to the development of the technical machinery required
for the proof of Proposition \ref{prop:QuantitativeDecayOutsideErgoregion}.
In particular, in this section, we will assume that we are given a
smooth function $\text{\textgreek{y}}:\mathcal{M}\rightarrow\mathbb{C}$
solving the wave equation (\ref{eq:WaveEquation}) on $\mathcal{D}(\text{\textgreek{S}})$
(i.\,e.~the domain of dependence of $\text{\textgreek{S}}$) with
compactly supported initial data on $\text{\textgreek{S}}$, such
that 
\begin{equation}
\mathcal{E}[\text{\textgreek{y}}]\doteq\sup_{\text{\textgreek{t}}\ge0}\int_{\text{\textgreek{S}}_{\text{\textgreek{t}}}}J^{N}(\text{\textgreek{y}})n^{\text{\textgreek{m}}}<+\infty.\label{eq:BoundednessEnergy}
\end{equation}
 We will also introduce the frequency parameters $\text{\textgreek{w}}_{+}>1$
and $0<\text{\textgreek{w}}_{0}<1$, and we will decompose the function
$\text{\textgreek{y}}$ into components with localised frequency support
(associated to the $t$ variable). We will always identify $\mathcal{M}\backslash\mathcal{H}^{-}$
with $\mathbb{R}\times\text{\textgreek{S}}$ under the flow of $T$
as explained in Section \ref{sub:CoordinateCharts}. The constructions
in this section will be similar to the associated constructions in
Section 4 of \cite{Moschidisb}.

\subsection{Weighted energy estimates for $\text{\textgreek{y}}$}

Before proceeding to cut off $\text{\textgreek{y}}$ in the frequency
space, we will first derive a few bounds for some suitable weighted
energies of $\text{\textgreek{y}}$.

In view of the finite speed of propagation for solutions to (\ref{eq:WaveEquation})
and the fact that $(\text{\textgreek{y}},T\text{\textgreek{y}})|_{\text{\textgreek{S}}_{0}}$
is compactly supported, we infer that $(\text{\textgreek{y}},T\text{\textgreek{y}})|_{\text{\textgreek{S}}_{\text{\textgreek{t}}}}$
is also compactly supported for any $\text{\textgreek{t}}\ge0$. The
following lemma is a straightforward application of the finite speed
of propagation property of equation (\ref{eq:WaveEquation}):
\begin{lem}
\label{lem:WeightedBoundedness}For any $a>0$, any $R\gg1$ (so that
$T$ is timelike in $\{r\ge R\}$), any $\text{\textgreek{t}}_{1}\ge0$
and any $\text{\textgreek{t}}\in\mathbb{R}$:
\begin{equation}
\int_{\text{\textgreek{S}}_{\text{\textgreek{t}}}\cap\mathcal{D}(\text{\textgreek{S}}_{\text{\textgreek{t}}_{1}}\cap\{r\ge R\})}\big(\log(r)\big)^{a}J_{\text{\textgreek{m}}}^{T}(\text{\textgreek{y}})n^{\text{\textgreek{m}}}\le C_{a}\big(\log(2+|\text{\textgreek{t}}-\text{\textgreek{t}}_{1}|)\big)^{a+1}\int_{\text{\textgreek{S}}_{\text{\textgreek{t}}_{1}}\cap\{r\ge R\}}\big(\log(r)\big)^{a}J_{\text{\textgreek{m}}}^{T}(\text{\textgreek{y}})n^{\text{\textgreek{m}}}\label{eq:WeightedBoundednessLogarithmic}
\end{equation}
and 
\begin{equation}
\int_{\text{\textgreek{S}}_{\text{\textgreek{t}}}\cap\mathcal{D}(\text{\textgreek{S}}_{\text{\textgreek{t}}_{1}}\cap\{r\ge R\})}r^{a}J_{\text{\textgreek{m}}}^{T}(\text{\textgreek{y}})n^{\text{\textgreek{m}}}\le C_{a}\big(1+|\text{\textgreek{t}}-\text{\textgreek{t}}_{1}|\big)^{a}\int_{\text{\textgreek{S}}_{\text{\textgreek{t}}_{1}}\cap\{r\ge R\}}r^{a}J_{\text{\textgreek{m}}}^{T}(\text{\textgreek{y}})n^{\text{\textgreek{m}}},\label{eq:WeightedBoundednessPolynomial}
\end{equation}
where $\mathcal{D}(\text{\textgreek{S}}_{\text{\textgreek{t}}_{1}}\cap\{r\ge R\})\subset\{r\ge R\}$
is the domain of dependence of $\text{\textgreek{S}}_{\text{\textgreek{t}}}\cap\{r\ge R\}$
and $C_{a}>0$ depends only on $a$ and the geometry of $(\mathcal{M},g)$.\end{lem}
\begin{proof}
Let us define for any $k\ge1$ the sets 
\begin{equation}
\mathcal{A}_{k}=\{2^{k}\le r\le2^{k+1}\}\subset\mathcal{M},
\end{equation}
and let us set 
\begin{equation}
\mathcal{A}_{0}=\{r\le1\}.
\end{equation}
Then, in view of the asymptotics (\ref{eq:metric}) of $g$ in each
connected component of the asymptotically flat region $\mathcal{I}_{as}$,
there exists a constant $C>0$ depending on the geometry of $(\mathcal{M},g)$
such that for any $\text{\textgreek{t}}_{1}\ge0$, $\text{\textgreek{t}}\in\mathbb{R}$
and any $k\in\mathbb{N}$: 
\begin{equation}
\text{\textgreek{S}}_{\text{\textgreek{t}}_{1}}\cap\Big(J^{+}(\mathcal{A}_{k}\cap\text{\textgreek{S}}_{\text{\textgreek{t}}})\cup J^{-}(\mathcal{A}_{k}\cap\text{\textgreek{S}}_{\text{\textgreek{t}}})\Big)\subset\bigcup_{n=\max\{0,k-\log_{2}(|\text{\textgreek{t}}-\text{\textgreek{t}}_{1}|+1)-C\}}^{k+\log_{2}(|\text{\textgreek{t}}-\text{\textgreek{t}}_{1}|+1)+C}\mathcal{A}_{n}\cap\text{\textgreek{S}}_{\text{\textgreek{t}}_{1}}.\label{eq:DyadicDecomposition:}
\end{equation}

Applying for any $k\in\mathbb{N}$ the conservation of the $T$-energy
flux in the spacetime region $J^{-}(\mathcal{A}_{k}\cap\text{\textgreek{S}}_{\text{\textgreek{t}}})\cap\mathcal{D}^{+}(\text{\textgreek{S}}_{\text{\textgreek{t}}_{1}}\cap\{r\ge R\})$,
in the case $\text{\textgreek{t}}\ge\text{\textgreek{t}}_{1}$, or
the region $J^{+}(\mathcal{A}_{k}\cap\text{\textgreek{S}}_{\text{\textgreek{t}}})\cap\mathcal{D}^{-}(\text{\textgreek{S}}_{\text{\textgreek{t}}_{1}}\cap\{r\ge R\})$,%
\footnote{Here, $\mathcal{D}^{+}(B)$ is the future domain of dependence of
the set $B\subset\mathcal{M}$, while $\mathcal{D}^{-}(B)$ is the
past domain of dependence%
} in the case $\text{\textgreek{t}}\le\text{\textgreek{t}}_{1}$, we
readily obtain in view of (\ref{eq:DyadicDecomposition:})(using also
the fact that $T$ is timelike for $r\ge R$): 
\begin{equation}
\int_{\mathcal{A}_{k}\cap\text{\textgreek{S}}_{\text{\textgreek{t}}}\cap\mathcal{D}(\text{\textgreek{S}}_{\text{\textgreek{t}}_{1}}\cap\{r\ge R\})}J_{\text{\textgreek{m}}}^{T}(\text{\textgreek{y}})n^{\text{\textgreek{m}}}\le\sum_{n=\max\{0,k-\log_{2}(|\text{\textgreek{t}}-\text{\textgreek{t}}_{1}|+1)-C\}}^{k+\log_{2}(|\text{\textgreek{t}}-\text{\textgreek{t}}_{1}|+1)+C}\int_{\mathcal{A}_{n}\cap\text{\textgreek{S}}_{\text{\textgreek{t}}_{1}}\cap\{r\ge R\})}J_{\text{\textgreek{m}}}^{T}(\text{\textgreek{y}})n^{\text{\textgreek{m}}}.\label{eq:DyadicEnergyIdentity}
\end{equation}
Multiplying (\ref{eq:DyadicEnergyIdentity}) with $k^{a}$ and summing
over $k\in\mathbb{N}$, we obtain: 
\begin{align}
\sum_{k=0}^{\infty}k^{a}\int_{\mathcal{A}_{k}\cap\text{\textgreek{S}}_{\text{\textgreek{t}}}\cap\mathcal{D}(\text{\textgreek{S}}_{\text{\textgreek{t}}_{1}}\cap\{r\ge R\})}J_{\text{\textgreek{m}}}^{T}(\text{\textgreek{y}})n^{\text{\textgreek{m}}} & \le\sum_{k=0}^{\infty}\Big(k^{a}\sum_{n=\max\{0,k-\log_{2}(|\text{\textgreek{t}}-\text{\textgreek{t}}_{1}|+1)-C\}}^{k+\log_{2}(|\text{\textgreek{t}}-\text{\textgreek{t}}_{1}|+1)+C}\int_{\mathcal{A}_{n}\cap\text{\textgreek{S}}_{\text{\textgreek{t}}_{1}}\cap\{r\ge R\})}J_{\text{\textgreek{m}}}^{T}(\text{\textgreek{y}})n^{\text{\textgreek{m}}}\Big)\le\label{eq:DyadicEnergyIdentity-1}\\
 & \le C\sum_{k=0}^{\infty}\Big(\big(\sum_{j=k}^{k+\log_{2}(|\text{\textgreek{t}}-\text{\textgreek{t}}_{1}|+1)+C}j^{a}\big)\int_{\mathcal{A}_{k}\cap\text{\textgreek{S}}_{\text{\textgreek{t}}_{1}}\cap\{r\ge R\})}J_{\text{\textgreek{m}}}^{T}(\text{\textgreek{y}})n^{\text{\textgreek{m}}}\Big)\le\nonumber \\
 & \le C_{a}\sum_{k=0}^{\infty}\Big((k+\log_{2}(|\text{\textgreek{t}}-\text{\textgreek{t}}_{1}|+1)+C)^{a+1}\int_{\mathcal{A}_{k}\cap\text{\textgreek{S}}_{\text{\textgreek{t}}_{1}}\cap\{r\ge R\})}J_{\text{\textgreek{m}}}^{T}(\text{\textgreek{y}})n^{\text{\textgreek{m}}}\Big).\nonumber 
\end{align}

Inequality (\ref{eq:WeightedBoundednessLogarithmic}) follows readily
from (\ref{eq:DyadicEnergyIdentity-1}). Inequality (\ref{eq:WeightedBoundednessPolynomial})
follows in the same way, after multiplying (\ref{eq:DyadicEnergyIdentity})
with $2^{ka}$ and summing over $k\in\mathbb{N}$.
\end{proof}
In view of (\ref{eq:BoundednessEnergy}) and the conservation of the
$T$-energy flux in the region $\{t_{-}\ge0\}\cap\{t\le0\}$, we can
bound: 
\begin{equation}
\sup_{\text{\textgreek{t}}\in\mathbb{R}}\int_{\text{\textgreek{S}}_{\text{\textgreek{t}}}\cap\{t_{-}\ge0\}}J_{\text{\textgreek{m}}}^{N}(\text{\textgreek{y}})n^{\text{\textgreek{m}}}\le\mathcal{E}[\text{\textgreek{y}}]\label{eq:BoundednessAtAlltimes}
\end{equation}
(note that $\text{\textgreek{S}}_{\text{\textgreek{t}}}\cap\{t_{-}\ge0\}=\text{\textgreek{S}}_{\text{\textgreek{t}}}$
when $\text{\textgreek{t}}\ge0$). Furthermore, in view of (\ref{eq:WeightedBoundednessLogarithmic})
for $\text{\textgreek{t}}_{1}=0$ and the Hardy inequality (\ref{eq:GeneralHardyLogarithmic}),
we can estimate: 
\begin{equation}
\sup_{\text{\textgreek{t}}\le0}\Big(\big(\log(2+|\text{\textgreek{t}}|)\big)^{-3}\int_{\text{\textgreek{S}}_{\text{\textgreek{t}}}\cap\{t_{-}>0\}}(1+r)^{-2}|\text{\textgreek{y}}|^{2}\Big)\le C\int_{\text{\textgreek{S}}_{0}}\big(\log(2+r)\big)^{3}J_{\text{\textgreek{m}}}^{N}(\text{\textgreek{y}})n^{\text{\textgreek{m}}}.\label{eq:BoundHardyNegativeTimes}
\end{equation}

\subsection{\label{sub:Frequency-cut-off}Frequency cut-off }

Let us fix a constant $R_{1}\gg1$ large in terms of the geometry
of $(\mathcal{M},g)$, as well as a smooth cut-off function $\text{\textgreek{q}}_{1}:[0,+\infty)\rightarrow[0,1]$
satisfying $\text{\textgreek{q}}_{1}(r)=0$ for $r\le R_{1}$ and
$\text{\textgreek{q}}_{1}(r)=1$ for $r\ge R_{1}+1$. As in Section
4 of \cite{Moschidisb}, we will define the following distorted time
function on $\mathcal{M}\backslash\mathcal{H}^{-}$: 
\begin{equation}
t_{-}=t+\frac{1}{2}\text{\textgreek{q}}_{1}(r)(r-R_{1}).\label{eq:DistortedTime-}
\end{equation}
 Note that $\{t=0\}\subset J^{+}(\{t_{-}=0\})$.

We will also fix another smooth cut-off function $\text{\textgreek{q}}_{2}:\mathbb{R}\rightarrow[0,1]$,
satisfying $\text{\textgreek{q}}_{2}\equiv0$ on $(-\infty,0]$ and
$\text{\textgreek{q}}_{2}\equiv1$ on $[1,+\infty)$, and we will
define the function $\text{\textgreek{y}}_{c}:\mathcal{M}\backslash\mathcal{H}^{-}\rightarrow\mathbb{C}$
as 
\begin{equation}
\text{\textgreek{y}}_{c}\doteq\begin{cases}
\text{\textgreek{q}}_{2}(t_{-})\cdot\text{\textgreek{y}}, & t_{-}\ge0,\\
0, & t_{-}\le0.
\end{cases}\label{eq:CutOffPsi}
\end{equation}
Since $\text{\textgreek{y}}$ solves (\ref{eq:WaveEquation}), $\text{\textgreek{y}}_{c}$
solves
\begin{equation}
\square_{g}\text{\textgreek{y}}_{c}=F,\label{eq:InhomogeneousCutOffWaveequation}
\end{equation}
where 
\begin{equation}
F=2\partial^{\text{\textgreek{m}}}\text{\textgreek{q}}_{2}(t_{-})\cdot\partial_{\text{\textgreek{m}}}\text{\textgreek{y}}+\square_{g}\text{\textgreek{q}}_{2}(t_{-})\cdot\text{\textgreek{y}}\label{eq:SourceTermCutOff}
\end{equation}
is supported in $\{0\le t_{-}\le1\}$.

Noting that $r\gtrsim|\text{\textgreek{t}}|$ on $\{t=\text{\textgreek{t}}\}\cap\{0\le t_{-}\le1\}$
for $\text{\textgreek{t}}\le0$, combining (\ref{eq:BoundednessAtAlltimes})
and (\ref{eq:BoundHardyNegativeTimes}) (in each asymptotically flat
end of $\text{\textgreek{S}}_{\text{\textgreek{t}}}$) with the Hardy-type
inequality (obtained after averaging (\ref{eq:GeneralHardyLogarithmic})
over $R_{2}$, using also a Poincare-type inequality in the near region
$\{r\lesssim1\}$):
\begin{align}
\int_{\text{\textgreek{S}}_{\text{\textgreek{t}}}\backslash\mathcal{D}(\text{\textgreek{S}}_{0}\cap\{r\ge R\})}(1+r)^{-2}|\text{\textgreek{y}}_{c}|^{2} & \le C\int_{\text{\textgreek{S}}_{\text{\textgreek{t}}}\backslash\mathcal{D}(\text{\textgreek{S}}_{0}\cap\{r\ge2R\})}\big(\log(2+r)\big)^{2}J_{\text{\textgreek{m}}}^{N}(\text{\textgreek{y}}_{c})n^{\text{\textgreek{m}}}+C\int_{\text{\textgreek{S}}_{\text{\textgreek{t}}}\cap\mathcal{D}(\text{\textgreek{S}}_{0}\cap\{R\le r\le2R\})}(1+r)^{-2}\log(r)|\text{\textgreek{y}}_{c}|^{2}\le\label{eq:HardyInterior}\\
 & \le C\big(\log(2+|\text{\textgreek{t}}|)\big)^{2}\int_{\text{\textgreek{S}}_{\text{\textgreek{t}}}}J_{\text{\textgreek{m}}}^{N}(\text{\textgreek{y}}_{c})n^{\text{\textgreek{m}}}+C\log(2+|\text{\textgreek{t}}|)\int_{\text{\textgreek{S}}_{\text{\textgreek{t}}}\cap\mathcal{D}(\text{\textgreek{S}}_{0}\cap\{r\ge R\})}(1+r)^{-2}|\text{\textgreek{y}}_{c}|^{2},\nonumber 
\end{align}
 we obtain in view of (\ref{eq:CutOffPsi}):
\begin{equation}
\sup_{\text{\textgreek{t}}\ge0}\int_{\text{\textgreek{S}}_{\text{\textgreek{t}}}}J^{N}(\text{\textgreek{y}}_{c})n^{\text{\textgreek{m}}}+\sup_{\text{\textgreek{t}}\le0}\Big(|\text{\textgreek{t}}|^{-2}\big(\log(2+|\text{\textgreek{t}}|)\big)^{-4}\int_{\text{\textgreek{S}}_{\text{\textgreek{t}}}}J_{\text{\textgreek{m}}}^{N}(\text{\textgreek{y}}_{c})n^{\text{\textgreek{m}}}\Big)+\sup_{\text{\textgreek{t}}\in\mathbb{R}}\Big(\big(\log(2+|\text{\textgreek{t}}|)\big)^{-4}\int_{\text{\textgreek{S}}_{\text{\textgreek{t}}}}\big(1+r\big)^{-2}|\text{\textgreek{y}}_{c}|^{2}\Big)\le C\mathcal{E}_{log}[\text{\textgreek{y}}],\label{eq:EnergyBoundCutOff}
\end{equation}
where 
\begin{equation}
\mathcal{E}_{log}[\text{\textgreek{y}}]\doteq\mathcal{E}[\text{\textgreek{y}}]+\int_{\text{\textgreek{S}}_{0}}\big(\log(2+r)\big)^{3}J_{\text{\textgreek{m}}}^{N}(\text{\textgreek{y}})n^{\text{\textgreek{m}}}.\label{eq:WeightedInitialEnergy}
\end{equation}

\begin{rem*}
In dimensions $d\ge3$, inequality (\ref{eq:EnergyBoundCutOff}),
as well as most of the estimates of this section, holds without the
logarithmic loss (since (\ref{eq:BoundHardyNegativeTimes}) holds
without a logarithmic loss in this case).
\end{rem*}
We will now proceed to perform a cut-off procedure on $\text{\textgreek{y}}_{c}$
in the frequency domain. Let $0<\text{\textgreek{w}}_{0}<1$ be a
(small) positive constant, and $\text{\textgreek{w}}_{+}\gg\text{\textgreek{w}}_{0}$
a (large) positive constant, and let us set 
\begin{equation}
n=\lceil\log_{2}\frac{\text{\textgreek{w}}_{+}}{\text{\textgreek{w}}_{0}}\rceil
\end{equation}
 and, for any integer $1\le k\le n$: 
\begin{equation}
\text{\textgreek{w}}_{k}=2^{k}\text{\textgreek{w}}_{0}.
\end{equation}
Fixing a third smooth cut-off function $\text{\textgreek{q}}_{3}:\mathbb{R}\rightarrow[0,1]$
satisfying $\text{\textgreek{q}}_{3}\equiv1$ on $[-1,1]$ and $\text{\textgreek{q}}_{3}\equiv0$
on $(-\infty,-2]\cup[2,+\infty)$, we will define the following Schwartz
functions on $\mathbb{R}$:

\begin{align}
\text{\textgreek{z}}_{0}(t) & =\int_{-\infty}^{+\infty}e^{i\text{\textgreek{w}}t}\text{\textgreek{q}}_{3}(\text{\textgreek{w}}_{0}^{-1}\text{\textgreek{w}})\, d\text{\textgreek{w}},\label{eq:FrequencyCutOffFunctions}\\
\text{\textgreek{z}}_{k}(t) & =\int_{-\infty}^{+\infty}e^{i\text{\textgreek{w}}t}\big(\text{\textgreek{q}}_{3}(\text{\textgreek{w}}_{k}^{-1}\text{\textgreek{w}})-\text{\textgreek{q}}_{3}(\text{\textgreek{w}}_{k-1}^{-1}\text{\textgreek{w}})\big)\, d\text{\textgreek{w}},\mbox{ for }1\le k\le n\nonumber \\
\text{\textgreek{z}}_{\le\text{\textgreek{w}}_{+}}(t) & =\sum_{k=0}^{n}\text{\textgreek{z}}_{k}(t).\nonumber 
\end{align}
Notice that the Fourier transform of $\text{\textgreek{z}}_{k}$ is
supported in $\{\text{\textgreek{w}}_{k-1}\le|\text{\textgreek{w}}|\le2\text{\textgreek{w}}_{k}\}$
(setting $\text{\textgreek{w}}_{-1}=0$), while the frequency support
of $\text{\textgreek{z}}_{\le\text{\textgreek{w}}_{+}}$ is contained
in $\{|\text{\textgreek{w}}|\le4\text{\textgreek{w}}_{+}\}$. Furthermore,
the following Schwartz bounds hold for any integers $m,m'\in\mathbb{N}$
and $0\le k\le n$:
\begin{equation}
\sup_{t\in\mathbb{R}}\big|\text{\textgreek{w}}_{k}^{-1-m'}(1+|\text{\textgreek{w}}_{k}t|^{m})\big(\frac{d}{dt}\big)^{m'}\text{\textgreek{z}}_{k}(t)\big|\le C_{m}\label{eq:SchwartzBoundk}
\end{equation}
and 
\begin{equation}
\sup_{t\in\mathbb{R}}\big|\text{\textgreek{w}}_{+}^{-1-m'}(1+|\text{\textgreek{w}}_{+}t|^{m})\big(\frac{d}{dt}\big)^{m'}\text{\textgreek{z}}_{\le\text{\textgreek{w}}_{+}}(t)\big|\le C_{m}.\label{eq:SchwartzBoundLargerCutOff}
\end{equation}

Using $\text{\textgreek{z}}_{k},\text{\textgreek{z}}_{\le\text{\textgreek{w}}_{+}}$,
we will define, for $0\le k\le n$, the ``frequency decomposed''
components $\text{\textgreek{y}}_{k},\text{\textgreek{y}}_{\le\text{\textgreek{w}}_{+}},\text{\textgreek{y}}_{\ge\text{\textgreek{w}}_{+}}:\mathcal{M}\backslash\mathcal{H}^{-}\rightarrow\mathbb{C}$
of $\text{\textgreek{y}}$ through the following relations (identifying
$\mathcal{M}\backslash\mathcal{H}^{-}$ with $\mathbb{R}\times\text{\textgreek{S}}$
through the flow of $T$):
\begin{equation}
\text{\textgreek{y}}_{k}(t,\cdot)=\int_{-\infty}^{+\infty}\text{\textgreek{z}}_{k}(t-s)\text{\textgreek{y}}_{c}(s,\cdot)\, ds,\label{eq:FrequencyLocalisedPart}
\end{equation}
\begin{equation}
\text{\textgreek{y}}_{\le\text{\textgreek{w}}_{+}}(t,\cdot)=\int_{-\infty}^{+\infty}\text{\textgreek{z}}_{\le\text{\textgreek{w}}_{+}}(t-s)\text{\textgreek{y}}_{c}(s,\cdot)\, ds\label{eq:LowFrequencyPart}
\end{equation}
and 
\begin{equation}
\text{\textgreek{y}}_{\ge\text{\textgreek{w}}_{+}}(t,\cdot)=\text{\textgreek{y}}_{c}(t,\cdot)-\text{\textgreek{y}}_{\le\text{\textgreek{w}}_{+}}(t,\cdot).\label{eq:HighFrequencyPart}
\end{equation}
Note that the integrals (\ref{eq:FrequencyLocalisedPart}) and (\ref{eq:LowFrequencyPart})
do not necessarily converge pointwise for all $(t,x)\in\mathbb{R}\times\text{\textgreek{S}}$,
since the bound (\ref{eq:BoundednessEnergy}) does not suffice to
exclude the pointwise exponential growth of $\text{\textgreek{y}}$
in the $t$ variable. Instead, in view of (\ref{eq:EnergyBoundCutOff}),
(\ref{eq:SchwartzBoundk}) and (\ref{eq:SchwartzBoundLargerCutOff}),
the restrictions $(\text{\textgreek{y}}_{k},T\text{\textgreek{y}}_{k})|_{\text{\textgreek{S}}_{\text{\textgreek{t}}}}$,
$(\text{\textgreek{y}}_{\le\text{\textgreek{w}}_{+}},T\text{\textgreek{y}}_{\le\text{\textgreek{w}}_{+}})|_{\text{\textgreek{S}}_{\text{\textgreek{t}}}}$
and $(\text{\textgreek{y}}{}_{\ge\text{\textgreek{w}}_{+}},T\text{\textgreek{y}}_{\ge\text{\textgreek{w}}_{+}})|_{\text{\textgreek{S}}_{\text{\textgreek{t}}}}$
are only defined as finite energy functions on $\text{\textgreek{S}}_{\text{\textgreek{t}}}$
for any $\text{\textgreek{t}}\in\mathbb{R}$, satisfying the following
bound for any $a>0$ (derived from (\ref{eq:EnergyBoundCutOff}),
(\ref{eq:SchwartzBoundk}), (\ref{eq:SchwartzBoundLargerCutOff})
and Young's inequality): 
\begin{align}
\sup_{\text{\textgreek{t}}\ge0} & \Big((1+\text{\textgreek{w}}_{k}^{-2-a})^{-1}\int_{\text{\textgreek{S}}}\big(|N\text{\textgreek{y}}_{k}(\text{\textgreek{t}},x)|^{2}+|\nabla_{g_{\text{\textgreek{S}}}}\text{\textgreek{y}}_{k}(\text{\textgreek{t}},x)|_{g_{\text{\textgreek{S}}}}^{2}\big)\, dg_{\text{\textgreek{S}}}\Big)+\label{eq:EnergyClassAPrioriBoundLocal}\\
+ & \sup_{\text{\textgreek{t}}\le0}\Big((1+\text{\textgreek{w}}_{k}^{-2-a})^{-1}|\text{\textgreek{t}}|^{-2}\big(\log(2+|\text{\textgreek{t}}|)\big)^{-4}\int_{\text{\textgreek{S}}}\big(|N\text{\textgreek{y}}_{k}(\text{\textgreek{t}},x)|^{2}+|\nabla_{g_{\text{\textgreek{S}}}}\text{\textgreek{y}}_{k}(\text{\textgreek{t}},x)|_{g_{\text{\textgreek{S}}}}^{2}\big)\, dg_{\text{\textgreek{S}}}\Big)+\nonumber \\
 & \hphantom{\sup_{\text{\textgreek{t}}\le0}\Big((1}+\sup_{\text{\textgreek{t}}\in\mathbb{R}}\Big((1+\text{\textgreek{w}}_{k}^{-a})^{-1}\big(\log(2+|\text{\textgreek{t}}|)\big)^{-4}\int_{\text{\textgreek{S}}}(1+r)^{-2}|\text{\textgreek{y}}_{k}(\text{\textgreek{t}},x)|^{2}\, dg_{\text{\textgreek{S}}}\Big)\le C_{a}\mathcal{E}_{log}[\text{\textgreek{y}}],\nonumber 
\end{align}
\begin{align}
\sup_{\text{\textgreek{t}}\ge0} & \int_{\text{\textgreek{S}}}\big(|N\text{\textgreek{y}}_{\le\text{\textgreek{w}}_{+}}(\text{\textgreek{t}},x)|^{2}+|\nabla_{g_{\text{\textgreek{S}}}}\text{\textgreek{y}}_{\le\text{\textgreek{w}}_{+}}(\text{\textgreek{t}},x)|_{g_{\text{\textgreek{S}}}}^{2}\big)\, dg_{\text{\textgreek{S}}}+\label{eq:EnergyClassAPrioriBoundLow}\\
+ & \sup_{\text{\textgreek{t}}\le0}\Big(|\text{\textgreek{t}}|^{-2}\big(\log(2+|\text{\textgreek{t}}|)\big)^{-4}\int_{\text{\textgreek{S}}}\big(|N\text{\textgreek{y}}_{\le\text{\textgreek{w}}_{+}}(\text{\textgreek{t}},x)|^{2}+|\nabla_{g_{\text{\textgreek{S}}}}\text{\textgreek{y}}_{\le\text{\textgreek{w}}_{+}}(\text{\textgreek{t}},x)|_{g_{\text{\textgreek{S}}}}^{2}\big)\, dg_{\text{\textgreek{S}}}\Big)+\nonumber \\
 & \hphantom{\sup_{\text{\textgreek{t}}\le0}\Big((1}+\sup_{\text{\textgreek{t}}\in\mathbb{R}}\Big(\big(\log(2+|\text{\textgreek{t}}|)\big)^{-4}\int_{\text{\textgreek{S}}}(1+r)^{-2}|\text{\textgreek{y}}_{\le\text{\textgreek{w}}_{+}}(\text{\textgreek{t}},x)|^{2}\, dg_{\text{\textgreek{S}}}\Big)\le C_{a}\mathcal{E}_{log}[\text{\textgreek{y}}]\nonumber 
\end{align}
and 
\begin{align}
\sup_{\text{\textgreek{t}}\ge0} & \int_{\text{\textgreek{S}}}\big(|N\text{\textgreek{y}}_{\ge\text{\textgreek{w}}_{+}}(\text{\textgreek{t}},x)|^{2}+|\nabla_{g_{\text{\textgreek{S}}}}\text{\textgreek{y}}_{\ge\text{\textgreek{w}}_{+}}(\text{\textgreek{t}},x)|_{g_{\text{\textgreek{S}}}}^{2}\big)\, dg_{\text{\textgreek{S}}}+\label{eq:EnergyClassAPrioriBoundHigh}\\
+ & \sup_{\text{\textgreek{t}}\le0}\Big(|\text{\textgreek{t}}|^{-2}\big(\log(2+|\text{\textgreek{t}}|)\big)^{-4}\int_{\text{\textgreek{S}}}\big(|N\text{\textgreek{y}}_{\ge\text{\textgreek{w}}_{+}}(\text{\textgreek{t}},x)|^{2}+|\nabla_{g_{\text{\textgreek{S}}}}\text{\textgreek{y}}_{\ge\text{\textgreek{w}}_{+}}(\text{\textgreek{t}},x)|_{g_{\text{\textgreek{S}}}}^{2}\big)\, dg_{\text{\textgreek{S}}}\Big)+\nonumber \\
 & \hphantom{\sup_{\text{\textgreek{t}}\le0}\Big((1}+\sup_{\text{\textgreek{t}}\in\mathbb{R}}\Big(\big(\log(2+|\text{\textgreek{t}}|)\big)^{-4}\int_{\text{\textgreek{S}}}(1+r)^{-2}|\text{\textgreek{y}}_{\ge\text{\textgreek{w}}_{+}}(\text{\textgreek{t}},x)|^{2}\, dg_{\text{\textgreek{S}}}\Big)\le C_{a}\mathcal{E}_{log}[\text{\textgreek{y}}].\nonumber 
\end{align}

Defining, similarly $F_{k}$, $F_{\le\text{\textgreek{w}}_{+}}$ and
$F_{\ge\text{\textgreek{w}}_{+}}$ in terms of $F$ as in (\ref{eq:FrequencyLocalisedPart})--(\ref{eq:HighFrequencyPart})
(replacing $\text{\textgreek{y}}_{c}$ with $F$), in view of (\ref{eq:InhomogeneousCutOffWaveequation})
we obtain the following relations (for any $0\le k\le n$):
\begin{equation}
\square_{g}\text{\textgreek{y}}_{k}=F_{k},\label{eq:InhomogeneousWaveEquationPsiK}
\end{equation}
\begin{equation}
\square_{g}\text{\textgreek{y}}_{\le\text{\textgreek{w}}_{+}}=F_{\le\text{\textgreek{w}}_{+}}\label{eq:InhomogeneousWaveEquationPsiLow}
\end{equation}
 and 
\begin{equation}
\square_{g}\text{\textgreek{y}}_{\ge\text{\textgreek{w}}_{+}}=F_{\ge\text{\textgreek{w}}_{+}}.\label{eq:InhomogeneousWaveEquationPsiHigh}
\end{equation}

\subsection{\label{sub:BoundsForPsiK}Bounds for the frequency-decomposed components}

In this section, we will establish some useful estimates for the energy
of $\text{\textgreek{y}}_{k},\text{\textgreek{y}}_{\le\text{\textgreek{w}}_{+}},\text{\textgreek{y}}_{\ge\text{\textgreek{w}}_{+}}$,
as well as for the ``error'' terms $F_{k},F_{\le\text{\textgreek{w}}_{+}},F_{\ge\text{\textgreek{w}}_{+}}$,
in terms of $\mathcal{E}_{log}[\text{\textgreek{y}}]$. 

We start with an estimate for weighted spacetime norms of the terms
$F_{k},F_{\le\text{\textgreek{w}}_{+}},F_{\ge\text{\textgreek{w}}_{+}}$.
\begin{lem}
\label{lem:BoundF}We can bound for any $0\le k\le n$, any $q,q'\in\mathbb{N}$
and any $0\le\text{\textgreek{t}}_{1}\le\text{\textgreek{t}}_{2}$:
\begin{equation}
\int_{\mathcal{R}(\text{\textgreek{t}}_{1},\text{\textgreek{t}}_{2})}r^{q}|F_{k}|^{2}\le C_{qq'}(1+\text{\textgreek{w}}_{k}^{-q-2})\cdot(1+\text{\textgreek{w}}_{k}\text{\textgreek{t}}_{1})^{-q'}\mathcal{E}_{log}[\text{\textgreek{y}}].\label{eq:SchwartzEstimateFk}
\end{equation}
 The same inequality also holds for $F_{\le\text{\textgreek{w}}_{+}},F_{\ge\text{\textgreek{w}}_{+}}$
in place of $F_{k}$ (with $\text{\textgreek{w}}_{+}$ in place of
$\text{\textgreek{w}}_{k}$).\end{lem}
\begin{proof}
In view of (\ref{eq:SourceTermCutOff}) and the fact that 
\begin{equation}
F_{k}(t,\cdot)=\int_{-\infty}^{+\infty}\text{\textgreek{z}}_{k}(t-s)F(s,\cdot)\, dt,
\end{equation}
we can estimate (denoting with $x$ the space variable in the splitting
$\mathcal{M}\backslash\mathcal{H}^{-}=\mathbb{R}\times\text{\textgreek{S}}$)

\begin{align}
\int_{\mathcal{R}(\text{\textgreek{t}}_{1},\text{\textgreek{t}}_{2})}r^{q}|F_{k}|^{2} & \le C\int_{\text{\textgreek{S}}}\int_{\text{\textgreek{t}}_{1}}^{\text{\textgreek{t}}_{2}}r^{q}\Big|\int_{-\infty}^{\infty}\text{\textgreek{z}}_{k}(t-s)F(s,x)\, ds\Big|^{2}\, dtdg_{\text{\textgreek{S}}}\label{eq:FInequality}\\
 & \le C\Big(\int_{\text{\textgreek{S}}}r^{q}\int_{\text{\textgreek{t}}_{1}}^{\text{\textgreek{t}}_{2}}\Big|\int_{-\frac{1}{2}\text{\textgreek{q}}_{1}(r)(r-R_{1})}^{1-\frac{1}{2}\text{\textgreek{q}}_{1}(r)(r-R_{1})}\text{\textgreek{z}}_{k}(t-s)F(s,x)\, ds\Big|^{2}\, dtdg_{\text{\textgreek{S}}}\Big).\nonumber 
\end{align}
From the Schwartz bound (\ref{eq:SchwartzBoundk}) for $m=q'+q+5$
and $m'=0$ (using also (\ref{eq:SourceTermCutOff})), we can estimate:
\begin{equation}
\begin{split}\int_{\text{\textgreek{S}}}r^{q}\int_{\text{\textgreek{t}}_{1}}^{\text{\textgreek{t}}_{2}} & \Big|\int_{-\frac{1}{2}\text{\textgreek{q}}_{1}(r)(r-R_{1})}^{1-\frac{1}{2}\text{\textgreek{q}}_{1}(r)(r-R_{1})}\text{\textgreek{z}}_{k}(t-s)F(s,x)\, ds\Big|^{2}\, dtdg_{\text{\textgreek{S}}}\le\\
\le & C_{qq'}\int_{\text{\textgreek{S}}}r^{q}\int_{\text{\textgreek{t}}_{1}}^{\text{\textgreek{t}}_{2}}\Big|\int_{-\frac{1}{2}\text{\textgreek{q}}_{1}(r)(r-R_{1})}^{1-\frac{1}{2}\text{\textgreek{q}}_{1}(r)(r-R_{1})}\frac{\text{\textgreek{w}}_{k}}{(1+\text{\textgreek{w}}_{k}|t-s|)^{q'+q+5}}F(s,x)\, ds\Big|^{2}\, dtdg_{\text{\textgreek{S}}}\le\\
\le & C_{qq'}\int_{\text{\textgreek{S}}}r^{q}\int_{\text{\textgreek{t}}_{1}}^{\text{\textgreek{t}}_{2}}\Big(\int_{-\infty}^{+\infty}\frac{\text{\textgreek{w}}_{k}}{(1+\text{\textgreek{w}}_{k}|\text{\textgreek{l}}|)^{q'+q+5}}\, d\text{\textgreek{l}}\Big)\Big(\int_{-\frac{1}{2}\text{\textgreek{q}}_{1}(r)(r-R_{1})}^{1-\frac{1}{2}\text{\textgreek{q}}_{1}(r)(r-R_{1})}\frac{\text{\textgreek{w}}_{k}}{(1+\text{\textgreek{w}}_{k}|t-s|)^{q'+q+5}}|F(s,x)|^{2}\, ds\Big)\, dtdg_{\text{\textgreek{S}}}\le\\
\le & C_{qq'}\int_{\text{\textgreek{S}}}r^{q}\int_{\text{\textgreek{t}}_{1}}^{\text{\textgreek{t}}_{2}}\Big(\int_{-\frac{1}{2}\text{\textgreek{q}}_{1}(r)(r-R_{1})}^{1-\frac{1}{2}\text{\textgreek{q}}_{1}(r)(r-R_{1})}\frac{\text{\textgreek{w}}_{k}}{(1+\text{\textgreek{w}}_{k}|t-s|)^{q'+q+5}}|F(s,x)|^{2}\, ds\Big)\, dtdg_{\text{\textgreek{S}}}\le\\
\le & C_{qq'}\int_{\text{\textgreek{S}}}r^{q}\Big(\int_{\text{\textgreek{t}}_{1}}^{\text{\textgreek{t}}_{2}}\frac{\text{\textgreek{w}}_{k}}{(1+\text{\textgreek{w}}_{k}|\text{\textgreek{t}}+\frac{1}{2}\text{\textgreek{q}}_{1}(r)(r-R_{1})|)^{q'+q+5}}d\text{\textgreek{t}}\Big)\Big(\int_{-\frac{1}{2}\text{\textgreek{q}}_{1}(r)(r-R_{1})}^{1-\frac{1}{2}\text{\textgreek{q}}_{1}(r)(r-R_{1})}|F(s,x)|^{2}\, ds\Big)\, dg_{\text{\textgreek{S}}}\le\\
\le & C_{qq'}\frac{1+\text{\textgreek{w}}_{k}^{-q}}{(1+\text{\textgreek{w}}_{k}\text{\textgreek{t}}_{1})^{q'}}\int_{\text{\textgreek{S}}}(1+\text{\textgreek{w}}_{k}r)^{-4}\Big(\int_{-\frac{1}{2}\text{\textgreek{q}}_{1}(r)(r-R_{1})}^{1-\frac{1}{2}\text{\textgreek{q}}_{1}(r)(r-R_{1})}|F(t,x)|^{2}\, dt\Big)\, dg_{\text{\textgreek{S}}}\le\\
\le & C_{qq'}\frac{1+\text{\textgreek{w}}_{k}^{-q}}{(1+\text{\textgreek{w}}_{k}\text{\textgreek{t}}_{1})^{q'}}\int_{\text{\textgreek{S}}}(1+\text{\textgreek{w}}_{k}r)^{-4}\Big(\int_{-\frac{1}{2}\text{\textgreek{q}}_{1}(r)(r-R_{1})}^{1-\frac{1}{2}\text{\textgreek{q}}_{1}(r)(r-R_{1})}\big(J_{\text{\textgreek{m}}}^{N}(\text{\textgreek{y}})n^{\text{\textgreek{m}}}+|\text{\textgreek{y}}|^{2}\big)\, dt\Big)\, dg_{\text{\textgreek{S}}}\le\\
\le & C_{qq'}\frac{1+\text{\textgreek{w}}_{k}^{-q-2}}{(1+\text{\textgreek{w}}_{k}\text{\textgreek{t}}_{1})^{q'}}\int_{-\infty}^{0}(1+\text{\textgreek{t}})^{-2}\Big(\int_{\text{\textgreek{S}}_{\text{\textgreek{t}}}\cap\{0\le t_{-}\le1\}}\big(J_{\text{\textgreek{m}}}^{N}(\text{\textgreek{y}})n^{\text{\textgreek{m}}}+(1+r)^{-2}|\text{\textgreek{y}}|^{2}\big)\Big)\, d\text{\textgreek{t}}
\end{split}
\label{eq:BoundFromSchwartzFk}
\end{equation}
(for the last inequality, we used the fact that $t\sim r$ on $\{0\le t_{-}\le1\}$).
Therefore, from (\ref{eq:FInequality}), (\ref{eq:BoundFromSchwartzFk}),
(\ref{eq:BoundednessAtAlltimes}) and (\ref{eq:BoundHardyNegativeTimes}),
we readily obtain (\ref{eq:SchwartzEstimateFk}).

The estimate for $F_{\le\text{\textgreek{w}}_{+}}$ and $F_{\ge\text{\textgreek{w}}_{+}}$
follows in exactly the same way.
\end{proof}
We will also need the following qualitative decay statement near spacelike
infinity for the functions $\text{\textgreek{y}}_{k}$, $\text{\textgreek{y}}_{\le\text{\textgreek{w}}_{+}}$
and $\text{\textgreek{y}}_{\ge\text{\textgreek{w}}_{+}}$:
\begin{lem}
\label{lem:DecayPsiKNearSpacelikeInfinity}For any $q\in\mathbb{N}$,
any $\text{\textgreek{t}}\ge0$ and any $0\le k\le n$: 
\begin{equation}
\limsup_{R\rightarrow+\infty}\Big(R^{q}\int_{\text{\textgreek{S}}_{\text{\textgreek{t}}}\cap\{R\le r\le R+1\}}\big(J_{\text{\textgreek{m}}}^{N}(\text{\textgreek{y}}_{k})n^{\text{\textgreek{m}}}+|\text{\textgreek{y}}_{k}|^{2}\big)\Big)=0.\label{eq:RapidDecayAtSpacelikeInfinity}
\end{equation}
The relation (\ref{eq:RapidDecayAtSpacelikeInfinity}) also holds
for $\text{\textgreek{y}}_{\le\text{\textgreek{w}}_{+}}$, $\text{\textgreek{y}}_{\ge\text{\textgreek{w}}_{+}}$
in place of $\text{\textgreek{y}}_{k}$.\end{lem}
\begin{proof}
The proof of Lemma \ref{lem:DecayPsiKNearSpacelikeInfinity} is a
straightforward consequence of the compact support of $(\text{\textgreek{y}},T\text{\textgreek{y}})|_{\text{\textgreek{S}}}$
and the Schwartz bounds (\ref{eq:SchwartzBoundk}), (\ref{eq:SchwartzBoundLargerCutOff}).

Let $R_{0}(\text{\textgreek{y}})$ be sufficiently large, so that
$(\text{\textgreek{y}},T\text{\textgreek{y}})|_{\text{\textgreek{S}}}$
is supported in $\{r\le R_{0}(\text{\textgreek{y}})-1\}$. Then, in
view of the finite speed of propagation property of equation (\ref{eq:WaveEquation}),
there exists a $C>0$ (depending only on the geometry of $(\mathcal{M},g)$,
so that the function $\text{\textgreek{y}}$ is supported in $\{r\le R_{0}(\text{\textgreek{y}})+C|t|\}\subset\mathcal{M}$.
Thus, 
\begin{equation}
\text{\textgreek{y}}\equiv0\mbox{ on }\{|t|\ge C^{-1}\big(r-R_{0}(\text{\textgreek{y}})\big).\label{eq:PsiVanishesAway}
\end{equation}
 Then, in view of (\ref{eq:CutOffPsi}), (\ref{eq:FrequencyLocalisedPart}),
(\ref{eq:SchwartzBoundk}) and (\ref{eq:PsiVanishesAway}), we can
bound for any $\text{\textgreek{t}}\ge0$, $R>R_{0}(\text{\textgreek{y}})+C\text{\textgreek{t}}$
and $0\le k\le n$: 
\begin{align}
\int_{\text{\textgreek{S}}_{\text{\textgreek{t}}}\cap\{R\le r\le R+1\}}\big(J_{\text{\textgreek{m}}}^{N}(\text{\textgreek{y}}_{k})n^{\text{\textgreek{m}}}+|\text{\textgreek{y}}_{k}|^{2}\big) & \le C\int_{\text{\textgreek{S}}\cap\{R\le r\le R+1\}}\sum_{j=0}^{1}\Big|\int_{-\infty}^{+\infty}\text{\textgreek{z}}_{k}(\text{\textgreek{t}}-s)\nabla^{j}\text{\textgreek{y}}_{c}(s,x)\, ds\Big|_{g_{ref}}^{2}\, dg_{\text{\textgreek{S}}}\le\label{eq:BoundForRapidDecay}\\
 & \le C\Big(\int_{-\infty}^{\infty}|\text{\textgreek{z}}_{k}(s)|\, ds\Big)\sum_{j=0}^{1}\int_{\text{\textgreek{S}}\cap\{R\le r\le R+1\}}\int_{-\infty}^{+\infty}|\text{\textgreek{z}}_{k}(\text{\textgreek{t}}-s)|\big|\nabla^{j}\text{\textgreek{y}}_{c}(s,x)\big|_{g_{ref}}^{2}\, dsdg_{\text{\textgreek{S}}}\le\nonumber \\
 & \le C_{q}\sum_{j=0}^{1}\int_{\text{\textgreek{S}}\cap\{R\le r\le R+1\}}\int_{-\infty}^{+\infty}\frac{\text{\textgreek{w}}_{k}}{1+(\text{\textgreek{w}}_{k}|\text{\textgreek{t}}-s|)^{q+4}}\big|\nabla^{j}\text{\textgreek{y}}_{c}(s,x)\big|_{g_{ref}}^{2}\, dsdg_{\text{\textgreek{S}}}\le\nonumber \\
 & \le C_{q}\sum_{j=0}^{1}\int_{\text{\textgreek{S}}\cap\{R\le r\le R+1\}}\int_{C^{-1}(R-R_{0}(\text{\textgreek{y}}))}^{+\infty}\frac{\text{\textgreek{w}}_{k}}{1+(\text{\textgreek{w}}_{k}|\text{\textgreek{t}}-s|)^{q+4}}\big|\nabla^{j}\text{\textgreek{y}}(s,x)\big|_{g_{ref}}^{2}\, dsdg_{\text{\textgreek{S}}}.\nonumber 
\end{align}

In view of the bounds (\ref{eq:BoundednessAtAlltimes}) and (\ref{eq:BoundHardyNegativeTimes}),
inequality (\ref{eq:BoundForRapidDecay}) yields: 
\begin{align}
\int_{\text{\textgreek{S}}_{\text{\textgreek{t}}}\cap\{R\le r\le R+1\}}\big(J_{\text{\textgreek{m}}}^{N}(\text{\textgreek{y}}_{k})n^{\text{\textgreek{m}}}+|\text{\textgreek{y}}_{k}|^{2}\big) & \le C_{q}R^{2}\mathcal{E}_{log}[\text{\textgreek{y}}]\int_{\text{\textgreek{S}}\cap\{R\le r\le R+1\}}\int_{C^{-1}(R-R_{0}(\text{\textgreek{y}}))}^{+\infty}\frac{\text{\textgreek{w}}_{k}}{1+(\text{\textgreek{w}}_{k}|\text{\textgreek{t}}-s|)^{q+4}}\big(\log(2+|s|)\big)^{3}\, dsdg_{\text{\textgreek{S}}}\le\label{eq:AlmostThereForRapiddecay}\\
 & \le C_{q}\frac{\text{\textgreek{w}}_{k}^{-1}R^{2}}{1+\big(\text{\textgreek{w}}_{k}(C^{-1}(R-R_{0}(\text{\textgreek{y}}))-\text{\textgreek{t}})\big)^{q+2}}\mathcal{E}_{log}[\text{\textgreek{y}}].\nonumber 
\end{align}
Thus, (\ref{eq:RapidDecayAtSpacelikeInfinity}) readily follows from
(\ref{eq:AlmostThereForRapiddecay}).

The relation (\ref{eq:RapidDecayAtSpacelikeInfinity}) for $\text{\textgreek{y}}_{\le\text{\textgreek{w}}_{+}}$,
$\text{\textgreek{y}}_{\ge\text{\textgreek{w}}_{+}}$ in place of
$\text{\textgreek{y}}_{k}$ follows in exactly the same way, using
(\ref{eq:SchwartzBoundLargerCutOff}) in place of (\ref{eq:SchwartzBoundk}).
\end{proof}
We will now proceed to obtain local in time estimates of the form
$\int_{-\infty}^{\infty}|\partial_{t}\text{\textgreek{y}}_{k}|^{2}\, dt\sim\text{\textgreek{w}}_{k}^{2}\int_{-\infty}^{\infty}|\text{\textgreek{y}}_{k}|^{2}\, dt$.
Let us define the following Schwartz functions on $\mathbb{R}$, similar
to (\ref{eq:FrequencyCutOffFunctions}):

\begin{align}
\text{\textgreek{x}}_{0}(t) & =\int_{-\infty}^{+\infty}e^{i\text{\textgreek{w}}t}\text{\textgreek{q}}_{3}(\frac{1}{2}\text{\textgreek{w}}_{0}^{-1}\text{\textgreek{w}})\, d\text{\textgreek{w}},\label{eq:SeedsReproducingFormula}\\
\text{\textgreek{x}}_{k}(t) & =\int_{-\infty}^{+\infty}e^{i\text{\textgreek{w}}t}\big(\text{\textgreek{q}}_{3}(\frac{1}{2}\text{\textgreek{w}}_{k}^{-1}\text{\textgreek{w}})-\text{\textgreek{q}}_{3}(2\text{\textgreek{w}}_{k-1}^{-1}\text{\textgreek{w}})\big)\, d\text{\textgreek{w}},\mbox{ for }1\le k\le n.\nonumber 
\end{align}
Notice that, for any $0\le k\le n$ (setting $\text{\textgreek{w}}_{-1}=0$),
$\text{\textgreek{q}}_{3}(\frac{1}{2}\text{\textgreek{w}}_{k}^{-1}\text{\textgreek{w}})-\text{\textgreek{q}}_{3}(2\text{\textgreek{w}}_{k-1}^{-1}\text{\textgreek{w}})=1$
for all $\text{\textgreek{w}}\in\mathbb{R}$ such that $\text{\textgreek{q}}_{3}(\text{\textgreek{w}}_{k}^{-1}\text{\textgreek{w}})-\text{\textgreek{q}}_{3}(\text{\textgreek{w}}_{k-1}^{-1}\text{\textgreek{w}})\neq0$,
and thus: 
\begin{equation}
\hat{\text{\textgreek{z}}}_{k}=\hat{\text{\textgreek{x}}}_{k}\cdot\hat{\text{\textgreek{z}}}_{k},\label{eq:BaseForReproducingFormula}
\end{equation}
where $\hat{\hphantom{\text{\textgreek{z}}}}$ denotes the Fourier
transform operator on $\mathbb{R}$. Moreover, the following Schwartz
bound holds for any integers $m,m'\in\mathbb{N}$ and $0\le k\le n$:
\begin{equation}
\sup_{t\in\mathbb{R}}\big|\text{\textgreek{w}}_{k}^{-1-m'}(1+|\text{\textgreek{w}}_{k}t|^{m})\big(\frac{d}{dt}\big)^{m'}\text{\textgreek{x}}_{k}(t)\big|\le C_{m,m'}.\label{eq:SchwartzBoundk-1}
\end{equation}

The relation (\ref{eq:BaseForReproducingFormula}), as well as the
definition (\ref{eq:FrequencyCutOffFunctions}), implies for any $0\le k\le n$
the following self reproducing formula for $\text{\textgreek{y}}_{k}$:

\begin{equation}
\text{\textgreek{y}}_{k}(t,\cdot)=\int_{-\infty}^{\infty}\text{\textgreek{x}}_{k}(t-s)\cdot\text{\textgreek{y}}_{k}(s,\cdot)\, ds,\label{eq:Reproducing Formula}
\end{equation}
where, again, the integral in the right hand side of (\ref{eq:Reproducing Formula})
converges with respect to the $\int_{\text{\textgreek{S}}_{t}}J_{\text{\textgreek{m}}}^{N}(\cdot)n^{\text{\textgreek{m}}}$
norm (in view of (\ref{eq:EnergyClassAPrioriBoundLocal}), (\ref{eq:SchwartzBoundk-1})
and Young's inequality).

For any $1\le k\le n$, we will also introduce the anti-derivatives
of $\text{\textgreek{x}}_{k}$, defined as 
\begin{equation}
\tilde{\text{\textgreek{x}}}_{k}(t)=\int_{-\infty}^{+\infty}\frac{1}{i\text{\textgreek{w}}}e^{i\text{\textgreek{w}}t}\big(\text{\textgreek{q}}_{3}(\frac{1}{2}\text{\textgreek{w}}_{k}^{-1}\text{\textgreek{w}})-\text{\textgreek{q}}_{3}(2\text{\textgreek{w}}_{k-1}^{-1}\text{\textgreek{w}})\big)\, d\text{\textgreek{w}},
\end{equation}
thus satisfying for any $m\in\mathbb{N}$ the Schwartz bound
\begin{equation}
\sup_{t\in\mathbb{R}}\big|(1+|\text{\textgreek{w}}_{k}t|^{m})\text{\textgreek{x}}_{k}(t)\big|\le C_{m},\label{eq:SchwartzBoundk-1-1}
\end{equation}
as well as the frequency-domain identity:
\begin{equation}
\hat{\text{\textgreek{z}}}_{k}=\hat{\tilde{\text{\textgreek{x}}}}_{k}\cdot i\text{\textgreek{w}}\hat{\text{\textgreek{z}}}_{k}.\label{eq:BaseForReproducingFormulaAntiderivative}
\end{equation}
In view of (\ref{eq:BaseForReproducingFormulaAntiderivative}), as
well as the definition (\ref{eq:FrequencyCutOffFunctions}), we obtain
for any $1\le k\le n$:
\begin{equation}
\text{\textgreek{y}}_{k}(t,\cdot)=\int_{-\infty}^{\infty}\tilde{\text{\textgreek{x}}}_{k}(t-s)\cdot T\text{\textgreek{y}}_{k}(s,\cdot)\, ds,\label{eq:Reproducing FormulaAntiderivative}
\end{equation}
where the integral in the right hand side of (\ref{eq:Reproducing FormulaAntiderivative})
converges with respect to the $\int_{\text{\textgreek{S}}_{t}}J_{\text{\textgreek{m}}}^{N}(\cdot)n^{\text{\textgreek{m}}}$
norm.

We can now establish the following lemma:
\begin{lem}
\label{lem:DtToOmegaInequalities}For any $1\le k\le n$, any $0\le\text{\textgreek{t}}_{1}\le\text{\textgreek{t}}_{2}$,
any $T$-invariant $L_{loc}^{\infty}$ function $\text{\textgreek{q}}:\mathcal{M}\backslash\mathcal{H}^{-}\rightarrow[0,+\infty)$,
any $R\ge0$ and any $0<a<1$, we can bound
\begin{multline}
c\text{\textgreek{w}}_{k}^{2}\int_{\mathcal{R}(\text{\textgreek{t}}_{1},\text{\textgreek{t}}_{2})\cap\{r\le R\}}\text{\textgreek{q}}|\text{\textgreek{y}}_{k}|^{2}-C_{a}\text{\textgreek{w}}_{k}^{2}(1+\text{\textgreek{w}}_{k}^{-5-a})\sup_{\{r\le R\}}\text{\textgreek{q}}\cdot\mathcal{E}_{log}[\text{\textgreek{y}}]\le\int_{\mathcal{R}(\text{\textgreek{t}}_{1},\text{\textgreek{t}}_{2})\cap\{r\le R\}}\text{\textgreek{q}}|T\text{\textgreek{y}}_{k}|^{2}\le\\
\le C\text{\textgreek{w}}_{k}^{2}\int_{\mathcal{R}(\text{\textgreek{t}}_{1},\text{\textgreek{t}}_{2})\cap\{r\le R\}}\text{\textgreek{q}}|\text{\textgreek{y}}_{k}|^{2}+C_{a}\text{\textgreek{w}}_{k}^{2}(1+\text{\textgreek{w}}_{k}^{-1-a})\big(\log(2+\text{\textgreek{t}}_{2})\big)^{4}R^{2}\sup_{\{r\le R\}}\text{\textgreek{q}}\cdot\mathcal{E}_{log}[\text{\textgreek{y}}],\label{eq:w-estimate}
\end{multline}
and similarly for $k=0$:

\begin{equation}
\int_{\mathcal{R}(\text{\textgreek{t}}_{1},\text{\textgreek{t}}_{2})\cap\{r\le R\}}\text{\textgreek{q}}|T\text{\textgreek{y}}_{0}|^{2}\le C\text{\textgreek{w}}_{0}^{2}\int_{\mathcal{R}(\text{\textgreek{t}}_{1},\text{\textgreek{t}}_{2})\cap\{r\le R\}}\text{\textgreek{q}}|\text{\textgreek{y}}_{k}|^{2}+C_{a}\text{\textgreek{w}}_{0}^{2}(1+\text{\textgreek{w}}_{0}^{-1-a})\big(\log(2+\text{\textgreek{t}}_{2})\big)^{4}R^{2}\sup_{\{r\le R\}}\text{\textgreek{q}}\cdot\mathcal{E}_{log}[\text{\textgreek{y}}].\label{eq:LowFrequencyEstimate}
\end{equation}
 
\end{lem}
\emph{Remark}: Notice that the constant multiplying the error term
in the right hand side of (\ref{eq:w-estimate}) depends on $R$ and
$\text{\textgreek{t}}_{2}$, while this is not the case in the left
hand side. 
\begin{proof}
For any $0\le k\le n$, from (\ref{eq:Reproducing Formula}) and (\ref{eq:SchwartzBoundk-1})
(for $m=5$, $m'=1$) we can estimate for any $\text{\textgreek{t}}\ge0$:
\begin{align}
\int_{\text{\textgreek{S}}\cap\{r\le R\}}\text{\textgreek{q}}(x)|T\text{\textgreek{y}}_{k}(\text{\textgreek{t}},x)|^{2}\, dg_{\text{\textgreek{S}}} & =\int_{\text{\textgreek{S}}\cap\{r\le R\}}\text{\textgreek{q}}(x)\Big|\int_{-\infty}^{+\infty}\text{\textgreek{x}}_{k}^{\prime}(\text{\textgreek{t}}-s)\text{\textgreek{y}}_{k}(s,x)\, ds\Big|^{2}\, dg_{\text{\textgreek{S}}}\le\label{eq:FirstEstimateFromReproducingFormula}\\
 & \le C\int_{\text{\textgreek{S}}\cap\{r\le R\}}\text{\textgreek{q}}(x)\Big|\int_{-\infty}^{+\infty}\frac{\text{\textgreek{w}}_{k}^{2}}{1+|\text{\textgreek{w}}_{k}(\text{\textgreek{t}}-s)|^{5}}\text{\textgreek{y}}_{k}(s,x)\, ds\Big|^{2}\, dg_{\text{\textgreek{S}}}\le\nonumber \\
 & \le C\text{\textgreek{w}}_{k}^{2}\int_{\text{\textgreek{S}}\cap\{r\le R\}}\text{\text{\textgreek{q}}}(x)\Big(\int_{-\infty}^{+\infty}\frac{\text{\textgreek{w}}_{k}}{1+|\text{\textgreek{w}}_{k}(\text{\textgreek{t}}-s)|^{5}}\, ds\Big)\Big(\int_{-\infty}^{+\infty}\frac{\text{\textgreek{w}}_{k}}{1+|\text{\textgreek{w}}_{k}(\text{\textgreek{t}}-s)|^{5}}|\text{\textgreek{y}}_{k}(s,x)|^{2}\, ds\Big)\, dg_{\text{\textgreek{S}}}\le\nonumber \\
 & \le C\text{\textgreek{w}}_{k}^{2}\int_{\text{\textgreek{S}}\cap\{r\le R\}}\int_{-\infty}^{+\infty}\frac{\text{\textgreek{w}}_{k}}{1+|\text{\textgreek{w}}_{k}(\text{\textgreek{t}}-s)|^{5}}\text{\textgreek{q}}(x)|\text{\textgreek{y}}_{k}(s,x)|^{2}\, dsdg_{\text{\textgreek{S}}}.\nonumber 
\end{align}
 Thus, integrating (\ref{eq:FirstEstimateFromReproducingFormula})
over $\{\text{\textgreek{t}}_{1}\le\text{\textgreek{t}}\le\text{\textgreek{t}}_{2}\}$
we obtain: 
\begin{align}
\int_{\mathcal{R}(\text{\textgreek{t}}_{1},\text{\textgreek{t}}_{2})\cap\{r\le R\}}\text{\textgreek{q}}|T\text{\textgreek{y}}_{k}|^{2} & \le C\text{\textgreek{w}}_{k}^{2}\int_{\text{\textgreek{t}}_{1}}^{\text{\textgreek{t}}_{2}}\int_{\text{\textgreek{S}}\cap\{r\le R\}}\int_{-\infty}^{+\infty}\frac{\text{\textgreek{w}}_{k}}{1+|\text{\textgreek{w}}_{k}(\text{\textgreek{t}}-s)|^{5}}\text{\textgreek{q}}(x)|\text{\textgreek{y}}_{k}(s,x)|^{2}\, dsdg_{\text{\textgreek{S}}}d\text{\textgreek{t}}\le\label{eq:SpacetimeUpperBoundTPsiK}\\
 & \le C\text{\textgreek{w}}_{k}^{2}\Big(\int_{\text{\textgreek{t}}_{1}}^{\text{\textgreek{t}}_{2}}\int_{\text{\textgreek{S}}\cap\{r\le R\}}\int_{\text{\textgreek{t}}_{1}}^{\text{\textgreek{t}}_{2}}\frac{\text{\textgreek{w}}_{k}}{1+|\text{\textgreek{w}}_{k}(\text{\textgreek{t}}-s)|^{5}}\text{\textgreek{q}}(x)|\text{\textgreek{y}}_{k}(s,x)|^{2}\, dsdg_{\text{\textgreek{S}}}d\text{\textgreek{t}}+\nonumber \\
 & \hphantom{\le C\text{\textgreek{w}}_{k}^{2}\Big(}+\int_{\text{\textgreek{t}}_{1}}^{\text{\textgreek{t}}_{2}}\int_{\text{\textgreek{S}}\cap\{r\le R\}}\int_{\mathbb{R}\backslash[\text{\textgreek{t}}_{1},\text{\textgreek{t}}_{2}]}\frac{\text{\textgreek{w}}_{k}}{1+|\text{\textgreek{w}}_{k}(\text{\textgreek{t}}-s)|^{5}}\text{\textgreek{q}}(x)|\text{\textgreek{y}}_{k}(s,x)|^{2}\, dsdg_{\text{\textgreek{S}}}d\text{\textgreek{t}}\Big)\le\nonumber \\
 & \le C\text{\textgreek{w}}_{k}^{2}\Big(\big(\int_{-\infty}^{+\infty}\frac{\text{\textgreek{w}}_{k}}{1+|\text{\textgreek{w}}_{k}\text{\textgreek{l}}|^{5}}\, d\text{\textgreek{l}}\big)\int_{\mathcal{R}(\text{\textgreek{t}}_{1},\text{\textgreek{t}}_{2})\cap\{r\le R\}}\text{\textgreek{q}}|\text{\textgreek{y}}_{k}|^{2}+\nonumber \\
 & \hphantom{\le C\text{\textgreek{w}}_{k}^{2}\Big(}+\int_{\text{\textgreek{S}}\cap\{r\le R\}}\int_{-\infty}^{\text{\textgreek{t}}_{1}}\big(\int_{\text{\textgreek{t}}_{1}}^{\text{\textgreek{t}}_{2}}\frac{\text{\textgreek{w}}_{k}}{1+|\text{\textgreek{w}}_{k}(\text{\textgreek{t}}-s)|^{5}}\, d\text{\textgreek{t}}\big)\text{\textgreek{q}}(x)|\text{\textgreek{y}}_{k}(s,x)|^{2}\, dsdg_{\text{\textgreek{S}}}+\nonumber \\
 & \hphantom{\le C\text{\textgreek{w}}_{k}^{2}\Big(}+\int_{\text{\textgreek{S}}\cap\{r\le R\}}\int_{\text{\textgreek{t}}_{2}}^{+\infty}\big(\int_{\text{\textgreek{t}}_{1}}^{\text{\textgreek{t}}_{2}}\frac{\text{\textgreek{w}}_{k}}{1+|\text{\textgreek{w}}_{k}(\text{\textgreek{t}}-s)|^{5}}\, d\text{\textgreek{t}}\big)\text{\textgreek{q}}(x)|\text{\textgreek{y}}_{k}(s,x)|^{2}\, dsdg_{\text{\textgreek{S}}}\Big)\le\nonumber \\
 & \le C\text{\textgreek{w}}_{k}^{2}\Big(\int_{\mathcal{R}(\text{\textgreek{t}}_{1},\text{\textgreek{t}}_{2})\cap\{r\le R\}}\text{\textgreek{q}}|\text{\textgreek{y}}_{k}|^{2}+\nonumber \\
 & \hphantom{\le C\text{\textgreek{w}}_{k}^{2}\Big(}+\sup_{r\le R}\text{\textgreek{q}}\cdot\int_{\text{\textgreek{S}}\cap\{r\le R\}}\int_{\mathbb{R}\backslash[\text{\textgreek{t}}_{1},\text{\textgreek{t}}_{2}]}\frac{1}{1+(\text{\textgreek{w}}_{k}\min\{|\text{\textgreek{t}}_{1}-s|,|\text{\textgreek{t}}_{2}-s|\})^{4}}|\text{\textgreek{y}}_{k}(s,x)|^{2}\, dsdg_{\text{\textgreek{S}}}\Big).\nonumber 
\end{align}

From (\ref{eq:EnergyClassAPrioriBoundLocal}) we can readily estimate
for any $0<a<1$:
\begin{equation}
\begin{split}\int_{\text{\textgreek{S}}\cap\{r\le R\}}\int_{\mathbb{R}\backslash[\text{\textgreek{t}}_{1},\text{\textgreek{t}}_{2}]} & \frac{1}{1+(\text{\textgreek{w}}_{k}\min\{|\text{\textgreek{t}}_{1}-s|,|\text{\textgreek{t}}_{2}-s|\})^{4}}|\text{\textgreek{y}}_{k}(s,x)|^{2}\, dsdg_{\text{\textgreek{S}}}\le\\
\le & C_{a}R^{2}\Big(\int_{\mathbb{R}\backslash[\text{\textgreek{t}}_{1},\text{\textgreek{t}}_{2}]}\frac{1}{1+(\text{\textgreek{w}}_{k}\min\{|\text{\textgreek{t}}_{1}-s|,|\text{\textgreek{t}}_{2}-s|\})^{4}}\big(\log\big((2+|s|)\big)^{4}\, ds\Big)(1+\text{\textgreek{w}}_{k}^{-a})\mathcal{E}_{log}[\text{\textgreek{y}}]\le\\
\le & C_{a}R^{2}(1+\text{\textgreek{w}}_{k}^{-1-2a})\big(\log(2+\text{\textgreek{t}}_{2})\big)^{4}\mathcal{E}_{log}[\text{\textgreek{y}}].
\end{split}
\label{eq:UpperBoundErrorTerm}
\end{equation}
Thus, from (\ref{eq:SpacetimeUpperBoundTPsiK}) and (\ref{eq:UpperBoundErrorTerm})
we readily infer the right ``half'' of inequality (\ref{eq:w-estimate}),
as well as inequality (\ref{eq:LowFrequencyEstimate}).

In order to establish the left ``half'' of inequality (\ref{eq:w-estimate}),
we will work similarly, using formula (\ref{eq:Reproducing FormulaAntiderivative})
in place of (\ref{eq:Reproducing Formula}). In particular, from (\ref{eq:Reproducing FormulaAntiderivative})
and (\ref{eq:SchwartzBoundk-1-1}) (for $m=5$) we obtain for any
$\text{\textgreek{t}}\ge0$ and any $1\le k\le n$:

\begin{align}
\int_{\text{\textgreek{S}}\cap\{r\le R\}}\text{\textgreek{q}}(x)|\text{\textgreek{y}}_{k}(\text{\textgreek{t}},x)|^{2}\, dg_{\text{\textgreek{S}}} & =\int_{\text{\textgreek{S}}\cap\{r\le R\}}\text{\textgreek{q}}(x)\Big|\int_{-\infty}^{+\infty}\tilde{\text{\textgreek{x}}}_{k}(\text{\textgreek{t}}-s)T\text{\textgreek{y}}_{k}(s,x)\, ds\Big|^{2}\, dg_{\text{\textgreek{S}}}\le\label{eq:FirstEstimateFromReproducingFormulaAntiderivative}\\
 & \le C\int_{\text{\textgreek{S}}\cap\{r\le R\}}\text{\textgreek{q}}(x)\Big|\int_{-\infty}^{+\infty}\frac{1}{1+|\text{\textgreek{w}}_{k}(\text{\textgreek{t}}-s)|^{5}}T\text{\textgreek{y}}_{k}(s,x)\, ds\Big|^{2}\, dg_{\text{\textgreek{S}}}\le\nonumber \\
 & \le C\text{\textgreek{w}}_{k}^{-2}\int_{\text{\textgreek{S}}\cap\{r\le R\}}\text{\text{\textgreek{q}}}(x)\Big(\int_{-\infty}^{+\infty}\frac{\text{\textgreek{w}}_{k}}{1+|\text{\textgreek{w}}_{k}(\text{\textgreek{t}}-s)|^{5}}\, ds\Big)\Big(\int_{-\infty}^{+\infty}\frac{\text{\textgreek{w}}_{k}}{1+|\text{\textgreek{w}}_{k}(\text{\textgreek{t}}-s)|^{5}}|T\text{\textgreek{y}}_{k}(s,x)|^{2}\, ds\Big)\, dg_{\text{\textgreek{S}}}\le\nonumber \\
 & \le C\text{\textgreek{w}}_{k}^{-2}\int_{\text{\textgreek{S}}\cap\{r\le R\}}\int_{-\infty}^{+\infty}\frac{\text{\textgreek{w}}_{k}}{1+|\text{\textgreek{w}}_{k}(\text{\textgreek{t}}-s)|^{5}}\text{\textgreek{q}}(x)|T\text{\textgreek{y}}_{k}(s,x)|^{2}\, dsdg_{\text{\textgreek{S}}}.\nonumber 
\end{align}
Integrating (\ref{eq:FirstEstimateFromReproducingFormulaAntiderivative})
over $\{\text{\textgreek{t}}_{1}\le\text{\textgreek{t}}\le\text{\textgreek{t}}_{2}\}$,
we obtain:
\begin{align}
\int_{\mathcal{R}(\text{\textgreek{t}}_{1},\text{\textgreek{t}}_{2})\cap\{r\le R\}}\text{\textgreek{q}}|\text{\textgreek{y}}_{k}|^{2} & \le C\text{\textgreek{w}}_{k}^{-2}\int_{\text{\textgreek{t}}_{1}}^{\text{\textgreek{t}}_{2}}\int_{\text{\textgreek{S}}\cap\{r\le R\}}\int_{-\infty}^{+\infty}\frac{\text{\textgreek{w}}_{k}}{1+|\text{\textgreek{w}}_{k}(\text{\textgreek{t}}-s)|^{5}}\text{\textgreek{q}}(x)|T\text{\textgreek{y}}_{k}(s,x)|^{2}\, dsdg_{\text{\textgreek{S}}}d\text{\textgreek{t}}\le\label{eq:SpacetimeLowerBoundTPsiK}\\
 & \le C\text{\textgreek{w}}_{k}^{-2}\Big(\int_{\text{\textgreek{t}}_{1}}^{\text{\textgreek{t}}_{2}}\int_{\text{\textgreek{S}}\cap\{r\le R\}}\int_{\text{\textgreek{t}}_{1}}^{\text{\textgreek{t}}_{2}}\frac{\text{\textgreek{w}}_{k}}{1+|\text{\textgreek{w}}_{k}(\text{\textgreek{t}}-s)|^{5}}\text{\textgreek{q}}(x)|T\text{\textgreek{y}}_{k}(s,x)|^{2}\, dsdg_{\text{\textgreek{S}}}d\text{\textgreek{t}}+\nonumber \\
 & \hphantom{\le C\text{\textgreek{w}}_{k}^{2}\Big(}+\int_{\text{\textgreek{t}}_{1}}^{\text{\textgreek{t}}_{2}}\int_{\text{\textgreek{S}}\cap\{r\le R\}}\int_{\mathbb{R}\backslash[\text{\textgreek{t}}_{1},\text{\textgreek{t}}_{2}]}\frac{\text{\textgreek{w}}_{k}}{1+|\text{\textgreek{w}}_{k}(\text{\textgreek{t}}-s)|^{5}}\text{\textgreek{q}}(x)|T\text{\textgreek{y}}_{k}(s,x)|^{2}\, dsdg_{\text{\textgreek{S}}}d\text{\textgreek{t}}\Big)\le\nonumber \\
 & \le C\text{\textgreek{w}}_{k}^{-2}\Big(\big(\int_{-\infty}^{+\infty}\frac{\text{\textgreek{w}}_{k}}{1+|\text{\textgreek{w}}_{k}\text{\textgreek{l}}|^{5}}\, d\text{\textgreek{l}}\big)\int_{\mathcal{R}(\text{\textgreek{t}}_{1},\text{\textgreek{t}}_{2})\cap\{r\le R\}}\text{\textgreek{q}}|T\text{\textgreek{y}}_{k}|^{2}+\nonumber \\
 & \hphantom{\le C\text{\textgreek{w}}_{k}^{2}\Big(}+\int_{\text{\textgreek{S}}\cap\{r\le R\}}\int_{-\infty}^{\text{\textgreek{t}}_{1}}\big(\int_{\text{\textgreek{t}}_{1}}^{\text{\textgreek{t}}_{2}}\frac{\text{\textgreek{w}}_{k}}{1+|\text{\textgreek{w}}_{k}(\text{\textgreek{t}}-s)|^{5}}\, d\text{\textgreek{t}}\big)\text{\textgreek{q}}(x)|T\text{\textgreek{y}}_{k}(s,x)|^{2}\, dsdg_{\text{\textgreek{S}}}+\nonumber \\
 & \hphantom{\le C\text{\textgreek{w}}_{k}^{2}\Big(}+\int_{\text{\textgreek{S}}\cap\{r\le R\}}\int_{\text{\textgreek{t}}_{2}}^{+\infty}\big(\int_{\text{\textgreek{t}}_{1}}^{\text{\textgreek{t}}_{2}}\frac{\text{\textgreek{w}}_{k}}{1+|\text{\textgreek{w}}_{k}(\text{\textgreek{t}}-s)|^{5}}\, d\text{\textgreek{t}}\big)\text{\textgreek{q}}(x)|T\text{\textgreek{y}}_{k}(s,x)|^{2}\, dsdg_{\text{\textgreek{S}}}\Big)\le\nonumber \\
 & \le C\text{\textgreek{w}}_{k}^{-2}\Big(\int_{\mathcal{R}(\text{\textgreek{t}}_{1},\text{\textgreek{t}}_{2})\cap\{r\le R\}}\text{\textgreek{q}}|\text{\textgreek{y}}_{k}|^{2}+\nonumber \\
 & \hphantom{\le C\text{\textgreek{w}}_{k}^{-2}\Big(}+\sup_{r\le R}\text{\textgreek{q}}\cdot\int_{\text{\textgreek{S}}\cap\{r\le R\}}\int_{\mathbb{R}\backslash[\text{\textgreek{t}}_{1},\text{\textgreek{t}}_{2}]}\frac{1}{1+(\text{\textgreek{w}}_{k}\min\{|\text{\textgreek{t}}_{1}-s|,|\text{\textgreek{t}}_{2}-s|\})^{4}}|T\text{\textgreek{y}}_{k}(s,x)|^{2}\, dsdg_{\text{\textgreek{S}}}\Big).\nonumber 
\end{align}
 From (\ref{eq:EnergyClassAPrioriBoundLocal}) we can estimate:
\begin{equation}
\begin{split}\int_{\text{\textgreek{S}}\cap\{r\le R\}}\int_{\mathbb{R}\backslash[\text{\textgreek{t}}_{1},\text{\textgreek{t}}_{2}]} & \frac{1}{1+(\text{\textgreek{w}}_{k}\min\{|\text{\textgreek{t}}_{1}-s|,|\text{\textgreek{t}}_{2}-s|\})^{4}}|T\text{\textgreek{y}}_{k}(s,x)|^{2}\, dsdg_{\text{\textgreek{S}}}\le\\
\le & C_{a}\Big(\int_{\mathbb{R}\backslash[\text{\textgreek{t}}_{1},\text{\textgreek{t}}_{2}]}\frac{(1+\max\{0,-s\})^{2}\big(\log(2+\max\{0,-s\})\big)^{4}}{1+(\text{\textgreek{w}}_{k}\min\{|\text{\textgreek{t}}_{1}-s|,|\text{\textgreek{t}}_{2}-s|\})^{4}}\, ds\Big)(1+\text{\textgreek{w}}_{k}^{-2-a})\mathcal{E}_{log}[\text{\textgreek{y}}]\le\\
\le & C_{a}(1+\text{\textgreek{w}}_{k}^{-5-2a})\mathcal{E}_{log}[\text{\textgreek{y}}].
\end{split}
\label{eq:UpperBoundErrorTerm-1}
\end{equation}
Thus, the left ``half'' of inequality (\ref{eq:w-estimate}) follows
from (\ref{eq:SpacetimeLowerBoundTPsiK}) and (\ref{eq:UpperBoundErrorTerm-1}).
\end{proof}
We will also need the following estimate in the case when $\text{\textgreek{y}}$
is of the form $T\text{\textgreek{f}}$, where $\text{\textgreek{f}}$
is a smooth solution to the wave equation on $\mathcal{D}(\text{\textgreek{S}})$:
\begin{lem}
\label{lem:NearZeroFrequencyEstimateDerivatives}Let $\text{\textgreek{y}}$
be of the form
\begin{equation}
\text{\textgreek{y}}=T\text{\textgreek{f}},\label{eq:DefinitionPsiLemma}
\end{equation}
 where $\text{\textgreek{f}}:\mathcal{D}(\text{\textgreek{S}})\rightarrow\mathbb{C}$
is a smooth function solving (\ref{eq:WaveEquation}) with compactly
supported initial data on $\text{\textgreek{S}}$, such that $\mathcal{E}[\text{\textgreek{f}}]<+\infty.$
Then, for any $0\le\text{\textgreek{t}}_{1}\le\text{\textgreek{t}}_{2}$,
any $0<a<1$ and any $R\ge0$ we can bound: 
\begin{align}
\int_{\mathcal{R}(\text{\textgreek{t}}_{1},\text{\textgreek{t}}_{2})\cap\{r\le R\}}\big(J_{\text{\textgreek{m}}}^{N}(\text{\textgreek{y}}_{0})N^{\text{\textgreek{m}}}+|\text{\textgreek{y}}_{0}|^{2}\big) & \le C\text{\textgreek{w}}_{0}^{2}\int_{\mathcal{R}(\text{\textgreek{t}}_{1},\text{\textgreek{t}}_{2})\cap\{r\le R\}}\big(J_{\text{\textgreek{m}}}^{N}(\text{\textgreek{f}})N^{\text{\textgreek{m}}}+|\text{\textgreek{f}}|^{2}\big)+\label{eq:LowFrequencyEstimateDerivative}\\
 & \hphantom{\le C\Big(\text{\textgreek{w}}_{0}^{2}}+C_{a}\big(\text{\textgreek{w}}_{0}^{2}(1+\text{\textgreek{w}}_{0}^{-1-a})R^{2}\big(\log(2+|\text{\textgreek{t}}_{2}|\big)^{4}+(1+\text{\textgreek{w}}_{0}\text{\textgreek{t}}_{1})^{-1}R^{2}\big)\mathcal{E}_{log}[\text{\textgreek{f}}].\nonumber 
\end{align}
\end{lem}
\begin{proof}
The bounds (\ref{eq:BoundednessAtAlltimes}), (\ref{eq:WeightedBoundednessLogarithmic})
for $\text{\textgreek{f}}$ in place of $\text{\textgreek{y}}$ (combined
with the Hardy-type inequalities (\ref{eq:GeneralHardyLogarithmic})
and (\ref{eq:HardyInterior})) imply that 
\begin{equation}
\sup_{\text{\textgreek{t}}\in\mathbb{R}}\int_{\text{\textgreek{S}}_{\text{\textgreek{t}}}\cap\{t_{-}\ge0\}}J^{N}(\text{\textgreek{f}})n^{\text{\textgreek{m}}}+\sup_{\text{\textgreek{t}}\in\mathbb{R}}\Big(\big(\log(2+|\text{\textgreek{t}}|)\big)^{4}\int_{\text{\textgreek{S}}_{\text{\textgreek{t}}}\cap\{t_{-}\ge0\}}\big(1+r\big)^{-2}|\text{\textgreek{f}}|^{2}\Big)\le C\mathcal{E}_{log}[\text{\textgreek{f}}].\label{eq:EnergyBoundPhi}
\end{equation}

From (\ref{eq:DefinitionPsiLemma}), (\ref{eq:CutOffPsi}) and (\ref{eq:FrequencyLocalisedPart})
we calculate:
\begin{align}
\int_{\text{\textgreek{S}}\cap\{r\le R\}}\int_{\text{\textgreek{t}}_{1}}^{\text{\textgreek{t}}_{2}}|\text{\textgreek{y}}_{0}(s,x)|^{2}\, d\text{\textgreek{t}}dg_{\text{\textgreek{S}}} & =\int_{\text{\textgreek{S}}\cap\{r\le R\}}\int_{\text{\textgreek{t}}_{1}}^{\text{\textgreek{t}}_{2}}\Big|\int_{-\infty}^{+\infty}\text{\textgreek{z}}_{0}(t-s)\text{\textgreek{y}}_{c}(s,x)\, ds\Big|^{2}\, dtdg_{\text{\textgreek{S}}}=\label{eq:FirstEqualityConvolution}\\
 & =\int_{\text{\textgreek{S}}\cap\{r\le R\}}\int_{\text{\textgreek{t}}_{1}}^{\text{\textgreek{t}}_{2}}\Big|\int_{-\infty}^{+\infty}\text{\textgreek{z}}_{0}(t-s)\text{\textgreek{q}}_{2}(s+\frac{1}{2}\text{\textgreek{q}}_{1}(r)(r-R_{1}))\partial_{t}\text{\textgreek{f}}(s,x)\, ds\Big|^{2}\, dtdg_{\text{\textgreek{S}}}=\nonumber \\
 & =\int_{\text{\textgreek{S}}\cap\{r\le R\}}\int_{\text{\textgreek{t}}_{1}}^{\text{\textgreek{t}}_{2}}\Big|-\int_{-\infty}^{+\infty}\frac{d}{ds}\big(\text{\textgreek{z}}_{0}(t-s)\text{\textgreek{q}}_{2}(s+\frac{1}{2}\text{\textgreek{q}}_{1}(r)(r-R_{1}))\big)\cdot\text{\textgreek{f}}(s,x)\, ds\Big|^{2}\, dtdg_{\text{\textgreek{S}}},\nonumber 
\end{align}
noting that the integrating by parts in the last step of (\ref{eq:FirstEqualityConvolution})
is possible in view of the Schwartz bound (\ref{eq:SchwartzBoundk})
on $\text{\textgreek{z}}_{0}$ and (\ref{eq:EnergyBoundPhi}).

In view of (\ref{eq:SchwartzBoundk}), the relation (\ref{eq:FirstEqualityConvolution})
yields: 
\begin{align}
\int_{\text{\textgreek{S}}\cap\{r\le R\}}\int_{\text{\textgreek{t}}_{1}}^{\text{\textgreek{t}}_{2}}|\text{\textgreek{y}}_{0}(s, & x)|^{2}\, d\text{\textgreek{t}}dg_{\text{\textgreek{S}}}\le C\int_{\text{\textgreek{S}}\cap\{r\le R\}}\int_{\text{\textgreek{t}}_{1}}^{\text{\textgreek{t}}_{2}}\Big(\Big|\int_{-\infty}^{+\infty}\frac{\text{\textgreek{w}}_{0}^{2}}{1+|\text{\textgreek{w}}_{0}(t-s)|^{3}}\text{\textgreek{q}}_{2}(s+\frac{1}{2}\text{\textgreek{q}}_{1}(r)(r-R_{1}))\cdot\text{\textgreek{f}}(s,x)\, ds\Big|^{2}+\label{eq:FirstEqualityConvolution-1}\\
 & \hphantom{x)|^{2}\, d\text{\textgreek{t}}dg_{\text{\textgreek{S}}}\le C\int_{\text{\textgreek{S}}\cap\{r\le R\}}\int_{\text{\textgreek{t}}_{1}}^{\text{\textgreek{t}}_{2}}\Big(}+\Big|\int_{-\infty}^{+\infty}\frac{\text{\textgreek{w}}_{0}}{1+|\text{\textgreek{w}}_{0}(t-s)|^{3}}\text{\textgreek{q}}_{2}^{\prime}(s+\frac{1}{2}\text{\textgreek{q}}_{1}(r)(r-R_{1}))\cdot\text{\textgreek{f}}(s,x)\, ds\Big|^{2}\Big)\, dtdg_{\text{\textgreek{S}}}\le\nonumber \\
 & \le C\Big(\int_{-\infty}^{+\infty}\frac{\text{\textgreek{w}}_{0}}{1+|\text{\textgreek{w}}_{0}(t-s)|^{3}}\, ds\Big)\Bigg\{\text{\textgreek{w}}_{0}^{2}\int_{\text{\textgreek{S}}\cap\{r\le R\}}\int_{\text{\textgreek{t}}_{1}}^{\text{\textgreek{t}}_{2}}\int_{-\frac{1}{2}\text{\textgreek{q}}_{1}(r)(r-R_{1})}^{+\infty}\frac{\text{\textgreek{w}}_{0}}{1+|\text{\textgreek{w}}_{0}(t-s)|^{3}}|\text{\textgreek{f}}(s,x)|^{2}\, dsdtdg_{\text{\textgreek{S}}}+\nonumber \\
 & \hphantom{\le C\Big(\int_{-\infty}^{+\infty}\frac{\text{\textgreek{w}}_{0}}{1+|\text{\textgreek{w}}_{0}(t-s)|^{3}}\, ds\Big)\Bigg\{}+\int_{\text{\textgreek{S}}\cap\{r\le R\}}\int_{\text{\textgreek{t}}_{1}}^{\text{\textgreek{t}}_{2}}\int_{-\frac{1}{2}\text{\textgreek{q}}_{1}(r)(r-R_{1})}^{1-\frac{1}{2}\text{\textgreek{q}}_{1}(r)(r-R_{1})}\frac{\text{\textgreek{w}}_{0}}{1+|\text{\textgreek{w}}_{0}(t-s)|^{3}}|\text{\textgreek{f}}(s,x)|^{2}\, dsdtdg_{\text{\textgreek{S}}}\Bigg\}\le\nonumber \\
 & \le C\Bigg\{\text{\textgreek{w}}_{0}^{2}\int_{\text{\textgreek{S}}\cap\{r\le R\}}\int_{[\text{\textgreek{t}}_{1},\text{\textgreek{t}}_{2}]\cap[-\frac{1}{2}\text{\textgreek{q}}_{1}(r)(r-R_{1}),+\infty)}\Big(\int_{\text{\textgreek{t}}_{1}}^{\text{\textgreek{t}}_{2}}\frac{\text{\textgreek{w}}_{0}}{1+|\text{\textgreek{w}}_{0}(t-s)|^{3}}\, dt\Big)|\text{\textgreek{f}}(s,x)|^{2}\, dsdg_{\text{\textgreek{S}}}+\nonumber \\
 & \hphantom{\le C\Bigg\{}+\text{\textgreek{w}}_{0}^{2}\int_{\text{\textgreek{S}}\cap\{r\le R\}}\int_{[-\frac{1}{2}\text{\textgreek{q}}_{1}(r)(r-R_{1}),+\infty)\backslash[\text{\textgreek{t}}_{1},\text{\textgreek{t}}_{2}]}\Big(\int_{\text{\textgreek{t}}_{1}}^{\text{\textgreek{t}}_{2}}\frac{\text{\textgreek{w}}_{0}}{1+|\text{\textgreek{w}}_{0}(t-s)|^{3}}\, dt\Big)|\text{\textgreek{f}}(s,x)|^{2}\, dsdg_{\text{\textgreek{S}}}+\nonumber \\
 & \hphantom{\le C\Bigg\{}+\int_{\text{\textgreek{S}}\cap\{r\le R\}}\int_{-\frac{1}{2}\text{\textgreek{q}}_{1}(r)(r-R_{1})}^{1-\frac{1}{2}\text{\textgreek{q}}_{1}(r)(r-R_{1})}\Big(\int_{\text{\textgreek{t}}_{1}}^{\text{\textgreek{t}}_{2}}\frac{\text{\textgreek{w}}_{0}}{1+|\text{\textgreek{w}}_{0}(t-s)|^{3}}\, dt\Big)|\text{\textgreek{f}}(s,x)|^{2}\, dsdg_{\text{\textgreek{S}}}\Bigg\}\le\nonumber \\
 & \le C\Bigg\{\text{\textgreek{w}}_{0}^{2}\int_{\text{\textgreek{S}}\cap\{r\le R\}}\int_{\text{\textgreek{t}}_{1}}^{\text{\textgreek{t}}_{2}}|\text{\textgreek{f}}(s,x)|^{2}\, dsdg_{\text{\textgreek{S}}}+\nonumber \\
 & \hphantom{\le C\Bigg\{}+\text{\textgreek{w}}_{0}^{2}\int_{\text{\textgreek{S}}\cap\{r\le R\}}\int_{[-\frac{1}{2}\text{\textgreek{q}}_{1}(r)(r-R_{1}),+\infty)\backslash[\text{\textgreek{t}}_{1},\text{\textgreek{t}}_{2}]}\frac{1}{1+\min\{|\text{\textgreek{w}}_{0}(\text{\textgreek{t}}_{1}-s)|^{2},|\text{\textgreek{w}}_{0}(\text{\textgreek{t}}_{2}-s)|^{2}\}}|\text{\textgreek{f}}(s,x)|^{2}\, dsdg_{\text{\textgreek{S}}}+\nonumber \\
 & \hphantom{\le C\Bigg\{}+\int_{\text{\textgreek{S}}\cap\{r\le R\}}\int_{-\frac{1}{2}\text{\textgreek{q}}_{1}(r)(r-R_{1})}^{1-\frac{1}{2}\text{\textgreek{q}}_{1}(r)(r-R_{1})}\frac{1}{1+|\text{\textgreek{w}}_{0}(\text{\textgreek{t}}_{1}-s)|^{2}}|\text{\textgreek{f}}(s,x)|^{2}\, dsdg_{\text{\textgreek{S}}}\Bigg\}.\nonumber 
\end{align}
From (\ref{eq:EnergyBoundPhi}), we can estimate for any $0<a<1$:
\begin{multline}
\text{\textgreek{w}}_{0}^{2}\int_{\text{\textgreek{S}}\cap\{r\le R\}}\int_{[-\frac{1}{2}\text{\textgreek{q}}_{1}(r)(r-R_{1}),+\infty)\backslash[\text{\textgreek{t}}_{1},\text{\textgreek{t}}_{2}]}\frac{1}{1+\min\{|\text{\textgreek{w}}_{0}(\text{\textgreek{t}}_{1}-s)|^{2},|\text{\textgreek{w}}_{0}(\text{\textgreek{t}}_{2}-s)|^{2}\}}|\text{\textgreek{f}}(s,x)|^{2}\, dsdg_{\text{\textgreek{S}}}\le\\
\le C_{a}\text{\textgreek{w}}_{0}^{2}(1+\text{\textgreek{w}}_{0}^{-1-a})R^{2}\big(\log(2+|\text{\textgreek{t}}_{2}|\big)^{4}\mathcal{E}_{log}[\text{\textgreek{f}}]
\end{multline}
and 
\begin{equation}
\int_{\text{\textgreek{S}}\cap\{r\le R\}}\int_{-\frac{1}{2}\text{\textgreek{q}}_{1}(r)(r-R_{1})}^{1-\frac{1}{2}\text{\textgreek{q}}_{1}(r)(r-R_{1})}\frac{1}{1+|\text{\textgreek{w}}_{0}(\text{\textgreek{t}}_{1}-s)|^{2}}|\text{\textgreek{f}}(s,x)|^{2}\, dsdg_{\text{\textgreek{S}}}\le C_{a}(1+\text{\textgreek{w}}_{0}^{-1-a})(1+\text{\textgreek{w}}_{0}\text{\textgreek{t}}_{1})^{-1}R^{2}\mathcal{E}_{log}[\text{\textgreek{f}}].
\end{equation}
 Thus, from (\ref{eq:FirstEqualityConvolution-1}) we obtain for any
$0<a<1$: 
\begin{align}
\int_{\text{\textgreek{S}}\cap\{r\le R\}}\int_{\text{\textgreek{t}}_{1}}^{\text{\textgreek{t}}_{2}}|\text{\textgreek{y}}_{0}(s,x)|^{2}\, d\text{\textgreek{t}}dg_{\text{\textgreek{S}}}\le C & \text{\textgreek{w}}_{0}^{2}\int_{\text{\textgreek{S}}\cap\{r\le R\}}\int_{\text{\textgreek{t}}_{1}}^{\text{\textgreek{t}}_{2}}|\text{\textgreek{f}}(s,x)|^{2}\, dsdg_{\text{\textgreek{S}}}+\label{eq:BoundZerothOrder}\\
 & +C_{a}\big(\text{\textgreek{w}}_{0}^{2}\big(\log(2+|\text{\textgreek{t}}_{2}|\big)^{4}+(1+\text{\textgreek{w}}_{0}\text{\textgreek{t}}_{1})^{-1}\big)(1+\text{\textgreek{w}}_{0}^{-1-a})R^{2}\mathcal{E}_{log}[\text{\textgreek{f}}].\nonumber 
\end{align}
Repeating the same procedure with $T\text{\textgreek{y}}_{0}$ and
$\nabla_{g_{\text{\textgreek{S}}}}\text{\textgreek{y}}_{0}$ in place
of $\text{\textgreek{y}}_{0}$, we similarly obtain: 
\begin{align}
\int_{\text{\textgreek{S}}\cap\{r\le R\}}\int_{\text{\textgreek{t}}_{1}}^{\text{\textgreek{t}}_{2}}|T\text{\textgreek{y}}_{0}(s,x)|^{2}\, d\text{\textgreek{t}}dg_{\text{\textgreek{S}}}\le C & \text{\textgreek{w}}_{0}^{2}\int_{\text{\textgreek{S}}\cap\{r\le R\}}\int_{\text{\textgreek{t}}_{1}}^{\text{\textgreek{t}}_{2}}|T\text{\textgreek{f}}(s,x)|^{2}\, dsdg_{\text{\textgreek{S}}}+\label{eq:BoundZerothOrder-1}\\
 & +C_{a}\big(\text{\textgreek{w}}_{0}^{2}\big(\log(2+|\text{\textgreek{t}}_{2}|\big)^{4}+(1+\text{\textgreek{w}}_{0}\text{\textgreek{t}}_{1})^{-1}\big)(1+\text{\textgreek{w}}_{0}^{-1-a})R^{2}\mathcal{E}_{log}[\text{\textgreek{f}}]\nonumber 
\end{align}
 and 
\begin{align}
\int_{\text{\textgreek{S}}\cap\{r\le R\}}\int_{\text{\textgreek{t}}_{1}}^{\text{\textgreek{t}}_{2}}|\nabla_{g_{\text{\textgreek{S}}}}\text{\textgreek{y}}_{0}(s,x)|_{g_{\text{\textgreek{S}}}}^{2}\, d\text{\textgreek{t}}dg_{\text{\textgreek{S}}}\le C & \text{\textgreek{w}}_{0}^{2}\int_{\text{\textgreek{S}}\cap\{r\le R\}}\int_{\text{\textgreek{t}}_{1}}^{\text{\textgreek{t}}_{2}}|\nabla_{g_{\text{\textgreek{S}}}}\text{\textgreek{f}}(s,x)|_{g_{\text{\textgreek{S}}}}^{2}\, dsdg_{\text{\textgreek{S}}}+\label{eq:BoundZerothOrder-2}\\
 & +C_{a}\big(\text{\textgreek{w}}_{0}^{2}\big(\log(2+|\text{\textgreek{t}}_{2}|\big)^{4}+(1+\text{\textgreek{w}}_{0}\text{\textgreek{t}}_{1})^{-1}\big)(1+\text{\textgreek{w}}_{0}^{-1-a})R^{2}\mathcal{E}_{log}[\text{\textgreek{f}}].\nonumber 
\end{align}

Inequality (\ref{eq:LowFrequencyEstimateDerivative}) readily follows
after adding (\ref{eq:BoundZerothOrder}), (\ref{eq:BoundZerothOrder-1})
and (\ref{eq:BoundZerothOrder-2}).
\end{proof}
We will finally establish the following bound for the energy of the
high frequency part $\text{\textgreek{y}}_{\ge\text{\textgreek{w}}_{+}}$
of $\text{\textgreek{y}}$:
\begin{lem}
\label{lem:HighFrequencies}For any $\text{\textgreek{t}}\ge0$ and
any $m\in\mathbb{N}$ such that 
\begin{equation}
\sum_{j=0}^{m}\mathcal{E}[T^{j}\text{\textgreek{y}}]<+\infty,\label{eq:BoundEnergyHigherDerivatives}
\end{equation}
there exists a constant $C_{m}>0$ depending only on $m$ such that:

\begin{equation}
\int_{\{t=\text{\textgreek{t}}\}}J_{\text{\textgreek{m}}}^{N}(\text{\textgreek{y}}_{\ge\text{\textgreek{w}}_{+}})n^{\text{\textgreek{m}}}\le\frac{C_{m}}{\text{\textgreek{w}}_{+}^{2m}}\Big(\sum_{j=0}^{m}\mathcal{E}[T^{j}\text{\textgreek{y}}]+\mathcal{E}_{log}[\text{\textgreek{y}}]\Big).\label{eq:HighFrequencyBoundLemma}
\end{equation}
\end{lem}
\begin{proof}
We can assume without loss of generality that $m\ge1$, since the
$m=0$ case is a direct consequence of (\ref{eq:EnergyClassAPrioriBoundHigh}).
Let us introduce the function $\breve{\text{\textgreek{x}}}_{m}:\mathbb{R}\backslash\{0\}\rightarrow\mathbb{C}$
by the formula 
\begin{equation}
\breve{\text{\textgreek{x}}}_{m}(t)=\int_{-\infty}^{+\infty}(i\text{\textgreek{w}}_{+}^{-1}\text{\textgreek{w}})^{-m}e^{i\text{\textgreek{w}}t}(1-\text{\textgreek{q}}(\text{\textgreek{w}}_{+}^{-1}\text{\textgreek{w}}))\, d\text{\textgreek{w}}.\label{eq:TemperedDistribution}
\end{equation}
Note that, when $m=1$, the right hand side of (\ref{eq:TemperedDistribution})
diverges when $t=0$. In view of the bound 
\begin{equation}
\Big|\int_{1}^{+\infty}\frac{1}{y^{m}}e^{i\text{\textgreek{l}}y}\, dy\Big|\le\begin{cases}
C\big(|\log(\text{\textgreek{l}})|+1\big), & m=1\\
C, & m>1,
\end{cases}
\end{equation}
as well as the relation 
\begin{equation}
t\breve{\text{\textgreek{x}}}_{m}(t)=i\int_{-\infty}^{+\infty}e^{i\text{\textgreek{w}}t}\frac{d}{d\text{\textgreek{w}}}\big((i\text{\textgreek{w}})^{-m}(1-\text{\textgreek{q}}(\text{\textgreek{w}}_{+}^{-1}\text{\textgreek{w}}))\big)\, d\text{\textgreek{w}}=-m\text{\textgreek{w}}_{+}^{-1}\breve{\text{\textgreek{x}}}_{m+1}(t)-i\text{\textgreek{w}}_{+}^{-1}\int_{-\infty}^{+\infty}e^{i\text{\textgreek{w}}t}(i\text{\textgreek{w}}_{+}^{-1}\text{\textgreek{w}})^{-m}\text{\textgreek{q}}^{\prime}(\text{\textgreek{w}}_{+}^{-1}\text{\textgreek{w}}))\, d\text{\textgreek{w}},
\end{equation}
from (\ref{eq:TemperedDistribution}) we infer that for any integer
$q\in\mathbb{N}$ and any $t\neq0$:
\begin{equation}
|\breve{\text{\textgreek{x}}}_{m}(t)|\le C_{qm}\text{\textgreek{w}}_{+}\frac{|\log(|\text{\textgreek{w}}_{+}t|)|+1}{|\text{\textgreek{w}}_{+}t|^{q}+1}.\label{eq:SchwartzBoundLogarithmic}
\end{equation}

Defining the tempered distribution 
\begin{equation}
\text{\textgreek{z}}_{\ge\text{\textgreek{w}}_{+}}\doteq\text{\textgreek{d}}_{D}-\text{\textgreek{z}}_{\le\text{\textgreek{w}}_{+}},
\end{equation}
 where $\text{\textgreek{d}}_{D}$ is Dirac's delta function and $\text{\textgreek{z}}_{\le\text{\textgreek{w}}_{+}}$
is defined by (\ref{eq:FrequencyCutOffFunctions}), the Fourier transforms
of $\breve{\text{\textgreek{x}}}_{m}$ and $\text{\textgreek{z}}_{\ge\text{\textgreek{w}}_{+}}$
satisfy the relation:
\begin{equation}
\hat{\text{\textgreek{z}}}_{\ge\text{\textgreek{w}}_{+}}=\text{\textgreek{w}}_{+}^{-m}\widehat{\breve{\text{\textgreek{x}}}}_{m}\cdot(i\text{\textgreek{w}})^{m}\hat{\text{\textgreek{z}}}_{\ge\text{\textgreek{w}}_{+}},
\end{equation}
yielding the following relation for $\text{\textgreek{y}}_{\ge\text{\textgreek{w}}_{+}}$
in physical space: 
\begin{equation}
\text{\textgreek{y}}_{\ge\text{\textgreek{w}}_{+}}(t,\cdot)=\text{\textgreek{w}}_{+}^{-m}\int_{-\infty}^{+\infty}\breve{\text{\textgreek{x}}}_{m}(t-s)T^{m}\text{\textgreek{y}}_{\ge\text{\textgreek{w}}_{+}}(s,\cdot)\, ds,\label{eq:ReproducingFormulaHigherFrequencyPart}
\end{equation}
where, again, the integral in the right hand side of (\ref{eq:ReproducingFormulaHigherFrequencyPart})
converges in the $\int_{\text{\textgreek{S}}_{t}}J_{\text{\textgreek{m}}}^{N}(\cdot)n^{\text{\textgreek{m}}}$
norm.

From (\ref{eq:ReproducingFormulaHigherFrequencyPart}) and (\ref{eq:SchwartzBoundLogarithmic})
we can estimate for any $\text{\textgreek{t}}\in\mathbb{R}$:
\begin{equation}
\begin{split}\int_{\text{\textgreek{S}}}\big(|T\text{\textgreek{y}}_{\ge\text{\textgreek{w}}_{+}} & (\text{\textgreek{t}},x)|^{2}+|\nabla_{g_{\text{\textgreek{S}}}}\text{\textgreek{y}}_{\ge\text{\textgreek{w}}_{+}}(\text{\textgreek{t}},x)|_{g_{\text{\textgreek{S}}}}^{2}\big)\, dg_{\text{\textgreek{S}}}=\\
= & \text{\textgreek{w}}_{+}^{-2m}\int_{\text{\textgreek{S}}}\Big(\Big|\int_{-\infty}^{+\infty}\breve{\text{\textgreek{x}}}_{m}(\text{\textgreek{t}}-s)T^{m+1}\text{\textgreek{y}}_{\ge\text{\textgreek{w}}_{+}}(s,x)\, ds\Big|^{2}+\Big|\int_{-\infty}^{+\infty}\breve{\text{\textgreek{x}}}_{m}(\text{\textgreek{t}}-s)\nabla_{g_{\text{\textgreek{S}}}}T^{m}\text{\textgreek{y}}_{\ge\text{\textgreek{w}}_{+}}(s,x)\, ds\Big|_{g_{\text{\textgreek{S}}}}^{2}\big)\, dg_{\text{\textgreek{S}}}\le\\
\le & C_{m}\text{\textgreek{w}}_{+}^{-2m}\Big\{\int_{\text{\textgreek{S}}}\Big(\Big|\int_{-\infty}^{+\infty}\text{\textgreek{w}}_{+}\frac{|\log(|\text{\textgreek{w}}_{+}(\text{\textgreek{t}}-s)|)|+1}{|\text{\textgreek{w}}_{+}(\text{\textgreek{t}}-s)|^{4}+1}T^{m+1}\text{\textgreek{y}}_{\ge\text{\textgreek{w}}_{+}}(s,x)\, ds\Big|^{2}+\\
 & \hphantom{C_{m}\Big\{\int_{\text{\textgreek{S}}}\Big(}+\Big|\int_{-\infty}^{+\infty}\text{\textgreek{w}}_{+}\frac{|\log(|\text{\textgreek{w}}_{+}(\text{\textgreek{t}}-s)|)|+1}{|\text{\textgreek{w}}_{+}(\text{\textgreek{t}}-s)|^{4}+1}\nabla_{g_{\text{\textgreek{S}}}}T^{m}\text{\textgreek{y}}_{\ge\text{\textgreek{w}}_{+}}(s,x)\, ds\Big|_{g_{\text{\textgreek{S}}}}^{2}\big)\, dg_{\text{\textgreek{S}}}\Big\}\le\\
\le & C_{m}\text{\textgreek{w}}_{+}^{-2m}\Big(\int_{-\infty}^{+\infty}\text{\textgreek{w}}_{+}\frac{|\log(|\text{\textgreek{w}}_{+}(\text{\textgreek{t}}-s)|)|+1}{|\text{\textgreek{w}}_{+}(\text{\textgreek{t}}-s)|^{4}+1}\, ds\Big)\times\\
 & \hphantom{C_{m}\text{\textgreek{w}}_{+}^{-2m}\Big(}\times\Big(\int_{\text{\textgreek{S}}}\int_{-\infty}^{+\infty}\text{\textgreek{w}}_{+}\frac{|\log(|\text{\textgreek{w}}_{+}(\text{\textgreek{t}}-s)|)|+1}{|\text{\textgreek{w}}_{+}(\text{\textgreek{t}}-s)|^{4}+1}\big(|T^{m+1}\text{\textgreek{y}}_{\ge\text{\textgreek{w}}_{+}}(s,x)|^{2}+|\nabla_{g_{\text{\textgreek{S}}}}T^{m}\text{\textgreek{y}}_{\ge\text{\textgreek{w}}_{+}}(s,x)|_{g_{\text{\textgreek{S}}}}^{2}\big)\, dsdg_{\text{\textgreek{S}}}\Big)\le\\
\le & C_{m}\text{\textgreek{w}}_{+}^{-2m}\Big(\int_{\text{\textgreek{S}}}\int_{-\infty}^{+\infty}\text{\textgreek{w}}_{+}\frac{|\log(|\text{\textgreek{w}}_{+}(\text{\textgreek{t}}-s)|)|+1}{|\text{\textgreek{w}}_{+}(\text{\textgreek{t}}-s)|^{4}+1}\big(|T^{m+1}\text{\textgreek{y}}_{\ge\text{\textgreek{w}}_{+}}(s,x)|^{2}+|\nabla_{g_{\text{\textgreek{S}}}}T^{m}\text{\textgreek{y}}_{\ge\text{\textgreek{w}}_{+}}(s,x)|_{g_{\text{\textgreek{S}}}}^{2}\big)\, dsdg_{\text{\textgreek{S}}}\Big).
\end{split}
\label{eq:BoundForHighFrequencies}
\end{equation}

In view of (\ref{eq:HighFrequencyPart}) and the Schwartz bounds (\ref{eq:SchwartzBoundLargerCutOff}),
we readily obtain that for any $\text{\textgreek{t}}\in\mathbb{R}$:
\begin{equation}
\begin{split}\int_{\text{\textgreek{S}}}\big(|T^{m+1}\text{\textgreek{y}}_{\ge\text{\textgreek{w}}_{+}}(\text{\textgreek{t}},x)|^{2} & +|\nabla_{g_{\text{\textgreek{S}}}}T^{m}\text{\textgreek{y}}_{\ge\text{\textgreek{w}}_{+}}(\text{\textgreek{t}},x)|_{g_{\text{\textgreek{S}}}}^{2}\big)\, dg_{\text{\textgreek{S}}}\le\\
\le & C\int_{-\infty}^{+\infty}\frac{\text{\textgreek{w}}_{+}}{(1+\text{\textgreek{w}}_{+}|\text{\textgreek{t}}-s|)^{4}}\Big(\int_{\text{\textgreek{S}}}\big(|T^{m+1}\text{\textgreek{y}}_{c}(s,x)|^{2}+|\nabla_{g_{\text{\textgreek{S}}}}T^{m}\text{\textgreek{y}}_{c}(s,x)|_{g_{\text{\textgreek{S}}}}^{2}\big)\, dg_{\text{\textgreek{S}}}\Big)\, ds.
\end{split}
\label{eq:BoundEnergyFromCutOff}
\end{equation}
In view of the definition (\ref{eq:CutOffPsi}) of $\text{\textgreek{y}}_{c}$,
(\ref{eq:BoundEnergyFromCutOff}) yields 
\begin{equation}
\begin{split}\int_{\text{\textgreek{S}}}\big(|T^{m+1}\text{\textgreek{y}}_{\ge\text{\textgreek{w}}_{+}}(\text{\textgreek{t}},x)|^{2} & +|\nabla_{g_{\text{\textgreek{S}}}}T^{m}\text{\textgreek{y}}_{\ge\text{\textgreek{w}}_{+}}(\text{\textgreek{t}},x)|_{g_{\text{\textgreek{S}}}}^{2}\big)\, dg_{\text{\textgreek{S}}}\le\\
\le & C\sum_{j=0}^{m}\int_{-\infty}^{+\infty}\frac{\text{\textgreek{w}}_{+}}{(1+\text{\textgreek{w}}_{+}|\text{\textgreek{t}}-s|)^{4}}\Big\{\int_{\text{\textgreek{S}}\cap\{t_{-}\ge0\}}\big(|T^{j+1}\text{\textgreek{y}}(s,x)|^{2}+|\nabla_{g_{\text{\textgreek{S}}}}T^{j}\text{\textgreek{y}}(s,x)|_{g_{\text{\textgreek{S}}}}^{2}\big)\, dg_{\text{\textgreek{S}}}+\\
 & \hphantom{C\sum_{j=0}^{m}\int_{-\infty}^{+\infty}\frac{\text{\textgreek{w}}_{+}}{(1+\text{\textgreek{w}}_{+}|\text{\textgreek{t}}-s|)^{4}}\Big(}+\int_{\text{\textgreek{S}}\cap\{0\le t_{-}\le1\}}|\text{\textgreek{y}}(s,x)|^{2}\, dg_{\text{\textgreek{S}}}\Big\}\, ds.
\end{split}
\label{eq:BoundEnergyFromCutOff-1}
\end{equation}
Thus, (\ref{eq:BoundForHighFrequencies}), (\ref{eq:BoundEnergyFromCutOff-1}),
(\ref{eq:BoundHardyNegativeTimes}), and (\ref{eq:BoundEnergyHigherDerivatives})
(combined with the conservation of the $T$-energy flux in the region
$\{t_{-}\ge0\}\cap\{t\le0\}$) imply: 
\begin{align}
\int_{\text{\textgreek{S}}}\big(|T\text{\textgreek{y}}_{\ge\text{\textgreek{w}}_{+}}( & \text{\textgreek{t}},x)|^{2}+|\nabla_{g_{\text{\textgreek{S}}}}\text{\textgreek{y}}_{\ge\text{\textgreek{w}}_{+}}(\text{\textgreek{t}},x)|_{g_{\text{\textgreek{S}}}}^{2}\big)\, dg_{\text{\textgreek{S}}}\le\\
 & \le C_{m}\text{\textgreek{w}}_{+}^{-2m}\Big(\int_{-\infty}^{+\infty}\text{\textgreek{w}}_{+}\frac{|\log(|\text{\textgreek{w}}_{+}(\text{\textgreek{t}}-s)|)|+1}{|\text{\textgreek{w}}_{+}(\text{\textgreek{t}}-s)|^{4}+1}(1+|s|^{2})\big(\log(2+|s|)\big)^{4}\, ds\Big)\Big(\sum_{j=0}^{m}\mathcal{E}[T^{j}\text{\textgreek{y}}]+\mathcal{E}_{log}[\text{\textgreek{y}}]\Big),\nonumber 
\end{align}
from which (\ref{eq:HighFrequencyBoundLemma}) readily follows.
\end{proof}

\section{\label{sec:Carleman}A Carleman-type estimate outside the extended
ergoregion}

In this section, we will establish the following estimate for solutions
$\text{\textgreek{f}}$ to the inhomogeneous wave equation 
\begin{equation}
\square_{g}\text{\textgreek{f}}=G\label{eq:InhomogeneousWaveEquation}
\end{equation}
on $(\mathcal{M},g)$: 
\begin{prop}
\label{prop:GeneralCarlemanEstimate} For any $s,R\gg1$ sufficiently
large in terms of the geometry of $(\mathcal{M},g)$ and any $0<\text{\textgreek{e}}_{0}<1$,
there exists a smooth $T$-invariant function $f:\mathcal{M}\backslash\mathcal{H}^{-}\rightarrow(0,+\infty)$
satisfying 
\begin{equation}
f=\begin{cases}
e^{2sw_{R}}+e^{2s\tilde{w}_{R}}, & r\le R,\\
C_{s}\Big(\frac{r}{R}-\frac{9}{10}\log(\frac{r}{R})\Big) & r\ge R,
\end{cases}\label{eq:FunctionInCarlemanProposition}
\end{equation}
 where the functions $w_{R},\tilde{w}_{R}:\{r\le R\}\rightarrow\mathbb{R}$
satisfy

\begin{enumerate}

\item $w_{R}\equiv\tilde{w}_{R}\mbox{ on }\{r\le\frac{1}{4}r_{0}\}\cup\mathscr{E}_{ext}\cup\{r\ge\frac{1}{2}R_{0}\},$

\item $\sup_{\{r\le R\}}w_{R}-\inf_{\{r\le R\}}w_{R}+\sup_{\{r\le R\}}\tilde{w}_{R}-\inf_{\{r\le R\}}\tilde{w}_{R}\le C\text{\textgreek{e}}_{0}^{-1}R^{3\text{\textgreek{e}}_{0}}\mbox{ for some absolute constant }C>0,$

\item $\inf_{\{\frac{1}{4}r_{0}\le r\le R\}\backslash\mathscr{E}_{2\text{\textgreek{d}}}}w_{R}\ge\max_{\mathscr{E}_{\text{\textgreek{d}}}}w_{R}+c_{\text{\textgreek{d}}}R^{-3\text{\textgreek{e}}_{0}}\mbox{ and }\inf_{\{\frac{1}{4}r_{0}\le r\le R\}\backslash\mathscr{E}_{2\text{\textgreek{d}}}}\tilde{w}_{R}\ge\max_{\mathscr{E}_{\text{\textgreek{d}}}}\tilde{w}_{R}+c_{\text{\textgreek{d}}}R^{-3\text{\textgreek{e}}_{0}}\mbox{ for any }0<\text{\textgreek{d}}\ll1,$

\item $\sum_{j=1}^{4}\big(|\nabla^{j}w_{R}|_{g_{ref}}+|\nabla^{j}\tilde{w}_{R}|_{g_{ref}}\big)\le C,$

\end{enumerate} so that the following statement holds: For any $0<\text{\textgreek{d}},\text{\textgreek{e}}_{0}\ll1$,
any $s,R\gg1$ satisfying $\text{\textgreek{e}}_{0}sR^{-9\text{\textgreek{e}}_{0}}\gg1$,
any $0\le\text{\textgreek{t}}_{1}\le\text{\textgreek{t}}_{2}$ and
any smooth function $\text{\textgreek{f}}:\mathcal{M}\backslash\mathcal{H}^{-}\rightarrow\mathbb{C}$
solving (\ref{eq:InhomogeneousWaveEquation}) with compact support
on the hypersurfaces $\{t=\text{\textgreek{t}}\}$ for any $\text{\textgreek{t}}_{1}\le\text{\textgreek{t}}\le\text{\textgreek{t}}_{2}$,
we can estimate: 
\begin{equation}
\begin{split}\int_{\mathcal{R}(\text{\textgreek{t}}_{1},\text{\textgreek{t}}_{2})\cap\{r\le R_{0}\}\backslash\mathscr{E}_{\text{\textgreek{d}}}}(f+\inf_{\{r\ge\frac{1}{4}r_{0}\}\backslash\mathscr{E}}f)\Bigg\{ & sR^{-3\text{\textgreek{e}}_{0}}|\nabla_{g_{\text{\textgreek{S}}}}\text{\textgreek{f}}|_{g_{\text{\textgreek{S}}}}^{2}-C_{\text{\textgreek{d}}}sR^{-3\text{\textgreek{e}}_{0}}|T\text{\textgreek{f}}|^{2}+s^{3}R^{-9\text{\textgreek{e}}_{0}}|\text{\textgreek{f}}|^{2}\Bigg\}\, dg+\\
+\int_{\mathcal{R}(\text{\textgreek{t}}_{1},\text{\textgreek{t}}_{2})\cap\{R_{0}\le r\le\frac{1}{2}R\}}f\Bigg\{ & sR^{-3\text{\textgreek{e}}_{0}}r^{-\frac{5}{2}}\big(\big|\partial_{r}\text{\textgreek{f}}\big|^{2}+r^{-2}|\partial_{\text{\textgreek{sv}}}\text{\textgreek{f}}|^{2}\big)+sR^{-3\text{\textgreek{e}}_{0}}r^{-2}|T\text{\textgreek{f}}|^{2}+\text{\textgreek{e}}_{0}s^{3}R^{-9\text{\textgreek{e}}_{0}}r^{-4}|\text{\textgreek{f}}|^{2}\Bigg\}\, dg+\\
+\int_{\mathcal{R}(\text{\textgreek{t}}_{1},\text{\textgreek{t}}_{2})\cap\{\frac{1}{2}R\le r\le R\}}f\Bigg\{ & r^{-\frac{5}{2}}\big(\big|\partial_{r}\text{\textgreek{f}}\big|^{2}+r^{-2}|\partial_{\text{\textgreek{sv}}}\text{\textgreek{f}}|^{2}\big)+R\partial_{r}w_{R}\big(cR^{-2}s|T\text{\textgreek{f}}|^{2}-CR^{-4}s^{3}|\text{\textgreek{f}}|^{2}\big)\Bigg\}\, dg+\\
+\int_{\mathcal{R}(\text{\textgreek{t}}_{1},\text{\textgreek{t}}_{2})\cap\{r\ge R\}}f(R)\Bigg\{ & r^{-\frac{5}{2}}\big(\big|\partial_{r}\text{\textgreek{f}}\big|^{2}+r^{-2}|\partial_{\text{\textgreek{sv}}}\text{\textgreek{f}}|^{2}\big)+r^{-2}|T\text{\textgreek{f}}|^{2}-CR^{-1}r^{-3}|\text{\textgreek{f}}|^{2}\Bigg\}\, dg\le\\
\le & C_{\text{\textgreek{d}}}\int_{\mathcal{R}(\text{\textgreek{t}}_{1},\text{\textgreek{t}}_{2})\cap\mathscr{E}_{\text{\textgreek{d}}}}f\Big\{ s^{2}R^{-6\text{\textgreek{e}}_{0}}|\nabla\text{\textgreek{f}}|_{g_{ref}}^{2}+s^{4}R^{-12\text{\textgreek{e}}_{0}}|\text{\textgreek{f}}|^{2}\Big\}\, dg+\\
 & +C\Big|\int_{\mathcal{R}(\text{\textgreek{t}}_{1},\text{\textgreek{t}}_{2})}G\big(\nabla^{\text{\textgreek{m}}}f\nabla_{\text{\textgreek{m}}}\bar{\text{\textgreek{f}}}+O\big(\sum_{j=1}^{2}(1+r)^{j-2}|\nabla^{j}f|_{g_{ref}}\big)\bar{\text{\textgreek{f}}}\big)\, dg\Big|+\\
 & +C\sum_{j=1}^{2}\int_{\text{\textgreek{S}}_{\text{\textgreek{t}}_{j}}}\Big(|\nabla f|_{g_{ref}}|\nabla\text{\textgreek{f}}|_{g_{ref}}^{2}+\big(\sum_{j=1}^{3}(1+r)^{j-3}|\nabla^{j}f|_{g_{ref}}\big)|\text{\textgreek{f}}|^{2}\Big)\, dg_{\text{\textgreek{S}}}.
\end{split}
\label{eq:MainCarlemanEstimate}
\end{equation}

\end{prop}
The proof of Proposition \ref{prop:GeneralCarlemanEstimate} will
be given in Section \ref{sub:ProofOfCarleman}. It will be based on
the construction of a suitable multiplier for the inhomogeneous wave
equation (\ref{eq:InhomogeneousWaveEquation}), which will be presented
in Sections \ref{sub:Some-auxiliary-lemmas}--\ref{sub:ChoiceOfSeedFunctions},
as well as an intricate integration-by-parts procedure, that will
be performed in Section \ref{sub:IntegrationsByParts}. 
\begin{rem*}
In fact, Proposition \ref{prop:GeneralCarlemanEstimate} also holds
in the case when $\mathscr{E}=\emptyset$. We should also remark that
Proposition \ref{prop:GeneralCarlemanEstimate} applies in the case
when $(\mathcal{M},g)$ has a $T$-invariant timelike boundary component
$\partial_{tim}\mathcal{M}$, with $\partial_{tim}\mathcal{M}\cap\text{\textgreek{S}}$
compact and $\partial_{tim}\mathcal{M}\cap\mathcal{H}=\emptyset$,
and $\text{\textgreek{f}}$ is assumed to satisfy either Dirichlet
or Neumann boundary conditions on $\partial_{tim}\mathcal{M}$ (see
Section \ref{sub:ProofWithBoundaryConditions} for more details). 

Furthermore, the proof of Proposition \ref{prop:GeneralCarlemanEstimate}
applies without any change after replacing equation (\ref{eq:InhomogeneousWaveEquation})
with 
\begin{equation}
\square_{g}\text{\textgreek{f}}-V\text{\textgreek{f}}=G,
\end{equation}
for any smooth and $T$-invariant function $V:\mathcal{M}\rightarrow\mathbb{R}$
satisfying either $\partial_{r}V<0$ and $V\rightarrow0$ as $r\rightarrow+\infty$
in the asymptotically flat region of $(\mathcal{M},g)$, or $\sup_{\mathcal{M}}\big((1+r)^{2+\text{\textgreek{h}}}|V|\big)<+\infty$
for some $\text{\textgreek{h}}>0$. In the later case, the constants
in the analogue of (\ref{eq:MainCarlemanEstimate}) can be chosen
to depending only on $\text{\textgreek{h}}$ and $\sup_{\mathcal{M}}\big((1+r)^{2+\text{\textgreek{h}}}|V|\big)$. 

Finally, let us remark that the estimate (\ref{eq:MainCarlemanEstimate})
can be readily used to show that any smooth solution $\text{\textgreek{f}}$
to equation (\ref{eq:WaveEquation}) on $\mathcal{M}$ of the form
$\text{\textgreek{f}}=e^{-i\text{\textgreek{w}}t}\text{\textgreek{f}}_{\text{\textgreek{w}}}$
with $\text{\textgreek{w}}\in\mathbb{R}\backslash\{0\}$, $T(\text{\textgreek{f}}_{\text{\textgreek{w}}})=0$
and 
\begin{equation}
\lim_{\text{\textgreek{r}}\rightarrow+\infty}\int_{\{r=\text{\textgreek{r}}\}\cap\{t=0\}}\big(|\text{\textgreek{f}}_{\text{\textgreek{w}}}|^{2}+|\nabla\text{\textgreek{f}}_{\text{\textgreek{w}}}|^{2}\big)=0
\end{equation}
 vanishes identically on $\mathcal{M}\backslash\mathscr{E}_{ext}$. 
\end{rem*}
As a corollary of Proposition \ref{prop:GeneralCarlemanEstimate},
given $\text{\textgreek{w}}_{+}>1$ and $0<\text{\textgreek{w}}_{0}<1$,
we will establish the following estimate for the frequency localised
components $\text{\textgreek{y}}_{k}$ of any solution $\text{\textgreek{y}}$
to the wave equation (\ref{eq:WaveEquation}) on $(\mathcal{M},g)$
satisfying the bound (\ref{eq:BoundednessEnergy}) (see the relevant
constructions in Section \ref{sec:FrequencyDecomposition}):
\begin{cor}
\label{cor:CarlemanForPsik}For any smooth solution $\text{\textgreek{y}}$
to (\ref{eq:WaveEquation}) satisfying (\ref{eq:BoundednessEnergy}),
any integer $1\le k\le n$, any $0<\text{\textgreek{d}}_{1},\text{\textgreek{d}}_{2},\text{\textgreek{e}}_{0}<1$,
any $R_{1}\ge0$, any $0\le\text{\textgreek{t}}_{1}\le\text{\textgreek{t}}_{2}$
we can bound: 
\begin{equation}
\begin{split}\int_{\mathcal{R}(\text{\textgreek{t}}_{1},\text{\textgreek{t}}_{2})\cap\{r\le R_{1}\}\backslash\mathscr{E}_{2\text{\textgreek{d}}_{1}}}\Big(J_{\text{\textgreek{m}}}^{N}(\text{\textgreek{y}}_{k})N^{\text{\textgreek{m}}} & +|\text{\textgreek{y}}_{k}|^{2}\Big)\le\\
\le & \text{\textgreek{d}}_{2}\int_{\mathcal{R}(\text{\textgreek{t}}_{1},\text{\textgreek{t}}_{2})\cap\mathscr{E}_{\text{\textgreek{d}}_{1}}}\Big(J_{\text{\textgreek{m}}}^{N}(\text{\textgreek{y}}_{k})N^{\text{\textgreek{m}}}+|\text{\textgreek{y}}_{k}|^{2}\Big)+\\
 & +C_{\text{\textgreek{e}}_{0}\text{\textgreek{d}}_{1}R_{1}}(1+\text{\textgreek{w}}_{k}^{-10})\big(\log(2+\text{\textgreek{t}}_{2})\big)^{4}\cdot e^{C_{\text{\textgreek{e}}_{0}\text{\textgreek{d}}_{1}}\cdot\max\big\{|\text{\textgreek{w}}_{k}|,|\text{\textgreek{w}}_{k}|^{-\text{\textgreek{e}}_{0}},-\log\text{\textgreek{d}}_{2}\big\}}\mathcal{E}_{log}[\text{\textgreek{y}}],
\end{split}
\label{eq:CarlemanEstimate}
\end{equation}
where $C_{\text{\textgreek{e}}_{0}\text{\textgreek{d}}_{1}R_{1}}$
depends only on $\text{\textgreek{e}}_{0},\text{\textgreek{d}}_{1},R_{1}$
and the geometry of $(\mathcal{M},g)$, while $C_{\text{\textgreek{e}}_{0}\text{\textgreek{d}}_{1}}$
depends only on $\text{\textgreek{e}}_{0},\text{\textgreek{d}}_{1}$
and the geometry of $(\mathcal{M},g)$.
\end{cor}
The proof of Corollary \ref{cor:CarlemanForPsik} will be presented
in Section \ref{sub:ProofOfCorollaryCarleman}.

Finally, let us sketch an additional application of Proposition \ref{prop:GeneralCarlemanEstimate}
in the Riemannian setting. Let $(\text{\textgreek{S}}^{d},\bar{g})$,
$d\ge3$, be an asymptotically conic Riemannian manifold, with the
asymptotics described in \cite{Rodnianski2011}, and let us consider
the unique solution $u\in L^{2}(\text{\textgreek{S}})$ of the inhomogeneous
Helmholtz equation 
\begin{equation}
\text{\textgreek{D}}_{\bar{g}}u+\text{\textgreek{w}}^{2}u-Vu=G\label{eq:HelmholtzEquationPotential}
\end{equation}
 on $(\text{\textgreek{S}},\bar{g})$ for a suitably decaying source
term $G:\text{\textgreek{S}}\rightarrow\mathbb{C}$, with $0<Im(\text{\textgreek{w}})\ll1$,
$Re(\text{\textgreek{w}})\neq0$ and a potential $V:\text{\textgreek{S}}\rightarrow\mathbb{R}$
satisfying either $\partial_{r}V<0$ in the asymptotically conic region
of $(\text{\textgreek{S}},\bar{g})$ and $V\rightarrow0$ as $r\rightarrow+\infty$
(where $r$ is the radial coordinate function in the asymptotically
conic region of $\text{\textgreek{S}}$, extended to a positive function
everywhere on $\text{\textgreek{S}}$), or 
\begin{equation}
\sup_{\text{\textgreek{S}}}\big((r{}^{-2-\text{\textgreek{h}}}+|\text{\textgreek{w}}|r^{-1-\text{\textgreek{h}}})|V|\big)<+\infty.\label{eq:BoundPotentialHelmholtz}
\end{equation}
Then, applying Proposition \ref{prop:GeneralCarlemanEstimate} on
the product spacetime $(\mathbb{R}\times\text{\textgreek{S}},g=-dt^{2}+\bar{g})$
for the function $\text{\textgreek{f}}=e^{-i\text{\textgreek{w}}t}u$
solving (in view of (\ref{eq:HelmholtzEquationPotential})) 
\begin{equation}
\square_{g}\text{\textgreek{f}}-V\text{\textgreek{f}}=e^{-i\text{\textgreek{w}}t}G,
\end{equation}
and using the charge estimate
\begin{equation}
Im(\text{\textgreek{w}}^{2})\int_{\text{\textgreek{S}}}|u|^{2}\, d\bar{g}=\int_{\text{\textgreek{S}}}Im(G\bar{u})\, d\bar{g}
\end{equation}
(combined with elliptic estimates for (\ref{eq:HelmholtzEquationPotential}),
as is done, for instance, in \cite{Rodnianski2011}) one readily obtains
the (quantitative in $V$) global Carleman-type estimates of \cite{Rodnianski2011},
albeit with a worse dependence on $\text{\textgreek{w}}$ as $Re(\text{\textgreek{w}})\rightarrow0$.
Thus, the proof of Proposition \ref{prop:GeneralCarlemanEstimate}
yields a proof of the Carleman-type estimates used in \cite{Rodnianski2011,Moschidisb}
based entirely on the method of first order multipliers. 
\begin{rem*}
A multiplier-based proof of a similar set of Carleman-type estimates
for equation (\ref{eq:HelmholtzEquationPotential}) restricted, however,
to the high frequency regime $\text{\textgreek{w}}\gg1$ was obtained
previously in \cite{Datchev2014}. 
\end{rem*}

\subsection{Parameters and cut-off functions in the proof of Proposition \ref{cor:CarlemanForPsik}}

Let $R_{0}\gg1$ be large in terms of the geometry of $(\mathcal{M},g)$,
such that $\{r\ge\frac{1}{8}R_{0}\}\subset\mathcal{I}_{as}$ ($R_{0}$
will be considered fixed and, thus, we will not use any special notation
to denote the dependence of constants on $R_{0}$). \textgreek{I}n
addition to the parameters $\text{\textgreek{d}},\text{\textgreek{e}}_{0},s,R$
appearing in the statement of Proposition \ref{cor:CarlemanForPsik},
we will introduce the parameters $R\gg R_{0}$ and $0<\text{\textgreek{d}}_{0},\text{\textgreek{d}}_{1},\text{\textgreek{d}}_{2}\ll1$.
We will assume without loss of generality that $0<\text{\textgreek{e}}_{0}\ll1$.
These additional parameters will be fixed in the proof of Proposition
\ref{cor:CarlemanForPsik}.

In the region $\{r\ge\frac{1}{2}R_{0}\}$, the vector field $\partial_{r}$
will simply denote the associated coordinate vector field in the $(t,r,\text{\textgreek{sv}})$
coordinate chart in each connected component of this region. 

Fixing a smooth function $\text{\textgreek{q}}_{4}:\mathbb{R}\rightarrow[0,1]$
satisfying $\text{\textgreek{q}}_{4}(x)=0$ for $x\le\frac{3}{4}$
and $\text{\textgreek{q}}_{4}(x)=1$ for $x\ge1$, we will define
the following smooth cut-off functions: 
\begin{gather}
\text{\textgreek{q}}_{\ge R_{0}}(r)\doteq\text{\textgreek{q}}_{4}(\frac{r}{R_{0}}),\label{eq:SmoothCut-OffNearRegion}\\
\text{\textgreek{q}}_{\le R}(r)\doteq\text{\textgreek{q}}_{4}(\frac{R}{r}).
\end{gather}

\begin{rem*}
Note that $\text{\textgreek{q}}_{\le R}\equiv1$ for $r\le R$ and
$\text{\textgreek{q}}_{\le R}\equiv0$ for $r\ge\frac{4}{3}R$, while
$\text{\textgreek{q}}_{\ge R_{0}}\equiv1$ for $r\ge R_{0}$ and $\text{\textgreek{q}}_{\ge R_{0}}\equiv0$
for $r\le\frac{3}{4}R_{0}$.
\end{rem*}

\subsection{\label{sub:Some-auxiliary-lemmas}Construction of the auxiliary functions
$w_{R},\tilde{w}_{R}$}

In this section, we will construct the pair of functions $w_{R},\tilde{w}_{R}:\mathcal{M}\backslash\mathcal{H}^{-}\rightarrow\mathbb{R}$
appearing in the statement of Proposition \ref{cor:CarlemanForPsik},
depending on the parameters $\text{\textgreek{d}}_{0},\text{\textgreek{d}}_{1},\text{\textgreek{e}}_{0},s,R$.
These functions will be used extensively in the next sections. 

First, we will establish the following lemma:
\begin{lem}
\label{lem:DoubleExponentForTheCarleman}There exists a smooth and
$T$-invariant function $\bar{w}:\mathcal{M}\backslash\mathcal{H}\rightarrow\mathbb{R}$
satisfying the following properties:

\begin{enumerate}

\item The restriction $\bar{w}|_{\text{\textgreek{S}}}$ of $\bar{w}$
on $\text{\textgreek{S}}$ is a Morse function on $\overline{\text{\textgreek{S}}\backslash\mathscr{E}_{ext}}$,
with no critical points on $\partial\mathscr{E}_{ext}$. Furthermore,
none of the (at most finite) critical points $\{x_{j}\}_{j=1}^{k}$
of $\bar{w}|_{\text{\textgreek{S}}}$ on $\text{\textgreek{S}}\backslash\mathscr{E}_{ext}$
is a point of local maximum of $\bar{w}|_{\text{\textgreek{S}}}$.

\item In the region $\{0<r\le\frac{1}{8}r_{0}\}$,%
\footnote{Recall that $\{r\le\frac{1}{8}r_{0}\}\cap\text{\textgreek{S}}$ is
a neighborhood of $\mathcal{H}^{+}\cap\text{\textgreek{S}}$ in $\text{\textgreek{S}}$%
} $\bar{w}$ is a function of $r$ and satisfies 
\begin{equation}
\nabla^{\text{\textgreek{m}}}r\nabla_{\text{\textgreek{m}}}\bar{w}>0.\label{eq:Positivity in r direction}
\end{equation}

\item In the region $\{r\ge\frac{1}{2}R_{0}\}$, $\bar{w}$ is a
function of $r$, and satisfies (\ref{eq:Positivity in r direction}).

\item On $\big(\overline{\mathcal{M}\backslash\mathscr{E}_{ext}\cup\mathcal{H}}\big)\backslash\big(\mathbb{R}\times\cup_{j=1}^{k}\{x_{j}\}\big)$
we have 
\begin{equation}
\nabla^{\text{\textgreek{m}}}\bar{w}\nabla_{\text{\textgreek{m}}}\bar{w}>0.
\end{equation}

\item For any $0<\text{\textgreek{d}}\ll1$, we have 
\begin{equation}
\inf_{\{r\ge\frac{1}{4}r_{0}\}\backslash\mathscr{E}_{2\text{\textgreek{d}}}}\bar{w}>\max_{\mathscr{E}_{\text{\textgreek{d}}}}\bar{w}.\label{eq:SmallnessInErgoregion}
\end{equation}

\end{enumerate}\end{lem}
\begin{proof}
The proof of Lemma \ref{lem:DoubleExponentForTheCarleman} will be
based on ideas from \cite{Rodnianski2011,Moschidisb}. 

Let $R_{0}\gg1$ be a fixed constant large in terms of the geometry
of $(\mathcal{M},g)$. For any $0\le\text{\textgreek{g}}<1$, let
$\bar{w}_{\text{\textgreek{g}}}:\big(\text{\textgreek{S}}\cap\{\frac{1}{6}r_{0}\le r\le\frac{1}{4}R_{0}\}\big)\backslash\mathscr{E}_{ext}\rightarrow\mathbb{R}$
be the (unique) smooth solution of the elliptic boundary value problem:
\begin{equation}
\begin{cases}
\text{\textgreek{D}}_{g_{\text{\textgreek{S}}}}\bar{w}_{\text{\textgreek{g}}}=\text{\textgreek{g}} & \mbox{on }\big(\text{\textgreek{S}}\cap\{\frac{1}{6}r_{0}<r<\frac{1}{4}R_{0}\}\big)\backslash\mathscr{E}_{ext},\\
\bar{w}_{\text{\textgreek{g}}}|_{r=\frac{1}{4}R_{0}}=2,\\
\bar{w}_{\text{\textgreek{g}}}|_{r=\frac{1}{6}r_{0}}=1,\\
\bar{w}_{\text{\textgreek{g}}}|_{\partial\mathscr{E}_{ext}}=1.
\end{cases}\label{eq:EllipticBoundaryProblem}
\end{equation}
 Let us extend $\bar{w}_{\text{\textgreek{g}}}$ on the whole of $\{\frac{1}{6}r_{0}<r<\frac{1}{4}R_{0}\}\backslash\mathscr{E}_{ext}\subset\mathcal{M}\backslash\mathcal{H}^{-}$
by the requirement that $T\bar{w}|_{\text{\textgreek{g}}}=0$. 

Since $g_{\text{\textgreek{S}}}$ is smooth, $\bar{w}_{\text{\textgreek{g}}}$
depends smoothly on $\text{\textgreek{g}}$ (see \cite{Gilbarg2001}).
In view of the fact that every connected component of $\mathcal{M}\backslash\mathscr{E}$
intersecting $\mathcal{H}^{+}$ also intersects $\mathcal{I}_{as}$
(see Assumption \hyperref[Assumption 3]{G3}), when $\text{\textgreek{g}}=0$,
the maximum principle and Hopf's lemma (see \cite{Gilbarg2001}) imply
that for any $\text{\textgreek{d}}>0$ 
\begin{equation}
\inf_{r=\frac{1}{4}R_{0}}\big(\nabla^{\text{\textgreek{m}}}r\nabla_{\text{\textgreek{m}}}\bar{w}_{0}\big),\mbox{ }\inf_{r=\frac{1}{6}r_{0}}\big(\nabla^{\text{\textgreek{m}}}r\nabla_{\text{\textgreek{m}}}\bar{w}_{0}\big),\mbox{ }\inf_{\partial\mathscr{E}_{ext}}\big(n_{\partial\mathscr{E}}(\bar{w}_{0})\big),\big(\inf_{\{r\ge\frac{1}{4}r_{0}\}\backslash\mathscr{E}_{\text{\textgreek{d}}}}\bar{w}_{0}-\max_{\partial\mathscr{E}_{ext}}\bar{w}_{0}\big)>0\label{eq:PositivityTransversalDerivativesW}
\end{equation}
(see Section \ref{sub:SpecialSubsets} for the definition of $n_{\partial\mathscr{E}}$).
Therefore, there exists a $\text{\textgreek{g}}_{0}\in(0,1)$ and
a $c_{0}>0$, such that: 
\begin{equation}
\inf_{r=\frac{1}{4}R_{0}}\big(\nabla^{\text{\textgreek{m}}}r\nabla_{\text{\textgreek{m}}}\bar{w}_{\text{\textgreek{g}}_{0}}\big),\mbox{ }\inf_{r=\frac{1}{6}r_{0}}\big(\nabla^{\text{\textgreek{m}}}r\nabla_{\text{\textgreek{m}}}\bar{w}_{\text{\textgreek{g}}_{0}}\big),\mbox{ }\inf_{\mathscr{E}_{c_{0}}\backslash\mathscr{E}_{ext}}\big(n_{\partial\mathscr{E}}(\bar{w}_{\text{\textgreek{g}}_{0}})\big)\ge c_{0}>0\label{eq:PositivityTransversalDerivativesW-1}
\end{equation}
and, for all $0<\text{\textgreek{d}}<1$ (and some fixed $c_{1}>0$):
\begin{equation}
\inf_{\{r\ge\frac{1}{4}r_{0}\}\backslash\mathscr{E}_{\text{\textgreek{d}}}}\bar{w}_{\text{\textgreek{g}}_{0}}-\max_{\partial\mathscr{E}_{ext}}\bar{w}_{\text{\textgreek{g}}_{0}}>c_{1}\text{\textgreek{d}}>0.\label{eq:DifferenceErgoregion}
\end{equation}

In view of \ref{eq:PositivityTransversalDerivativesW-1} and (\ref{eq:DifferenceErgoregion}),
we can extend $\bar{w}_{\text{\textgreek{g}}_{0}}$ as a $T$-invariant
function on the whole of $\mathcal{M}\backslash\mathcal{H}^{-}$ in
such a way, so that 
\begin{gather}
\inf_{\{\frac{1}{16}r_{0}\le r\le\frac{1}{6}r_{0}\}\cup\{\frac{1}{4}R_{0}\le r\le R_{0}\}}\big(\nabla^{\text{\textgreek{m}}}r\nabla_{\text{\textgreek{m}}}\bar{w}_{\text{\textgreek{g}}_{0}}\big)\ge\frac{1}{10}c_{0}>0,\label{eq:PositivityTransversalDerivativesW-1-1}\\
\inf_{\{\text{\textgreek{d}}\le r\le\frac{1}{6}r_{0}\}}\big(\nabla^{\text{\textgreek{m}}}r\nabla_{\text{\textgreek{m}}}\bar{w}_{\text{\textgreek{g}}_{0}}\big)\ge c(\text{\textgreek{d}})>0,\label{eq:NonNegativityNearHorizon}\\
\inf_{\{r\ge\frac{1}{4}r_{0}\}\backslash\mathscr{E}_{2\text{\textgreek{d}}}}\bar{w}_{\text{\textgreek{g}}_{0}}>\max_{\mathscr{E}_{\text{\textgreek{d}}}}\bar{w}_{\text{\textgreek{g}}_{0}}+\frac{1}{2}c_{1}\text{\textgreek{d}}\label{eq:SmallnessInErgoregion-1}
\end{gather}
 for all $0<\text{\textgreek{d}}\ll1$ (where $c(\text{\textgreek{d}})$
is a positive function of $\text{\textgreek{d}}>0$) and, in addition,
$\bar{w}_{\text{\textgreek{g}}_{0}}$ is a function of $r$ in the
region $\{r\le\frac{1}{8}r_{0}\}\cup\{r\ge\frac{1}{2}R_{0}\}$. Notice
also that, since $\bar{w}_{\text{\textgreek{g}}_{0}}|_{\partial\mathscr{E}_{ext}}=1$
and $n_{\partial\mathscr{E}}(\bar{w}_{\text{\textgreek{g}}_{0}})|_{\partial\mathscr{E}_{ext}}\ge c_{0}$,
we have 
\begin{equation}
\big(\nabla^{\text{\textgreek{m}}}\bar{w}_{\text{\textgreek{g}}_{0}}\nabla_{\text{\textgreek{m}}}\bar{w}_{\text{\textgreek{g}}_{0}}\big)\big|_{\partial\mathscr{E}_{ext}}\ge\frac{1}{2}c_{0}^{2}.\label{eq:GradientOnBoundryErgoregion}
\end{equation}

With $\bar{w}_{\text{\textgreek{g}}_{0}}$ constructed as above, we
can thus readily choose $\bar{w}$ so that it satisfies the following
conditions:

\begin{enumerate}

\item \label{enu:density} $\bar{w}|_{\text{\textgreek{S}}}$ and
$\bar{w}_{\text{\textgreek{g}}_{0}}|_{\text{\textgreek{S}}}$ satisfy
\begin{equation}
\big|\bar{w}|_{\text{\textgreek{S}}}-\bar{w}_{\text{\textgreek{g}}_{0}}|_{\text{\textgreek{S}}}\big|_{C^{2}(\text{\textgreek{S}})}<\frac{1}{100}\min\{\text{\textgreek{g}}_{0},c_{0}\}.\label{eq:MorseClosedness}
\end{equation}

\item \label{enu:Morse} $\bar{w}|_{\text{\textgreek{S}}}$ is a
Morse function on an open neighborhood of $\overline{\big(\text{\textgreek{S}}\cap\{\frac{1}{6}r_{0}\le r\le\frac{1}{4}R_{0}\}\big)\backslash\mathscr{E}_{ext}}$.

\item \label{enu:TrivialExtension} $\bar{w}|_{\text{\textgreek{S}}}=\bar{w}_{\text{\textgreek{g}}_{0}}|_{\text{\textgreek{S}}}$
in the region $\mathscr{E}\cup\{r\le\frac{1}{6}r_{0}\}\cup\{r\ge\frac{1}{4}R_{0}\}$.

\item $T\bar{w}=0$.

\end{enumerate}

\noindent \emph{Remark.} Note that the compatibility of conditions
\ref{enu:density} and \ref{enu:Morse} follows from the density of
the set of Morse functions on $\overline{\big(\text{\textgreek{S}}\cap\{\frac{1}{6}r_{0}\le r\le\frac{1}{4}R_{0}\}\big)\backslash\mathscr{E}_{ext}}$
in $C^{2}\Big(\overline{\big(\text{\textgreek{S}}\cap\{\frac{1}{6}r_{0}\le r\le\frac{1}{4}R_{0}\}\big)\backslash\mathscr{E}_{ext}}\Big)$.

\medskip{}

In view of (\ref{eq:EllipticBoundaryProblem}), (\ref{eq:PositivityTransversalDerivativesW-1-1}),
(\ref{eq:NonNegativityNearHorizon}) and (\ref{eq:MorseClosedness}),
$\bar{w}$ satisfies 
\begin{equation}
\text{\textgreek{D}}_{g_{\text{\textgreek{S}}}}\bar{w}|_{\text{\textgreek{S}}}>0
\end{equation}
on $\overline{\{\frac{1}{6}r_{0}<r<\frac{1}{4}R_{0}\}\backslash\mathscr{E}_{ext}}$
and 
\begin{gather}
\inf_{\{\frac{1}{16}r_{0}\le r\le\frac{1}{6}r_{0}\}\cup\{r\ge\frac{1}{4}R_{0}\}}\big(\nabla^{\text{\textgreek{m}}}r\nabla_{\text{\textgreek{m}}}\bar{w}\big)\ge\frac{1}{10}c_{0}>0,\label{eq:PositivityTransversalDerivativesW-1-1-1}\\
\inf_{\{\text{\textgreek{d}}\le r\le\frac{1}{6}r_{0}\}}\big(\nabla^{\text{\textgreek{m}}}r\nabla_{\text{\textgreek{m}}}\bar{w}\big)\ge c(\text{\textgreek{d}})>0.
\end{gather}
Therefore, none of the critical points of $\bar{w}|_{\text{\textgreek{S}}}$
on $\overline{\text{\textgreek{S}}\backslash\mathscr{E}_{ext}}$ is
a point of local maximum. Furthermore, Conditions \ref{enu:Morse}
and \ref{enu:TrivialExtension} imply that $\bar{w}|_{\text{\textgreek{S}}}$
is a Morse function on $\overline{\text{\textgreek{S}}\backslash\mathscr{E}_{ext}}$.
Since $T(\bar{w})=0$ and $T$ is strictly timelike on $\mathcal{M}\backslash(\mathscr{E}_{ext}\cup\mathcal{H})$,
in view of (\ref{eq:GradientOnBoundryErgoregion}) and Condition \ref{enu:density}
we have 
\begin{equation}
\nabla^{\text{\textgreek{m}}}\bar{w}\nabla_{\text{\textgreek{m}}}\bar{w}>0\mbox{ on }\big(\overline{\mathcal{M}\backslash\mathscr{E}_{ext}\cup\mathcal{H}}\big)\backslash\big(\mathbb{R}\times\cup_{j=1}^{k}\{x_{j}\}\big),
\end{equation}
 where $\{x_{j}\}_{j=1}^{k}$ are the (at most finite) critical points
of $\bar{w}|_{\text{\textgreek{S}}}$ on $\overline{\text{\textgreek{S}}\backslash\mathscr{E}_{ext}}$,
none of which lies on $\partial\mathscr{E}_{ext}$ (in view of (\ref{eq:GradientOnBoundryErgoregion})
and Condition \ref{enu:density}). Finally, in view of (\ref{eq:PositivityTransversalDerivativesW-1}),
(\ref{eq:SmallnessInErgoregion-1}) and (\ref{eq:MorseClosedness}),
inequality (\ref{eq:SmallnessInErgoregion}) holds for all $0<\text{\textgreek{d}}\ll1$.\end{proof}
\begin{lem}
\label{lem:PositivityHessianW}For any $0<\text{\textgreek{d}}_{0}\ll1$
small in terms of the geometry of $(\mathcal{M},g)$, there exists
a pair of smooth and $T$-invariant functions $w,\tilde{w}:\mathcal{M}\backslash\mathcal{H}^{-}\rightarrow\mathbb{R}$,
as well as a finite number of points $\{x_{j}\}_{j=1}^{k},\{\tilde{x}_{j}\}_{j=1}^{k}\in\text{\textgreek{S}}\cap\{r\le\frac{1}{2}R_{0}+4\text{\textgreek{d}}_{0}\}\backslash\mathscr{E}_{8\text{\textgreek{d}}_{0}}$,
such that the following statements hold:

\begin{enumerate}

\item Defining for any $\text{\textgreek{r}}>0$ the subsets 
\begin{gather}
\mathcal{B}_{crit}(\text{\textgreek{r}})=\mathbb{R}\times\big(\cup_{j=1}^{k}B_{g_{\text{\textgreek{S}}}}(x_{j},\text{\textgreek{r}})\big),\\
\tilde{\mathcal{B}}_{crit}(\text{\textgreek{r}})=\mathbb{R}\times\big(\cup_{j=1}^{k}B_{g_{\text{\textgreek{S}}}}(\tilde{x}_{j},\text{\textgreek{r}})\big)
\end{gather}
 of $\mathcal{M}\backslash\mathcal{H}^{-}$, where $B_{g_{\text{\textgreek{S}}}}(x_{j},\text{\textgreek{r}})\subset(\text{\textgreek{S}},g_{\text{\textgreek{S}}})$
is the closed Riemannian ball of radius $\text{\textgreek{r}}$ centered
at $x_{j}$, we have: 
\begin{gather}
\mathcal{B}_{crit}(\text{\textgreek{d}}_{0})\subset\tilde{\mathcal{B}}_{crit}(4\text{\textgreek{d}}_{0}),\\
\tilde{\mathcal{B}}_{crit}(\text{\textgreek{d}}_{0})\subset\mathcal{B}_{crit}(4\text{\textgreek{d}}_{0}),\\
\mathcal{B}_{crit}(\text{\textgreek{d}}_{0})\cap\tilde{\mathcal{B}}_{crit}(\text{\textgreek{d}}_{0})=\emptyset.
\end{gather}

\item The functions $w,\tilde{w}$ coincide outside $\mathcal{B}_{crit}(4\text{\textgreek{d}}_{0})$:
\begin{equation}
w\equiv\tilde{w}\mbox{ on }\mathcal{M}\backslash\big(\mathcal{B}_{crit}(4\text{\textgreek{d}}_{0})\cup\mathcal{H}^{-}\big).\label{eq:IdenticalWsAway}
\end{equation}

\item The functions $w,\tilde{w}$ satisfy the following non-degeneracy
conditions for some absolute constant $c_{0}>0$ (independent of $\text{\textgreek{d}}_{0}$):
\begin{gather}
\inf_{\{r\ge\frac{1}{8}r_{0}\}\backslash(\mathscr{E}_{ext}\cup\mathcal{B}_{crit}(\text{\textgreek{d}}_{0}))}\nabla^{\text{\textgreek{m}}}\nabla^{\text{\textgreek{n}}}w\nabla_{\text{\textgreek{m}}}w\nabla_{\text{\textgreek{n}}}w\ge c_{0}>0,\label{eq:LowerBoundHessian}\\
\inf_{\{r\ge\frac{1}{8}r_{0}\}\backslash(\mathscr{E}_{ext}\cup\tilde{\mathcal{B}}_{crit}(\text{\textgreek{d}}_{0}))}\nabla^{\text{\textgreek{m}}}\nabla^{\text{\textgreek{n}}}\tilde{w}\nabla_{\text{\textgreek{m}}}\tilde{w}\nabla_{\text{\textgreek{n}}}\tilde{w}\ge c_{0}>0.\label{eq:LowerBoundHessianDistorted}
\end{gather}

\item For any $T$-invariant vector fields $X,\tilde{X}$ on $\{r\ge\frac{1}{8}r_{0}\}\backslash(\mathscr{E}_{ext}\cup\mathcal{B}_{crit}(\text{\textgreek{d}}_{0}))$
such that $X(w)=0$ and $\tilde{X}(\tilde{w})=0$, the following one
sided bounds hold: 
\begin{gather}
\nabla_{\text{\textgreek{m}}}\nabla_{\text{\textgreek{n}}}wX^{\text{\textgreek{m}}}X^{\text{\textgreek{n}}}>-\text{\textgreek{d}}_{0}\frac{\nabla^{\text{\textgreek{m}}}\nabla^{\text{\textgreek{n}}}w\nabla_{\text{\textgreek{m}}}w\nabla_{\text{\textgreek{n}}}w}{g_{ref}(dw,dw)}g_{ref}(X,X),\label{eq:OneSidedBoundSideHessian}\\
\nabla_{\text{\textgreek{m}}}\nabla_{\text{\textgreek{n}}}\tilde{w}\tilde{X}^{\text{\textgreek{m}}}\tilde{X}^{\text{\textgreek{n}}}>-\text{\textgreek{d}}_{0}\frac{\nabla^{\text{\textgreek{m}}}\nabla^{\text{\textgreek{n}}}\tilde{w}\nabla_{\text{\textgreek{m}}}\tilde{w}\nabla_{\text{\textgreek{n}}}\tilde{w}}{g_{ref}(d\tilde{w},d\tilde{w})}g_{ref}(\tilde{X},\tilde{X}),\label{eq:OneSidedBoundHessianDistorted}
\end{gather}
where $g_{ref}$ is the reference Riemannian metric (\ref{eq:ReferenceRiemannianMetric}).

\item The functions $w,\tilde{w}$ satisfy 
\begin{gather}
\max_{\mathcal{B}_{crit}(\text{\textgreek{d}}_{0})}w<\min_{\mathcal{B}_{crit}(\text{\textgreek{d}}_{0})}\tilde{w},\label{eq:InequalitiesForPairOfDeformedFunctions}\\
\max_{\tilde{\mathcal{B}}_{crit}(\text{\textgreek{d}}_{0})}\tilde{w}<\min_{\tilde{\mathcal{B}}_{crit}(\text{\textgreek{d}}_{0})}w.\label{eq:InequalitiesForPair2}
\end{gather}

\item For any $0<\text{\textgreek{d}}\ll1$: 
\begin{gather}
\inf_{\{r\ge\frac{1}{4}r_{0}\}\backslash\mathscr{E}_{2\text{\textgreek{d}}}}w>\max_{\mathscr{E}_{\text{\textgreek{d}}}}w,\label{eq:SmallerinErgoregion}\\
\inf_{\{r\ge\frac{1}{4}r_{0}\}\backslash\mathscr{E}_{2\text{\textgreek{d}}}}\tilde{w}>\max_{\mathscr{E}_{\text{\textgreek{d}}}}\tilde{w}\label{eq:SmallerInErgoregionDistorted}
\end{gather}
and 
\begin{gather}
\inf_{\{r\ge\frac{1}{2}r_{0}\}\backslash\mathscr{E}_{ext}}w>\max_{\{r\le\frac{1}{4}r_{0}\}}w,\label{eq:SmallerNearHorizon}\\
\inf_{\{r\ge\frac{1}{2}r_{0}\}\backslash\mathscr{E}_{ext}}\tilde{w}>\max_{\{r\le\frac{1}{4}r_{0}\}}\tilde{w}\label{eq:SmallerNearHorizonDistorted}
\end{gather}

\end{enumerate}\end{lem}
\begin{proof}
Let $\bar{w}:\mathcal{M}\backslash\mathcal{H}^{-}\rightarrow\mathbb{R}$
be as in the statement of Lemma \ref{lem:DoubleExponentForTheCarleman},
and let $\{x_{j}\}_{j=1}^{k}$ be the (at most finite) critical points
of $\bar{w}|_{\text{\textgreek{S}}}$ in $\overline{\text{\textgreek{S}}\backslash\mathscr{E}_{ext}}$.
According to Lemma \ref{lem:DoubleExponentForTheCarleman}, none of
these points lies on $\partial\mathscr{E}_{ext}$ or on $\{r\le\frac{1}{8}r_{0}\}\cup\{r\ge\frac{1}{2}R_{0}\}$
and, thus, provided $\text{\textgreek{d}}_{0}\ll1$, we have 
\begin{equation}
\{x_{j}\}_{j=1}^{k}\in\text{\textgreek{S}}\cap\{\frac{1}{8}r_{0}\le r\le\frac{1}{2}R_{0}\}\backslash\mathscr{E}_{16\text{\textgreek{d}}_{0}}.\label{eq:RegionOfCriticalPoints}
\end{equation}

Let $l>0$ be large in terms of $\text{\textgreek{d}}_{0}$, and let
us define the $T$-invariant function $w:\mathcal{M}\backslash\mathcal{H}^{-}\rightarrow\mathbb{R}$
as 
\begin{equation}
w=e^{l\bar{w}}.\label{eq:TheFunctionWL}
\end{equation}
Then, the one sided bounds (\ref{eq:LowerBoundHessian}) and (\ref{eq:OneSidedBoundSideHessian})
readily follow from the properties of $\bar{w}$ (see Lemma \ref{lem:DoubleExponentForTheCarleman}),
as well as the identities 
\begin{equation}
\nabla_{\text{\textgreek{m}}}w=l(\nabla_{\text{\textgreek{m}}}\bar{w})e^{l\bar{w}}\label{eq:FirstDerivativeW}
\end{equation}
and 
\begin{equation}
\nabla_{\text{\textgreek{m}}}\nabla_{\text{\textgreek{n}}}w=\big(l^{2}(\nabla_{\text{\textgreek{m}}}\bar{w})(\nabla_{\text{\textgreek{n}}}\bar{w})+l\nabla_{\text{\textgreek{m}}}\nabla_{\text{\textgreek{n}}}\bar{w}\big)e^{l\bar{w}},\label{eq:SecondDerivativeW}
\end{equation}
provided $l$ is sufficiently large in terms of $\text{\textgreek{d}}_{0}$.
Inequality (\ref{eq:SmallerinErgoregion}) follows readily from (\ref{eq:SmallnessInErgoregion}).

Since the points $\{x_{j}\}_{j=1}^{k}$ satisfy (\ref{eq:RegionOfCriticalPoints})
and none of them is a point of local maximum for $\bar{w}|_{\text{\textgreek{S}}}$
(see Lemma \ref{lem:DoubleExponentForTheCarleman}), for any $0<\text{\textgreek{d}}_{0}\ll1$,
there exists a diffeomorphism $\mathcal{X}:\text{\textgreek{S}}\rightarrow\mathbb{\text{\textgreek{S}}}$
such that $\mathcal{X}=Id$ on $\text{\textgreek{S}}\backslash\cup_{j=1}^{k}B_{g_{\text{\textgreek{S}}}}(x_{j},4\text{\textgreek{d}}_{0})$
and for all $1\le j\le k$: 
\begin{equation}
2\text{\textgreek{d}}_{0}<dist_{g_{\text{\textgreek{S}}}}(x_{j},\mathcal{X}(x_{j}))<4\text{\textgreek{d}}_{0},\label{eq:DeltaSeperatedCriticalPoints}
\end{equation}
\begin{equation}
\mathcal{X}\big(B_{g_{\text{\textgreek{S}}}}(x_{j},\text{\textgreek{d}}_{0})\big)=B_{g_{\text{\textgreek{S}}}}(\mathcal{X}(x_{j}),\text{\textgreek{d}}_{0}),\label{eq:SitchingDiffeomorphism}
\end{equation}
\begin{equation}
\mathcal{X}\big(B_{g_{\text{\textgreek{S}}}}(\mathcal{X}(x_{j}),\text{\textgreek{d}}_{0})\big)=B_{g_{\text{\textgreek{S}}}}(x_{j},\text{\textgreek{d}}_{0}),\label{eq:SitchingDiffeomorphism-1}
\end{equation}
and 
\begin{equation}
\max_{B_{g_{\text{\textgreek{S}}}}(x_{j},\text{\textgreek{d}}_{0})}\bar{w}<\min_{B_{g_{\text{\textgreek{S}}}}(\mathcal{X}(x_{j}),\text{\textgreek{d}}_{0})}\bar{w}.\label{eq:NonLocalMaximum}
\end{equation}
Setting $\tilde{x}_{j}\doteq\mathcal{X}(x_{j})$ for $j=1,\ldots,k$,
provided $\text{\textgreek{d}}_{0}$ is sufficiently small in terms
of the geometry of $(\mathcal{M},g)$, we have: 
\begin{equation}
\{\tilde{x}_{j}\}_{j=1}^{k}\in\text{\textgreek{S}}\cap\{r\le\frac{1}{2}R_{0}+4\text{\textgreek{d}}_{0}\}\backslash\mathscr{E}_{8\text{\textgreek{d}}_{0}}.
\end{equation}

Extending $\mathcal{X}$ on the whole of $\mathcal{M}\backslash\mathcal{H}^{-}$
by the requirement that it commutes with the flow of $T$, i.\,e.:
\begin{equation}
\mathcal{L}_{T}\circ\mathcal{X}=\mathcal{X}\circ\mathcal{L}_{T},
\end{equation}
and defining the function $\tilde{w}:\mathcal{M}\backslash\mathcal{H}^{-}\rightarrow\mathbb{R}$
as 
\begin{equation}
\tilde{w}\doteq w\circ\mathcal{X},\label{eq:DeformedWeightFunction}
\end{equation}
we infer that, in view of (\ref{eq:FirstDerivativeW}), (\ref{eq:SecondDerivativeW})
and the properties of $\mathcal{X}$, the relations (\ref{eq:IdenticalWsAway}),
(\ref{eq:LowerBoundHessianDistorted}) and (\ref{eq:OneSidedBoundHessianDistorted})
hold, provided $l$ is sufficiently large in terms of $\text{\textgreek{d}}$
and the precise choice of $\mathcal{X}$. Furthermore, in view of
(\ref{eq:NonLocalMaximum}), inequalities (\ref{eq:InequalitiesForPairOfDeformedFunctions})
and (\ref{eq:InequalitiesForPair2}) hold. Finally, inequalities (\ref{eq:SmallerinErgoregion})
and (\ref{eq:SmallerInErgoregionDistorted}) follow trivially from
(\ref{eq:SmallnessInErgoregion}) and the fact that $\mathcal{X}=Id$
on $\mathscr{E}_{\text{\textgreek{d}}}$ for $\text{\textgreek{d}}\le4\text{\textgreek{d}}_{0}$,
while (\ref{eq:SmallerNearHorizon}) and (\ref{eq:SmallerNearHorizonDistorted})
follow from (\ref{eq:Positivity in r direction}).\end{proof}
\begin{lem}
\label{lem:ConstructionWR}For any $R_{0}\gg1$,$0<\text{\textgreek{d}}_{0}\ll1$,
$0<\text{\textgreek{d}}_{1}\ll1$, $s\gg1$, $0<\text{\textgreek{e}}_{0}\ll1$
and $R\gg\max\{R_{0},\text{\textgreek{e}}_{0}^{-1}\}$, there exists
a pair of smooth and $T$-invariant function $w_{R},\tilde{w}_{R}:\{r\le R\}\subset\mathcal{M}\backslash\mathcal{H}^{-}\rightarrow\mathbb{R}$
satisfying the following properties:

\begin{enumerate}

\item \label{enu:nearRegion} In the region $\{r\le R_{0}\}$: 
\begin{equation}
w_{R}=R^{-3\text{\textgreek{e}}_{0}}w
\end{equation}
and 
\begin{equation}
\tilde{w}_{R}=R^{-3\text{\textgreek{e}}_{0}}\tilde{w},
\end{equation}
where $w,\tilde{w}$ are the functions from Lemma \ref{lem:PositivityHessianW}.

\item In the region $\{R_{0}\le r\le R^{\text{\textgreek{e}}_{0}}\}$,
$w_{R}$ is a function of $r$ and $\tilde{w}_{R}=w_{R}$. The following
bounds are also satisfied for some constants depending only on $R_{0}$
and $\text{\textgreek{d}}_{0}$ (and the precise choice of $w$):
\begin{gather}
0<cR^{-3\text{\textgreek{e}}_{0}}\le\partial_{r}w_{R}\le C,\label{eq:BoundDerivativeIntermediate}\\
\partial_{r}^{2}w_{R}+r^{-1}\partial_{r}w_{R}\ge cR^{-3\text{\textgreek{e}}_{0}}+\big|r^{-\frac{1}{2}}\partial_{r}^{2}w_{R}\big|+\big|r^{-\frac{3}{2}}\partial_{r}w_{R}\big|,\label{eq:BoundSecondDerivativeIntermediate}\\
|\partial_{r}^{2}w_{R}|,|\partial_{r}^{3}w_{R}|,|\partial_{r}^{4}w_{R}|\le C.\label{eq:UpperboundsIntermediate}
\end{gather}

\item \label{enu:NearIntermediateRegion}In the region $\{R^{\text{\textgreek{e}}_{0}}\le r\le\frac{1}{2}R\}$:
\begin{equation}
w_{R}=\tilde{w}_{R}=C_{1}\text{\textgreek{e}}_{0}^{-1}(\frac{r}{R})^{\text{\textgreek{e}}_{0}}+C_{2}\label{eq:W_Intermediate_Region}
\end{equation}
 for some constants $C_{1},C_{2}$ depending only on $R_{0},\text{\textgreek{d}}_{0}$
(and the precise choice of $w_{l}$).

\item In the region $\{\frac{1}{2}R\le r\le R\}$:
\begin{equation}
w_{R}=\tilde{w}_{R}=v_{s}(\frac{r}{R})+C_{3}\label{eq:DefinitionWRIntermediate+}
\end{equation}
for some constant $C_{3}$ depending on $R_{0},\text{\textgreek{d}}_{0},\text{\textgreek{d}}_{1}$
(and the precise choice of $w$), where the function $v_{s}:[\frac{1}{2},1]\rightarrow\mathbb{R}$
depends on $s,\text{\textgreek{e}}_{0},\text{\textgreek{d}}_{0},\text{\textgreek{d}}_{1}$
and satisfies (for some constants $c_{\text{\textgreek{d}}_{0}},C_{\text{\textgreek{d}}_{0}}>0$
depending on $\text{\textgreek{d}}_{0},R_{0}$ and the precise choice
of $w$):
\begin{gather}
\frac{dv_{s}}{dx}\ge c_{\text{\textgreek{d}}_{0}}s^{-1},\label{eq:LowerBound}\\
\big|\frac{d^{2}v_{s}}{dx^{2}}\big|\le C_{\text{\textgreek{d}}_{0}}(\text{\textgreek{d}}_{1}s+\text{\textgreek{d}}_{1}^{-1})\frac{dv_{s}}{dx}\label{eq:UpperBoundSecondDerivative}\\
\big|\frac{d^{2}v_{s}}{dx^{2}}\big|,\big|\frac{d^{3}v_{s}}{dx^{3}}\big|,\big|\frac{d^{3}v_{s}}{dx^{3}}\big|\le C_{\text{\textgreek{d}}_{0}}\text{\textgreek{d}}_{1}^{-1}\label{eq:UpperBounds}
\end{gather}
and, for $x\in[\frac{3}{4},1]$: 
\begin{equation}
v_{s}(x)=\frac{1}{2s}\log\big(x-\frac{9}{10}\log(x)\big).\label{eq:ValueLog}
\end{equation}

\end{enumerate}\end{lem}
\begin{rem*}
Notice that we can bound on $\{R^{\text{\textgreek{e}}_{0}}\le r\le\frac{1}{2}R\}$
\begin{gather}
\partial_{r}^{2}w_{R}+r^{-1}\partial_{r}w_{R}>c_{0}\text{\textgreek{e}}_{0}R^{-\text{\textgreek{e}}_{0}}\cdot r^{-2+\text{\textgreek{e}}_{0}}+|r^{-\frac{1}{2}}\partial_{r}^{2}w_{R}|+|r^{-\frac{3}{2}}\partial_{r}w_{R}|,\label{eq:Bound_S^3_Terms_Final}\\
\partial_{r}w_{R}>c_{0}R^{-\text{\textgreek{e}}_{0}}r^{-1+\text{\textgreek{e}}_{0}},\label{eq:BoundDerivative}\\
\sum_{j=1}^{4}|r^{j}\partial_{r}^{j}w_{R}|<C_{0}R^{-\text{\textgreek{e}}_{0}}r^{\text{\textgreek{e}}_{0}}.\label{eq:BoundDerivativesUpper}
\end{gather}
\end{rem*}
\begin{proof}
The construction of $w_{R}$ (and, similarly, $\tilde{w}_{R}$) can
be readily performed in view of the following observations:

\begin{itemize}

\item{ In view of Condition \ref{enu:nearRegion} and the properties
of the function $w$, for $r=R_{0}$ we have: 
\begin{gather}
\partial_{r}w_{R}(R_{0})\sim_{R_{0}}R^{-3\text{\textgreek{e}}_{0}},\label{eq:BoundDerivativeIntermediate-1}\\
\partial_{r}^{2}w_{R}(R_{0}),r^{-1}\partial_{r}w_{R}(R_{0})\sim_{R_{0}}R^{-3\text{\textgreek{e}}_{0}},\label{eq:BoundSecondDerivativeIntermediate-1}
\end{gather}
while Condition \ref{enu:NearIntermediateRegion} requires that, for
$r=R^{\text{\textgreek{e}}_{0}}$:
\begin{gather}
\partial_{r}w_{R}(R^{\text{\textgreek{e}}_{0}})=C_{1}R^{-2\text{\textgreek{e}}_{0}+\text{\textgreek{e}}_{0}^{2}}\gg\partial_{r}w_{R}(R_{0}),\label{eq:BoundDerivativeIntermediate-1-1}\\
\partial_{r}^{2}w_{R}(R^{\text{\textgreek{e}}_{0}})=C_{1}(\text{\textgreek{e}}_{0}-1)R^{-3\text{\textgreek{e}}_{0}+\text{\textgreek{e}}_{0}^{2}}\\
\partial_{r}^{2}w_{R}(R^{\text{\textgreek{e}}_{0}})+r^{-1}\partial_{r}w_{R}(R^{\text{\textgreek{e}}_{0}})\ge c\text{\textgreek{e}}_{0}R^{-3\text{\textgreek{e}}_{0}+\text{\textgreek{e}}_{0}^{2}}+\big|r^{-\frac{1}{2}}\partial_{r}^{2}w_{R}(R^{\text{\textgreek{e}}_{0}})\big|+\big|r^{-\frac{3}{2}}\partial_{r}w_{R}(R^{\text{\textgreek{e}}_{0}})\big|.\label{eq:BoundSecondDerivativeIntermediate-1-1}
\end{gather}
Therefore, we can readily construct the function $\partial_{r}w_{R}$
(as a function of $r$) on the interval $\{R_{0}\le r\le R^{\text{\textgreek{e}}_{0}}\}$
(and then integrate in order to obtain $w_{R}$ and the constant $C_{2}$
in (\ref{eq:W_Intermediate_Region})), so that (\ref{eq:BoundDerivativeIntermediate})--(\ref{eq:UpperboundsIntermediate})
are satisfied. In particular, $\partial_{r}w_{R}$ can be constructed
as an increasing function of $r$ (i.\,e.~with $\partial_{r}^{2}w_{R}>0$)
up to $r=R^{\text{\textgreek{e}}_{0}}-1$, while for $r\in[R^{\text{\textgreek{e}}_{0}}-1,R^{\text{\textgreek{e}}_{0}}]$,
$\partial_{r}w_{R}$ is constructed a smooth function of $r$ extending
(\ref{eq:W_Intermediate_Region}) from $\{r\ge R^{\text{\textgreek{e}}_{0}}\}$
under the requirement that it satisfies the one sided bound 
\begin{equation}
\partial_{r}^{2}w_{R}(r)\ge-(1+\text{\textgreek{e}}_{0}^{2})|\partial_{r}^{2}w_{R}(R^{\text{\textgreek{e}}_{0}})|.
\end{equation}

}

\item{ Let $\tilde{v}:[\frac{3}{5},1]\rightarrow\mathbb{R}$ be a
smooth and strictly increasing function such that $\tilde{v}(x)=-(x-1)^{2}-10$
for $x\in[\frac{3}{5},\frac{7}{10}]$ and $\tilde{v}(x)=\log\big(x-\frac{9}{10}\log(x)\big)$
for $x\in[\frac{3}{4},1]$. Then, provided $s\gg1$ and $\text{\textgreek{d}}_{1}<1$,
it canbe readily inferred that there exists a $C^{5}$ and piecewise
$C^{6}$ function $\tilde{v}_{s}:[\frac{1}{2},1]\rightarrow\mathbb{R}$,
which is smooth on $[\frac{1}{2},1]\backslash\{\frac{3}{5}\}$, satisfying
$\frac{d\tilde{v}_{s}}{dx}\ge\frac{1}{10s}C_{1}$, $\sum_{j=1}^{4}\big|\frac{d^{j}\tilde{v}_{s}}{dx^{j}}\big|\le10C_{1}\text{\textgreek{d}}_{1}^{-1}$
and 
\[
\tilde{v}_{s}(x)=\begin{cases}
C_{1}\text{\textgreek{e}}_{0}^{-1}x^{\text{\textgreek{e}}_{0}}+C_{2}-C_{3} & \mbox{ for }x\in[\frac{1}{2},\frac{11}{20}]\\
-\text{\textgreek{d}}_{1}(x-\frac{3}{5})^{6}-\frac{1}{2s}\big((x-1)^{2}+10\big) & \mbox{ for }x\in[\frac{23}{40},\frac{3}{5})\\
\frac{1}{2s}\tilde{v}(x), & \mbox{ for }x\in[\frac{3}{5},1]
\end{cases}
\]
for a suitable constant $C_{3}>0$ depending on $C_{1},C_{2}$. The
function $v_{s}$ is then constructed by mollifying $\tilde{v}_{s}$
around $x=\frac{3}{5}$.}

\end{itemize}
\end{proof}

\subsection{\label{sub:ChoiceOfSeedFunctions}The seed functions $f$, $\tilde{f}$,
$h$ and $\tilde{h}$}

In this section, we will construct (using the auxiliary functions
from the previous section) the seed functions for the multipliers
that will be used in the proof of Proposition \ref{cor:CarlemanForPsik}.

We will assume without loss of generality that $0<\text{\textgreek{d}}_{0},\text{\textgreek{d}}_{1}\ll1$,
$s\gg1$, $0<\text{\textgreek{e}}_{0}\ll1$ and $R\gg\max\{R_{0},\text{\textgreek{e}}_{0}^{-1}\}$.
Let $w_{R},\tilde{w}_{R}:\{r\le R\}\rightarrow\mathbb{R}$ be the
functions from Lemma \ref{lem:ConstructionWR} (associated to the
parameters $s,R,\text{\textgreek{e}}_{0},\text{\textgreek{d}}_{0},\text{\textgreek{d}}_{1}$).
We define the smooth and $T$-invariant functions $f,\tilde{f}:\mathcal{M}\backslash\mathcal{H}^{-}\rightarrow(0,+\infty)$
as follows: 
\begin{equation}
f=\begin{cases}
e^{2sw_{R}}, & \mbox{on }\{r\le R\}\\
C_{4}^{2s}\cdot\Big(\frac{r}{R}-\frac{9}{10}\log(\frac{r}{R})\Big), & \mbox{on }\{r\ge R\},
\end{cases}\label{eq:Function_F}
\end{equation}
and 
\begin{equation}
\tilde{f}=\begin{cases}
e^{2s\tilde{w}_{R}}, & \mbox{on }\{r\le R\}\\
C_{4}^{2s}\cdot\Big(\frac{r}{R}-\frac{9}{10}\log(\frac{r}{R})\Big), & \mbox{on }\{r\ge R\},
\end{cases}\label{eq:Function_F-1}
\end{equation}
where $C_{4}>0$ is chosen so that $f$ and $\tilde{f}$ are smooth
at $r=R$ (which is possible in view of (\ref{eq:ValueLog})). 

Let $h:\mathcal{M}\backslash\mathcal{H}^{-}\rightarrow\mathbb{R}$
be a smooth and $T$-invariant function satisfying the following conditions
(provided $\text{\textgreek{d}}_{1},\text{\textgreek{d}}_{2}\ll1$): 

\begin{enumerate}

\item In the region $\{r\le\frac{4}{3}R\}$:
\begin{equation}
h=-\text{\textgreek{q}}_{\le\frac{1}{2}R_{0}}s\text{\textgreek{d}}_{1}\frac{\nabla_{\text{\textgreek{m}}}\nabla_{\text{\textgreek{n}}}w_{R}\nabla^{\text{\textgreek{m}}}w_{R}\nabla^{\text{\textgreek{n}}}w_{R}}{g_{ref}(dw_{R},dw_{R})}e^{2sw_{R}}+\text{\textgreek{q}}_{\ge R_{0}}(r^{-1}-r^{-\frac{3}{2}})\partial_{r}f,\label{eq:hAlmostR}
\end{equation}
where 
\begin{equation}
\text{\textgreek{q}}_{\le\frac{1}{2}R_{0}}\doteq\text{\textgreek{q}}_{4}(\frac{R_{0}}{2r}).
\end{equation}

\item In the region $\{\frac{4}{3}R\le r\le\text{\textgreek{d}}_{2}^{-1}R\}$,
$h$ is a function of $r$, satisfying 
\begin{equation}
cf(R)r^{-2}<h\le\min\{(r^{-1}-r^{-\frac{3}{2}})\partial_{r}f,(1-r^{-\frac{1}{2}})\partial_{r}^{2}f\}\label{eq:hIntermediateAway}
\end{equation}
 and 
\begin{equation}
-\square_{g}h\le CR^{-4}f(R)\label{eq:BoundBoxh}
\end{equation}
for some absolute constants $C,c>0$.

\item In the region $\{r\ge\text{\textgreek{d}}_{2}^{-1}R\}$: 
\begin{equation}
h=\frac{1}{2}\partial_{r}^{2}f.\label{eq:hFarAway}
\end{equation}

\end{enumerate}

Notice that $h$ can indeed be defined as above on the interval $\{(1+2\text{\textgreek{d}}_{2})R\le r\le\text{\textgreek{d}}_{2}^{-1}R\}$
(provided $\text{\textgreek{d}}_{2}$ is smaller than an absolute
constant), in view of the fact that
\begin{equation}
\min\{\partial_{r}^{2}f,r^{-1}\partial_{r}f\}\gtrsim f(R)r^{-2},
\end{equation}
\begin{equation}
\square_{g}h=\big(1+O(r^{-1})\big)\partial_{r}^{2}h+\big((d-1)r^{-1}+O(r^{-2})\big)\partial_{r}h
\end{equation}
 and 
\begin{equation}
\sum_{j=1}^{4}r^{j-4}|\partial_{r}^{j}f|\le CR^{-4}f(R)
\end{equation}
on that interval, while $(r^{-1}-r^{-\frac{3}{2}})\partial_{r}f<\partial_{r}^{2}f$
for $R\le r<\frac{4}{3}R$ (provided $R\gg1$) and $r^{-1}\partial_{r}f>\partial_{r}^{2}f$
for $r\ge\text{\textgreek{d}}_{2}^{-1}R$ (provided $\text{\textgreek{d}}_{2}\ll1$). 

We also define $\tilde{h}:\mathcal{M}\backslash\mathcal{H}^{-}\rightarrow\mathbb{R}$
in the same way as $h$, but with $\tilde{w}_{R}$ and $\tilde{f}$
in place of $w_{R}$ and $f$, respectively.

\subsection{\label{sub:IntegrationsByParts}The integration-by-parts scheme}

In this section, we will establish a general identity obtained from
equation (\ref{eq:InhomogeneousWaveEquation}) and a suitable first
order multiplier, after successively integrating by parts over $\mathcal{R}(\text{\textgreek{t}}_{1},\text{\textgreek{t}}_{2})$.
This identity will lie at the core of the proof of Proposition \ref{prop:GeneralCarlemanEstimate}.

Let $f,h$ be as in Section \ref{sub:ChoiceOfSeedFunctions}. We introduce
the following multiplier for equation (\ref{eq:InhomogeneousWaveEquation})
\begin{equation}
2\nabla^{\text{\textgreek{m}}}f\cdot\nabla_{\text{\textgreek{m}}}\text{\textgreek{f}}+\square_{g}f\cdot\text{\textgreek{f}}.\label{eq:Multiplier}
\end{equation}

Multiplying (\ref{eq:InhomogeneousWaveEquationPsiK}) with the complex
conjugate of (\ref{eq:Multiplier}) and integrating by parts over
$\mathcal{R}(\text{\textgreek{t}}_{1},\text{\textgreek{t}}_{2})$,
we obtain: 
\begin{equation}
\begin{split}\int_{\mathcal{R}(\text{\textgreek{t}}_{1},\text{\textgreek{t}}_{2})}Re\Big\{ & 2\nabla^{\text{\textgreek{m}}}\nabla^{\text{\textgreek{n}}}f\nabla_{\text{\textgreek{m}}}\text{\textgreek{f}}\nabla_{\text{\textgreek{n}}}\bar{\text{\textgreek{f}}}-\frac{1}{2}\square_{g}^{2}f|\text{\textgreek{f}}|^{2}\big\}\, dg=-\int_{\mathcal{R}(\text{\textgreek{t}}_{1},\text{\textgreek{t}}_{2})}Re\big\{ G\big(2\nabla^{\text{\textgreek{m}}}f\nabla_{\text{\textgreek{m}}}\bar{\text{\textgreek{f}}}+(\square_{g}f)\bar{\text{\textgreek{f}}}\big)\big\}\, dg-\\
- & \sum_{j=1}^{2}(-1)^{j}\int_{\text{\textgreek{S}}_{\text{\textgreek{t}}_{j}}}Re\big\{\big(2\nabla^{\text{\textgreek{m}}}f\nabla_{\text{\textgreek{m}}}\bar{\text{\textgreek{f}}}\nabla_{\text{\textgreek{n}}}\text{\textgreek{f}}+(\square_{g}f)\text{\textgreek{f}}\nabla_{\text{\textgreek{n}}}\bar{\text{\textgreek{f}}}-\nabla_{\text{\textgreek{n}}}f\nabla^{\text{\textgreek{m}}}\text{\textgreek{f}}\nabla_{\text{\textgreek{m}}}\bar{\text{\textgreek{f}}}-\frac{1}{2}(\nabla_{\text{\textgreek{n}}}(\square_{g}f))|\text{\textgreek{f}}|^{2}\big)n_{\text{\textgreek{S}}_{\text{\textgreek{t}}_{j}}}^{\text{\textgreek{n}}}\big\}\, dg_{\text{\textgreek{S}}_{\text{\textgreek{t}}_{j}}}-\\
- & \int_{\mathcal{H}^{+}\cap\mathcal{R}(\text{\textgreek{t}}_{1},\text{\textgreek{t}}_{2})}Re\big\{\big(2\nabla^{\text{\textgreek{m}}}f\nabla_{\text{\textgreek{m}}}\bar{\text{\textgreek{f}}}\nabla_{\text{\textgreek{n}}}\text{\textgreek{f}}+(\square_{g}f)\text{\textgreek{f}}\nabla_{\text{\textgreek{n}}}\bar{\text{\textgreek{f}}}-\nabla_{\text{\textgreek{n}}}f\nabla^{\text{\textgreek{m}}}\text{\textgreek{f}}\nabla_{\text{\textgreek{m}}}\bar{\text{\textgreek{f}}}-\frac{1}{2}(\nabla_{\text{\textgreek{n}}}(\square_{g}f))|\text{\textgreek{f}}|^{2}\big)n_{\mathcal{H}^{+}}^{\text{\textgreek{n}}}\big\}\, dvol_{\mathcal{H}^{+}}.
\end{split}
\label{eq:Fundamental_Current}
\end{equation}

Let us split the left hand side of (\ref{eq:Fundamental_Current})
as 
\begin{equation}
\begin{split}\int_{\mathcal{R}(\text{\textgreek{t}}_{1},\text{\textgreek{t}}_{2})}Re\Big\{2\nabla^{\text{\textgreek{m}}}\nabla^{\text{\textgreek{n}}}f\nabla_{\text{\textgreek{m}}}\text{\textgreek{f}}\nabla_{\text{\textgreek{n}}}\bar{\text{\textgreek{f}}}-\frac{1}{2}\square_{g}^{2}f|\text{\textgreek{f}}|^{2}\Big\}\, dg= & +\int_{\mathcal{R}(\text{\textgreek{t}}_{1},\text{\textgreek{t}}_{2})}\text{\textgreek{q}}_{\le R}Re\Big\{2\nabla^{\text{\textgreek{m}}}\nabla^{\text{\textgreek{n}}}f\nabla_{\text{\textgreek{m}}}\text{\textgreek{f}}\nabla_{\text{\textgreek{n}}}\bar{\text{\textgreek{f}}}-\frac{1}{2}\square_{g}^{2}f|\text{\textgreek{f}}|^{2}\Big\}\, dg+\\
 & \hphantom{\int_{R}}+\int_{\mathcal{R}(\text{\textgreek{t}}_{1},\text{\textgreek{t}}_{2})}(1-\text{\textgreek{q}}_{\le R})Re\Big\{2\nabla^{\text{\textgreek{m}}}\nabla^{\text{\textgreek{n}}}f\nabla_{\text{\textgreek{m}}}\text{\textgreek{f}}\nabla_{\text{\textgreek{n}}}\bar{\text{\textgreek{f}}}-\frac{1}{2}\square_{g}^{2}f|\text{\textgreek{f}}|^{2}\Big\}\, dg.
\end{split}
\label{eq:SplitCutOffs}
\end{equation}
 Using the identity 
\begin{equation}
\nabla_{\text{\textgreek{m}}}\text{\textgreek{f}}\nabla_{\text{\textgreek{n}}}\bar{\text{\textgreek{f}}}=f^{-1}\nabla_{\text{\textgreek{m}}}(f^{\frac{1}{2}}\text{\textgreek{f}})\nabla_{\text{\textgreek{n}}}(f^{\frac{1}{2}}\bar{\text{\textgreek{f}}})-\frac{1}{2}f^{-1}\big(\nabla_{\text{\textgreek{m}}}f\text{\textgreek{f}}\nabla_{\text{\textgreek{n}}}\bar{\text{\textgreek{f}}}+\nabla_{\text{\textgreek{n}}}f\bar{\text{\textgreek{f}}}\nabla_{\text{\textgreek{m}}}\text{\textgreek{f}}\big)-\frac{1}{4}f^{-2}\nabla_{\text{\textgreek{m}}}f\nabla_{\text{\textgreek{n}}}f|\text{\textgreek{f}}|^{2}\label{eq:RescaledPhiIdentity}
\end{equation}
 and integrating by parts in the $\text{\textgreek{f}}\nabla\text{\textgreek{f}}$
terms, we have: 
\begin{equation}
\begin{split}\int_{\mathcal{R}(\text{\textgreek{t}}_{1},\text{\textgreek{t}}_{2})} & \text{\textgreek{q}}_{\le R}Re\Big\{2\nabla^{\text{\textgreek{m}}}\nabla^{\text{\textgreek{n}}}f\nabla_{\text{\textgreek{m}}}\text{\textgreek{f}}\nabla_{\text{\textgreek{n}}}\bar{\text{\textgreek{f}}}-\frac{1}{2}\square_{g}^{2}f|\text{\textgreek{f}}|^{2}\Big\}\, dg=\\
= & \int_{\mathcal{R}(\text{\textgreek{t}}_{1},\text{\textgreek{t}}_{2})}\text{\textgreek{q}}_{\le R}Re\Big\{2f^{-1}\nabla^{\text{\textgreek{m}}}\nabla^{\text{\textgreek{n}}}f\nabla_{\text{\textgreek{m}}}(f^{\frac{1}{2}}\text{\textgreek{f}})\nabla_{\text{\textgreek{n}}}(f^{\frac{1}{2}}\bar{\text{\textgreek{f}}})+\big(\nabla_{\text{\textgreek{n}}}(f^{-1}\nabla^{\text{\textgreek{m}}}\nabla^{\text{\textgreek{n}}}f\nabla_{\text{\textgreek{m}}}f)-\frac{1}{2}f^{-2}\nabla^{\text{\textgreek{m}}}\nabla^{\text{\textgreek{n}}}f\nabla_{\text{\textgreek{m}}}f\nabla_{\text{\textgreek{n}}}f-\frac{1}{2}\square_{g}^{2}f\big)|\text{\textgreek{f}}|^{2}\Big\}\, dg+\\
 & +\int_{\mathcal{R}(\text{\textgreek{t}}_{1},\text{\textgreek{t}}_{2})}\nabla_{\text{\textgreek{n}}}\text{\textgreek{q}}_{\le R}\cdot f^{-1}\nabla^{\text{\textgreek{m}}}\nabla^{\text{\textgreek{n}}}f\nabla_{\text{\textgreek{m}}}f|\text{\textgreek{f}}|^{2}\, dg+\sum_{j=1}^{2}(-1)^{j}\int_{\text{\textgreek{S}}_{\text{\textgreek{t}}_{j}}}\text{\textgreek{q}}_{\le R}\cdot f^{-1}\nabla_{\text{\textgreek{m}}}\nabla_{\text{\textgreek{n}}}f\nabla^{\text{\textgreek{m}}}f|\text{\textgreek{f}}|^{2}n_{\text{\textgreek{S}}_{\text{\textgreek{t}}_{j}}}^{\text{\textgreek{n}}}\, dg_{\text{\textgreek{S}}_{\text{\textgreek{t}}_{j}}}+\\
 & +\int_{\mathcal{H}^{+}\cap\mathcal{R}(\text{\textgreek{t}}_{1},\text{\textgreek{t}}_{2})}\text{\textgreek{q}}_{\le R}\cdot f^{-1}\nabla_{\text{\textgreek{m}}}\nabla_{\text{\textgreek{n}}}f\nabla^{\text{\textgreek{m}}}f|\text{\textgreek{f}}|^{2}n_{\mathcal{H}^{+}}^{\text{\textgreek{n}}}\, dvol_{\mathcal{H}^{+}}.
\end{split}
\label{eq:IntegrationByPartsIntermediate}
\end{equation}
Thus, in view of (\ref{eq:IntegrationByPartsIntermediate}), the identity
(\ref{eq:SplitCutOffs}) yields: 
\begin{equation}
\begin{split}\int_{\mathcal{R}(\text{\textgreek{t}}_{1},\text{\textgreek{t}}_{2})} & Re\Big\{2\nabla^{\text{\textgreek{m}}}\nabla^{\text{\textgreek{n}}}f\nabla_{\text{\textgreek{m}}}\text{\textgreek{f}}\nabla_{\text{\textgreek{n}}}\bar{\text{\textgreek{f}}}-\frac{1}{2}\square_{g}^{2}f|\text{\textgreek{f}}|^{2}\Big\}\, dg=\\
= & \int_{\mathcal{R}(\text{\textgreek{t}}_{1},\text{\textgreek{t}}_{2})}Re\Big\{2\text{\textgreek{q}}_{\le R}f^{-1}\nabla^{\text{\textgreek{m}}}\nabla^{\text{\textgreek{n}}}f\nabla_{\text{\textgreek{m}}}(f^{\frac{1}{2}}\text{\textgreek{f}})\nabla_{\text{\textgreek{n}}}(f^{\frac{1}{2}}\bar{\text{\textgreek{f}}})+2(1-\text{\textgreek{q}}_{\le R})\nabla^{\text{\textgreek{m}}}\nabla^{\text{\textgreek{n}}}f\nabla_{\text{\textgreek{m}}}\text{\textgreek{f}}\nabla_{\text{\textgreek{n}}}\bar{\text{\textgreek{f}}}\Big\}\, dg+\\
 & +\int_{\mathcal{R}(\text{\textgreek{t}}_{1},\text{\textgreek{t}}_{2})}\Big(\text{\textgreek{q}}_{\le R}\nabla_{\text{\textgreek{n}}}\big(f^{-1}\nabla^{\text{\textgreek{m}}}\nabla^{\text{\textgreek{n}}}f\nabla_{\text{\textgreek{m}}}f\big)-\frac{1}{2}\text{\textgreek{q}}_{\le R}f^{-2}\nabla^{\text{\textgreek{m}}}\nabla^{\text{\textgreek{n}}}f\nabla_{\text{\textgreek{m}}}f\nabla_{\text{\textgreek{n}}}f-\frac{1}{2}\square_{g}^{2}f\Big)|\text{\textgreek{f}}|^{2}\, dg+\\
 & +\int_{\mathcal{R}(\text{\textgreek{t}}_{1},\text{\textgreek{t}}_{2})}\big(\nabla_{\text{\textgreek{n}}}\text{\textgreek{q}}_{\le R}\cdot f^{-1}\nabla^{\text{\textgreek{m}}}\nabla^{\text{\textgreek{n}}}f\nabla_{\text{\textgreek{m}}}f\big)|\text{\textgreek{f}}|^{2}\, dg+\sum_{j=1}^{2}(-1)^{j}\int_{\text{\textgreek{S}}_{\text{\textgreek{t}}_{j}}}\text{\textgreek{q}}_{\le R}f^{-1}\nabla_{\text{\textgreek{m}}}\nabla_{\text{\textgreek{n}}}f\nabla^{\text{\textgreek{m}}}f|\text{\textgreek{f}}|^{2}n_{\text{\textgreek{S}}_{\text{\textgreek{t}}_{j}}}^{\text{\textgreek{n}}}\, dg_{\text{\textgreek{S}}_{\text{\textgreek{t}}_{j}}}+\\
 & +\int_{\mathcal{H}^{+}\cap\mathcal{R}(\text{\textgreek{t}}_{1},\text{\textgreek{t}}_{2})}\text{\textgreek{q}}_{\le R}\cdot f^{-1}\nabla_{\text{\textgreek{m}}}\nabla_{\text{\textgreek{n}}}f\nabla^{\text{\textgreek{m}}}f|\text{\textgreek{f}}|^{2}n_{\mathcal{H}^{+}}^{\text{\textgreek{n}}}\, dvol_{\mathcal{H}^{+}}.
\end{split}
\label{eq:BeforeSplittingAwayRegion}
\end{equation}

Adding to (\ref{eq:BeforeSplittingAwayRegion}) the identity 
\begin{equation}
\begin{split}0=\int_{\mathcal{R}(\text{\textgreek{t}}_{1},\text{\textgreek{t}}_{2})}\Big(-2\text{\textgreek{q}}_{\ge R_{0}} & \text{\textgreek{q}}_{\le R}r^{-1}f^{-1}(\partial_{r}f)\big|\partial_{r}(f^{\frac{1}{2}}\text{\textgreek{f}})\big|^{2}+2\text{\textgreek{q}}_{\ge R_{0}}\text{\textgreek{q}}_{\le R}r^{-1}(\partial_{r}f)|\partial_{r}\text{\textgreek{f}}|^{2}+\\
+ & \text{\textgreek{q}}_{\ge R_{0}}\text{\textgreek{q}}_{\le R}r^{-1}f^{-1}(\partial_{r}f)^{2}\partial_{r}(|\text{\textgreek{f}}|^{2})+\frac{1}{2}\text{\textgreek{q}}_{\ge R_{0}}\text{\textgreek{q}}_{\le R}r^{-1}f^{-2}(\partial_{r}f)^{3}|\text{\textgreek{f}}|^{2}\Big)\, dg
\end{split}
\label{eq:LIdentity}
\end{equation}
(recall that $\partial_{r}$ is the coordinate vector field in the
$(t,r,\text{\textgreek{sv}})$ coordinate chart in each connected
component of the region $\{r\ge\frac{1}{4}R_{0}\}$) and integrating
by parts in the $\partial_{r}(|\text{\textgreek{f}}|^{2})$ term,
we obtain: 
\begin{equation}
\begin{split}\int_{\mathcal{R}(\text{\textgreek{t}}_{1},\text{\textgreek{t}}_{2})} & Re\Big\{2\nabla^{\text{\textgreek{m}}}\nabla^{\text{\textgreek{n}}}f\nabla_{\text{\textgreek{m}}}\text{\textgreek{f}}\nabla_{\text{\textgreek{n}}}\bar{\text{\textgreek{f}}}-\frac{1}{2}\square_{g}^{2}f|\text{\textgreek{f}}|^{2}\Big\}\, dg=\\
= & \int_{\mathcal{R}(\text{\textgreek{t}}_{1},\text{\textgreek{t}}_{2})}Re\Big\{2\text{\textgreek{q}}_{\le R}\Big(f^{-1}\nabla^{\text{\textgreek{m}}}\nabla^{\text{\textgreek{n}}}f\nabla_{\text{\textgreek{m}}}(f^{\frac{1}{2}}\text{\textgreek{f}})\nabla_{\text{\textgreek{n}}}(f^{\frac{1}{2}}\bar{\text{\textgreek{f}}})-\text{\textgreek{q}}_{\ge R_{0}}r^{-1}f^{-1}(\partial_{r}f)\big|\partial_{r}(f^{\frac{1}{2}}\text{\textgreek{f}})\big|^{2}\Big)+\\
 & \hphantom{\int_{\mathcal{R}(\text{\textgreek{t}}_{1},\text{\textgreek{t}}_{2})}Re\Big\{}+2(1-\text{\textgreek{q}}_{\le R})\nabla^{\text{\textgreek{m}}}\nabla^{\text{\textgreek{n}}}f\nabla_{\text{\textgreek{m}}}\text{\textgreek{f}}\nabla_{\text{\textgreek{n}}}\bar{\text{\textgreek{f}}}+2\text{\textgreek{q}}_{\ge R_{0}}\text{\textgreek{q}}_{\le R}r^{-1}(\partial_{r}f)|\partial_{r}\text{\textgreek{f}}|^{2}+\mathcal{A}_{f}^{(R)}|\text{\textgreek{f}}|^{2}\Big\}\, dg+\\
 & +\sum_{j=1}^{2}(-1)^{j}\int_{\text{\textgreek{S}}_{\text{\textgreek{t}}_{j}}}\text{\textgreek{q}}_{\le R}f^{-1}\nabla_{\text{\textgreek{m}}}\nabla_{\text{\textgreek{n}}}f\nabla^{\text{\textgreek{m}}}f|\text{\textgreek{f}}|^{2}n_{\text{\textgreek{S}}_{\text{\textgreek{t}}_{j}}}^{\text{\textgreek{n}}}\, dg_{\text{\textgreek{S}}_{\text{\textgreek{t}}_{j}}}+\\
 & +\int_{\mathcal{H}^{+}\cap\mathcal{R}(\text{\textgreek{t}}_{1},\text{\textgreek{t}}_{2})}\text{\textgreek{q}}_{\le R}\cdot f^{-1}\nabla_{\text{\textgreek{m}}}\nabla_{\text{\textgreek{n}}}f\nabla^{\text{\textgreek{m}}}f|\text{\textgreek{f}}|^{2}n_{\mathcal{H}^{+}}^{\text{\textgreek{n}}}\, dvol_{\mathcal{H}^{+}},
\end{split}
\label{eq:AfterSplittingAwayRegion}
\end{equation}
where 
\begin{align}
\mathcal{A}_{f}^{(R)}\doteq & \text{\textgreek{q}}_{\le R}\nabla_{\text{\textgreek{n}}}\big(f^{-1}\nabla^{\text{\textgreek{m}}}\nabla^{\text{\textgreek{n}}}f\nabla_{\text{\textgreek{m}}}f\big)-\frac{1}{2}\text{\textgreek{q}}_{\le R}f^{-2}\nabla^{\text{\textgreek{m}}}\nabla^{\text{\textgreek{n}}}f\nabla_{\text{\textgreek{m}}}f\nabla_{\text{\textgreek{n}}}f-\frac{1}{2}\square_{g}^{2}f+\label{eq:AlmostFinalZerothOrderTerm}\\
 & -\text{\textgreek{q}}_{\ge R_{0}}\text{\textgreek{q}}_{\le R}\partial_{r}\big(r^{-1}f^{-1}(\partial_{r}f)^{2}\big)-\text{\textgreek{q}}_{\ge R_{0}}\text{\textgreek{q}}_{\le R}r^{-1}f^{-1}(\partial_{r}f)^{2}div(\partial_{r})+\frac{1}{2}\text{\textgreek{q}}_{\ge R_{0}}\text{\textgreek{q}}_{\le R}r^{-1}f^{-2}(\partial_{r}f)^{3}-\nonumber \\
 & +\nabla_{\text{\textgreek{n}}}\text{\textgreek{q}}_{\le R}\cdot f^{-1}\nabla^{\text{\textgreek{m}}}\nabla^{\text{\textgreek{n}}}f\nabla_{\text{\textgreek{m}}}f-\partial_{r}(\text{\textgreek{q}}_{\ge R_{0}}\text{\textgreek{q}}_{\le R})r^{-1}f^{-1}(\partial_{r}f)^{2}.\nonumber 
\end{align}
Thus, (\ref{eq:Fundamental_Current}) and (\ref{eq:AfterSplittingAwayRegion})
yield: 
\begin{equation}
\begin{split}\int_{\mathcal{R}(\text{\textgreek{t}}_{1},\text{\textgreek{t}}_{2})}Re\Big\{2\text{\textgreek{q}}_{\le R}\Big(f^{-1}\nabla^{\text{\textgreek{m}}}\nabla^{\text{\textgreek{n}}}f\nabla_{\text{\textgreek{m}}}(f^{\frac{1}{2}}\text{\textgreek{f}})\nabla_{\text{\textgreek{n}}}(f^{\frac{1}{2}}\bar{\text{\textgreek{f}}})-\text{\textgreek{q}}_{\ge R_{0}} & r^{-1}f^{-1}(\partial_{r}f)\big|\partial_{r}(f^{\frac{1}{2}}\text{\textgreek{f}})\big|^{2}\Big)+\\
\hphantom{\int_{\mathcal{R}(\text{\textgreek{t}}_{1},\text{\textgreek{t}}_{2})}Re\Big\{}+2(1-\text{\textgreek{q}}_{\le R})\nabla^{\text{\textgreek{m}}}\nabla^{\text{\textgreek{n}}}f\nabla_{\text{\textgreek{m}}}\text{\textgreek{f}}\nabla_{\text{\textgreek{n}}}\bar{\text{\textgreek{f}}}+2\text{\textgreek{q}}_{\ge R_{0}} & \text{\textgreek{q}}_{\le R}r^{-1}(\partial_{r}f)|\partial_{r}\text{\textgreek{f}}|^{2}+\mathcal{A}_{f}^{(R)}|\text{\textgreek{f}}|^{2}\Big\}\, dg=\\
= & -\int_{\mathcal{R}(\text{\textgreek{t}}_{1},\text{\textgreek{t}}_{2})}Re\big\{ G\big(2\nabla^{\text{\textgreek{m}}}f\nabla_{\text{\textgreek{m}}}\bar{\text{\textgreek{f}}}+(\square_{g}f)\bar{\text{\textgreek{f}}}\big)\big\}\, dg-\mathcal{B}_{f}^{(R)}[\text{\textgreek{f}};\text{\textgreek{t}}_{1},\text{\textgreek{t}}_{2}],
\end{split}
\label{eq:FundamentalCurrentBeforeFull-1}
\end{equation}
where 
\begin{equation}
\begin{split}\mathcal{B}_{f}^{(R)}[\text{\textgreek{f}};\text{\textgreek{t}}_{1},\text{\textgreek{t}}_{2}]= & \sum_{j=1}^{2}(-1)^{j}\int_{\text{\textgreek{S}}_{\text{\textgreek{t}}_{j}}}Re\Big\{\Big(2\nabla^{\text{\textgreek{m}}}f\nabla_{\text{\textgreek{m}}}\bar{\text{\textgreek{f}}}\nabla_{\text{\textgreek{n}}}\text{\textgreek{f}}+(\square_{g}f)\text{\textgreek{f}}\nabla_{\text{\textgreek{n}}}\bar{\text{\textgreek{f}}}-\nabla_{\text{\textgreek{n}}}f\nabla^{\text{\textgreek{m}}}\text{\textgreek{f}}\nabla_{\text{\textgreek{m}}}\bar{\text{\textgreek{f}}}+\\
 & \hphantom{\sum_{j=1}^{2}(-1)^{j}\int_{\text{\textgreek{S}}_{\text{\textgreek{t}}_{j}}}Re\big\{}+\big(\text{\textgreek{q}}_{\le R}f^{-1}\nabla_{\text{\textgreek{m}}}\nabla_{\text{\textgreek{n}}}f\nabla^{\text{\textgreek{m}}}f-\frac{1}{2}(\nabla_{\text{\textgreek{n}}}(\square_{g}f))\big)|\text{\textgreek{f}}|^{2}\Big)n_{\text{\textgreek{S}}_{\text{\textgreek{t}}_{j}}}^{\text{\textgreek{n}}}\Big\}\, dg_{\text{\textgreek{S}}_{\text{\textgreek{t}}_{j}}}+\\
 & \hphantom{\sum_{j=1}^{2}}+\int_{\mathcal{H}^{+}\cap\mathcal{R}(\text{\textgreek{t}}_{1},\text{\textgreek{t}}_{2})}Re\Big\{\Big(2\nabla^{\text{\textgreek{m}}}f\nabla_{\text{\textgreek{m}}}\bar{\text{\textgreek{f}}}\nabla_{\text{\textgreek{n}}}\text{\textgreek{f}}+(\square_{g}f)\text{\textgreek{f}}\nabla_{\text{\textgreek{n}}}\bar{\text{\textgreek{f}}}-\nabla_{\text{\textgreek{n}}}f\nabla^{\text{\textgreek{m}}}\text{\textgreek{f}}\nabla_{\text{\textgreek{m}}}\bar{\text{\textgreek{f}}}+\\
 & \hphantom{\hphantom{\sum_{j=1}^{2}}+\int_{\mathcal{H}^{+}\cap\mathcal{R}(\text{\textgreek{t}}_{1},\text{\textgreek{t}}_{2})}Re\Big\{}+\big(\text{\textgreek{q}}_{\le R}f^{-1}\nabla_{\text{\textgreek{m}}}\nabla_{\text{\textgreek{n}}}f\nabla^{\text{\textgreek{m}}}f-\frac{1}{2}(\nabla_{\text{\textgreek{n}}}(\square_{g}f))\big)|\text{\textgreek{f}}|^{2}\Big)n_{\mathcal{H}^{+}}^{\text{\textgreek{n}}}\Big\}\, dvol_{\mathcal{H}^{+}}.
\end{split}
\label{eq:BoundaryTerms}
\end{equation}

Finally, for $h:\mathcal{M}\backslash\mathcal{H}^{-}\rightarrow\mathbb{R}$
as in Section \ref{sub:ChoiceOfSeedFunctions}, adding to (\ref{eq:FundamentalCurrentBeforeFull-1})
the Lagrangean identity 
\begin{equation}
\begin{split}\int_{\mathcal{R}(\text{\textgreek{t}}_{1},\text{\textgreek{t}}_{2})}\big(-2h\nabla^{\text{\textgreek{m}}}\text{\textgreek{f}}\nabla_{\text{\textgreek{m}}}\bar{\text{\textgreek{f}}} & +(\square_{g}h)|\text{\textgreek{f}}|^{2}\big)\, dg=\int_{\mathcal{R}(\text{\textgreek{t}}_{1},\text{\textgreek{t}}_{2})}Re\big\{ G\cdot2h\bar{\text{\textgreek{f}}}\big\}\, dg+\\
+ & \sum_{j=1}^{2}(-1)^{j}\int_{\text{\textgreek{S}}_{\text{\textgreek{t}}_{j}}}Re\big\{\big(2h\nabla_{\text{\textgreek{n}}}\text{\textgreek{f}}\bar{\text{\textgreek{f}}}-\nabla_{\text{\textgreek{n}}}h|\text{\textgreek{f}}|^{2}\big)n_{\text{\textgreek{S}}_{\text{\textgreek{t}}_{j}}}^{\text{\textgreek{n}}}\big\}\, dg_{\text{\textgreek{S}}_{\text{\textgreek{t}}_{j}}}+\\
+ & \int_{\mathcal{H}^{+}\cap\mathcal{R}(\text{\textgreek{t}}_{1},\text{\textgreek{t}}_{2})}Re\big\{\big(2h\nabla_{\text{\textgreek{n}}}\text{\textgreek{f}}\bar{\text{\textgreek{f}}}-\nabla_{\text{\textgreek{n}}}h|\text{\textgreek{f}}|^{2}\big)n_{\mathcal{H}^{+}}^{\text{\textgreek{n}}}\big\}\, dvol_{\mathcal{H}^{+}},
\end{split}
\label{eq:Lagrange}
\end{equation}
we obtain: 
\begin{equation}
\begin{split}\int_{\mathcal{R}(\text{\textgreek{t}}_{1},\text{\textgreek{t}}_{2})}Re\Big\{2\text{\textgreek{q}}_{\le R}f^{-1}\nabla^{\text{\textgreek{m}}}\nabla^{\text{\textgreek{n}}}f\nabla_{\text{\textgreek{m}}}(f^{\frac{1}{2}}\text{\textgreek{f}})\nabla_{\text{\textgreek{n}}}(f^{\frac{1}{2}}\bar{\text{\textgreek{f}}})-2\text{\textgreek{q}}_{\le R}\text{\textgreek{q}}_{\ge R_{0}} & r^{-1}f^{-1}(\partial_{r}f)\big|\partial_{r}(f^{\frac{1}{2}}\text{\textgreek{f}})\big|^{2}+\\
\hphantom{\int_{\mathcal{R}(\text{\textgreek{t}}_{1},\text{\textgreek{t}}_{2})}Re\Big\{}+2(1-\text{\textgreek{q}}_{\le R})\nabla^{\text{\textgreek{m}}}\nabla^{\text{\textgreek{n}}}f\nabla_{\text{\textgreek{m}}}\text{\textgreek{f}}\nabla_{\text{\textgreek{n}}}\bar{\text{\textgreek{f}}}+2\text{\textgreek{q}}_{\ge R_{0}} & \text{\textgreek{q}}_{\le R}r^{-1}(\partial_{r}f)|\partial_{r}\text{\textgreek{f}}|^{2}-2h\nabla^{\text{\textgreek{m}}}\text{\textgreek{f}}\nabla_{\text{\textgreek{m}}}\bar{\text{\textgreek{f}}}+\mathcal{A}_{f,h}^{(R)}|\text{\textgreek{f}}|^{2}\Big\}\, dg=\\
= & -\int_{\mathcal{R}(\text{\textgreek{t}}_{1},\text{\textgreek{t}}_{2})}Re\big\{ G\big(2\nabla^{\text{\textgreek{m}}}f\nabla_{\text{\textgreek{m}}}\bar{\text{\textgreek{f}}}+(\square_{g}f-2h)\bar{\text{\textgreek{f}}}\big)\big\}\, dg-\mathcal{B}_{f,h}^{(R)}[\text{\textgreek{f}};\text{\textgreek{t}}_{1},\text{\textgreek{t}}_{2}],
\end{split}
\label{eq:FundamentalCurrentFull}
\end{equation}
where 
\begin{align}
\mathcal{A}_{f,h}^{(R)}\doteq & \square_{g}h+\text{\textgreek{q}}_{\le R}\nabla_{\text{\textgreek{n}}}\big(f^{-1}\nabla^{\text{\textgreek{m}}}\nabla^{\text{\textgreek{n}}}f\nabla_{\text{\textgreek{m}}}f\big)-\frac{1}{2}\text{\textgreek{q}}_{\le R}f^{-2}\nabla^{\text{\textgreek{m}}}\nabla^{\text{\textgreek{n}}}f\nabla_{\text{\textgreek{m}}}f\nabla_{\text{\textgreek{n}}}f-\frac{1}{2}\square_{g}^{2}f-\label{eq:FinalZerothOrderTerm}\\
 & -\text{\textgreek{q}}_{\ge R_{0}}\text{\textgreek{q}}_{\le R}\partial_{r}\big(r^{-1}f^{-1}(\partial_{r}f)^{2}\big)-\text{\textgreek{q}}_{\ge R_{0}}\text{\textgreek{q}}_{\le R}r^{-1}f^{-1}(\partial_{r}f)^{2}div(\partial_{r})+\frac{1}{2}\text{\textgreek{q}}_{\ge R_{0}}\text{\textgreek{q}}_{\le R}r^{-1}f^{-2}(\partial_{r}f)^{3}-\nonumber \\
 & +\nabla_{\text{\textgreek{n}}}\text{\textgreek{q}}_{\le R}\cdot f^{-1}\nabla^{\text{\textgreek{m}}}\nabla^{\text{\textgreek{n}}}f\nabla_{\text{\textgreek{m}}}f-\partial_{r}(\text{\textgreek{q}}_{\ge R_{0}}\text{\textgreek{q}}_{\le R})r^{-1}f^{-1}(\partial_{r}f)^{2}\nonumber 
\end{align}
and 
\begin{equation}
\begin{split}\mathcal{B}_{f,h}^{(R)}[\text{\textgreek{f}};\text{\textgreek{t}}_{1},\text{\textgreek{t}}_{2}]= & \sum_{j=1}^{2}(-1)^{j}\int_{\text{\textgreek{S}}_{\text{\textgreek{t}}_{j}}}Re\Big\{\Big(2\nabla^{\text{\textgreek{m}}}f\nabla_{\text{\textgreek{m}}}\bar{\text{\textgreek{f}}}\nabla_{\text{\textgreek{n}}}\text{\textgreek{f}}+(\square_{g}f-2h)\text{\textgreek{f}}\nabla_{\text{\textgreek{n}}}\bar{\text{\textgreek{f}}}-\nabla_{\text{\textgreek{n}}}f\nabla^{\text{\textgreek{m}}}\text{\textgreek{f}}\nabla_{\text{\textgreek{m}}}\bar{\text{\textgreek{f}}}+\\
 & \hphantom{\sum_{j=1}^{2}(-1)^{j}\int_{\text{\textgreek{S}}_{\text{\textgreek{t}}_{j}}}Re\big\{}+\big(\text{\textgreek{q}}_{\le R}f^{-1}\nabla_{\text{\textgreek{m}}}\nabla_{\text{\textgreek{n}}}f\nabla^{\text{\textgreek{m}}}f+\nabla_{\text{\textgreek{n}}}h-\frac{1}{2}(\nabla_{\text{\textgreek{n}}}(\square_{g}f))\big)|\text{\textgreek{f}}|^{2}\Big)n_{\text{\textgreek{S}}_{\text{\textgreek{t}}_{j}}}^{\text{\textgreek{n}}}\Big\}\, dg_{\text{\textgreek{S}}_{\text{\textgreek{t}}_{j}}}+\\
 & \hphantom{\sum_{j=1}^{2}}+\int_{\mathcal{H}^{+}\cap\mathcal{R}(\text{\textgreek{t}}_{1},\text{\textgreek{t}}_{2})}Re\Big\{\Big(2\nabla^{\text{\textgreek{m}}}f\nabla_{\text{\textgreek{m}}}\bar{\text{\textgreek{f}}}\nabla_{\text{\textgreek{n}}}\text{\textgreek{f}}+(\square_{g}f-2h)\text{\textgreek{f}}\nabla_{\text{\textgreek{n}}}\bar{\text{\textgreek{f}}}-\nabla_{\text{\textgreek{n}}}f\nabla^{\text{\textgreek{m}}}\text{\textgreek{f}}\nabla_{\text{\textgreek{m}}}\bar{\text{\textgreek{f}}}+\\
 & \hphantom{\hphantom{\sum_{j=1}^{2}}+\int_{\mathcal{H}^{+}\cap\mathcal{R}(\text{\textgreek{t}}_{1},\text{\textgreek{t}}_{2})}Re\Big\{}+\big(\text{\textgreek{q}}_{\le R}f^{-1}\nabla_{\text{\textgreek{m}}}\nabla_{\text{\textgreek{n}}}f\nabla^{\text{\textgreek{m}}}f+\nabla_{\text{\textgreek{n}}}h-\frac{1}{2}(\nabla_{\text{\textgreek{n}}}(\square_{g}f))\big)|\text{\textgreek{f}}|^{2}\Big)n_{\mathcal{H}^{+}}^{\text{\textgreek{n}}}\Big\}\, dvol_{\mathcal{H}^{+}}.
\end{split}
\label{eq:BoundaryTermsFinal}
\end{equation}

In the next sections, we will esablish a number of estimates for the
left hand side of (\ref{eq:FundamentalCurrentFull}) that will lead
to the proof of Proposition \ref{prop:GeneralCarlemanEstimate}.

\subsection{\label{sub:PropertiesOfAfh}Estimates for the zeroth order term}

In this section, we will establish some bounds for the coefficient
$\mathcal{A}_{f,h}^{(R)}$ of the zeroth order term appearing in the
left hand side of (\ref{eq:FundamentalCurrentFull}).

In view of the choice of the funcions $f,h$ in Section \ref{sub:ChoiceOfSeedFunctions},
we can readily calculate that on $\{r\le R\}$ (where $\text{\textgreek{q}}_{R}\equiv1$),
the quantity $\mathcal{A}_{f,h}^{(R)}$ in (\ref{eq:FinalZerothOrderTerm})
has the form 
\begin{equation}
\mathcal{A}_{f,h}^{(R)}=\Big\{\mathcal{A}_{w_{R},3}s^{3}+\mathcal{A}_{w_{R},2}s^{2}+\mathcal{A}_{w_{R},1}s\Big\} e^{2sw_{R}},\label{eq:GeneralFormOfAfhR}
\end{equation}
where 
\begin{align}
\mathcal{A}_{w_{R},3}= & \Big(4-\text{\textgreek{q}}_{\le\frac{1}{2}R_{0}}\text{\textgreek{d}}_{1}\frac{\nabla^{\text{\textgreek{a}}}w_{R}\nabla_{\text{\textgreek{a}}}w_{R}}{g_{ref}(dw_{R},dw_{R})}\Big)\nabla^{\text{\textgreek{m}}}\nabla^{\text{\textgreek{n}}}w_{R}\nabla_{\text{\textgreek{m}}}w_{R}\nabla_{\text{\textgreek{n}}}w_{R}+4\text{\textgreek{q}}_{\ge R_{0}}r^{-1}\big(1+O(r^{-\frac{1}{2}})\big)(\partial_{r}w_{R})^{3},\label{eq:CoefficientOfS3}\\
\nonumber \\
\mathcal{A}_{w_{R},2}= & 4\nabla_{\text{\textgreek{n}}}\nabla^{\text{\textgreek{m}}}\nabla^{\text{\textgreek{n}}}w_{R}\nabla_{\text{\textgreek{m}}}w_{R}-4\nabla_{\text{\textgreek{n}}}\nabla^{\text{\textgreek{n}}}\nabla^{\text{\textgreek{m}}}w_{R}\nabla_{\text{\textgreek{m}}}w_{R}-\label{eq:CoefficientOfS2}\\
 & -4\nabla^{\text{\textgreek{m}}}(\square_{g}w_{R})\nabla_{\text{\textgreek{m}}}w_{R}-2(\square_{g}w_{R})^{2}+\nonumber \\
 & +4\text{\textgreek{q}}_{\ge R_{0}}r^{-1}\big(1+O(r^{-\frac{1}{2}})\big)\partial_{r}^{2}w_{R}\partial_{r}w_{R}+4\text{\textgreek{q}}_{\ge R_{0}}r^{-2}\big(1+O(r^{-\frac{1}{2}})\big)(\partial_{r}w_{R})^{2}+\nonumber \\
 & +O(\text{\textgreek{d}}_{1})\text{\textgreek{q}}_{\le\frac{1}{2}R_{0}}\sum_{j=0}^{1}|\nabla^{2+j}w_{R}|_{g_{ref}}|\nabla^{2-j}w_{R}|_{g_{ref}}+\nonumber \\
 & +O(|\nabla\text{\textgreek{q}}_{\le R_{0}}|_{g_{ref}}+|\nabla\text{\textgreek{q}}_{\le\frac{1}{2}R_{0}}|_{g_{ref}})\big(|\nabla^{2}w_{R}|_{g_{ref}}^{2}+|\nabla w_{R}|_{g_{ref}}^{2}\big),\nonumber \\
\nonumber \\
\mathcal{A}_{w_{R},1}= & -\square_{g}^{2}w_{R}+2\text{\textgreek{q}}_{\ge R_{0}}r^{-1}\big(1+O(r^{-\frac{1}{2}})\big)\partial_{r}^{3}w+2(d-3)\text{\textgreek{q}}_{\ge R_{0}}r^{-2}\big(1+O(r^{-\frac{1}{2}})\big)\partial_{r}^{2}w_{R}-\label{eq:CoefficientOfS1}\\
 & -2(d-3)\text{\textgreek{q}}_{\ge R_{0}}r^{-3}\big(1+O(r^{-\frac{1}{2}})\big)\partial_{r}w_{R}+\nonumber \\
 & +O(\text{\textgreek{d}}_{1})\text{\textgreek{q}}_{\le\frac{1}{2}R_{0}}\sum_{j_{1}+j_{2}+j_{3}=1}^{2}\frac{|\nabla^{2+j_{1}}w_{R}|_{g_{ref}}|\nabla^{1+j_{2}}w_{R}|_{g_{ref}}|\nabla^{1+j_{3}}w_{R}|_{g_{ref}}}{|\nabla w_{R}|_{g_{ref}}^{2}}+\nonumber \\
 & +2(\square_{g}\text{\textgreek{q}}_{\le R_{0}})r^{-1}\partial_{r}w_{R}-\text{\textgreek{d}}_{1}(\square_{g}\text{\textgreek{q}}_{\le\frac{1}{2}R_{0}})\nabla^{\text{\textgreek{m}}}\nabla^{\text{\textgreek{n}}}w_{R}\nabla_{\text{\textgreek{m}}}w_{R}\nabla_{\text{\textgreek{n}}}w_{R}\frac{\nabla^{\text{\textgreek{a}}}w_{R}\nabla_{\text{\textgreek{a}}}w_{R}}{g_{ref}(dw_{R},dw_{R})}+\nonumber \\
 & +O\Big(\sum_{j=1}^{2}(|\nabla^{j}\text{\textgreek{q}}_{\le R_{0}}|_{g_{ref}}+|\nabla^{j}\text{\textgreek{q}}_{\le\frac{1}{2}R_{0}}|_{g_{ref}})\Big)\big(|\nabla^{2}w_{R}|_{g_{ref}}+|\nabla w_{R}|_{g_{ref}}\big)\nonumber 
\end{align}
(with the constants implicit in the $O(\cdot)$ notation depending
only on the geometry of $(\mathcal{M},g)$).
\begin{rem*}
Notice the cancellation of the $O(s^{4})$ terms that were expected
to apper in (\ref{eq:GeneralFormOfAfhR}).
\end{rem*}

\subsubsection{Bound on $\{r\le R_{0}\}$}

In view of (\ref{eq:GeneralFormOfAfhR})--(\ref{eq:CoefficientOfS1}),
the properties of the function $w_{R}$ (see Lemma \ref{lem:ConstructionWR})
and the form (\ref{eq:metric}) of the metric $g$ in the region $r\gg1$
imply that in the region $\{r\le R_{0}\}$: 
\begin{equation}
\mathcal{A}_{f,h}^{(R)}=\Big\{\mathcal{A}_{w,3;R_{0}}R^{-9\text{\textgreek{e}}_{0}}s^{3}+O_{\text{\textgreek{d}}_{0}}(1)R^{-6\text{\textgreek{e}}_{0}}s^{2}+O_{\text{\textgreek{d}}_{0}}(1)R^{-3\text{\textgreek{e}}_{0}}s\Big\} e^{2sw_{R}},\label{eq:IntermediateRegionAfh}
\end{equation}
where 
\begin{equation}
\mathcal{A}_{w,3;R_{0}}=\Big(4-\text{\textgreek{q}}_{\le\frac{1}{2}R_{0}}\text{\textgreek{d}}_{1}\frac{\nabla^{\text{\textgreek{a}}}w\nabla_{\text{\textgreek{a}}}w}{g_{ref}(dw,dw)}\Big)\nabla^{\text{\textgreek{m}}}\nabla^{\text{\textgreek{n}}}w\nabla_{\text{\textgreek{m}}}w\nabla_{\text{\textgreek{n}}}w+4\text{\textgreek{q}}_{\ge R_{0}}r^{-1}\big(1+O(r^{-\frac{1}{2}})\big)(\partial_{r}w)^{3},\label{eq:TopCoefficientInIntermediateRegion}
\end{equation}
$w$ is the function from Lemma \ref{lem:PositivityHessianW} and
the constants implicit in the $O_{\text{\textgreek{d}}_{0}}(1)$ notation
depend only on $R_{0},\text{\textgreek{d}}_{0}$.

\subsubsection{Bounds on $\{R_{0}\le r\le R\}$}

In the region $\{R_{0}\le r\le R\}$, the expressions (\ref{eq:CoefficientOfS3}),
(\ref{eq:CoefficientOfS2}) and (\ref{eq:CoefficientOfS1}) simplify
as follows, in view of (\ref{eq:metric}) and the fact that $w_{R}$
is a function of $r$ for $r\ge R_{0}$: 
\begin{align}
\mathcal{A}_{w_{R},3}= & 4\big(1+O(r^{-\frac{1}{2}})\big)\partial_{r}^{2}w_{R}(\partial_{r}w_{R})^{2}+4r^{-1}\big(1+O(r^{-\frac{1}{2}})\big)(\partial_{r}w_{R})^{3},\\
\mathcal{A}_{w_{R},2}= & -4\big(1+O(r^{-\frac{1}{2}})\big)\partial_{r}w_{R}\partial_{r}^{3}w_{R}-2\big(1+O(r^{-\frac{1}{2}})\big)(\partial_{r}^{2}w_{R})^{2}+\\
 & \hphantom{+}+(4-8(d-1))r^{-1}\big(1+O(r^{-\frac{1}{2}})\big)\partial_{r}^{2}w_{R}\partial_{r}w_{R}+\nonumber \\
 & \hphantom{+}+(4(d-2)-2(d-1)^{2})r^{-2}\big(1+O(r^{-\frac{1}{2}})\big)(\partial_{r}w_{R})^{2},\nonumber \\
\mathcal{A}_{w_{R},1}= & -\big(1+O(r^{-\frac{1}{2}})\big)\partial_{r}^{4}w_{R}-2(d-1)r^{-1}\big(1+O(r^{-\frac{1}{2}})\big)\partial_{r}^{3}w_{R}-\\
 & \hphantom{+}-(d-3)^{2}r^{-1}\big(1+O(r^{-\frac{1}{2}})\big)\partial_{r}^{2}w_{R}+(d-3)^{2}r^{-3}\big(1+O(r^{-\frac{1}{2}})\big)\partial_{r}w_{R}.\nonumber 
\end{align}
Therefore, the properties of the function $w_{R}$ (see Lemma \ref{lem:ConstructionWR})
imply the following relations for $\mathcal{A}_{f,h}^{(R)}$ on $\{R_{0}\le r\le R\}$
(provided $R_{0}$is sufficiently large in terms of the geometry of
$(\mathcal{M},g)$):

\begin{enumerate}

\item In the region $\{R_{0}\le r\le R^{\text{\textgreek{e}}_{0}}\}$,
(\ref{eq:BoundDerivativeIntermediate})--(\ref{eq:UpperboundsIntermediate})
yield: 
\begin{equation}
\mathcal{A}_{f,h}^{(R)}\ge\Big\{ c_{\text{\textgreek{d}}_{0}}R^{-9\text{\textgreek{e}}_{0}}s^{3}-C_{\text{\textgreek{d}}_{0}}s^{2}-C_{\text{\textgreek{d}}_{0}}s\Big\} e^{2sw_{R}}\label{eq:LowerBoundAfhFirstAwayRegion}
\end{equation}
 for some constants $c_{\text{\textgreek{d}}_{0}},C_{\text{\textgreek{d}}_{0}}>0$
depending on $\text{\textgreek{d}}_{0},R_{0}$.

\item In the region $\{R^{\text{\textgreek{e}}_{0}}\le r\le\frac{1}{2}R\}$,
(\ref{eq:Bound_S^3_Terms_Final})--(\ref{eq:BoundDerivativesUpper})
yield: 
\begin{equation}
\mathcal{A}_{f,h}^{(R)}\ge\Big\{ c_{\text{\textgreek{d}}_{0}}\text{\textgreek{e}}_{0}R^{-3\text{\textgreek{e}}_{0}}\cdot r^{-4+3\text{\textgreek{e}}_{0}}s^{3}-C_{\text{\textgreek{d}}_{0}}R^{-2\text{\textgreek{e}}_{0}}r^{-4+2\text{\textgreek{e}}_{0}}s^{2}-C_{\text{\textgreek{d}}_{0}}R^{-\text{\textgreek{e}}_{0}}r^{-4+\text{\textgreek{e}}_{0}}s\Big\} e^{2sw_{R}}.\label{eq:LowerBoundAfhSecondAwayRegion}
\end{equation}

\item In the region $\{\frac{1}{2}R\le r\le R\}$, (\ref{eq:DefinitionWRIntermediate+})--(\ref{eq:UpperBounds})
yield: 
\begin{equation}
\mathcal{A}_{f,h}^{(R)}\ge-C_{\text{\textgreek{d}}_{0}}R^{-4}\Big\{ v_{s}^{\prime}(\frac{r}{R})s^{3}+s^{2}+s\Big\} e^{2sw_{R}},\label{eq:LowerBoundAfhThirdAwayRegion}
\end{equation}
where $C_{\text{\textgreek{d}}_{0}}>0$ depends only on on $\text{\textgreek{d}}_{0},R_{0}$.

\end{enumerate}

\subsubsection{Bound on $\{r\ge R\}$}

In the region $\{R\le r\le\text{\textgreek{d}}_{2}^{-1}R\}$, (\ref{eq:Function_F}),
(\ref{eq:hAlmostR}) and (\ref{eq:BoundBoxh}) yield:
\begin{equation}
\mathcal{A}_{f,h}^{(R)}(r)\ge-CR^{-4}f(R)\label{eq:LowerBoundAfAlmostInfinity}
\end{equation}
for some absolute constant $C>0$, while for $r\ge\text{\textgreek{d}}_{2}^{-1}R$
we have (provided $\text{\textgreek{d}}_{2}\ll1$):
\begin{align}
\mathcal{A}_{f,h}^{(R)} & =-\frac{(d-1)}{2}r^{-1}\big(1+O(r^{-\frac{1}{2}})\big)\partial_{r}^{3}f-\frac{(d-1)(d-3)}{2}r^{-2}\big(1+O(r^{-\frac{1}{2}})\big)\partial_{r}^{2}f+\frac{(d-1)(d-3)}{2}r^{-3}\big(1+O(r^{-\frac{1}{2}})\big)\partial_{r}f=\label{eq:LowerBoundAfhNearInfinity}\\
 & =\frac{1}{2}(d-1)r^{-4}f(R)\big((d-3)(\frac{r+O(r^{\frac{1}{2}})}{R})+\frac{9}{5}+O(r^{-\frac{1}{2}})\big)\ge\nonumber \\
 & \ge f(R)\Big(\frac{1}{2}(d-1)(d-3)R^{-1}r^{-3}\big(1+O(r^{-\frac{1}{2}})\big)+cr^{-4}\Big)\nonumber 
\end{align}
for some absolute constant $c>0$.

\subsection{Estimates for the first order terms}

In this Section, we will establish various bounds for the quantity
\begin{equation}
\begin{split}Re\Big\{2\text{\textgreek{q}}_{\le R}f^{-1}\nabla^{\text{\textgreek{m}}}\nabla^{\text{\textgreek{n}}}f\nabla_{\text{\textgreek{m}}}(f^{\frac{1}{2}}\text{\textgreek{f}})\nabla_{\text{\textgreek{n}}}(f^{\frac{1}{2}}\bar{\text{\textgreek{f}}})-2\text{\textgreek{q}}_{\le R}\text{\textgreek{q}}_{\ge R_{0}} & r^{-1}f^{-1}(\partial_{r}f)\big|\partial_{r}(f^{\frac{1}{2}}\text{\textgreek{f}})\big|^{2}+\\
\hphantom{Re\Big\{}+2(1-\text{\textgreek{q}}_{\le R})\nabla^{\text{\textgreek{m}}}\nabla^{\text{\textgreek{n}}}f\nabla_{\text{\textgreek{m}}}\text{\textgreek{f}}\nabla_{\text{\textgreek{n}}}\bar{\text{\textgreek{f}}}+2\text{\textgreek{q}}_{\ge R_{0}} & \text{\textgreek{q}}_{\le R}r^{-1}(\partial_{r}f)|\partial_{r}\text{\textgreek{f}}|^{2}-2h\nabla^{\text{\textgreek{m}}}\text{\textgreek{f}}\nabla_{\text{\textgreek{m}}}\bar{\text{\textgreek{f}}}
\end{split}
\label{eq:FundamentalCurrentFull-1}
\end{equation}
appearing in the integral in left hand side of (\ref{eq:FundamentalCurrentFull}).
Thus, combined with the bounds of Section \ref{sub:PropertiesOfAfh}
for the zeroth order term $\mathcal{A}_{f,h}^{(R)}|\text{\textgreek{f}}|^{2}$
in left hand side of (\ref{eq:FundamentalCurrentFull}), the results
of this section will provide all the necessary estimates leading to
the proof of Proposition \ref{prop:GeneralCarlemanEstimate}.

Let us denote with $g^{-1}$ the natural extension of the metric (\ref{eq:metric})
on the cotangent bundle $T^{*}\mathcal{M}$ of $\mathcal{M}$. Since
we have identified $\mathcal{M}\backslash\mathcal{H}^{-}$ with $\mathbb{R}\times\text{\textgreek{S}}$
under the flow of $T$, $g^{-1}$ splits naturally in any local coordinate
chart $(t,x^{1},\ldots,x^{d})$ on $\mathbb{R}\times\text{\textgreek{S}}$
as 
\begin{equation}
g^{-1}=g^{00}T\otimes T+\frac{1}{2}g^{0i}(T\otimes\partial_{x^{i}}+\partial_{x^{i}}\otimes T)+(g^{-1})_{\text{\textgreek{S}}}^{ij}\partial_{x^{i}}\otimes\partial_{x^{j}},\label{eq:InverseOfThemetricSplit}
\end{equation}
where $(g^{-1})_{\text{\textgreek{S}}}$ is a symmetric $(2,0)$-tensor
on $\text{\textgreek{S}}$. In view of Assumption \hyperref[Assumption 3]{G3},
the expression (\ref{eq:InverseOfThemetricSplit}) and the fact that
$g^{-1}$ is non-degenerate and has Lorentzian signature imply that
$(g^{-1})_{\text{\textgreek{S}}}$ has Riemannian signature on $\text{\textgreek{S}}\backslash(\mathscr{E}\cup\mathcal{H}^{+})$
and Lorentzian signature on $\text{\textgreek{S}}\cap int(\mathscr{E})$,
while $(g^{-1})_{\text{\textgreek{S}}}$ degenerates on $\text{\textgreek{S}}\cap(\partial\mathscr{E}\cup\mathcal{H}^{+})$.
Using the tensor $(g^{-1})_{\text{\textgreek{S}}}$, we can conveniently
bound for any $\text{\textgreek{f}}\in C^{1}(\mathcal{M}\backslash\mathcal{H}^{-})$
(for some constant $C>0$ depending only on the geometry of $(\mathcal{M},g)$):
\begin{equation}
\nabla^{\text{\textgreek{m}}}\text{\textgreek{f}}\nabla_{\text{\textgreek{m}}}\bar{\text{\textgreek{f}}}\ge(g^{-1})_{\text{\textgreek{S}}}^{ij}\partial_{i}\text{\textgreek{f}}\partial_{j}\bar{\text{\textgreek{f}}}-C|\nabla_{g_{\text{\textgreek{S}}}}\text{\textgreek{f}}|_{g_{\text{\textgreek{S}}}}|T\text{\textgreek{f}}|-C|T\text{\textgreek{f}}|^{2},\label{eq:BoundForgNorm}
\end{equation}
where the indices $i,j$ in the abstract index notation $(g^{-1})_{\text{\textgreek{S}}}^{ij}\partial_{i}\text{\textgreek{f}}\partial_{j}\bar{\text{\textgreek{f}}}$
run over the variables $\{x^{i}\}_{i=1}^{d}$ in any local coordinate
chart on $(\mathcal{M}\backslash\mathcal{H}^{-})$ of the form $(t,x^{1},\ldots,x^{d})$.

\subsubsection{Bound on $\mathscr{E}_{ext}\cup\{r\le\frac{1}{4}r_{0}\}$}

For any $0\le\text{\textgreek{t}}_{1}\le\text{\textgreek{t}}_{2}$,
we can readily bound from above on $\mathcal{R}(\text{\textgreek{t}}_{1},\text{\textgreek{t}}_{2})\cap\big(\mathscr{E}_{ext}\cup\{r\le\frac{1}{4}r_{0}\}\big)$
in view of (\ref{eq:Function_F}) and (\ref{eq:hAlmostR}): 
\begin{equation}
\begin{split}\int_{\mathcal{R}(\text{\textgreek{t}}_{1},\text{\textgreek{t}}_{2})\cap\big(\mathscr{E}_{ext}\cup\{r\le\frac{1}{4}r_{0}\}\big)}Re\Big\{ & 2\text{\textgreek{q}}_{\le R}f^{-1}\nabla^{\text{\textgreek{m}}}\nabla^{\text{\textgreek{n}}}f\nabla_{\text{\textgreek{m}}}(f^{\frac{1}{2}}\text{\textgreek{f}})\nabla_{\text{\textgreek{n}}}(f^{\frac{1}{2}}\bar{\text{\textgreek{f}}})-2\text{\textgreek{q}}_{\le R}\text{\textgreek{q}}_{\ge R_{0}}r^{-1}f^{-1}(\partial_{r}f)\big|\partial_{r}(f^{\frac{1}{2}}\text{\textgreek{f}})\big|^{2}+\\
 & +2(1-\text{\textgreek{q}}_{\le R})\nabla^{\text{\textgreek{m}}}\nabla^{\text{\textgreek{n}}}f\nabla_{\text{\textgreek{m}}}\text{\textgreek{f}}\nabla_{\text{\textgreek{n}}}\bar{\text{\textgreek{f}}}+2\text{\textgreek{q}}_{\ge R_{0}}\text{\textgreek{q}}_{\le R}r^{-1}(\partial_{r}f)|\partial_{r}\text{\textgreek{f}}|^{2}-2h\nabla^{\text{\textgreek{m}}}\text{\textgreek{f}}\nabla_{\text{\textgreek{m}}}\bar{\text{\textgreek{f}}}+\mathcal{A}_{f,h}^{(R)}|\text{\textgreek{f}}|^{2}\Big\}\, dg\le\\
 & \hphantom{++}\le C_{\text{\textgreek{d}}_{0}}\int_{\mathcal{R}(\text{\textgreek{t}}_{1},\text{\textgreek{t}}_{2})\cap\big(\mathscr{E}_{ext}\cup\{r\le\frac{1}{4}r_{0}\}\big)}e^{2sw_{R}}\Big((R^{-6\text{\textgreek{e}}_{0}}s^{2}+1)|\nabla\text{\textgreek{f}}|_{g_{ref}}^{2}+(R^{-12\text{\textgreek{e}}_{0}}s^{4}+1)|\text{\textgreek{f}}|^{2}\Big)\, dg.
\end{split}
\label{eq:BoundCurrentErgoregion}
\end{equation}

\subsubsection{Bound on $\{\frac{1}{4}r_{0}\le r\le\frac{1}{2}R_{0}\}\backslash\mathscr{E}_{ext}$}

In the region $\{\frac{1}{4}r_{0}\le r\le\frac{1}{2}R_{0}\}$, in
view of (\ref{eq:OneSidedBoundSideHessian}) and (\ref{eq:Function_F}),
we can estimate (using a Cauchy--Schwarz inequality): 
\begin{align}
Re\big\{2f^{-1}\nabla^{\text{\textgreek{m}}}\nabla^{\text{\textgreek{n}}}f\nabla_{\text{\textgreek{m}}}(f^{\frac{1}{2}}\text{\textgreek{f}})\nabla_{\text{\textgreek{n}}}(f^{\frac{1}{2}}\bar{\text{\textgreek{f}}})\big\} & =e^{2sw_{R}}\Big\{8s^{2}\big|\nabla^{\text{\textgreek{m}}}w_{R}\nabla_{\text{\textgreek{m}}}\text{\textgreek{f}}\big|^{2}+16s^{3}\nabla^{\text{\textgreek{m}}}w_{R}\nabla_{\text{\textgreek{m}}}w_{R}Re\{\text{\textgreek{f}}\cdot\nabla^{\text{\textgreek{n}}}w_{R}\nabla_{\text{\textgreek{n}}}\bar{\text{\textgreek{f}}}\}+\label{eq:UpperBoundDerivativesOusideErgoregion}\\
 & \hphantom{=e^{2sw_{R}}\Big\{}8s^{4}\big(\nabla^{\text{\textgreek{m}}}w_{R}\nabla_{\text{\textgreek{m}}}w_{R}\big)^{2}|\text{\textgreek{f}}|^{2}+4s\nabla^{\text{\textgreek{m}}}\nabla^{\text{\textgreek{n}}}w_{R}\nabla_{\text{\textgreek{m}}}\text{\textgreek{f}}\nabla_{\text{\textgreek{n}}}\bar{\text{\textgreek{f}}}+\nonumber \\
 & \hphantom{=e^{2sw_{R}}\Big\{}+8s\nabla^{\text{\textgreek{m}}}\nabla^{\text{\textgreek{n}}}w_{R}\nabla_{\text{\textgreek{m}}}w_{R}Re\{\text{\textgreek{f}}\cdot\nabla_{\text{\textgreek{n}}}\bar{\text{\textgreek{f}}}\}+4s^{3}\nabla^{\text{\textgreek{m}}}\nabla^{\text{\textgreek{n}}}w_{R}\nabla_{\text{\textgreek{m}}}w_{R}\nabla_{\text{\textgreek{n}}}w_{R}|\text{\textgreek{f}}|^{2}\Big\}\ge\nonumber \\
 & \ge e^{2sw_{R}}\Big\{8s^{2}\big|\nabla^{\text{\textgreek{m}}}w_{R}\nabla_{\text{\textgreek{m}}}\text{\textgreek{f}}+s^{2}\nabla^{\text{\textgreek{m}}}w_{R}\nabla_{\text{\textgreek{m}}}w_{R}\cdot\text{\textgreek{f}}\big|^{2}-Cs\text{\textgreek{d}}_{0}\frac{\nabla^{\text{\textgreek{m}}}\nabla^{\text{\textgreek{n}}}w_{R}\nabla_{\text{\textgreek{m}}}w_{R}\nabla_{\text{\textgreek{n}}}w_{R}}{g_{ref}(dw_{R},dw_{R})}\big|\nabla\text{\textgreek{f}}\big|_{g_{ref}}^{2}+\nonumber \\
 & \hphantom{\ge e^{2sw_{R}}\Big\{}+3s^{3}\nabla^{\text{\textgreek{m}}}\nabla^{\text{\textgreek{n}}}w_{R}\nabla_{\text{\textgreek{m}}}w_{R}\nabla_{\text{\textgreek{n}}}w_{R}|\text{\textgreek{f}}|^{2}\Big\}\nonumber 
\end{align}
for some absolute constant $C>0$. Thus, (\ref{eq:UpperBoundDerivativesOusideErgoregion}),
(\ref{eq:hAlmostR}) and (\ref{eq:BoundForgNorm}) imply that we can
bound from below on the region $\mathcal{R}(\text{\textgreek{t}}_{1},\text{\textgreek{t}}_{2})\cap\{\frac{1}{4}r_{0}\le r\le\frac{1}{2}R_{0}\}\backslash\mathscr{E}_{ext}$:

\begin{align}
\int_{\mathcal{R}(\text{\textgreek{t}}_{1},\text{\textgreek{t}}_{2})\cap\{\frac{1}{4}r_{0}\le r\le\frac{1}{2}R_{0}\}\backslash\mathscr{E}_{ext}} & Re\big\{2f^{-1}\nabla^{\text{\textgreek{m}}}\nabla^{\text{\textgreek{n}}}f\nabla_{\text{\textgreek{m}}}(f^{\frac{1}{2}}\text{\textgreek{f}})\nabla_{\text{\textgreek{n}}}(f^{\frac{1}{2}}\bar{\text{\textgreek{f}}})-2h\nabla^{\text{\textgreek{m}}}\text{\textgreek{f}}\nabla_{\text{\textgreek{m}}}\bar{\text{\textgreek{f}}}\big\}\, dg\ge\label{eq:LowerBoundCurrentOutsideErgoregion}\\
\ge & \int_{\mathcal{R}(\text{\textgreek{t}}_{1},\text{\textgreek{t}}_{2})\cap\{\frac{1}{4}r_{0}\le r\le\frac{1}{2}R_{0}\}\backslash\mathscr{E}_{ext}}e^{2sw_{R}}\Bigg\{\frac{\nabla^{\text{\textgreek{m}}}\nabla^{\text{\textgreek{n}}}w_{R}\nabla_{\text{\textgreek{m}}}w_{R}\nabla_{\text{\textgreek{n}}}w_{R}}{g_{ref}(dw_{R},dw_{R})}\Big(cs\text{\textgreek{d}}_{1}(g^{-1})_{\text{\textgreek{S}}}^{ij}\partial_{i}\text{\textgreek{f}}\partial_{j}\bar{\text{\textgreek{f}}}+cs^{3}g_{ref}(dw_{R},dw_{R})|\text{\textgreek{f}}|^{2}\Big)-\nonumber \\
 & \hphantom{\int_{\mathcal{R}(\text{\textgreek{t}}_{1},\text{\textgreek{t}}_{2})\cap\{\frac{1}{4}r_{0}\le r\le\frac{1}{2}R_{0}\}\backslash\mathscr{E}_{ext}}e^{2sw_{R}}\Bigg\{}-Cs\big(\sum_{j=1}^{2}|\nabla^{j}w_{R}|_{g_{ref}}\big)|\nabla_{g_{\text{\textgreek{S}}}}\text{\textgreek{f}}|_{g_{\text{\textgreek{S}}}}|T\text{\textgreek{f}}|-Cs\big(\sum_{j=1}^{2}|\nabla^{j}w_{R}|_{g_{ref}}\big)|T\text{\textgreek{f}}|^{2}\Bigg\}\, dg.\nonumber 
\end{align}
 for some absolute constants $c,C>0$.

\subsubsection{Bound on $\{\frac{1}{2}R_{0}\le r\le R_{0}\}$}

In the region $\{r\ge\frac{1}{2}R_{0}\}$, the functions $f,h$ depend
only on $r$. In particular, we compute in the $(t,r,\text{\textgreek{sv}})$
coordinate system in each connected component of the region $\{r\ge\frac{1}{2}R_{0}\}$
for any function $\text{\textgreek{y}}\in C^{1}(\mathcal{M})$: 
\begin{equation}
\nabla^{\text{\textgreek{m}}}\nabla^{\text{\textgreek{n}}}f\nabla_{\text{\textgreek{m}}}\text{\textgreek{y}}\nabla_{\text{\textgreek{n}}}\bar{\text{\textgreek{y}}}=\big((1+O(r^{-1}))\partial_{r}^{2}f+O(r^{-2})\partial_{r}f\big)|\partial_{r}\text{\textgreek{y}}|^{2}+r^{-3}(1+O(r^{-1}))\partial_{r}f|\partial_{\text{\textgreek{sv}}}\text{\textgreek{y}}|^{2}+O(r^{-2})\partial_{r}f|T\text{\textgreek{y}}|^{2}.\label{eq:ExpressionHessianAsFlat}
\end{equation}
Therefore, (\ref{eq:Function_F}) and (\ref{eq:hAlmostR}) yield the
following lower bound: 
\begin{equation}
\begin{split}\int_{\mathcal{R}(\text{\textgreek{t}}_{1},\text{\textgreek{t}}_{2})\cap\{\frac{1}{2}R_{0}\le r\le R_{0}\}}Re\Big\{2\text{\textgreek{q}}_{\le R} & f^{-1}\nabla^{\text{\textgreek{m}}}\nabla^{\text{\textgreek{n}}}f\nabla_{\text{\textgreek{m}}}(f^{\frac{1}{2}}\text{\textgreek{f}})\nabla_{\text{\textgreek{n}}}(f^{\frac{1}{2}}\bar{\text{\textgreek{f}}})-2\text{\textgreek{q}}_{\le R}\text{\textgreek{q}}_{\ge R_{0}}r^{-1}f^{-1}(\partial_{r}f)\big|\partial_{r}(f^{\frac{1}{2}}\text{\textgreek{f}})\big|^{2}+\\
+2( & 1-\text{\textgreek{q}}_{\le R})\nabla^{\text{\textgreek{m}}}\nabla^{\text{\textgreek{n}}}f\nabla_{\text{\textgreek{m}}}\text{\textgreek{f}}\nabla_{\text{\textgreek{n}}}\bar{\text{\textgreek{f}}}+2\text{\textgreek{q}}_{\ge R_{0}}\text{\textgreek{q}}_{\le R}r^{-1}(\partial_{r}f)|\partial_{r}\text{\textgreek{f}}|^{2}-2h\nabla^{\text{\textgreek{m}}}\text{\textgreek{f}}\nabla_{\text{\textgreek{m}}}\bar{\text{\textgreek{f}}}\Big\}\, dg\ge\\
\ge & c\int_{\mathcal{R}(\text{\textgreek{t}}_{1},\text{\textgreek{t}}_{2})\cap\{\frac{1}{2}R_{0}\le r\le R_{0}\}}e^{2sw_{R}}\Bigg\{\big(s^{2}(\partial_{r}w_{R})^{2}+O(1)s\sum_{j=1}^{2}|\nabla^{j}w_{R}|_{g_{ref}}\big)e^{-2sw_{R}}\big|\partial_{r}(e^{sw_{R}}\text{\textgreek{f}})\big|^{2}+\\
 & \hphantom{C\int_{\mathcal{R}(\text{\textgreek{t}}_{1},}}+s(\partial_{r}w_{R})\big(\text{\textgreek{q}}_{\ge R_{0}}(r^{-\frac{3}{2}}+O(r^{-2}))+O(\text{\textgreek{d}}_{1})|\nabla w_{R}|_{g_{ref}}^{-1}|\nabla^{2}w_{R}|_{g_{ref}}\big)|\partial_{r}\text{\textgreek{f}}|^{2}+\\
 & \hphantom{C\int_{\mathcal{R}(\text{\textgreek{t}}_{1},}}+s(\partial_{r}w_{R})\big(r^{-3}+O(r^{-4})+\text{\textgreek{q}}_{\ge R_{0}}(r^{-\frac{7}{2}}-r^{-3})+O(\text{\textgreek{d}}_{1})|\nabla w_{R}|_{g_{ref}}^{-1}|\nabla^{2}w_{R}|_{g_{ref}}\big)|\partial_{\text{\textgreek{sv}}}\text{\textgreek{f}}|^{2}+\\
 & \hphantom{C\int_{\mathcal{R}(\text{\textgreek{t}}_{1},}}+s(\partial_{r}w_{R})\big(\text{\textgreek{q}}_{\ge R_{0}}(r^{-1}+O(r^{-\frac{3}{2}}))+O(\text{\textgreek{d}}_{1})|\nabla w_{R}|_{g_{ref}}^{-1}|\nabla^{2}w_{R}|_{g_{ref}}\big)|T\text{\textgreek{f}}|^{2}\Bigg\}\, dg
\end{split}
\label{eq:LowerboundCurrentAlmostAsFlat}
\end{equation}
 for some $c>0$ depending on the geometry of $(\mathcal{M},g)$.

\subsubsection{Bound on $\{R_{0}\le r\le R\}$}

In the region $\{R_{0}\le r\le R\}$, in view of (\ref{eq:Function_F}),
(\ref{eq:hAlmostR}) and (\ref{eq:ExpressionHessianAsFlat}), we can
readily estimate from below: 
\begin{equation}
\begin{split}\int_{\mathcal{R}(\text{\textgreek{t}}_{1},\text{\textgreek{t}}_{2})\cap\{R_{0}\le r\le R\}}Re\Big\{2\text{\textgreek{q}}_{\le R} & f^{-1}\nabla^{\text{\textgreek{m}}}\nabla^{\text{\textgreek{n}}}f\nabla_{\text{\textgreek{m}}}(f^{\frac{1}{2}}\text{\textgreek{f}})\nabla_{\text{\textgreek{n}}}(f^{\frac{1}{2}}\bar{\text{\textgreek{f}}})-2\text{\textgreek{q}}_{\le R}\text{\textgreek{q}}_{\ge R_{0}}r^{-1}f^{-1}(\partial_{r}f)\big|\partial_{r}(f^{\frac{1}{2}}\text{\textgreek{f}})\big|^{2}+\\
+2(1- & \text{\textgreek{q}}_{\le R})\nabla^{\text{\textgreek{m}}}\nabla^{\text{\textgreek{n}}}f\nabla_{\text{\textgreek{m}}}\text{\textgreek{f}}\nabla_{\text{\textgreek{n}}}\bar{\text{\textgreek{f}}}+2\text{\textgreek{q}}_{\ge R_{0}}\text{\textgreek{q}}_{\le R}r^{-1}(\partial_{r}f)|\partial_{r}\text{\textgreek{f}}|^{2}-2h\nabla^{\text{\textgreek{m}}}\text{\textgreek{f}}\nabla_{\text{\textgreek{m}}}\bar{\text{\textgreek{f}}}\Big\}\, dg\ge\\
\ge c & \int_{\mathcal{R}(\text{\textgreek{t}}_{1},\text{\textgreek{t}}_{2})\cap\{R_{0}\le r\le R\}}e^{2sw_{R}}\Big\{\big(s^{2}(1+O(r^{-1}))(\partial_{r}w_{R})^{2}+s(\partial_{r}^{2}w_{R}+O(r^{-2})\partial_{r}w_{R})\big)e^{-2sw_{R}}\big|\partial_{r}(e^{sw_{R}}\text{\textgreek{f}})\big|^{2}+\\
 & \hphantom{c\int_{\mathcal{R}(\text{\textgreek{t}}_{1},\text{\textgreek{t}}_{2})\cap\{R_{0}\le r\le R\}}e^{2sw_{R}}\Big\{}+s(\partial_{r}w_{R})\big(r^{-\frac{3}{2}}+O(r^{-2})\big)|\partial_{r}\text{\textgreek{f}}|^{2}+s(\partial_{r}w_{R})\big(r^{-\frac{7}{2}}+O(r^{-4})\big)|\partial_{\text{\textgreek{sv}}}\text{\textgreek{f}}|^{2}+\\
 & \hphantom{c\int_{\mathcal{R}(\text{\textgreek{t}}_{1},\text{\textgreek{t}}_{2})\cap\{R_{0}\le r\le R\}}e^{2sw_{R}}\Big\{}+cs(\partial_{r}w_{R})\big(r^{-1}+O(r^{-\frac{3}{2}})\big)|T\text{\textgreek{f}}|^{2}\Big\}\, dg
\end{split}
\label{eq:LowerboundCurrentFirstAway}
\end{equation}
for some $c>0$ depending on the geometry of $(\mathcal{M},g)$.

\subsubsection{Bound on $\{r\ge R\}$}

In the region $\{r\ge R\}$, we can estimate in view of in view of
(\ref{eq:ExpressionHessianAsFlat}): 
\begin{equation}
\begin{split}\int_{\mathcal{R}(\text{\textgreek{t}}_{1},\text{\textgreek{t}}_{2})\cap\{r\ge R\}}Re\Big\{2\text{\textgreek{q}}_{\le R} & f^{-1}\nabla^{\text{\textgreek{m}}}\nabla^{\text{\textgreek{n}}}f\nabla_{\text{\textgreek{m}}}(f^{\frac{1}{2}}\text{\textgreek{f}})\nabla_{\text{\textgreek{n}}}(f^{\frac{1}{2}}\bar{\text{\textgreek{f}}})-2\text{\textgreek{q}}_{\le R}\text{\textgreek{q}}_{\ge R_{0}}r^{-1}f^{-1}(\partial_{r}f)\big|\partial_{r}(f^{\frac{1}{2}}\text{\textgreek{f}})\big|^{2}+\\
+2( & 1-\text{\textgreek{q}}_{\le R})\nabla^{\text{\textgreek{m}}}\nabla^{\text{\textgreek{n}}}f\nabla_{\text{\textgreek{m}}}\text{\textgreek{f}}\nabla_{\text{\textgreek{n}}}\bar{\text{\textgreek{f}}}+2\text{\textgreek{q}}_{\ge R_{0}}\text{\textgreek{q}}_{\le R}r^{-1}(\partial_{r}f)|\partial_{r}\text{\textgreek{f}}|^{2}-2h\nabla^{\text{\textgreek{m}}}\text{\textgreek{f}}\nabla_{\text{\textgreek{m}}}\bar{\text{\textgreek{f}}}\Big\}\, dg\ge\\
\ge & \int_{\mathcal{R}(\text{\textgreek{t}}_{1},\text{\textgreek{t}}_{2})\cap\{r\ge R\}}\Big\{\text{\textgreek{q}}_{\le R}\big(2(1+O(r^{-1}))\partial_{r}^{2}f-2(r^{-1}-r^{-\frac{3}{2}}+O(r^{-2}))\partial_{r}f\big)f^{-1}\big|\partial_{r}(f^{\frac{1}{2}}\text{\textgreek{f}})\big|^{2}+\\
 & \hphantom{\int_{\mathcal{R}(\text{\textgreek{t}}_{1},\text{\textgreek{t}}_{2})\cap\{r\ge R\}}\Big\{}+\Big(2(1-\text{\textgreek{q}}_{\le R})\big((1+O(r^{-1}))\partial_{r}^{2}f+O(r^{-2})\partial_{r}f\big)-2(1+O(r^{-1}))h\Big)|\partial_{r}\text{\textgreek{f}}|^{2}+\\
 & \hphantom{\int_{\mathcal{R}(\text{\textgreek{t}}_{1},\text{\textgreek{t}}_{2})\cap\{r\ge R\}}\Big\{}+\Big(2r^{-3}(1+O(r^{-1}))\partial_{r}f-2(r^{-2}+O(r^{-3}))h\Big)|\partial_{\text{\textgreek{sv}}}\text{\textgreek{f}}|^{2}+\big(2h+O(r^{-2})\partial_{r}f\big)|T\text{\textgreek{f}}|^{2}\Big\}\, dg.
\end{split}
\label{eq:LowerboundCurrentNearInfinity}
\end{equation}

\subsection{\label{sub:ProofOfCarleman}Proof of Proposition \ref{prop:GeneralCarlemanEstimate}}

\begin{enumerate}

\item In view of (\ref{eq:IntermediateRegionAfh}), (\ref{eq:TopCoefficientInIntermediateRegion}),
(\ref{eq:LowerBoundCurrentOutsideErgoregion}) and Lemma \ref{lem:PositivityHessianW},
we can bound: 
\begin{equation}
\begin{split} & \int_{\mathcal{R}(\text{\textgreek{t}}_{1},\text{\textgreek{t}}_{2})\cap\{\frac{1}{4}r_{0}\le r\le\frac{1}{2}R_{0}\}\backslash\mathscr{E}_{ext}}Re\Big\{2\text{\textgreek{q}}_{\le R}f^{-1}\nabla^{\text{\textgreek{m}}}\nabla^{\text{\textgreek{n}}}f\nabla_{\text{\textgreek{m}}}(f^{\frac{1}{2}}\text{\textgreek{f}})\nabla_{\text{\textgreek{n}}}(f^{\frac{1}{2}}\bar{\text{\textgreek{f}}})-2\text{\textgreek{q}}_{\le R}\text{\textgreek{q}}_{\ge R_{0}}r^{-1}f^{-1}(\partial_{r}f)\big|\partial_{r}(f^{\frac{1}{2}}\text{\textgreek{f}})\big|^{2}+\\
 & \hphantom{\int_{\mathcal{R}(\text{\textgreek{t}}_{1},\text{\textgreek{t}}_{2})\cap\{\frac{1}{4}r_{0}\le r}}+2(1-\text{\textgreek{q}}_{\le R})\nabla^{\text{\textgreek{m}}}\nabla^{\text{\textgreek{n}}}f\nabla_{\text{\textgreek{m}}}\text{\textgreek{f}}\nabla_{\text{\textgreek{n}}}\bar{\text{\textgreek{f}}}+2\text{\textgreek{q}}_{\ge R_{0}}\text{\textgreek{q}}_{\le R}r^{-1}(\partial_{r}f)|\partial_{r}\text{\textgreek{f}}|^{2}-2h\nabla^{\text{\textgreek{m}}}\text{\textgreek{f}}\nabla_{\text{\textgreek{m}}}\bar{\text{\textgreek{f}}}+\mathcal{A}_{f,h}^{(R)}|\text{\textgreek{f}}|^{2}\Big\}\, dg\ge\\
 & \hphantom{\int_{\mathcal{R}(\text{\textgreek{t}}_{1},\text{\textgreek{t}}_{2})\cap\{\frac{1}{4}r_{0}\le r}+2}\ge\int_{\mathcal{R}(\text{\textgreek{t}}_{1},\text{\textgreek{t}}_{2})\cap\{\frac{1}{4}r_{0}\le r\le\frac{1}{2}R_{0}\}\backslash\big(\mathscr{E}_{ext}\cup\mathcal{B}_{crit}(\text{\textgreek{d}}_{0})\big)}e^{2sw_{R}}\Bigg\{ c_{\text{\textgreek{d}}_{0}\text{\textgreek{d}}_{1}}sR^{-3\text{\textgreek{e}}_{0}}(g^{-1})_{\text{\textgreek{S}}}^{ij}\partial_{i}\text{\textgreek{f}}\partial_{j}\bar{\text{\textgreek{f}}}-CsR^{-3\text{\textgreek{e}}_{0}}|\nabla_{g_{\text{\textgreek{S}}}}\text{\textgreek{f}}|_{g_{\text{\textgreek{S}}}}|T\text{\textgreek{f}}|-\\
 & \hphantom{\hphantom{\int_{\mathcal{R}(\text{\textgreek{t}}_{1},\text{\textgreek{t}}_{2})\cap\{\frac{1}{4}r_{0}\le r}}+2\ge\int_{\mathcal{R}(\text{\textgreek{t}}_{1},\text{\textgreek{t}}_{2})\cap\{\frac{1}{4}r_{0}\le r\le\frac{1}{2}R_{0}\}\backslash}}-CsR^{-3\text{\textgreek{e}}_{0}}|T\text{\textgreek{f}}|^{2}+\big(c_{\text{\textgreek{d}}_{0}}s^{3}R^{-9\text{\textgreek{e}}_{0}}-C_{\text{\textgreek{d}}_{0}}s^{2}R^{-6\text{\textgreek{e}}_{0}}-C_{\text{\textgreek{d}}_{0}}sR^{-3\text{\textgreek{e}}_{0}}\big)|\text{\textgreek{f}}|^{2}\Bigg\}\, dg-\\
 & \hphantom{\int_{\mathcal{R}(\text{\textgreek{t}}_{1},\text{\textgreek{t}}_{2})\cap\{\frac{1}{4}r_{0}\le r}+2\ge}-C_{\text{\textgreek{d}}_{0}}\int_{\mathcal{R}(\text{\textgreek{t}}_{1},\text{\textgreek{t}}_{2})\cap\mathcal{B}_{crit}(\text{\textgreek{d}}_{0})}e^{2sw_{R}}\Big\{ sR^{-3\text{\textgreek{e}}_{0}}|\nabla\text{\textgreek{f}}|_{g_{ref}}^{2}+s^{3}R^{-9\text{\textgreek{e}}_{0}}|\text{\textgreek{f}}|^{2}\Big\}\, dg.
\end{split}
\label{eq:TotalBoundOutsideErgoregion}
\end{equation}
Repeating the same procedure for $\tilde{f}$ in place of $f$, $\tilde{h}$
in place of $h$ and $\tilde{w}_{R}$ in place of $w_{R}$ (see Lemma
\ref{lem:ConstructionWR}), from (the analogues of) (\ref{eq:IntermediateRegionAfh}),
(\ref{eq:TopCoefficientInIntermediateRegion}), (\ref{eq:LowerBoundCurrentOutsideErgoregion})
for $w_{R}$ in place of $\tilde{w}_{R}$ we obtain: 
\begin{equation}
\begin{split} & \int_{\mathcal{R}(\text{\textgreek{t}}_{1},\text{\textgreek{t}}_{2})\cap\{\frac{1}{4}r_{0}\le r\le\frac{1}{2}R_{0}\}\backslash\mathscr{E}_{ext}}Re\Big\{2\text{\textgreek{q}}_{\le R}\tilde{f}^{-1}\nabla^{\text{\textgreek{m}}}\nabla^{\text{\textgreek{n}}}f\nabla_{\text{\textgreek{m}}}(\tilde{f}^{\frac{1}{2}}\text{\textgreek{f}})\nabla_{\text{\textgreek{n}}}(\tilde{f}^{\frac{1}{2}}\bar{\text{\textgreek{f}}})-2\text{\textgreek{q}}_{\le R}\text{\textgreek{q}}_{\ge R_{0}}r^{-1}\tilde{f}^{-1}(\partial_{r}\tilde{f})\big|\partial_{r}(\tilde{f}^{\frac{1}{2}}\text{\textgreek{f}})\big|^{2}+\\
 & \hphantom{\int_{\mathcal{R}(\text{\textgreek{t}}_{1},\text{\textgreek{t}}_{2})\cap\{\frac{1}{4}r_{0}\le r}}+2(1-\text{\textgreek{q}}_{\le R})\nabla^{\text{\textgreek{m}}}\nabla^{\text{\textgreek{n}}}\tilde{f}\nabla_{\text{\textgreek{m}}}\text{\textgreek{f}}\nabla_{\text{\textgreek{n}}}\bar{\text{\textgreek{f}}}+2\text{\textgreek{q}}_{\ge R_{0}}\text{\textgreek{q}}_{\le R}r^{-1}(\partial_{r}\tilde{f})|\partial_{r}\text{\textgreek{f}}|^{2}-2\tilde{h}\nabla^{\text{\textgreek{m}}}\text{\textgreek{f}}\nabla_{\text{\textgreek{m}}}\bar{\text{\textgreek{f}}}+\mathcal{A}_{\tilde{f},\tilde{h}}^{(R)}|\text{\textgreek{f}}|^{2}\Big\}\, dg\ge\\
 & \hphantom{\int_{\mathcal{R}(\text{\textgreek{t}}_{1},\text{\textgreek{t}}_{2})\cap\{\frac{1}{4}r_{0}\le r}+2}\ge\int_{\mathcal{R}(\text{\textgreek{t}}_{1},\text{\textgreek{t}}_{2})\cap\{\frac{1}{4}r_{0}\le r\le\frac{1}{2}R_{0}\}\backslash\big(\mathscr{E}_{ext}\cup\tilde{\mathcal{B}}_{crit}(\text{\textgreek{d}}_{0})\big)}e^{2s\tilde{w}_{R}}\Bigg\{ c_{\text{\textgreek{d}}_{0}\text{\textgreek{d}}_{1}}sR^{-3\text{\textgreek{e}}_{0}}(g^{-1})_{\text{\textgreek{S}}}^{ij}\partial_{i}\text{\textgreek{f}}\partial_{j}\bar{\text{\textgreek{f}}}-CsR^{-3\text{\textgreek{e}}_{0}}|\nabla_{g_{\text{\textgreek{S}}}}\text{\textgreek{f}}|_{g_{\text{\textgreek{S}}}}|T\text{\textgreek{f}}|-\\
 & \hphantom{\hphantom{\int_{\mathcal{R}(\text{\textgreek{t}}_{1},\text{\textgreek{t}}_{2})\cap\{\frac{1}{4}r_{0}\le r}}+2\ge\int_{\mathcal{R}(\text{\textgreek{t}}_{1},\text{\textgreek{t}}_{2})\cap\{\frac{1}{4}r_{0}\le r\le\frac{1}{2}R_{0}\}\backslash}}-CsR^{-3\text{\textgreek{e}}_{0}}|T\text{\textgreek{f}}|^{2}+\big(c_{\text{\textgreek{d}}_{0}}s^{3}R^{-9\text{\textgreek{e}}_{0}}-C_{\text{\textgreek{d}}_{0}}s^{2}R^{-6\text{\textgreek{e}}_{0}}-C_{\text{\textgreek{d}}_{0}}sR^{-3\text{\textgreek{e}}_{0}}\big)|\text{\textgreek{f}}|^{2}\Bigg\}\, dg-\\
 & \hphantom{\int_{\mathcal{R}(\text{\textgreek{t}}_{1},\text{\textgreek{t}}_{2})\cap\{\frac{1}{4}r_{0}\le r}+2\ge}-C_{\text{\textgreek{d}}_{0}}\int_{\mathcal{R}(\text{\textgreek{t}}_{1},\text{\textgreek{t}}_{2})\cap\tilde{\mathcal{B}}_{crit}(\text{\textgreek{d}}_{0})}e^{2s\tilde{w}_{R}}\Big\{ sR^{-3\text{\textgreek{e}}_{0}}|\nabla\text{\textgreek{f}}|_{g_{ref}}^{2}+s^{3}R^{-9\text{\textgreek{e}}_{0}}|\text{\textgreek{f}}|^{2}\Big\}\, dg.
\end{split}
\label{eq:TotalBoundOutsideErgoregionDistorted}
\end{equation}
Adding (\ref{eq:TotalBoundOutsideErgoregion}) and (\ref{eq:TotalBoundOutsideErgoregionDistorted})
and using (\ref{eq:InequalitiesForPairOfDeformedFunctions}) and (\ref{eq:InequalitiesForPair2}),
we obtain provided $s$ is large in terms of $\text{\textgreek{d}}_{0}$:
\begin{equation}
\begin{split} & \int_{\mathcal{R}(\text{\textgreek{t}}_{1},\text{\textgreek{t}}_{2})\cap\{\frac{1}{4}r_{0}\le r\le\frac{1}{2}R_{0}\}\backslash\mathscr{E}_{ext}}Re\Big\{2\text{\textgreek{q}}_{\le R}f^{-1}\nabla^{\text{\textgreek{m}}}\nabla^{\text{\textgreek{n}}}f\nabla_{\text{\textgreek{m}}}(f^{\frac{1}{2}}\text{\textgreek{f}})\nabla_{\text{\textgreek{n}}}(f^{\frac{1}{2}}\bar{\text{\textgreek{f}}})-2\text{\textgreek{q}}_{\le R}\text{\textgreek{q}}_{\ge R_{0}}r^{-1}f^{-1}(\partial_{r}f)\big|\partial_{r}(f^{\frac{1}{2}}\text{\textgreek{f}})\big|^{2}+\\
 & \hphantom{\int_{\mathcal{R}(\text{\textgreek{t}}_{1},\text{\textgreek{t}}_{2})\cap\{\frac{1}{4}r_{0}\le r}}+2(1-\text{\textgreek{q}}_{\le R})\nabla^{\text{\textgreek{m}}}\nabla^{\text{\textgreek{n}}}f\nabla_{\text{\textgreek{m}}}\text{\textgreek{f}}\nabla_{\text{\textgreek{n}}}\bar{\text{\textgreek{f}}}+2\text{\textgreek{q}}_{\ge R_{0}}\text{\textgreek{q}}_{\le R}r^{-1}(\partial_{r}f)|\partial_{r}\text{\textgreek{f}}|^{2}-2h\nabla^{\text{\textgreek{m}}}\text{\textgreek{f}}\nabla_{\text{\textgreek{m}}}\bar{\text{\textgreek{f}}}+\mathcal{A}_{f,h}^{(R)}|\text{\textgreek{f}}|^{2}\Big\}\, dg\ge\\
 & \int_{\mathcal{R}(\text{\textgreek{t}}_{1},\text{\textgreek{t}}_{2})\cap\{\frac{1}{4}r_{0}\le r\le\frac{1}{2}R_{0}\}\backslash\mathscr{E}_{ext}}Re\Big\{2\text{\textgreek{q}}_{\le R}\tilde{f}^{-1}\nabla^{\text{\textgreek{m}}}\nabla^{\text{\textgreek{n}}}f\nabla_{\text{\textgreek{m}}}(\tilde{f}^{\frac{1}{2}}\text{\textgreek{f}})\nabla_{\text{\textgreek{n}}}(\tilde{f}^{\frac{1}{2}}\bar{\text{\textgreek{f}}})-2\text{\textgreek{q}}_{\le R}\text{\textgreek{q}}_{\ge R_{0}}r^{-1}\tilde{f}^{-1}(\partial_{r}\tilde{f})\big|\partial_{r}(\tilde{f}^{\frac{1}{2}}\text{\textgreek{f}})\big|^{2}+\\
 & \hphantom{\int_{\mathcal{R}(\text{\textgreek{t}}_{1},\text{\textgreek{t}}_{2})\cap\{\frac{1}{4}r_{0}\le r}}+2(1-\text{\textgreek{q}}_{\le R})\nabla^{\text{\textgreek{m}}}\nabla^{\text{\textgreek{n}}}\tilde{f}\nabla_{\text{\textgreek{m}}}\text{\textgreek{f}}\nabla_{\text{\textgreek{n}}}\bar{\text{\textgreek{f}}}+2\text{\textgreek{q}}_{\ge R_{0}}\text{\textgreek{q}}_{\le R}r^{-1}(\partial_{r}\tilde{f})|\partial_{r}\text{\textgreek{f}}|^{2}-2\tilde{h}\nabla^{\text{\textgreek{m}}}\text{\textgreek{f}}\nabla_{\text{\textgreek{m}}}\bar{\text{\textgreek{f}}}+\mathcal{A}_{\tilde{f},\tilde{h}}^{(R)}|\text{\textgreek{f}}|^{2}\Big\}\, dg\ge\\
 & \hphantom{\int_{\mathcal{R}(\text{\textgreek{t}}_{1},\text{\textgreek{t}}_{2})\cap\{\frac{1}{4}r_{0}\le r}+2}\ge\int_{\mathcal{R}(\text{\textgreek{t}}_{1},\text{\textgreek{t}}_{2})\cap\{\frac{1}{4}r_{0}\le r\le\frac{1}{2}R_{0}\}\backslash\mathscr{E}_{ext}}(e^{2sw_{R}}+e^{2s\tilde{w}_{R}})\Bigg\{ c_{\text{\textgreek{d}}_{0}\text{\textgreek{d}}_{1}}sR^{-3\text{\textgreek{e}}_{0}}(g^{-1})_{\text{\textgreek{S}}}^{ij}\partial_{i}\text{\textgreek{f}}\partial_{j}\bar{\text{\textgreek{f}}}-CsR^{-3\text{\textgreek{e}}_{0}}|\nabla_{g_{\text{\textgreek{S}}}}\text{\textgreek{f}}|_{g_{\text{\textgreek{S}}}}|T\text{\textgreek{f}}|-\\
 & \hphantom{\hphantom{\int_{\mathcal{R}(\text{\textgreek{t}}_{1},\text{\textgreek{t}}_{2})\cap\{\frac{1}{4}r_{0}\le r}}+2\ge\int_{\mathcal{R}(\text{\textgreek{t}}_{1},\text{\textgreek{t}}_{2})\cap\{\frac{1}{4}r_{0}\le r\le\frac{1}{2}R_{0}\}\backslash\mathscr{E}_{ext}}}-CsR^{-3\text{\textgreek{e}}_{0}}|T\text{\textgreek{f}}|^{2}+\big(c_{\text{\textgreek{d}}_{0}}s^{3}R^{-9\text{\textgreek{e}}_{0}}-C_{\text{\textgreek{d}}_{0}}s^{2}R^{-6\text{\textgreek{e}}_{0}}-C_{\text{\textgreek{d}}_{0}}sR^{-3\text{\textgreek{e}}_{0}}\big)|\text{\textgreek{f}}|^{2}\Bigg\}\, dg.
\end{split}
\label{eq:TotalBoundOutsideErgoregionSum}
\end{equation}

\item In view of (\ref{eq:IntermediateRegionAfh}), (\ref{eq:TopCoefficientInIntermediateRegion}),
(\ref{eq:LowerboundCurrentAlmostAsFlat}) and Lemma \ref{lem:PositivityHessianW},
we can estimate 
\begin{equation}
\begin{split}\int_{\mathcal{R}(\text{\textgreek{t}}_{1},\text{\textgreek{t}}_{2})\cap\{\frac{1}{2}R_{0}\le r\le R_{0}\}}Re\Big\{2\text{\textgreek{q}}_{\le R} & f^{-1}\nabla^{\text{\textgreek{m}}}\nabla^{\text{\textgreek{n}}}f\nabla_{\text{\textgreek{m}}}(f^{\frac{1}{2}}\text{\textgreek{f}})\nabla_{\text{\textgreek{n}}}(f^{\frac{1}{2}}\bar{\text{\textgreek{f}}})-2\text{\textgreek{q}}_{\le R}\text{\textgreek{q}}_{\ge R_{0}}r^{-1}f^{-1}(\partial_{r}f)\big|\partial_{r}(f^{\frac{1}{2}}\text{\textgreek{f}})\big|^{2}+\\
+2( & 1-\text{\textgreek{q}}_{\le R})\nabla^{\text{\textgreek{m}}}\nabla^{\text{\textgreek{n}}}f\nabla_{\text{\textgreek{m}}}\text{\textgreek{f}}\nabla_{\text{\textgreek{n}}}\bar{\text{\textgreek{f}}}+2\text{\textgreek{q}}_{\ge R_{0}}\text{\textgreek{q}}_{\le R}r^{-1}(\partial_{r}f)|\partial_{r}\text{\textgreek{f}}|^{2}-2h\nabla^{\text{\textgreek{m}}}\text{\textgreek{f}}\nabla_{\text{\textgreek{m}}}\bar{\text{\textgreek{f}}}+\mathcal{A}_{f,h}^{(R)}|\text{\textgreek{f}}|^{2}\Big\}\, dg\ge\\
\ge & \int_{\mathcal{R}(\text{\textgreek{t}}_{1},\text{\textgreek{t}}_{2})\cap\{\frac{1}{2}R_{0}\le r\le R_{0}\}}e^{2sw_{R}}\Big\{\big(c_{\text{\textgreek{d}}_{0}}s^{2}R^{-6\text{\textgreek{e}}_{0}}-C_{\text{\textgreek{d}}_{0}}sR^{-3\text{\textgreek{e}}_{0}}\big)e^{-2sw_{R}}\big|\partial_{r}(e^{sw_{R}}\text{\textgreek{f}})\big|^{2}+\\
 & \hphantom{C\int_{\mathcal{R}(\text{\textgreek{t}}_{1},\text{\textgreek{t}}_{2})\cap\{\frac{1}{2}R_{0}\le r\le R_{0}\}}}+sR^{-3\text{\textgreek{e}}_{0}}\big(c_{\text{\textgreek{d}}_{0}}\text{\textgreek{q}}_{\ge R_{0}}(r^{-\frac{3}{2}}+O(r^{-2}))-C_{\text{\textgreek{d}}_{0}}\text{\textgreek{d}}_{1}\big)|\partial_{r}\text{\textgreek{f}}|^{2}+\\
 & \hphantom{C\int_{\mathcal{R}(\text{\textgreek{t}}_{1},\text{\textgreek{t}}_{2})\cap\{\frac{1}{2}R_{0}\le r\le R_{0}\}}}+sR^{-3\text{\textgreek{e}}_{0}}\big(c_{\text{\textgreek{d}}_{0}}\big(r^{-3}+O(r^{-4})+\text{\textgreek{q}}_{\ge R_{0}}(r^{-\frac{7}{2}}-r^{-3})\big)-C_{\text{\textgreek{d}}_{0}}\text{\textgreek{d}}_{1}\big)|\partial_{\text{\textgreek{sv}}}\text{\textgreek{f}}|^{2}-\\
 & \hphantom{C\int_{\mathcal{R}(\text{\textgreek{t}}_{1},\text{\textgreek{t}}_{2})\cap\{\frac{1}{2}R_{0}\le r\le R_{0}\}}}-C_{\text{\textgreek{d}}_{0}}sR^{-3\text{\textgreek{e}}_{0}}|T\text{\textgreek{f}}|^{2}+\big(c_{\text{\textgreek{d}}_{0}}s^{3}R^{-9\text{\textgreek{e}}_{0}}-C_{\text{\textgreek{d}}_{0}}s^{2}R^{-6\text{\textgreek{e}}_{0}}-C_{\text{\textgreek{d}}_{0}}sR^{-3\text{\textgreek{e}}_{0}}\big)|\text{\textgreek{f}}|^{2}\Bigg\}\, dg.
\end{split}
\label{eq:AlmostTotalBoundNearR_0}
\end{equation}
Provided $sR^{-3\text{\textgreek{e}}_{0}}$ is sufficiently large
in terms of $\text{\textgreek{d}}_{0}$, we can estimate 
\begin{equation}
\big(c_{\text{\textgreek{d}}_{0}}s^{2}R^{-6\text{\textgreek{e}}_{0}}-C_{\text{\textgreek{d}}_{0}}sR^{-3\text{\textgreek{e}}_{0}}\big)e^{-2sw_{R}}\big|\partial_{r}(e^{sw_{R}}\text{\textgreek{f}})\big|^{2}-C_{\text{\textgreek{d}}_{0}}\text{\textgreek{d}}_{1}sR^{-3\text{\textgreek{e}}_{0}}|\partial_{r}\text{\textgreek{f}}|^{2}\ge c_{\text{\textgreek{d}}_{0}}\text{\textgreek{d}}_{1}sR^{-3\text{\textgreek{e}}_{0}}|\partial_{r}\text{\textgreek{f}}|^{2}-C_{\text{\textgreek{d}}_{0}}\text{\textgreek{d}}_{1}s^{3}R^{-9\text{\textgreek{e}}_{0}}|\text{\textgreek{f}}|^{2}
\end{equation}
 and thus (\ref{eq:AlmostTotalBoundNearR_0}) yields: 
\begin{equation}
\begin{split}\int_{\mathcal{R}(\text{\textgreek{t}}_{1},\text{\textgreek{t}}_{2})\cap\{\frac{1}{2}R_{0}\le r\le R_{0}\}} & Re\Big\{2\text{\textgreek{q}}_{\le R}f^{-1}\nabla^{\text{\textgreek{m}}}\nabla^{\text{\textgreek{n}}}f\nabla_{\text{\textgreek{m}}}(f^{\frac{1}{2}}\text{\textgreek{f}})\nabla_{\text{\textgreek{n}}}(f^{\frac{1}{2}}\bar{\text{\textgreek{f}}})-2\text{\textgreek{q}}_{\le R}\text{\textgreek{q}}_{\ge R_{0}}r^{-1}f^{-1}(\partial_{r}f)\big|\partial_{r}(f^{\frac{1}{2}}\text{\textgreek{f}})\big|^{2}+\\
+2( & 1-\text{\textgreek{q}}_{\le R})\nabla^{\text{\textgreek{m}}}\nabla^{\text{\textgreek{n}}}f\nabla_{\text{\textgreek{m}}}\text{\textgreek{f}}\nabla_{\text{\textgreek{n}}}\bar{\text{\textgreek{f}}}+2\text{\textgreek{q}}_{\ge R_{0}}\text{\textgreek{q}}_{\le R}r^{-1}(\partial_{r}f)|\partial_{r}\text{\textgreek{f}}|^{2}-2h\nabla^{\text{\textgreek{m}}}\text{\textgreek{f}}\nabla_{\text{\textgreek{m}}}\bar{\text{\textgreek{f}}}+\mathcal{A}_{f,h}^{(R)}|\text{\textgreek{f}}|^{2}\Big\}\, dg\ge\\
\ge & \int_{\mathcal{R}(\text{\textgreek{t}}_{1},\text{\textgreek{t}}_{2})\cap\{\frac{1}{2}R_{0}\le r\le R_{0}\}}e^{2sw_{R}}\Big\{ c_{\text{\textgreek{d}}_{0}}\text{\textgreek{d}}_{1}sR^{-3\text{\textgreek{e}}_{0}}\big|\partial_{r}\text{\textgreek{f}}\big|^{2}+sR^{-3\text{\textgreek{e}}_{0}}\big(c_{\text{\textgreek{d}}_{0}}r^{-\frac{7}{2}}-C_{\text{\textgreek{d}}_{0}}\text{\textgreek{d}}_{1}\big)|\partial_{\text{\textgreek{sv}}}\text{\textgreek{f}}|^{2}-\\
 & \hphantom{\int_{\mathcal{R}(\text{\textgreek{t}}_{1},\text{\textgreek{t}}_{2})\cap\{\frac{1}{2}R_{0}\le}}-C_{\text{\textgreek{d}}_{0}}sR^{-3\text{\textgreek{e}}_{0}}|T\text{\textgreek{f}}|^{2}+\big((c_{\text{\textgreek{d}}_{0}}-C_{\text{\textgreek{d}}_{0}}\text{\textgreek{d}}_{1})s^{3}R^{-9\text{\textgreek{e}}_{0}}-C_{\text{\textgreek{d}}_{0}}s^{2}R^{-6\text{\textgreek{e}}_{0}}-C_{\text{\textgreek{d}}_{0}}sR^{-3\text{\textgreek{e}}_{0}}\big)|\text{\textgreek{f}}|^{2}\Bigg\}\, dg.
\end{split}
\label{eq:TotalBoundNearR_0}
\end{equation}

\item In view of (\ref{eq:LowerBoundAfhFirstAwayRegion}), (\ref{eq:LowerboundCurrentFirstAway})
and Lemma \ref{lem:ConstructionWR}, we can bound: 
\begin{equation}
\begin{split}\int_{\mathcal{R}(\text{\textgreek{t}}_{1},\text{\textgreek{t}}_{2})\cap\{R_{0}\le r\le R^{\text{\textgreek{e}}_{0}}\}}Re\Big\{2\text{\textgreek{q}}_{\le R} & f^{-1}\nabla^{\text{\textgreek{m}}}\nabla^{\text{\textgreek{n}}}f\nabla_{\text{\textgreek{m}}}(f^{\frac{1}{2}}\text{\textgreek{f}})\nabla_{\text{\textgreek{n}}}(f^{\frac{1}{2}}\bar{\text{\textgreek{f}}})-2\text{\textgreek{q}}_{\le R}\text{\textgreek{q}}_{\ge R_{0}}r^{-1}f^{-1}(\partial_{r}f)\big|\partial_{r}(f^{\frac{1}{2}}\text{\textgreek{f}})\big|^{2}+\\
+2( & 1-\text{\textgreek{q}}_{\le R})\nabla^{\text{\textgreek{m}}}\nabla^{\text{\textgreek{n}}}f\nabla_{\text{\textgreek{m}}}\text{\textgreek{f}}\nabla_{\text{\textgreek{n}}}\bar{\text{\textgreek{f}}}+2\text{\textgreek{q}}_{\ge R_{0}}\text{\textgreek{q}}_{\le R}r^{-1}(\partial_{r}f)|\partial_{r}\text{\textgreek{f}}|^{2}-2h\nabla^{\text{\textgreek{m}}}\text{\textgreek{f}}\nabla_{\text{\textgreek{m}}}\bar{\text{\textgreek{f}}}+\mathcal{A}_{f,h}^{(R)}|\text{\textgreek{f}}|^{2}\Big\}\, dg\ge\\
\ge & \int_{\mathcal{R}(\text{\textgreek{t}}_{1},\text{\textgreek{t}}_{2})\cap\{R_{0}\le r\le R^{\text{\textgreek{e}}_{0}}\}}e^{2sw_{R}}\Big\{\big(c_{\text{\textgreek{d}}_{0}}s^{2}R^{-6\text{\textgreek{e}}_{0}}-C_{\text{\textgreek{d}}_{0}}s\big)e^{-2sw_{R}}\big|\partial_{r}(e^{sw_{R}}\text{\textgreek{f}})\big|^{2}+\\
 & \hphantom{c\int_{\mathcal{R}(\text{\textgreek{t}}_{1},\text{\textgreek{t}}_{2})\cap\{R_{0}\le r\le R\}}e^{2sw_{R}}\Big\{}+c_{\text{\textgreek{d}}_{0}}sR^{-3\text{\textgreek{e}}_{0}}r^{-\frac{3}{2}}|\partial_{r}\text{\textgreek{f}}|^{2}+c_{\text{\textgreek{d}}_{0}}sR^{-3\text{\textgreek{e}}_{0}}r^{-\frac{7}{2}}|\partial_{\text{\textgreek{sv}}}\text{\textgreek{f}}|^{2}+\\
 & \hphantom{c\int_{\mathcal{R}(\text{\textgreek{t}}_{1},\text{\textgreek{t}}_{2})\cap\{R_{0}\le r\le R\}}e^{2sw_{R}}\Big\{}+c_{\text{\textgreek{d}}_{0}}sR^{-3\text{\textgreek{e}}_{0}}r^{-1}|T\text{\textgreek{f}}|^{2}+\big(c_{\text{\textgreek{d}}_{0}}R^{-9\text{\textgreek{e}}_{0}}s^{3}-C_{\text{\textgreek{d}}_{0}}s^{2}-C_{\text{\textgreek{d}}_{0}}s\big)|\text{\textgreek{f}}|^{2}\Big\}\, dg.
\end{split}
\label{eq:TotalBoundFirstAway}
\end{equation}

\item In view of (\ref{eq:LowerBoundAfhSecondAwayRegion}), (\ref{eq:LowerboundCurrentFirstAway})
and (\ref{eq:Bound_S^3_Terms_Final})--(\ref{eq:BoundDerivativesUpper}),
we can bound: 
\begin{equation}
\begin{split}\int_{\mathcal{R}(\text{\textgreek{t}}_{1},\text{\textgreek{t}}_{2})\cap\{R^{\text{\textgreek{e}}_{0}}\le r\le\frac{1}{2}R\}} & Re\Big\{2\text{\textgreek{q}}_{\le R}f^{-1}\nabla^{\text{\textgreek{m}}}\nabla^{\text{\textgreek{n}}}f\nabla_{\text{\textgreek{m}}}(f^{\frac{1}{2}}\text{\textgreek{f}})\nabla_{\text{\textgreek{n}}}(f^{\frac{1}{2}}\bar{\text{\textgreek{f}}})-2\text{\textgreek{q}}_{\le R}\text{\textgreek{q}}_{\ge R_{0}}r^{-1}f^{-1}(\partial_{r}f)\big|\partial_{r}(f^{\frac{1}{2}}\text{\textgreek{f}})\big|^{2}+\\
+2( & 1-\text{\textgreek{q}}_{\le R})\nabla^{\text{\textgreek{m}}}\nabla^{\text{\textgreek{n}}}f\nabla_{\text{\textgreek{m}}}\text{\textgreek{f}}\nabla_{\text{\textgreek{n}}}\bar{\text{\textgreek{f}}}+2\text{\textgreek{q}}_{\ge R_{0}}\text{\textgreek{q}}_{\le R}r^{-1}(\partial_{r}f)|\partial_{r}\text{\textgreek{f}}|^{2}-2h\nabla^{\text{\textgreek{m}}}\text{\textgreek{f}}\nabla_{\text{\textgreek{m}}}\bar{\text{\textgreek{f}}}+\mathcal{A}_{f,h}^{(R)}|\text{\textgreek{f}}|^{2}\Big\}\, dg\ge\\
\ge & \int_{\mathcal{R}(\text{\textgreek{t}}_{1},\text{\textgreek{t}}_{2})\cap\{R^{\text{\textgreek{e}}_{0}}\le r\le\frac{1}{2}R\}}e^{2sw_{R}}\Big\{ r^{-2}\big(c_{\text{\textgreek{d}}_{0}}r^{2\text{\textgreek{e}}_{0}}s^{2}R^{-2\text{\textgreek{e}}_{0}}-C_{\text{\textgreek{d}}_{0}}s(r^{\text{\textgreek{e}}_{0}}R^{-\text{\textgreek{e}}_{0}}+r^{2\text{\textgreek{e}}_{0}}R^{-2\text{\textgreek{e}}_{0}})\big)e^{-2sw_{R}}\big|\partial_{r}(e^{sw_{R}}\text{\textgreek{f}})\big|^{2}+\\
 & \hphantom{c\int_{\mathcal{R}(\text{\textgreek{t}}_{1},\text{\textgreek{t}}_{2})\cap\{R_{0}\le r\le R\}}e^{2sw_{R}}\Big\{}+c_{\text{\textgreek{d}}_{0}}r^{-\frac{5}{2}+\text{\textgreek{e}}_{0}}sR^{-\text{\textgreek{e}}_{0}}|\partial_{r}\text{\textgreek{f}}|^{2}+c_{\text{\textgreek{d}}_{0}}r^{-\frac{9}{2}+\text{\textgreek{e}}_{0}}sR^{-\text{\textgreek{e}}_{0}}|\partial_{\text{\textgreek{sv}}}\text{\textgreek{f}}|^{2}+c_{\text{\textgreek{d}}_{0}}r^{-2+\text{\textgreek{e}}_{0}}sR^{-\text{\textgreek{e}}_{0}}|T\text{\textgreek{f}}|^{2}+\\
 & \hphantom{c\int_{\mathcal{R}(\text{\textgreek{t}}_{1},\text{\textgreek{t}}_{2})\cap\{R_{0}\le r\le R\}}e^{2sw_{R}}\Big\{}+r^{-4}\big(c_{\text{\textgreek{d}}_{0}}\text{\textgreek{e}}_{0}r^{3\text{\textgreek{e}}_{0}}s^{3}R^{-3\text{\textgreek{e}}_{0}}-C_{\text{\textgreek{d}}_{0}}r^{2\text{\textgreek{e}}_{0}}s^{2}R^{-2\text{\textgreek{e}}_{0}}-C_{\text{\textgreek{d}}_{0}}r^{\text{\textgreek{e}}_{0}}sR^{-\text{\textgreek{e}}_{0}}\big)|\text{\textgreek{f}}|^{2}\Big\}\, dg.
\end{split}
\label{eq:TotalBoundSecondAway}
\end{equation}

\item In view of (\ref{eq:LowerBoundAfhThirdAwayRegion}), (\ref{eq:LowerboundCurrentFirstAway})
and (\ref{eq:LowerBound})--(\ref{eq:UpperBounds}), we can bound:
\begin{equation}
\begin{split}\int_{\mathcal{R}(\text{\textgreek{t}}_{1},\text{\textgreek{t}}_{2})\cap\{\frac{1}{2}R\le r\le R\}}Re\Big\{2\text{\textgreek{q}}_{\le R} & f^{-1}\nabla^{\text{\textgreek{m}}}\nabla^{\text{\textgreek{n}}}f\nabla_{\text{\textgreek{m}}}(f^{\frac{1}{2}}\text{\textgreek{f}})\nabla_{\text{\textgreek{n}}}(f^{\frac{1}{2}}\bar{\text{\textgreek{f}}})-2\text{\textgreek{q}}_{\le R}\text{\textgreek{q}}_{\ge R_{0}}r^{-1}f^{-1}(\partial_{r}f)\big|\partial_{r}(f^{\frac{1}{2}}\text{\textgreek{f}})\big|^{2}+\\
+2( & 1-\text{\textgreek{q}}_{\le R})\nabla^{\text{\textgreek{m}}}\nabla^{\text{\textgreek{n}}}f\nabla_{\text{\textgreek{m}}}\text{\textgreek{f}}\nabla_{\text{\textgreek{n}}}\bar{\text{\textgreek{f}}}+2\text{\textgreek{q}}_{\ge R_{0}}\text{\textgreek{q}}_{\le R}r^{-1}(\partial_{r}f)|\partial_{r}\text{\textgreek{f}}|^{2}-2h\nabla^{\text{\textgreek{m}}}\text{\textgreek{f}}\nabla_{\text{\textgreek{m}}}\bar{\text{\textgreek{f}}}+\mathcal{A}_{f,h}^{(R)}|\text{\textgreek{f}}|^{2}\Big\}\, dg\ge\\
\ge & \int_{\mathcal{R}(\text{\textgreek{t}}_{1},\text{\textgreek{t}}_{2})\cap\{\frac{1}{2}R\le r\le R\}}e^{2sw_{R}}\Big\{ R^{-2}\big(c_{\text{\textgreek{d}}_{0}}-C_{\text{\textgreek{d}}_{0}}(\text{\textgreek{d}}_{1}+s^{-1}\text{\textgreek{d}}_{1}^{-1})\big)e^{-2sw_{R}}\big|\partial_{r}(e^{sw_{R}}\text{\textgreek{f}})\big|^{2}+\\
 & \hphantom{c\int_{\mathcal{R}(\text{\textgreek{t}}_{1},\text{\textgreek{t}}_{2})\cap\{R_{0}\le r\le R\}}e^{2sw_{R}}\Big\{}+c_{\text{\textgreek{d}}_{0}}R^{-\frac{5}{2}}|\partial_{r}\text{\textgreek{f}}|^{2}+c_{\text{\textgreek{d}}_{0}}R^{-\frac{9}{2}}|\partial_{\text{\textgreek{sv}}}\text{\textgreek{f}}|^{2}+\\
 & \hphantom{c\int_{\mathcal{R}(\text{\textgreek{t}}_{1},\text{\textgreek{t}}_{2})\cap\{R_{0}\le r\le R\}}e^{2sw_{R}}\Big\{}+c_{\text{\textgreek{d}}_{0}}R^{-2}v_{s}^{\prime}(\frac{r}{R})s|T\text{\textgreek{f}}|^{2}-C_{\text{\textgreek{d}}_{0}}R^{-4}\big(v_{s}^{\prime}(\frac{r}{R})s^{3}+s^{2}+s\big)\Big\}\, dg.
\end{split}
\label{eq:LowerboundCurrentFirstAway-1}
\end{equation}

\item In view of (\ref{eq:LowerBoundAfAlmostInfinity}), (\ref{eq:LowerboundCurrentNearInfinity}),
(\ref{eq:Function_F}), (\ref{eq:hAlmostR}) and (\ref{eq:hIntermediateAway}),
we can estimate: 
\begin{equation}
\begin{split}\int_{\mathcal{R}(\text{\textgreek{t}}_{1},\text{\textgreek{t}}_{2})\cap\{R\le r\le\text{\textgreek{d}}_{2}^{-1}R\}}Re\Big\{2\text{\textgreek{q}}_{\le R} & f^{-1}\nabla^{\text{\textgreek{m}}}\nabla^{\text{\textgreek{n}}}f\nabla_{\text{\textgreek{m}}}(f^{\frac{1}{2}}\text{\textgreek{f}})\nabla_{\text{\textgreek{n}}}(f^{\frac{1}{2}}\bar{\text{\textgreek{f}}})-2\text{\textgreek{q}}_{\le R}\text{\textgreek{q}}_{\ge R_{0}}r^{-1}f^{-1}(\partial_{r}f)\big|\partial_{r}(f^{\frac{1}{2}}\text{\textgreek{f}})\big|^{2}+\\
+2( & 1-\text{\textgreek{q}}_{\le R})\nabla^{\text{\textgreek{m}}}\nabla^{\text{\textgreek{n}}}f\nabla_{\text{\textgreek{m}}}\text{\textgreek{f}}\nabla_{\text{\textgreek{n}}}\bar{\text{\textgreek{f}}}+2\text{\textgreek{q}}_{\ge R_{0}}\text{\textgreek{q}}_{\le R}r^{-1}(\partial_{r}f)|\partial_{r}\text{\textgreek{f}}|^{2}-2h\nabla^{\text{\textgreek{m}}}\text{\textgreek{f}}\nabla_{\text{\textgreek{m}}}\bar{\text{\textgreek{f}}}+\mathcal{A}_{f,h}^{(R)}|\text{\textgreek{f}}|^{2}\Big\}\, dg\ge\\
\ge & \int_{\mathcal{R}(\text{\textgreek{t}}_{1},\text{\textgreek{t}}_{2})\cap\{R\le r\le\text{\textgreek{d}}_{2}^{-1}R\}}\Big\{2\text{\textgreek{q}}_{\le R}\frac{R^{2}}{r^{2}}(\frac{9}{5}-\frac{r}{R})R^{-2}f(R)\cdot f^{-1}\big|\partial_{r}(f^{\frac{1}{2}}\text{\textgreek{f}})\big|^{2}+\\
 & \hphantom{\int_{\mathcal{R}(\text{\textgreek{t}}_{1},\text{\textgreek{t}}_{2})\cap\{R\le}}+\Big((1-\text{\textgreek{q}}_{\le R})f(R)r^{-\frac{5}{2}}-2\text{\textgreek{q}}_{\le R}\big(\frac{R^{2}}{r^{2}}(\frac{r}{R}-\frac{9}{10})(R^{-2}+O(R^{-\frac{5}{2}}))f(R)\Big)|\partial_{r}\text{\textgreek{f}}|^{2}+\\
 & \hphantom{\int_{\mathcal{R}(\text{\textgreek{t}}_{1},\text{\textgreek{t}}_{2})\cap\{R\le}}+cf(R)r^{-\frac{9}{2}}|\partial_{\text{\textgreek{sv}}}\text{\textgreek{f}}|^{2}+cf(R)r^{-2}|T\text{\textgreek{f}}|^{2}-CR^{-4}f(R)|\text{\textgreek{f}}|^{2}\Big\}\, dg.
\end{split}
\label{eq:AlmostTotalBoundNearInfinity}
\end{equation}

\begin{rem*}
Notice that the positivity of the coefficient of $\big|\partial_{r}(f^{\frac{1}{2}}\text{\textgreek{f}})\big|^{2}$
in the right hand side of (\ref{eq:AlmostTotalBoundNearInfinity})
follows from the fact that, in view of (\ref{eq:Function_F}), provided
$R$ is sufficiently large in terms of the geometry of $(\mathcal{M},g)$,
we can bound for $R\le r\le\frac{4}{3}R$ (i.\,e.~on $supp(\text{\textgreek{q}}_{\le R})\cap\{r\ge R\}$):
\begin{equation}
\partial_{r}^{2}f+r^{-\frac{3}{2}}\partial_{r}f>(r^{-1}+O(r^{-2}))\partial_{r}f+O(r^{-1})\partial_{r}^{2}f+f(R)(\frac{9}{5}-\frac{r}{R})R^{-2}.
\end{equation}

\end{rem*}
Applying the product rule and a Cauchy--Schwarz inequality on the
first term of the right hand side of (\ref{eq:AlmostTotalBoundNearInfinity}),
we obtain: 
\begin{equation}
2\text{\textgreek{q}}_{\le R}\frac{R^{2}}{r^{2}}(\frac{9}{5}-\frac{r}{R})R^{-2}f(R)\cdot f^{-1}\big|\partial_{r}(f^{\frac{1}{2}}\text{\textgreek{f}})\big|^{2}\ge2\text{\textgreek{q}}_{\le R}\frac{R^{2}}{r^{2}}(\frac{9}{5}-\frac{r}{R}-\frac{1}{100})R^{-2}f(R)|\partial_{r}\text{\textgreek{f}}|^{2}-C\text{\textgreek{q}}_{\le R}R^{-4}f(R)|\text{\textgreek{f}}|^{2}
\end{equation}
for some absolute constant $C>0$. Thus, (\ref{eq:AlmostTotalBoundNearInfinity})
yields (provided $R\gg1$): 
\begin{equation}
\begin{split}\int_{\mathcal{R}(\text{\textgreek{t}}_{1},\text{\textgreek{t}}_{2})\cap\{R\le r\le\text{\textgreek{d}}_{2}^{-1}R\}}Re\Big\{2\text{\textgreek{q}}_{\le R} & f^{-1}\nabla^{\text{\textgreek{m}}}\nabla^{\text{\textgreek{n}}}f\nabla_{\text{\textgreek{m}}}(f^{\frac{1}{2}}\text{\textgreek{f}})\nabla_{\text{\textgreek{n}}}(f^{\frac{1}{2}}\bar{\text{\textgreek{f}}})-2\text{\textgreek{q}}_{\le R}\text{\textgreek{q}}_{\ge R_{0}}r^{-1}f^{-1}(\partial_{r}f)\big|\partial_{r}(f^{\frac{1}{2}}\text{\textgreek{f}})\big|^{2}+\\
+2( & 1-\text{\textgreek{q}}_{\le R})\nabla^{\text{\textgreek{m}}}\nabla^{\text{\textgreek{n}}}f\nabla_{\text{\textgreek{m}}}\text{\textgreek{f}}\nabla_{\text{\textgreek{n}}}\bar{\text{\textgreek{f}}}+2\text{\textgreek{q}}_{\ge R_{0}}\text{\textgreek{q}}_{\le R}r^{-1}(\partial_{r}f)|\partial_{r}\text{\textgreek{f}}|^{2}-2h\nabla^{\text{\textgreek{m}}}\text{\textgreek{f}}\nabla_{\text{\textgreek{m}}}\bar{\text{\textgreek{f}}}+\mathcal{A}_{f,h}^{(R)}|\text{\textgreek{f}}|^{2}\Big\}\, dg\ge\\
\ge & \int_{\mathcal{R}(\text{\textgreek{t}}_{1},\text{\textgreek{t}}_{2})\cap\{R\le r\le\text{\textgreek{d}}_{2}^{-1}R\}}f(R)\Big\{ cr^{-\frac{5}{2}}|\partial_{r}\text{\textgreek{f}}|^{2}+cr^{-\frac{9}{2}}|\partial_{\text{\textgreek{sv}}}\text{\textgreek{f}}|^{2}+cr^{-2}|T\text{\textgreek{f}}|^{2}-CR^{-4}|\text{\textgreek{f}}|^{2}\Big\}\, dg.
\end{split}
\label{eq:TotalBoundAlmostNearInfinity}
\end{equation}

\item In view of (\ref{eq:LowerBoundAfhNearInfinity}), (\ref{eq:LowerboundCurrentNearInfinity}),
(\ref{eq:Function_F}) and (\ref{eq:hFarAway}), we can bound: 
\begin{equation}
\begin{split}\int_{\mathcal{R}(\text{\textgreek{t}}_{1},\text{\textgreek{t}}_{2})\cap\{r\ge\text{\textgreek{d}}_{2}^{-1}R\}}Re\Big\{2\text{\textgreek{q}}_{\le R} & f^{-1}\nabla^{\text{\textgreek{m}}}\nabla^{\text{\textgreek{n}}}f\nabla_{\text{\textgreek{m}}}(f^{\frac{1}{2}}\text{\textgreek{f}})\nabla_{\text{\textgreek{n}}}(f^{\frac{1}{2}}\bar{\text{\textgreek{f}}})-2\text{\textgreek{q}}_{\le R}\text{\textgreek{q}}_{\ge R_{0}}r^{-1}f^{-1}(\partial_{r}f)\big|\partial_{r}(f^{\frac{1}{2}}\text{\textgreek{f}})\big|^{2}+\\
+2( & 1-\text{\textgreek{q}}_{\le R})\nabla^{\text{\textgreek{m}}}\nabla^{\text{\textgreek{n}}}f\nabla_{\text{\textgreek{m}}}\text{\textgreek{f}}\nabla_{\text{\textgreek{n}}}\bar{\text{\textgreek{f}}}+2\text{\textgreek{q}}_{\ge R_{0}}\text{\textgreek{q}}_{\le R}r^{-1}(\partial_{r}f)|\partial_{r}\text{\textgreek{f}}|^{2}-2h\nabla^{\text{\textgreek{m}}}\text{\textgreek{f}}\nabla_{\text{\textgreek{m}}}\bar{\text{\textgreek{f}}}+\mathcal{A}_{f,h}^{(R)}|\text{\textgreek{f}}|^{2}\Big\}\, dg\ge\\
 & \hphantom{\text{\textgreek{f}}|^{2}}\ge\int_{\mathcal{R}(\text{\textgreek{t}}_{1},\text{\textgreek{t}}_{2})\cap\{r\ge\text{\textgreek{d}}_{2}^{-1}R\}}f(R)\Big\{ cr^{-2}|\partial_{r}\text{\textgreek{f}}|^{2}+cR^{-1}r^{-3}|\partial_{\text{\textgreek{sv}}}\text{\textgreek{f}}|^{2}+cr^{-2}|T\text{\textgreek{f}}|^{2}-CR^{-1}r^{-3}|\text{\textgreek{f}}|^{2}\Big\}\, dg.
\end{split}
\label{eq:TotalBoundNearInfinity}
\end{equation}

\end{enumerate} 

In view of (\ref{eq:BoundCurrentErgoregion}), (\ref{eq:TotalBoundOutsideErgoregionSum}),
(\ref{eq:TotalBoundNearR_0}), (\ref{eq:TotalBoundFirstAway}), (\ref{eq:TotalBoundSecondAway}),
(\ref{eq:TotalBoundAlmostNearInfinity}), (\ref{eq:TotalBoundNearInfinity}),
as well as the fact that $\tilde{f}\equiv f$ and $\tilde{h}\equiv h$
on $\mbox{ on }\mathcal{M}\backslash\big(\mathcal{B}_{crit}(4\text{\textgreek{d}}_{0})\cup\mathcal{H}^{-}\big)$,
we obtain from (\ref{eq:FundamentalCurrentFull}) (provided $\text{\textgreek{d}}_{0},\text{\textgreek{d}}_{2}>0$
are sufficiently small in terms of $R_{0}$ and the geometry of $(\mathcal{M},g)$,
$\text{\textgreek{d}}_{1}>0$ is sufficiently small in terms of $\text{\textgreek{d}}_{0}$
and $R_{0}$, $s$ is sufficiently large in terms of $\text{\textgreek{d}}_{0},\text{\textgreek{d}}_{1}$
and $\text{\textgreek{e}}_{0}sR^{-9\text{\textgreek{e}}_{0}}$ is
sufficiently large in terms of $\text{\textgreek{d}}_{0},\text{\textgreek{d}}_{1}$):
\begin{equation}
\begin{split}\int_{\mathcal{R}(\text{\textgreek{t}}_{1},\text{\textgreek{t}}_{2})\cap\{\frac{1}{4}r_{0}\le r\le\frac{1}{2}R_{0}\}\backslash\mathscr{E}_{ext}}(f+\tilde{f})\Bigg\{ & c_{\text{\textgreek{d}}_{0}\text{\textgreek{d}}_{1}}sR^{-3\text{\textgreek{e}}_{0}}(g^{-1})_{\text{\textgreek{S}}}^{ij}\partial_{i}\text{\textgreek{f}}\partial_{j}\bar{\text{\textgreek{f}}}-CsR^{-3\text{\textgreek{e}}_{0}}|\nabla_{g_{\text{\textgreek{S}}}}\text{\textgreek{f}}|_{g_{\text{\textgreek{S}}}}|T\text{\textgreek{f}}|-CsR^{-3\text{\textgreek{e}}_{0}}|T\text{\textgreek{f}}|^{2}+\\
 & +c_{\text{\textgreek{d}}_{0}}s^{3}R^{-9\text{\textgreek{e}}_{0}}|\text{\textgreek{f}}|^{2}\Bigg\}\, dg+\\
+c_{\text{\textgreek{d}}_{0}}\int_{\mathcal{R}(\text{\textgreek{t}}_{1},\text{\textgreek{t}}_{2})\cap\{r\ge\frac{1}{2}R_{0}\}}(f+\tilde{f})\Bigg\{ & \text{\textgreek{q}}_{\{\frac{1}{2}R_{0}\le r\le R_{0}\}}\text{\textgreek{d}}_{1}sR^{-3\text{\textgreek{e}}_{0}}+\text{\textgreek{q}}_{\{R_{0}\le r\le R^{\text{\textgreek{e}}_{0}}\}}sR^{-3\text{\textgreek{e}}_{0}}r^{-\frac{3}{2}}+\text{\textgreek{q}}_{\{R^{\text{\textgreek{e}}_{0}}\le r\le\frac{1}{2}R\}}sR^{-\text{\textgreek{e}}_{0}}r^{-\frac{5}{2}+\text{\textgreek{e}}_{0}}+\\
 & +\text{\textgreek{q}}_{\{\frac{1}{2}R\le r\le\text{\textgreek{d}}_{2}^{-1}R\}}r^{-\frac{5}{2}}+\text{\textgreek{q}}_{\{r\ge\text{\textgreek{d}}_{2}^{-1}R\}}r^{-2}\Bigg\}\big(\big|\partial_{r}\text{\textgreek{f}}\big|^{2}+r^{-2}|\partial_{\text{\textgreek{sv}}}\text{\textgreek{f}}|^{2}\big)\, dg+\\
+\int_{\mathcal{R}(\text{\textgreek{t}}_{1},\text{\textgreek{t}}_{2})\cap\{\frac{1}{2}R_{0}\le r\le\frac{1}{2}R\}}(f+\tilde{f})\Bigg\{ & \text{\textgreek{q}}_{\{\frac{1}{2}R_{0}\le r\le R_{0}\}}\Big(-C_{\text{\textgreek{d}}_{0}}sR^{-3\text{\textgreek{e}}_{0}}|T\text{\textgreek{f}}|^{2}+c_{\text{\textgreek{d}}_{0}}s^{3}R^{-9\text{\textgreek{e}}_{0}}|\text{\textgreek{f}}|^{2}\Big)+\\
 & +c_{\text{\textgreek{d}}_{0}}\text{\textgreek{q}}_{\{R_{0}\le r\le R^{\text{\textgreek{e}}_{0}}\}}\Big(sR^{-3\text{\textgreek{e}}_{0}}r^{-1}|T\text{\textgreek{f}}|^{2}+s^{3}R^{-9\text{\textgreek{e}}_{0}}|\text{\textgreek{f}}|^{2}\Big)+\\
 & +\text{\textgreek{q}}_{\{R^{\text{\textgreek{e}}_{0}}\le r\le\frac{1}{2}R\}}\Big(sR^{-\text{\textgreek{e}}_{0}}r^{-2+\text{\textgreek{e}}_{0}}|T\text{\textgreek{f}}|^{2}+c_{\text{\textgreek{d}}_{0}}\text{\textgreek{e}}_{0}s^{3}R^{-3\text{\textgreek{e}}_{0}}r^{-4+3\text{\textgreek{e}}_{0}}|\text{\textgreek{f}}|^{2}\Big)\Bigg\}\, dg+\\
+\int_{\mathcal{R}(\text{\textgreek{t}}_{1},\text{\textgreek{t}}_{2})\cap\{\frac{1}{2}R\le r\le R\}}(f+\tilde{f})v_{s}^{\prime}(\frac{r}{R})\Big( & c_{\text{\textgreek{d}}_{0}}R^{-2}s|T\text{\textgreek{f}}|^{2}-C_{\text{\textgreek{d}}_{0}}R^{-4}s^{3}|\text{\textgreek{f}}|^{2}\Big)\, dg+\\
+\int_{\mathcal{R}(\text{\textgreek{t}}_{1},\text{\textgreek{t}}_{2})\cap\{r\ge R\}}(f(R)+\tilde{f}(R))\Big( & cr^{-2}|T\text{\textgreek{f}}|^{2}-CR^{-1}r^{-3}|\text{\textgreek{f}}|^{2}\Big)\, dg\le\\
\le & C_{\text{\textgreek{d}}_{0}}\int_{\mathcal{R}(\text{\textgreek{t}}_{1},\text{\textgreek{t}}_{2})\cap\big(\mathscr{E}_{ext}\cup\{r\le\frac{1}{4}r_{0}\}\big)}(f+\tilde{f})\Big\{ s^{2}R^{-6\text{\textgreek{e}}_{0}}|\nabla\text{\textgreek{f}}|_{g_{ref}}^{2}+s^{4}R^{-12\text{\textgreek{e}}_{0}}|\text{\textgreek{f}}|^{2}\Big\}\, dg-\\
 & -\int_{\mathcal{R}(\text{\textgreek{t}}_{1},\text{\textgreek{t}}_{2})}Re\big\{ G\big(2(\nabla^{\text{\textgreek{m}}}f+\nabla^{\text{\textgreek{m}}}\tilde{f})\nabla_{\text{\textgreek{m}}}\bar{\text{\textgreek{f}}}+(\square_{g}f+\square_{g}\tilde{f}-2h-2\tilde{h})\bar{\text{\textgreek{f}}}\big)\big\}\, dg-\\
 & -\mathcal{B}_{f,h}^{(R)}[\text{\textgreek{f}};\text{\textgreek{t}}_{1},\text{\textgreek{t}}_{2}]-\mathcal{B}_{\tilde{f},\tilde{h}}^{(R)}[\text{\textgreek{f}};\text{\textgreek{t}}_{1},\text{\textgreek{t}}_{2}],
\end{split}
\label{eq:MainCarlemanEstimateFromundamentalCurrent}
\end{equation}
where, for any set $A\subset\mathcal{M}$, we denote with $\text{\textgreek{q}}_{A}$
the characteristic function of $A$.

In view of the fact that $(g^{-1})_{\text{\textgreek{S}}}$ is positive
definite on $\mathcal{M}\backslash\mathscr{E}\cup\mathcal{H}$, we
can estimate on $\mathcal{M}\backslash\mathscr{E}_{\text{\textgreek{d}}}\cup\{r\ge\frac{1}{4}r_{0}\}$
for any $\text{\textgreek{d}}>0$: 
\begin{equation}
c_{\text{\textgreek{d}}_{0}\text{\textgreek{d}}_{1}}sR^{-3\text{\textgreek{e}}_{0}}(g^{-1})_{\text{\textgreek{S}}}^{ij}\partial_{i}\text{\textgreek{f}}\partial_{j}\bar{\text{\textgreek{f}}}-CsR^{-3\text{\textgreek{e}}_{0}}|\nabla_{g_{\text{\textgreek{S}}}}\text{\textgreek{f}}|_{g_{\text{\textgreek{S}}}}|T\text{\textgreek{f}}|\ge c_{\text{\textgreek{d}}\text{\textgreek{d}}_{0}\text{\textgreek{d}}_{1}}sR^{-3\text{\textgreek{e}}_{0}}|\nabla_{g_{\text{\textgreek{S}}}}\text{\textgreek{f}}|_{g_{\text{\textgreek{S}}}}-C_{\text{\textgreek{d}}}sR^{-3\text{\textgreek{e}}_{0}}|T\text{\textgreek{f}}|^{2}.\label{eq:CauchySchwarzNearBoundaryErgoregion}
\end{equation}
Furthermore, if $\text{\textgreek{q}}_{r_{0}}:\mathcal{M}\backslash\mathcal{H}^{-}\rightarrow[0,1]$
is a smooth $T$-invariant function supported in $\{r\le r_{0}\}$
such that $\text{\textgreek{q}}_{r_{0}}\equiv1$ on $\{r\le\frac{1}{2}r_{0}\}$,
then, after integrating by parts in the identity 
\begin{equation}
\int_{\mathcal{R}(\text{\textgreek{t}}_{1},\text{\textgreek{t}}_{2})}Re\Big\{ N(\text{\textgreek{q}}_{r_{0}}\bar{\text{\textgreek{f}}})\square_{g}(\text{\textgreek{q}}_{r_{0}}\text{\textgreek{f}})\Big\}\, dg=\int_{\mathcal{R}(\text{\textgreek{t}}_{1},\text{\textgreek{t}}_{2})}Re\Big\{ N(\text{\textgreek{q}}_{r_{0}}\bar{\text{\textgreek{f}}})\big(\text{\textgreek{q}}_{r_{0}}G+2\nabla^{\text{\textgreek{m}}}\text{\textgreek{q}}_{r_{0}}\nabla_{\text{\textgreek{m}}}\text{\textgreek{f}}+\square_{g}\text{\textgreek{q}}_{r_{0}}\text{\textgreek{f}}\big)\Big\}\, dg,
\end{equation}
 using also the bounds (\ref{eq:PositiveKN}) and (\ref{eq:UpperBoundKN})
from Assumption \hyperref[Assumption 2]{G2} (as well as a Poincare-type
inequality), we readily obtain the red-shift-type estimate 
\begin{align}
c\int_{\mathcal{H}^{+}\cap\mathcal{R}(\text{\textgreek{t}}_{1},\text{\textgreek{t}}_{2})} & \big(J_{\text{\textgreek{m}}}^{N}(\text{\textgreek{f}})n_{\mathcal{H}^{+}}^{\text{\textgreek{m}}}+|\text{\textgreek{f}}|^{2}\big)\, dvol_{\mathcal{H}^{+}}+c\int_{\mathcal{R}(\text{\textgreek{t}}_{1},\text{\textgreek{t}}_{2})\cap\{r\le\frac{1}{2}r_{0}\}}\big(|\nabla\text{\textgreek{f}}|_{g_{ref}}^{2}+|\text{\textgreek{f}}|^{2}\big)\, dg\le\label{eq:FromRedShift}\\
\le C & \int_{\mathcal{R}(\text{\textgreek{t}}_{1},\text{\textgreek{t}}_{2})\cap\{\frac{1}{2}r_{0}\le r\le r_{0}\}}\big(|\nabla\text{\textgreek{f}}|_{g_{ref}}^{2}+|\text{\textgreek{f}}|^{2}\big)\, dg+C\int_{\text{\textgreek{S}}_{\text{\textgreek{t}}_{1}}\cap\{r\le r_{0}\}}|\nabla\text{\textgreek{f}}|_{g_{ref}}^{2}-\int_{\mathcal{R}(\text{\textgreek{t}}_{1},\text{\textgreek{t}}_{2})}Re\Big\{\text{\textgreek{q}}_{r_{0}}N(\text{\textgreek{q}}_{r_{0}}\bar{\text{\textgreek{f}}})\cdot G\Big\}\, dg.\nonumber 
\end{align}
Therefore, in view of (\ref{eq:CauchySchwarzNearBoundaryErgoregion}),
(\ref{eq:FromRedShift}) and the fact that 
\begin{equation}
\frac{\sup_{\{r\le\frac{1}{4}r_{0}\}}(e^{2sw_{R}}+e^{2s\tilde{w}_{R}})}{\inf_{\{r\ge\frac{1}{2}r_{0}\}}(e^{2sw_{R}}+e^{2s\tilde{w}_{R}})}>e^{csR^{-3\text{\textgreek{e}}_{0}}}
\end{equation}
(following from (\ref{eq:SmallerNearHorizon}), (\ref{eq:SmallerNearHorizonDistorted}),
(\ref{eq:Function_F}) and (\ref{eq:Function_F-1})), (\ref{eq:MainCarlemanEstimateFromundamentalCurrent})
yields (provided $sR^{-3\text{\textgreek{e}}_{0}}\gg1$): 
\begin{equation}
\begin{split}\int_{\mathcal{R}(\text{\textgreek{t}}_{1},\text{\textgreek{t}}_{2})\cap\{r\le\frac{1}{2}R_{0}\}\backslash\mathscr{E}_{\text{\textgreek{d}}}}(f+\tilde{f}+\sup_{\{r\le\frac{1}{4}r_{0}\}}f)\Bigg\{ & c_{\text{\textgreek{d}}\text{\textgreek{d}}_{0}\text{\textgreek{d}}_{1}}sR^{-3\text{\textgreek{e}}_{0}}|\nabla_{g_{\text{\textgreek{S}}}}\text{\textgreek{f}}|_{g_{\text{\textgreek{S}}}}^{2}-C_{\text{\textgreek{d}}}sR^{-3\text{\textgreek{e}}_{0}}|T\text{\textgreek{f}}|^{2}+c_{\text{\textgreek{d}}_{0}}s^{3}R^{-9\text{\textgreek{e}}_{0}}|\text{\textgreek{f}}|^{2}\Bigg\}\, dg+\\
+c_{\text{\textgreek{d}}_{0}}\int_{\mathcal{R}(\text{\textgreek{t}}_{1},\text{\textgreek{t}}_{2})\cap\{r\ge\frac{1}{2}R_{0}\}}(f+\tilde{f})\Bigg\{ & \text{\textgreek{q}}_{\{\frac{1}{2}R_{0}\le r\le R_{0}\}}\text{\textgreek{d}}_{1}sR^{-3\text{\textgreek{e}}_{0}}+\text{\textgreek{q}}_{\{R_{0}\le r\le R^{\text{\textgreek{e}}_{0}}\}}sR^{-3\text{\textgreek{e}}_{0}}r^{-\frac{3}{2}}+\text{\textgreek{q}}_{\{R^{\text{\textgreek{e}}_{0}}\le r\le\frac{1}{2}R\}}sR^{-\text{\textgreek{e}}_{0}}r^{-\frac{5}{2}+\text{\textgreek{e}}_{0}}+\\
 & +\text{\textgreek{q}}_{\{\frac{1}{2}R\le r\le\text{\textgreek{d}}_{2}^{-1}R\}}r^{-\frac{5}{2}}+\text{\textgreek{q}}_{\{r\ge\text{\textgreek{d}}_{2}^{-1}R\}}r^{-2}\Bigg\}\big(\big|\partial_{r}\text{\textgreek{f}}\big|^{2}+r^{-2}|\partial_{\text{\textgreek{sv}}}\text{\textgreek{f}}|^{2}\big)\, dg+\\
+\int_{\mathcal{R}(\text{\textgreek{t}}_{1},\text{\textgreek{t}}_{2})\cap\{\frac{1}{2}R_{0}\le r\le\frac{1}{2}R\}}(f+\tilde{f})\Bigg\{ & \text{\textgreek{q}}_{\{\frac{1}{2}R_{0}\le r\le R_{0}\}}\Big(-C_{\text{\textgreek{d}}_{0}}sR^{-3\text{\textgreek{e}}_{0}}|T\text{\textgreek{f}}|^{2}+c_{\text{\textgreek{d}}_{0}}s^{3}R^{-9\text{\textgreek{e}}_{0}}|\text{\textgreek{f}}|^{2}\Big)+\\
 & +c_{\text{\textgreek{d}}_{0}}\text{\textgreek{q}}_{\{R_{0}\le r\le R^{\text{\textgreek{e}}_{0}}\}}\Big(sR^{-3\text{\textgreek{e}}_{0}}r^{-1}|T\text{\textgreek{f}}|^{2}+s^{3}R^{-9\text{\textgreek{e}}_{0}}|\text{\textgreek{f}}|^{2}\Big)+\\
 & +\text{\textgreek{q}}_{\{R^{\text{\textgreek{e}}_{0}}\le r\le\frac{1}{2}R\}}\Big(sR^{-\text{\textgreek{e}}_{0}}r^{-2+\text{\textgreek{e}}_{0}}|T\text{\textgreek{f}}|^{2}+c_{\text{\textgreek{d}}_{0}}\text{\textgreek{e}}_{0}s^{3}R^{-3\text{\textgreek{e}}_{0}}r^{-4+3\text{\textgreek{e}}_{0}}|\text{\textgreek{f}}|^{2}\Big)\Bigg\}\, dg+\\
+\int_{\mathcal{R}(\text{\textgreek{t}}_{1},\text{\textgreek{t}}_{2})\cap\{\frac{1}{2}R\le r\le R\}}(f+\tilde{f})v_{s}^{\prime}(\frac{r}{R})\Big( & c_{\text{\textgreek{d}}_{0}}R^{-2}s|T\text{\textgreek{f}}|^{2}-C_{\text{\textgreek{d}}_{0}}R^{-4}s^{3}|\text{\textgreek{f}}|^{2}\Big)\, dg+\\
+\int_{\mathcal{R}(\text{\textgreek{t}}_{1},\text{\textgreek{t}}_{2})\cap\{r\ge R\}}(f(R)+\tilde{f}(R))\Big( & cr^{-2}|T\text{\textgreek{f}}|^{2}-CR^{-1}r^{-3}|\text{\textgreek{f}}|^{2}\Big)\, dg\le\\
\le & C_{\text{\textgreek{d}}_{0}\text{\textgreek{d}}}\int_{\mathcal{R}(\text{\textgreek{t}}_{1},\text{\textgreek{t}}_{2})\cap\mathscr{E}_{\text{\textgreek{d}}}}(f+\tilde{f})\Big\{ s^{2}R^{-6\text{\textgreek{e}}_{0}}|\nabla\text{\textgreek{f}}|_{g_{ref}}^{2}+s^{4}R^{-12\text{\textgreek{e}}_{0}}|\text{\textgreek{f}}|^{2}\Big\}\, dg-\\
 & -\int_{\mathcal{R}(\text{\textgreek{t}}_{1},\text{\textgreek{t}}_{2})}Re\big\{ G\big(2(\nabla^{\text{\textgreek{m}}}f+\nabla^{\text{\textgreek{m}}}\tilde{f})\nabla_{\text{\textgreek{m}}}\bar{\text{\textgreek{f}}}+(\square_{g}f+\square_{g}\tilde{f}-2h-2\tilde{h})\bar{\text{\textgreek{f}}}\big)\big\}\, dg+\\
 & +C\bar{\mathcal{B}}_{f,h;\tilde{f},\tilde{h}}^{(R)}[\text{\textgreek{f}};\text{\textgreek{t}}_{1},\text{\textgreek{t}}_{2}],-\mathcal{B}_{f,h}^{(R)}[\text{\textgreek{f}};\text{\textgreek{t}}_{1},\text{\textgreek{t}}_{2}]-\mathcal{B}_{\tilde{f},\tilde{h}}^{(R)}[\text{\textgreek{f}};\text{\textgreek{t}}_{1},\text{\textgreek{t}}_{2}]+\\
 & +C\sup_{\{r\le\frac{1}{4}r_{0}\}}f\int_{\text{\textgreek{S}}_{\text{\textgreek{t}}_{1}}\cap\{r\le r_{0}\}}|\nabla\text{\textgreek{f}}|_{g_{ref}}^{2},
\end{split}
\label{eq:MainCarlemanEstimateFinal}
\end{equation}
where 
\begin{equation}
\begin{split}\bar{\mathcal{B}}_{f,h;\tilde{f},\tilde{h}}^{(R)}[\text{\textgreek{f}};\text{\textgreek{t}}_{1},\text{\textgreek{t}}_{2}]\doteq & \sum_{j=1}^{2}\Bigg|\int_{\text{\textgreek{S}}_{\text{\textgreek{t}}_{j}}}Re\Big\{\Big(2\nabla^{\text{\textgreek{m}}}f\nabla_{\text{\textgreek{m}}}\bar{\text{\textgreek{f}}}\nabla_{\text{\textgreek{n}}}\text{\textgreek{f}}+(\square_{g}f-2h)\text{\textgreek{f}}\nabla_{\text{\textgreek{n}}}\bar{\text{\textgreek{f}}}-\nabla_{\text{\textgreek{n}}}f\nabla^{\text{\textgreek{m}}}\text{\textgreek{f}}\nabla_{\text{\textgreek{m}}}\bar{\text{\textgreek{f}}}+\\
 & \hphantom{\sum_{j=1}^{2}(-1)^{j}\int_{\text{\textgreek{S}}_{\text{\textgreek{t}}_{j}}}Re\big\{}+\big(\text{\textgreek{q}}_{\le R}f^{-1}\nabla_{\text{\textgreek{m}}}\nabla_{\text{\textgreek{n}}}f\nabla^{\text{\textgreek{m}}}f+\nabla_{\text{\textgreek{n}}}h-\frac{1}{2}(\nabla_{\text{\textgreek{n}}}(\square_{g}f))\big)|\text{\textgreek{f}}|^{2}\Big)n_{\text{\textgreek{S}}_{\text{\textgreek{t}}_{j}}}^{\text{\textgreek{n}}}\Big\}\, dg_{\text{\textgreek{S}}_{\text{\textgreek{t}}_{j}}}\Bigg|+\\
 & +\sum_{j=1}^{2}\Bigg|\int_{\text{\textgreek{S}}_{\text{\textgreek{t}}_{j}}}Re\Big\{\Big(2\nabla^{\text{\textgreek{m}}}\tilde{f}\nabla_{\text{\textgreek{m}}}\bar{\text{\textgreek{f}}}\nabla_{\text{\textgreek{n}}}\text{\textgreek{f}}+(\square_{g}\tilde{f}-2\tilde{h})\text{\textgreek{f}}\nabla_{\text{\textgreek{n}}}\bar{\text{\textgreek{f}}}-\nabla_{\text{\textgreek{n}}}\tilde{f}\nabla^{\text{\textgreek{m}}}\text{\textgreek{f}}\nabla_{\text{\textgreek{m}}}\bar{\text{\textgreek{f}}}+\\
 & \hphantom{+\sum_{j=1}^{2}(-1)^{j}\int_{\text{\textgreek{S}}_{\text{\textgreek{t}}_{j}}}Re\big\{}+\big(\text{\textgreek{q}}_{\le R}\tilde{f}^{-1}\nabla_{\text{\textgreek{m}}}\nabla_{\text{\textgreek{n}}}\tilde{f}\nabla^{\text{\textgreek{m}}}\tilde{f}+\nabla_{\text{\textgreek{n}}}\tilde{h}-\frac{1}{2}(\nabla_{\text{\textgreek{n}}}(\square_{g}\tilde{f}))\big)|\text{\textgreek{f}}|^{2}\Big)n_{\text{\textgreek{S}}_{\text{\textgreek{t}}_{j}}}^{\text{\textgreek{n}}}\Big\}\, dg_{\text{\textgreek{S}}_{\text{\textgreek{t}}_{j}}}\Bigg|+\\
 & +\sup_{\{r\le\frac{1}{4}r_{0}\}}f\int_{\text{\textgreek{S}}_{\text{\textgreek{t}}_{1}}\cap\{r\le r_{0}\}}|\nabla\text{\textgreek{f}}|_{g_{ref}}^{2}
\end{split}
\end{equation}
 (note that the boundary terms on $\mathcal{H}^{+}$in the right hand
side of (\ref{eq:MainCarlemanEstimateFromundamentalCurrent}) where
absorbed by the term in the left hand side of (\ref{eq:FromRedShift})).
Inequality (\ref{eq:MainCarlemanEstimate}) now readily follows from
(\ref{eq:MainCarlemanEstimateFinal}) in view of (\ref{eq:Function_F}),
(\ref{eq:Function_F-1}). \qed

\subsection{\label{sub:ProofOfCorollaryCarleman}Proof of Corollary \ref{cor:CarlemanForPsik}}

For any $1\le k\le n$ and $0<\text{\textgreek{d}}_{1},\text{\textgreek{d}}_{2},\text{\textgreek{e}}_{0}\ll1$,
let us choose the parameters $R,s$ to be sufficiently large in terms
of $\text{\textgreek{d}}_{1},\text{\textgreek{d}}_{2},\text{\textgreek{e}}_{0}$
and the geometry of $(\mathcal{M},g)$, satisfying in addition: 
\begin{gather}
R\ge C_{\text{\textgreek{d}}_{1}\text{\textgreek{e}}_{0}}\max\{1,\text{\textgreek{w}}_{k}{}^{-\frac{1}{1-9\text{\textgreek{e}}_{0}}},\big(-\log\text{\textgreek{d}}_{2}\big)^{\frac{1}{1-9\text{\textgreek{e}}_{0}}}\},\label{eq:LowerBoundR}\\
C_{\text{\textgreek{d}}_{1}\text{\textgreek{e}}_{0}}^{\frac{1}{3}}\max\big\{(1+\text{\textgreek{w}}_{k})R^{9\text{\textgreek{e}}_{0}},-\log\text{\textgreek{d}}_{2}\big\}\le s\le C_{\text{\textgreek{d}}_{1}\text{\textgreek{e}}_{0}}^{-\frac{1}{3}}R\text{\textgreek{w}}_{k},\label{eq:BoundsS}
\end{gather}
for some constant $C_{\text{\textgreek{d}}_{1}\text{\textgreek{e}}_{0}}>1$
large in terms of $\text{\textgreek{d}}_{1},\text{\textgreek{e}}_{0}$
and the geometry of $(\mathcal{M},g)$ (notice that the bound (\ref{eq:LowerBoundR})
guarantees that an $s$ satisfying (\ref{eq:BoundsS}) exists).

By approximating the functions $\text{\textgreek{y}}_{k}$, $1\le k\le n$,
by smooth solutions to (\ref{eq:WaveEquation}) with compact support
in space and using Lemma \ref{lem:DecayPsiKNearSpacelikeInfinity}
on the decay of $\text{\textgreek{y}}_{k}$ as $r\rightarrow+\infty$,
we infer that Proposition \ref{prop:GeneralCarlemanEstimate} also
applies for the functions $\text{\textgreek{y}}_{k}$. Therefore,
using the values of $s,R$ chosen above, we obtain for any $0\le\text{\textgreek{t}}_{1}\le\text{\textgreek{t}}_{2}$
and any $1\le k\le n$: 
\begin{equation}
\begin{split}\int_{\mathcal{R}(\text{\textgreek{t}}_{1},\text{\textgreek{t}}_{2})\cap\{r\le R_{0}\}\backslash\mathscr{E}_{\text{\textgreek{d}}_{1}}}(f+\inf_{\{r\ge\frac{1}{4}r_{0}\}\backslash\mathscr{E}}f)\Bigg\{ & sR^{-3\text{\textgreek{e}}_{0}}|\nabla_{g_{\text{\textgreek{S}}}}\text{\textgreek{y}}_{k}|_{g_{\text{\textgreek{S}}}}^{2}-C_{\text{\textgreek{d}}_{1}}sR^{-3\text{\textgreek{e}}_{0}}|T\text{\textgreek{y}}_{k}|^{2}+s^{3}R^{-9\text{\textgreek{e}}_{0}}|\text{\textgreek{y}}_{k}|^{2}\Bigg\}\, dg+\\
+\int_{\mathcal{R}(\text{\textgreek{t}}_{1},\text{\textgreek{t}}_{2})\cap\{R_{0}\le r\le\frac{1}{2}R\}}f\Bigg\{ & sR^{-3\text{\textgreek{e}}_{0}}r^{-\frac{5}{2}}\big(\big|\partial_{r}\text{\textgreek{y}}_{k}\big|^{2}+r^{-2}|\partial_{\text{\textgreek{sv}}}\text{\textgreek{y}}_{k}|^{2}\big)+sR^{-3\text{\textgreek{e}}_{0}}r^{-2}|T\text{\textgreek{y}}_{k}|^{2}+\text{\textgreek{e}}_{0}s^{3}R^{-9\text{\textgreek{e}}_{0}}r^{-4}|\text{\textgreek{y}}_{k}|^{2}\Bigg\}\, dg+\\
+\int_{\mathcal{R}(\text{\textgreek{t}}_{1},\text{\textgreek{t}}_{2})\cap\{\frac{1}{2}R\le r\le R\}}f\Bigg\{ & r^{-\frac{5}{2}}\big(\big|\partial_{r}\text{\textgreek{y}}_{k}\big|^{2}+r^{-2}|\partial_{\text{\textgreek{sv}}}\text{\textgreek{y}}_{k}|^{2}\big)+R\partial_{r}w_{R}\big(cR^{-2}s|T\text{\textgreek{y}}_{k}|^{2}-CR^{-4}s^{3}|\text{\textgreek{y}}_{k}|^{2}\big)\Bigg\}\, dg+\\
+\int_{\mathcal{R}(\text{\textgreek{t}}_{1},\text{\textgreek{t}}_{2})\cap\{r\ge R\}}f(R)\Bigg\{ & r^{-\frac{5}{2}}\big(\big|\partial_{r}\text{\textgreek{y}}_{k}\big|^{2}+r^{-2}|\partial_{\text{\textgreek{sv}}}\text{\textgreek{y}}_{k}|^{2}\big)+r^{-2}|T\text{\textgreek{y}}_{k}|^{2}-CR^{-1}r^{-3}|\text{\textgreek{y}}_{k}|^{2}\Bigg\}\, dg\le\\
\le & C_{\text{\textgreek{d}}_{1}}\int_{\mathcal{R}(\text{\textgreek{t}}_{1},\text{\textgreek{t}}_{2})\cap\mathscr{E}_{\text{\textgreek{d}}_{1}}}f\Big\{ s^{2}R^{-6\text{\textgreek{e}}_{0}}|\nabla\text{\textgreek{y}}_{k}|_{g_{ref}}^{2}+s^{4}R^{-12\text{\textgreek{e}}_{0}}|\text{\textgreek{y}}_{k}|^{2}\Big\}\, dg+\\
 & +C\Big|\int_{\mathcal{R}(\text{\textgreek{t}}_{1},\text{\textgreek{t}}_{2})}F_{k}\big(\nabla^{\text{\textgreek{m}}}f\nabla_{\text{\textgreek{m}}}\bar{\text{\textgreek{y}}}_{k}+O\big(\sum_{j=1}^{2}(1+r)^{j-2}|\nabla^{j}f|_{g_{ref}}\big)\bar{\text{\textgreek{y}}}_{k}\big)\, dg\Big|+\\
 & +C\sum_{j=1}^{2}\int_{\text{\textgreek{S}}_{\text{\textgreek{t}}_{j}}}\Big(|\nabla f|_{g_{ref}}|\nabla\text{\textgreek{y}}_{k}|_{g_{ref}}^{2}+\big(\sum_{j=1}^{3}(1+r)^{j-3}|\nabla^{j}f|_{g_{ref}}\big)|\text{\textgreek{y}}_{k}|^{2}\Big)\, dg_{\text{\textgreek{S}}}.
\end{split}
\label{eq:MainCarlemanEstimate-1}
\end{equation}
In view Lemma \ref{lem:DtToOmegaInequalities}, the bound (\ref{eq:MainCarlemanEstimate-1})
implies: 
\begin{equation}
\begin{split}\int_{\mathcal{R}(\text{\textgreek{t}}_{1},\text{\textgreek{t}}_{2})\cap\{r\le R_{0}\}\backslash\mathscr{E}_{\text{\textgreek{d}}_{1}}}(f+\inf_{\{r\ge\frac{1}{4}r_{0}\}\backslash\mathscr{E}}f)\Bigg\{ & sR^{-3\text{\textgreek{e}}_{0}}|\nabla_{g_{\text{\textgreek{S}}}}\text{\textgreek{y}}_{k}|_{g_{\text{\textgreek{S}}}}^{2}+sR^{-3\text{\textgreek{e}}_{0}}|T\text{\textgreek{y}}_{k}|^{2}+(s^{3}R^{-9\text{\textgreek{e}}_{0}}-C_{\text{\textgreek{d}}_{1}}\text{\textgreek{w}}_{k}^{2}sR^{-3\text{\textgreek{e}}_{0}})|\text{\textgreek{y}}_{k}|^{2}\Bigg\}\, dg+\\
+\int_{\mathcal{R}(\text{\textgreek{t}}_{1},\text{\textgreek{t}}_{2})\cap\{R_{0}\le r\le\frac{1}{2}R\}}f\Bigg\{ & sR^{-3\text{\textgreek{e}}_{0}}r^{-\frac{5}{2}}\big(\big|\partial_{r}\text{\textgreek{y}}_{k}\big|^{2}+r^{-2}|\partial_{\text{\textgreek{sv}}}\text{\textgreek{y}}_{k}|^{2}\big)+sR^{-3\text{\textgreek{e}}_{0}}r^{-2}|T\text{\textgreek{y}}_{k}|^{2}+\text{\textgreek{e}}_{0}s^{3}R^{-9\text{\textgreek{e}}_{0}}r^{-4}|\text{\textgreek{y}}_{k}|^{2}\Bigg\}\, dg+\\
+\int_{\mathcal{R}(\text{\textgreek{t}}_{1},\text{\textgreek{t}}_{2})\cap\{\frac{1}{2}R\le r\le R\}}f\Bigg\{ & r^{-\frac{5}{2}}\big(\big|\partial_{r}\text{\textgreek{y}}_{k}\big|^{2}+r^{-2}|\partial_{\text{\textgreek{sv}}}\text{\textgreek{y}}_{k}|^{2}\big)+R^{-1}s\partial_{r}w_{R}|T\text{\textgreek{y}}_{k}|^{2}+R^{-1}s\partial_{r}w_{R}\big(c\text{\textgreek{w}}_{k}^{2}-CR^{-2}s^{2}\big)|\text{\textgreek{y}}_{k}|^{2}\Bigg\}\, dg+\\
+\int_{\mathcal{R}(\text{\textgreek{t}}_{1},\text{\textgreek{t}}_{2})\cap\{r\ge R\}}f(R)\Bigg\{ & r^{-\frac{5}{2}}\big(\big|\partial_{r}\text{\textgreek{y}}_{k}\big|^{2}+r^{-2}|\partial_{\text{\textgreek{sv}}}\text{\textgreek{y}}_{k}|^{2}\big)+cr^{-2}|T\text{\textgreek{y}}_{k}|^{2}+\big(c\text{\textgreek{w}}_{k}^{2}-CR^{-2}\big)r^{-2}|\text{\textgreek{y}}_{k}|^{2}\Bigg\}\, dg\le\\
\le & C_{\text{\textgreek{d}}_{1}}\int_{\mathcal{R}(\text{\textgreek{t}}_{1},\text{\textgreek{t}}_{2})\cap\mathscr{E}_{\text{\textgreek{d}}_{1}}}f\Big\{ s^{2}R^{-6\text{\textgreek{e}}_{0}}|\nabla\text{\textgreek{y}}_{k}|_{g_{ref}}^{2}+s^{4}R^{-12\text{\textgreek{e}}_{0}}|\text{\textgreek{y}}_{k}|^{2}\Big\}\, dg+\\
 & +C\Big|\int_{\mathcal{R}(\text{\textgreek{t}}_{1},\text{\textgreek{t}}_{2})}F_{k}\big(\nabla^{\text{\textgreek{m}}}f\nabla_{\text{\textgreek{m}}}\bar{\text{\textgreek{y}}}_{k}+O\big(\sum_{j=1}^{2}(1+r)^{j-2}|\nabla^{j}f|_{g_{ref}}\big)\bar{\text{\textgreek{y}}}_{k}\big)\, dg\Big|+\\
 & +C\sum_{j=1}^{2}\int_{\text{\textgreek{S}}_{\text{\textgreek{t}}_{j}}}\Big(|\nabla f|_{g_{ref}}|\nabla\text{\textgreek{y}}_{k}|_{g_{ref}}^{2}+\big(\sum_{j=1}^{3}(1+r)^{j-3}|\nabla^{j}f|_{g_{ref}}\big)|\text{\textgreek{y}}_{k}|^{2}\Big)\, dg_{\text{\textgreek{S}}}+\\
 & +C\int_{\mathcal{R}(\text{\textgreek{t}}_{1},\text{\textgreek{t}}_{2})\cap\mathcal{H}^{+}}\Big(|\nabla f|_{g_{ref}}J_{\text{\textgreek{m}}}^{N}(\text{\textgreek{y}}_{k})n_{\mathcal{H}^{+}}^{\text{\textgreek{m}}}+\big(\sum_{j=1}^{3}(1+r)^{j-3}|\nabla^{j}f|_{g_{ref}}\big)|\text{\textgreek{y}}_{k}|^{2}\Big)\, dvol_{\mathcal{H}^{+}}.\\
 & +C\text{\textgreek{w}}_{k}^{2}(1+\text{\textgreek{w}}_{k}^{-2})\big(\log(2+\text{\textgreek{t}}_{2})\big)^{4}R_{0}^{2}\sup_{\{r\le R_{0}\}}f\cdot\mathcal{E}_{log}[\text{\textgreek{y}}]+\\
 & +C\text{\textgreek{w}}_{k}^{2}(1+\text{\textgreek{w}}_{k}^{-6})\sup_{\{r\le R\}}f\cdot\mathcal{E}_{log}[\text{\textgreek{y}}].
\end{split}
\label{eq:MainCarlemanEstimate-1-1}
\end{equation}

In view of the bound (\ref{eq:BoundsS}) for the parameters $R,s$,
as well as the properties of the function (\ref{eq:FunctionInCarlemanProposition}),
inequality (\ref{eq:MainCarlemanEstimate-1-1}) yields (using also
(\ref{eq:EnergyClassAPrioriBoundLocal}) and Lemma \ref{lem:BoundF},
combined with a Cauchy--Schwarz inequality, to estimate the second
and third terms in the right hand side of (\ref{eq:MainCarlemanEstimate-1-1})):
\begin{align}
 & \int_{\mathcal{R}(\text{\textgreek{t}}_{1},\text{\textgreek{t}}_{2})\backslash\mathscr{E}_{2\text{\textgreek{d}}_{1}}}\big((1+r)^{-\frac{5}{2}}|\nabla\text{\textgreek{y}}_{k}|_{g_{ref}}^{2}+(\text{\textgreek{w}}_{k}^{2}r^{-2}+r^{-4})|\text{\textgreek{y}}_{k}|^{2}\big)\, dg\le\label{eq:AlmostDoneWithTheUsefulCarleman}\\
 & \hphantom{++}\le\sum_{j=1}^{4}\big(|\nabla^{j}w_{R}|_{g_{ref}}+|\nabla^{j}\tilde{w}_{R}|_{g_{ref}}\big)\Bigg\{ C_{\text{\textgreek{d}}_{1}}\frac{\sup_{\mathscr{E}_{\text{\textgreek{d}}_{1}}}\big(e^{2sw_{R}}+e^{2s\tilde{w}_{R}}\big)}{\inf_{\{r\ge\frac{1}{4}r_{0}\}\backslash\mathscr{E}_{2\text{\textgreek{d}}_{1}}}\big(e^{2sw_{R}}+e^{2s\tilde{w}_{R}}\big)}(sR^{-3\text{\textgreek{e}}_{0}})\int_{\mathcal{R}(\text{\textgreek{t}}_{1},\text{\textgreek{t}}_{2})\backslash\mathscr{E}_{\text{\textgreek{d}}_{1}}}\big(|\nabla\text{\textgreek{y}}_{k}|_{g_{ref}}^{2}+|\text{\textgreek{y}}_{k}|^{2}\big)\, dg+\nonumber \\
 & \hphantom{++\le\sum_{j=1}^{4}\big(|\nabla^{j}w_{R}|_{g_{ref}}+|\nabla^{j}\tilde{w}_{R}|_{g_{ref}}\big)\Bigg\{}+C_{\text{\textgreek{d}}_{1}}(1+\text{\textgreek{w}}_{k}^{-10})\big(\log(2+\text{\textgreek{t}}_{2})\big)^{4}\frac{\sup_{\{r\le R\}}\big(e^{2sw_{R}}+e^{2s\tilde{w}_{R}}\big)}{\inf_{\{r\le R\}}\big(e^{2sw_{R}}+e^{2s\tilde{w}_{R}}\big)}\mathcal{E}_{log}[\text{\textgreek{y}}]\Bigg\}.\nonumber 
\end{align}
In view of the properties of the function \ref{eq:FunctionInCarlemanProposition},
we can estimate 
\begin{equation}
\sup_{\{r\le R\}}w_{R}-\inf_{\{r\le R\}}w_{R}+\sup_{\{r\le R\}}\tilde{w}_{R}-\inf_{\{r\le R\}}\tilde{w}_{R}\le C\text{\textgreek{e}}_{0}^{-1}R^{3\text{\textgreek{e}}_{0}},
\end{equation}
\begin{equation}
\inf_{\{\frac{1}{4}r_{0}\le R\}\backslash\mathscr{E}_{2\text{\textgreek{d}}_{1}}}w_{R}\ge\max_{\mathscr{E}_{\text{\textgreek{d}}_{1}}}w_{R}+c_{\text{\textgreek{d}}_{1}}R^{-3\text{\textgreek{e}}_{0}},
\end{equation}
\begin{equation}
\inf_{\{\frac{1}{4}r_{0}\le r\le R\}\backslash\mathscr{E}_{2\text{\textgreek{d}}_{1}}}\tilde{w}_{R}\ge\max_{\mathscr{E}_{\text{\textgreek{d}}_{1}}}\tilde{w}_{R}+c_{\text{\textgreek{d}}_{1}}R^{-3\text{\textgreek{e}}_{0}}
\end{equation}
and 
\begin{equation}
\sum_{j=1}^{4}\big(|\nabla^{j}w_{R}|_{g_{ref}}+|\nabla^{j}\tilde{w}_{R}|_{g_{ref}}\big)\le C.
\end{equation}
Therefore, inequality (\ref{eq:CarlemanEstimate}) readily follows
from \ref{eq:AlmostDoneWithTheUsefulCarleman}, provided $C_{\text{\textgreek{d}}_{1}\text{\textgreek{e}}_{0}}$
in (\ref{eq:BoundsS}) is sufficiently large in terms of $\text{\textgreek{d}}_{1},\text{\textgreek{e}}_{0}$.
\qed

\subsection{\label{sub:ProofWithBoundaryConditions}Proof of Proposition \ref{prop:GeneralCarlemanEstimate}
in the case of Dirchlet or Neumann boundary conditions}

In this section, we will briefly sketch how the proof of Proposition
\ref{prop:GeneralCarlemanEstimate} can be applied to the case when
the boundary $\partial\mathcal{M}$ of $(\mathcal{M},g)$ is allowed
to have a non-trivial timelike component $\partial_{tim}\mathcal{M}$
and equation (\ref{eq:InhomogeneousWaveEquation}) is supplemented
with Dirichlet or Neumann boundary conditions for $\text{\textgreek{f}}$
on $\partial_{tim}\mathcal{M}$. 

We will first describe the class of Lorentzian manifolds with such
a boundary component on which Proposition \ref{prop:GeneralCarlemanEstimate}
will apply. Let $(\mathcal{M}^{d+1},g)$, $d\ge2$, be a smooth Lorentzian
manifold with piecewise smooth boundary $\partial\mathcal{M}$ splitting
as 
\begin{equation}
\partial\mathcal{M}=\partial_{hor}\mathcal{M}\cup\partial_{tim}\mathcal{M},
\end{equation}
where $\partial_{hor}\mathcal{M}$ has the structure of a piecewise
smooth null hypersurface and $\partial_{tim}\mathcal{M}$ is a smooth
timelike hypersurface, with $\partial_{hor}\mathcal{M}\cap\partial_{tim}\mathcal{M}=\emptyset$.
For the discussion of this section, we will assume that $\partial_{tim}\mathcal{M}\neq\emptyset$,
but $\partial_{hor}\mathcal{M}$ will be allowed to be empty. Let
$(\widetilde{\mathcal{M}},\tilde{g})$ be the double of $(\mathcal{M},g)$
across $\partial_{tim}\mathcal{M}$, which is defined as the disjoint
union of two copies of $(\mathcal{M},g)$ glued along $\partial_{tim}\mathcal{M}$
(for the relevant definitions, see e.\,g.~\cite{Lee2003}). Let
$i_{1},i_{2}:\mathcal{M}\rightarrow\widetilde{\mathcal{M}}$ be the
two natural isometric embeddings of $(\mathcal{M},g)$ into $(\widetilde{\mathcal{M}},\tilde{g})$.
Note that $\widetilde{\mathcal{M}}=i_{1}(\mathcal{M})\cup i_{2}(\mathcal{M})$
and $i_{1}(\partial_{tim}\mathcal{M})=i_{2}(\partial_{tim}\mathcal{M})$.
Furthermore, $\widetilde{\mathcal{M}}$ is a smooth manifold, and
the metric $\tilde{g}$ is continuous and piecewise smooth on $\widetilde{\mathcal{M}}$
and smooth on $\widetilde{\mathcal{M}}\backslash i_{1}(\partial_{tim}\mathcal{M})$.%
\footnote{The metric $\tilde{g}$ is continuous across $i_{1}(\partial_{tim}\mathcal{M})$,
but fails to be $C^{1}$ at all the points of $i_{1}(\partial_{tim}\mathcal{M})$
on which the second fundamental form of $i_{1}(\partial_{tim}\mathcal{M})$
is non-zero.%
} We will always identify $\mathcal{M}$ with $i_{1}(\mathcal{M})\subset\widetilde{\mathcal{M}}$. 

We will assume that $(\widetilde{\mathcal{M}},\tilde{g})$ is a globally
hyperbolic Lorentzian manifold (with the regularity of $\tilde{g}$
as described before), satisfying Assumptions \hyperref[Assumption 1]{G1},
\hyperref[Assumption 2]{G2} and \hyperref[Assumption 3]{G3} of Section
\ref{sec:StatementAssumptionsResults} (for the discussion of this
Section, we can also allow the case $\mathscr{E}=\emptyset$). Additionally,
we will assume that the stationary Killing field $T$ of $\widetilde{\mathcal{M}}$
(defined by Assumption \hyperref[Assumption 1]{G1}) is tangent to
$i_{1}(\partial_{tim}\mathcal{M})$. Let also $\text{\textgreek{S}}_{\widetilde{\mathcal{M}}},\mathcal{S}_{\widetilde{\mathcal{M}}},\mathcal{H}_{\widetilde{\mathcal{M}}}^{\pm}\subset\widetilde{\mathcal{M}}$,
$t_{\widetilde{\mathcal{M}}}:\widetilde{\mathcal{M}}\backslash\mathcal{H}_{\widetilde{\mathcal{M}}}^{-}\rightarrow\mathbb{R}$
and $r_{\widetilde{\mathcal{M}}}:\widetilde{\mathcal{M}}\backslash\mathcal{H}_{\widetilde{\mathcal{M}}}^{-}\rightarrow[0,+\infty)$
be as defined under Assumption \hyperref[Assumption 1]{G1}. We will
assume without loss of generality that $\text{\textgreek{S}}_{\widetilde{\mathcal{M}}},\mathcal{S}_{\widetilde{\mathcal{M}}}$
intersect $i_{1}(\partial_{tim}\mathcal{M})$ transversally, and that
$\text{\textgreek{S}}_{\widetilde{\mathcal{M}}}\cap i_{1}(\partial_{tim}\mathcal{M}),\mathcal{S}_{\widetilde{\mathcal{M}}}\cap i_{1}(\partial_{tim}\mathcal{M})$
are compact. Note that the restriction of $\mathcal{H}_{\widetilde{\mathcal{M}}}$
on $\mathcal{M}$ coincides with $\partial_{hor}\mathcal{M}$. 
\begin{rem*}
We will use the notation $\text{\textgreek{S}},\mathcal{S},\mathcal{H}^{\pm},t$
and $r$ for the restriction of the hypersurfaces $\text{\textgreek{S}}_{\widetilde{\mathcal{M}}},\mathcal{S}_{\widetilde{\mathcal{M}}},\mathcal{H}_{\widetilde{\mathcal{M}}}^{\pm}$
and the functions $t_{\widetilde{\mathcal{M}}},r_{\widetilde{\mathcal{M}}}$
on $\mathcal{M}\simeq i_{1}(\mathcal{M})$. 
\end{rem*}
For any $F\in C^{\infty}(\mathcal{M})$ and any $(\text{\textgreek{f}}_{0},\text{\textgreek{f}}_{1})\in C^{\infty}(\text{\textgreek{S}})\times C^{\infty}(\text{\textgreek{S}})$,
the initial-boundary value problem 
\begin{equation}
\begin{cases}
\square_{g}\text{\textgreek{f}}=G & \mbox{ on }\{t\ge0\}\\
(\text{\textgreek{f}},T\text{\textgreek{f}})=(\text{\textgreek{f}}_{0},\text{\textgreek{f}}_{1}) & \mbox{ on }\{t=0\}\\
\text{\textgreek{f}}=0 & \mbox{ on }\partial_{tim}\mathcal{M}
\end{cases}\label{eq:DirichletProblem}
\end{equation}
is well posed on $\{t\ge0\}\subset\mathcal{M}$. This follows from
the assumption that $(\widetilde{\mathcal{M}},\tilde{g})$ is globally
hyperbolic. The Dirichlet boundary condition $\text{\textgreek{f}}|_{\partial_{tim}\mathcal{M}}=0$
in (\ref{eq:DirichletProblem}) can also be replaced by the Neumann
boundary condition 
\begin{equation}
n_{\partial_{tim}\mathcal{M}}(\text{\textgreek{f}})|_{\partial_{tim}\mathcal{M}}=0,\label{eq:NeumannBoundaryCondition}
\end{equation}
where $n_{\partial_{tim}\mathcal{M}}$ is the unit normal vector field
on $\partial_{tim}\mathcal{M}$, pointing towards the interior of
$\mathcal{M}$.

On a spacetime $(\mathcal{M},g)$ as above, we will extend Proposition
\ref{prop:GeneralCarlemanEstimate} as follows:
\begin{prop}
\label{prop:CarlemanBoundary} Let $(\mathcal{M},g)$ be a Lorentzian
manifold with boundary as above. For any $s,R\gg1$ sufficiently large
in terms of the geometry of $(\mathcal{M},g)$ and any $0<\text{\textgreek{e}}_{0}<1$,
there exists a a smooth $T$-invariant function $f:\mathcal{M}\backslash\mathcal{H}^{-}\rightarrow(0,+\infty)$
as in Proposition \ref{prop:GeneralCarlemanEstimate}, so that (provided
$\text{\textgreek{e}}_{0}sR^{-9\text{\textgreek{e}}_{0}}\gg1$), for
any $0<\text{\textgreek{d}}\ll1$ , any $0\le\text{\textgreek{t}}_{1}\le\text{\textgreek{t}}_{2}$
and any smooth function $\text{\textgreek{f}}:\mathcal{M}\backslash\mathcal{H}^{-}\rightarrow\mathbb{C}$
with compact support on the hypersurfaces $\{t=const\}$ solving (\ref{eq:InhomogeneousWaveEquation})
and satisfying on $\partial_{tim}\mathcal{M}$ either the Dirichlet
condition $\text{\textgreek{f}}=0$ or the Neumann condition $n_{\partial_{tim}\mathcal{M}}(\text{\textgreek{f}})=0$,
the estimate (\ref{eq:MainCarlemanEstimate}) holds.\end{prop}
\begin{proof}
The proof of Proposition \ref{prop:CarlemanBoundary} follows in almost
exactly the same way as the proof of Proposition \ref{prop:GeneralCarlemanEstimate},
the only difference being the following: When using the multiplier
(\ref{eq:Multiplier}) for equation (\ref{eq:InhomogeneousWaveEquation})
as in Section \ref{sub:IntegrationsByParts} and after performing
the same integration-by-parts procedure, one obtains instead of (\ref{eq:FundamentalCurrentFull})
the following relation: 
\begin{equation}
\begin{split}\int_{\mathcal{R}(\text{\textgreek{t}}_{1},\text{\textgreek{t}}_{2})}Re\Big\{2\text{\textgreek{q}}_{\le R}f^{-1}\nabla^{\text{\textgreek{m}}}\nabla^{\text{\textgreek{n}}}f\nabla_{\text{\textgreek{m}}}(f^{\frac{1}{2}}\text{\textgreek{f}})\nabla_{\text{\textgreek{n}}}(f^{\frac{1}{2}}\bar{\text{\textgreek{f}}})-2\text{\textgreek{q}}_{\le R}\text{\textgreek{q}}_{\ge R_{0}} & r^{-1}f^{-1}(\partial_{r}f)\big|\partial_{r}(f^{\frac{1}{2}}\text{\textgreek{f}})\big|^{2}+\\
\hphantom{\int_{\mathcal{R}(\text{\textgreek{t}}_{1},\text{\textgreek{t}}_{2})}Re\Big\{}+2(1-\text{\textgreek{q}}_{\le R})\nabla^{\text{\textgreek{m}}}\nabla^{\text{\textgreek{n}}}f\nabla_{\text{\textgreek{m}}}\text{\textgreek{f}}\nabla_{\text{\textgreek{n}}}\bar{\text{\textgreek{f}}}+2\text{\textgreek{q}}_{\ge R_{0}} & \text{\textgreek{q}}_{\le R}r^{-1}(\partial_{r}f)|\partial_{r}\text{\textgreek{f}}|^{2}-2h\nabla^{\text{\textgreek{m}}}\text{\textgreek{f}}\nabla_{\text{\textgreek{m}}}\bar{\text{\textgreek{f}}}+\mathcal{A}_{f,h}^{(R)}|\text{\textgreek{f}}|^{2}\Big\}\, dg=\\
=-\int_{\mathcal{R}(\text{\textgreek{t}}_{1},\text{\textgreek{t}}_{2})}Re\big\{ & G\big(2\nabla^{\text{\textgreek{m}}}f\nabla_{\text{\textgreek{m}}}\bar{\text{\textgreek{f}}}+(\square_{g}f-2h)\bar{\text{\textgreek{f}}}\big)\big\}\, dg-\mathcal{B}_{f,h}^{(R)}[\text{\textgreek{f}};\text{\textgreek{t}}_{1},\text{\textgreek{t}}_{2}]-\mathcal{B}_{f,h}^{(b)}[\text{\textgreek{f}};\text{\textgreek{t}}_{1}\text{\textgreek{t}}_{2}],
\end{split}
\label{eq:FundamentalCurrentFull-2}
\end{equation}
where 
\begin{align}
\mathcal{B}_{f,h}^{(b)}[\text{\textgreek{f}};\text{\textgreek{t}}_{1}\text{\textgreek{t}}_{2}]=\int_{\partial_{tim}\mathcal{M}\cap\mathcal{R}(\text{\textgreek{t}}_{1},\text{\textgreek{t}}_{2})}Re\Big\{ & \Big(2\nabla^{\text{\textgreek{m}}}f\nabla_{\text{\textgreek{m}}}\bar{\text{\textgreek{f}}}\nabla_{\text{\textgreek{n}}}\text{\textgreek{f}}+(\square_{g}f-2h)\text{\textgreek{f}}\nabla_{\text{\textgreek{n}}}\bar{\text{\textgreek{f}}}-\nabla_{\text{\textgreek{n}}}f\nabla^{\text{\textgreek{m}}}\text{\textgreek{f}}\nabla_{\text{\textgreek{m}}}\bar{\text{\textgreek{f}}}+\label{eq:AdditionalBoundaryTerms}\\
 & +\big(f^{-1}\nabla_{\text{\textgreek{m}}}\nabla_{\text{\textgreek{n}}}f\nabla^{\text{\textgreek{m}}}f+\nabla_{\text{\textgreek{n}}}h-\frac{1}{2}(\nabla_{\text{\textgreek{n}}}(\square_{g}f))\big)|\text{\textgreek{f}}|^{2}\Big)n_{\partial_{tim}\mathcal{M}}^{\text{\textgreek{n}}}\Big\}\, dg_{\partial_{tim}\mathcal{M}},\nonumber 
\end{align}
 $g_{\partial_{tim}\mathcal{M}}$ being the induced (Lorentzian) metric
on $\partial_{tim}\mathcal{M}$. Notice that (\ref{eq:FundamentalCurrentFull-2})
differs from (\ref{eq:FundamentalCurrentFull}) only by the term (\ref{eq:AdditionalBoundaryTerms})
in the right hand side.

Let us assume, without loss of generality, that the function $\bar{w}$
of Lemma \ref{lem:DoubleExponentForTheCarleman} has been chosen so
that it additionally satisfies $n_{\partial_{tim}\mathcal{M}}(\bar{w})>0$,
with $\bar{w}$ being constant on $\partial_{tim}\mathcal{M}$ (it
can be readily checked that Lemma \ref{lem:DoubleExponentForTheCarleman}
can be established under this additional assumption). In the case
when $\text{\textgreek{f}}$ satisfies the Dirichlet boundary condition
$\text{\textgreek{f}}|_{\partial_{tim}\mathcal{M}}=0$, it is straightforward
to check that this choice of $\bar{w}$ implies (in view of the choice
of the functions $f,h$ in Section \ref{sub:ChoiceOfSeedFunctions})
that the term (\ref{eq:AdditionalBoundaryTerms}) is non-negative,
and in particular 
\begin{align}
\int_{\partial_{tim}\mathcal{M}\cap\mathcal{R}(\text{\textgreek{t}}_{1},\text{\textgreek{t}}_{2})}Re\Big\{ & \Big(2\nabla^{\text{\textgreek{m}}}f\nabla_{\text{\textgreek{m}}}\bar{\text{\textgreek{f}}}\nabla_{\text{\textgreek{n}}}\text{\textgreek{f}}+(\square_{g}f-2h)\text{\textgreek{f}}\nabla_{\text{\textgreek{n}}}\bar{\text{\textgreek{f}}}-\nabla_{\text{\textgreek{n}}}f\nabla^{\text{\textgreek{m}}}\text{\textgreek{f}}\nabla_{\text{\textgreek{m}}}\bar{\text{\textgreek{f}}}+\label{eq:AdditionalBoundaryTerms-1}\\
 & +\big(f^{-1}\nabla_{\text{\textgreek{m}}}\nabla_{\text{\textgreek{n}}}f\nabla^{\text{\textgreek{m}}}f+\nabla_{\text{\textgreek{n}}}h-\frac{1}{2}(\nabla_{\text{\textgreek{n}}}(\square_{g}f))\big)|\text{\textgreek{f}}|^{2}\Big)n_{\partial_{tim}\mathcal{M}}^{\text{\textgreek{n}}}\Big\}\, dg_{\partial_{tim}\mathcal{M}}\ge\nonumber \\
 & \hphantom{+\big(f^{-1}\nabla_{\text{\textgreek{m}}}\nabla_{\text{\textgreek{n}}}f\nabla^{\text{\textgreek{m}}}f+\nabla_{\text{\textgreek{n}}}h-}\ge c\int_{\partial_{tim}\mathcal{M}\cap\mathcal{R}(\text{\textgreek{t}}_{1},\text{\textgreek{t}}_{2})}n_{\partial_{tim}\mathcal{M}}(f)\big|n_{\partial_{tim}\mathcal{M}}(\text{\textgreek{f}})\big|^{2}\, dg_{\partial_{tim}\mathcal{M}}\ge0\nonumber 
\end{align}
 for some $c>0$. Thus, the term (\ref{eq:AdditionalBoundaryTerms})
can be dropped from the right hand side of (\ref{eq:FundamentalCurrentFull-2})
(thus yielding (\ref{eq:FundamentalCurrentFull})) and one can proceed
as before to establish (\ref{eq:MainCarlemanEstimate}).

In the case when $\text{\textgreek{f}}$ satisfies the Neumann boundary
condition $n_{\partial_{tim}\mathcal{M}}(\text{\textgreek{f}})|_{\partial_{tim}\mathcal{M}}=0$,
(\ref{eq:AdditionalBoundaryTerms}) is not necessarily non-negative,
since the term 
\[
n_{\partial_{tim}\mathcal{M}}(f)\nabla^{\text{\textgreek{m}}}\text{\textgreek{f}}\nabla_{\text{\textgreek{m}}}\bar{\text{\textgreek{f}}}
\]
 in (\ref{eq:AdditionalBoundaryTerms}) does not nessarily have a
sign (as is the case when $\text{\textgreek{f}}|_{\partial_{tim}\mathcal{M}}=0$).
In order to absorb this term, we proceed as follows: Let $\mathcal{U}\subset\mathcal{M}$
be a (small) $T$-invariant tubular neighborhood of $\partial_{tim}\mathcal{M}$
(so that $\mathcal{U}\cap\big(\mathcal{H}\cup\mathcal{B}_{crit}(8\text{\textgreek{d}}_{0})\big)=\emptyset$),
split as $\mathcal{U}\simeq[0,1)\times\partial_{tim}\mathcal{M}$,
where the projection onto the factor $[0,1)$ is given by a smooth
function $\bar{r}:\mathcal{U}\rightarrow[0,1)$ such that $\nabla^{\text{\textgreek{m}}}\bar{r}|_{\partial_{tim}\mathcal{M}}=n_{\partial_{tim}\mathcal{M}}^{\text{\textgreek{m}}}$,
and the projection onto $\partial_{tim}\mathcal{M}$ is given by a
smooth map $\bar{\text{\textgreek{sv}}}:\mathcal{U}\rightarrow\partial_{tim}\mathcal{M}$.
We will extend $n_{\partial_{tim}\mathcal{M}}$ on the whole of $\mathcal{U}$
by the relation 
\begin{equation}
n_{\partial_{tim}\mathcal{M}}^{\text{\textgreek{m}}}=\nabla^{\text{\textgreek{m}}}\bar{r}.
\end{equation}
Let $\text{\textgreek{q}}_{c}:[0,1)\rightarrow[0,1]$ be a smooth
function satisfying $\text{\textgreek{q}}_{c}\equiv1$ on $[0,\frac{1}{4}]$
and $\text{\textgreek{q}}_{c}\equiv0$ on $[\frac{1}{2},1)$, and
let us define the function $\breve{f}:\mathcal{M}\rightarrow\mathbb{R}$
by the relation 
\[
\breve{f}(\bar{r},\bar{\text{\textgreek{sv}}})\doteq\text{\textgreek{q}}_{c}(\bar{r})\cdot\big(n_{\partial_{tim}\mathcal{M}}(f)\big)|_{\partial_{tim}\mathcal{M}}(\bar{\text{\textgreek{sv}}})\mbox{ on }\mathcal{U}\simeq[0,1)\times\partial_{tim}\mathcal{M}
\]
 (where $\big(n_{\partial_{tim}\mathcal{M}}(f)\big)|_{\partial_{tim}\mathcal{M}}$
is the value of $n_{\partial_{tim}\mathcal{M}}(f)$ on $\{\bar{r}=0\}$)
and 
\[
\breve{f}\equiv0\mbox{ on }\mathcal{M}\backslash\mathcal{U}.
\]

Adding to (\ref{eq:FundamentalCurrentFull-2}) the identity 
\begin{equation}
\begin{split}\int_{\mathcal{R}(\text{\textgreek{t}}_{1},\text{\textgreek{t}}_{2})}Re\Big\{2 & \nabla^{\text{\textgreek{m}}}(\breve{f}n_{\partial_{tim}\mathcal{M}}^{\text{\textgreek{n}}})\nabla_{\text{\textgreek{m}}}\text{\textgreek{f}}\nabla_{\text{\textgreek{n}}}\bar{\text{\textgreek{f}}}-\nabla_{\text{\textgreek{n}}}(\breve{f}n_{\partial_{tim}\mathcal{M}}^{\text{\textgreek{n}}})\nabla^{\text{\textgreek{m}}}\text{\textgreek{f}}\nabla_{\text{\textgreek{m}}}\bar{\text{\textgreek{f}}}\Big\}\, dg=\\
= & -\int_{\mathcal{R}(\text{\textgreek{t}}_{1},\text{\textgreek{t}}_{2})}Re\Big\{ G\cdot2\breve{f}n_{\partial_{tim}\mathcal{M}}(\bar{\text{\textgreek{f}}})\Big\}\, dg-\\
 & -\int_{\partial_{tim}\mathcal{M}\cap\mathcal{R}(\text{\textgreek{t}}_{1},\text{\textgreek{t}}_{2})}Re\Big\{\big(2\breve{f}n_{\partial_{tim}\mathcal{M}}(\bar{\text{\textgreek{f}}})\nabla_{\text{\textgreek{n}}}\text{\textgreek{f}}-\breve{f}(n_{\partial_{tim}\mathcal{M}})_{\text{\textgreek{n}}}\nabla^{\text{\textgreek{m}}}\text{\textgreek{f}}\nabla_{\text{\textgreek{m}}}\bar{\text{\textgreek{f}}}\big)n_{\partial_{tim}\mathcal{M}}^{\text{\textgreek{n}}}\Big\}\, dg_{\partial_{tim}\mathcal{M}}
\end{split}
\label{eq:AddedIdentityForNeumann}
\end{equation}
and using the Neumann condition $n_{\partial_{tim}\mathcal{M}}(\text{\textgreek{f}})|_{\partial_{tim}\mathcal{M}}=0$,
we thus infer: 
\begin{equation}
\begin{split}\int_{\mathcal{R}(\text{\textgreek{t}}_{1},\text{\textgreek{t}}_{2})}Re\Big\{2\text{\textgreek{q}}_{\le R}f^{-1}\nabla^{\text{\textgreek{m}}}\nabla^{\text{\textgreek{n}}}f\nabla_{\text{\textgreek{m}}}(f^{\frac{1}{2}}\text{\textgreek{f}})\nabla_{\text{\textgreek{n}}}(f^{\frac{1}{2}}\bar{\text{\textgreek{f}}})+ & 2\nabla^{\text{\textgreek{m}}}(\breve{f}n_{\partial_{tim}\mathcal{M}}^{\text{\textgreek{n}}})\nabla_{\text{\textgreek{m}}}\text{\textgreek{f}}\nabla_{\text{\textgreek{n}}}\bar{\text{\textgreek{f}}}-2\text{\textgreek{q}}_{\le R}\text{\textgreek{q}}_{\ge R_{0}}r^{-1}f^{-1}(\partial_{r}f)\big|\partial_{r}(f^{\frac{1}{2}}\text{\textgreek{f}})\big|^{2}+\\
\hphantom{\int_{\mathcal{R}(}}+2(1-\text{\textgreek{q}}_{\le R})\nabla^{\text{\textgreek{m}}}\nabla^{\text{\textgreek{n}}}f\nabla_{\text{\textgreek{m}}}\text{\textgreek{f}}\nabla_{\text{\textgreek{n}}}\bar{\text{\textgreek{f}}}+2\text{\textgreek{q}}_{\ge R_{0}}\text{\textgreek{q}}_{\le R}r^{-1} & (\partial_{r}f)|\partial_{r}\text{\textgreek{f}}|^{2}-\big(2h+\nabla_{\text{\textgreek{n}}}(\breve{f}n_{\partial_{tim}\mathcal{M}}^{\text{\textgreek{n}}})\big)\nabla^{\text{\textgreek{m}}}\text{\textgreek{f}}\nabla_{\text{\textgreek{m}}}\bar{\text{\textgreek{f}}}+\mathcal{A}_{f,h}^{(R)}|\text{\textgreek{f}}|^{2}\Big\}\, dg=\\
=-\int_{\mathcal{R}(\text{\textgreek{t}}_{1},\text{\textgreek{t}}_{2})}Re\big\{ & G\big(2\nabla^{\text{\textgreek{m}}}f\nabla_{\text{\textgreek{m}}}\bar{\text{\textgreek{f}}}+(\square_{g}f-2h)\bar{\text{\textgreek{f}}}\big)\big\}\, dg-\mathcal{B}_{f,h}^{(R)}[\text{\textgreek{f}};\text{\textgreek{t}}_{1},\text{\textgreek{t}}_{2}]-\breve{\mathcal{B}}_{f,h}^{(b)}[\text{\textgreek{f}};\text{\textgreek{t}}_{1}\text{\textgreek{t}}_{2}],
\end{split}
\label{eq:FundamentalCurrentFull-2-1}
\end{equation}
where 
\begin{equation}
\breve{\mathcal{B}}_{f,h}^{(b)}[\text{\textgreek{f}};\text{\textgreek{t}}_{1}\text{\textgreek{t}}_{2}]=\int_{\partial_{tim}\mathcal{M}\cap\mathcal{R}(\text{\textgreek{t}}_{1},\text{\textgreek{t}}_{2})}\big(f^{-1}\nabla_{\text{\textgreek{m}}}\nabla_{\text{\textgreek{n}}}f\nabla^{\text{\textgreek{m}}}f+\nabla_{\text{\textgreek{n}}}h-\frac{1}{2}(\nabla_{\text{\textgreek{n}}}(\square_{g}f))\big)|\text{\textgreek{f}}|^{2}\Big)n_{\partial_{tim}\mathcal{M}}^{\text{\textgreek{n}}}\, dg_{\partial_{tim}\mathcal{M}}.\label{eq:AdditionalBoundaryTermsNeumann}
\end{equation}
Notice that, if $sR^{-3\text{\textgreek{e}}_{0}}\gg1$, the term (\ref{eq:AdditionalBoundaryTermsNeumann})
is non-negative (in view of the properties of the functions $f,h$,
see Section \ref{sub:ChoiceOfSeedFunctions}), and thus it can be
dropped from the right hand side of (\ref{eq:FundamentalCurrentFull-2-1}).
Furthermore, if $l\gg1$ in Lemma \ref{lem:DoubleExponentForTheCarleman}
and $sR^{-3\text{\textgreek{e}}_{0}}\gg1$, the terms $2\nabla^{\text{\textgreek{m}}}(\breve{f}n_{\partial_{tim}\mathcal{M}}^{\text{\textgreek{n}}})\nabla_{\text{\textgreek{m}}}\text{\textgreek{f}}\nabla_{\text{\textgreek{n}}}\bar{\text{\textgreek{f}}}$
and $\nabla_{\text{\textgreek{n}}}(\breve{f}n_{\partial_{tim}\mathcal{M}}^{\text{\textgreek{n}}})\big)\nabla^{\text{\textgreek{m}}}\text{\textgreek{f}}\nabla_{\text{\textgreek{m}}}\bar{\text{\textgreek{f}}}$
in the left hand side of (\ref{eq:FundamentalCurrentFull-2-1}) (restricted
to the complement of $\mathscr{E}_{ext}$) can be absorbed into the
right hand side of \ref{eq:LowerBoundCurrentOutsideErgoregion}. Thus,
following exactly the same steps as we did in order to obtain (\ref{eq:MainCarlemanEstimate})
from (\ref{eq:FundamentalCurrentFull}) in the case $\partial_{tim}\mathcal{M}=\emptyset$,
we can also obtain (\ref{eq:MainCarlemanEstimate}) from (\ref{eq:FundamentalCurrentFull-2-1})
in the case when $\partial_{tim}\mathcal{M}\neq\emptyset$ and $n_{\partial_{tim}\mathcal{M}}(\text{\textgreek{f}})|_{\partial_{tim}\mathcal{M}}=0$.
\end{proof}

\section{\label{sec:Proof-of-Proposition}Proof of Proposition \ref{prop:QuantitativeDecayOutsideErgoregion}}

Let us introduce the parameters $0<\text{\textgreek{w}}_{0}\ll1$,
$\text{\textgreek{w}}_{+}\gg1$ and $\text{\textgreek{t}}_{1}\ge\bar{\text{\textgreek{t}}}_{0}+\text{\textgreek{e}}^{-2}\text{\textgreek{t}}_{*}$
depending on $\text{\textgreek{e}},\text{\textgreek{d}}_{1},R,\text{\textgreek{t}}_{*},\bar{\text{\textgreek{t}}}_{0},\mathcal{E}_{log}[\text{\textgreek{f}}],\mathcal{E}_{log}[\text{\textgreek{y}}],\mathcal{E}_{log}[T\text{\textgreek{y}}]$
and $\mathcal{E}[T^{2}\text{\textgreek{y}}]$ in the statement of
Proposition \ref{prop:QuantitativeDecayOutsideErgoregion} (we will
fix $\text{\textgreek{w}}_{0}$, $\text{\textgreek{w}}_{+}$ and $\text{\textgreek{t}}_{1}$
later), and, for $n=\lceil\log_{2}(\frac{\text{\textgreek{w}}_{+}}{\text{\textgreek{w}}_{0}})\rceil$,
let us decompose $\text{\textgreek{y}}$ and $T\text{\textgreek{y}}$
into their frequency localised components $\{\text{\textgreek{y}}_{k}\}_{k=0}^{n}$,
$\text{\textgreek{y}}_{\ge\text{\textgreek{w}}_{+}}$ and $\{(T\text{\textgreek{y}})_{k}\}_{k=0}^{n}$,
$(T\text{\textgreek{y}})_{\ge\text{\textgreek{w}}_{+}}$, respectively,
as in Section \ref{sub:Frequency-cut-off} (notice that (\ref{eq:BoundednessEnergy})
is satisfied in view of (\ref{BoundednessEnergyPsi})). 

In view of Lemma \ref{lem:NearZeroFrequencyEstimateDerivatives},
(\ref{eq:BoundForContradiction}) (and (\ref{BoundednessEnergyPsi})),
as well as a Hardy-type inequality (of the form (\ref{eq:HardyInterior})),
we obtain for any $\text{\textgreek{t}}\ge2\text{\textgreek{t}}_{*}$
and any $0<a\ll1$:
\begin{align}
\int_{\mathcal{R}(\text{\textgreek{t}}-\text{\textgreek{t}}_{*},\text{\textgreek{t}}+\text{\textgreek{t}}_{*})\cap\{r\le R\}}\big(J_{\text{\textgreek{m}}}^{N}(\text{\textgreek{y}}_{0})N^{\text{\textgreek{m}}}+|\text{\textgreek{y}}_{0}|^{2}\big) & \le C\text{\textgreek{w}}_{0}^{2}\int_{\mathcal{R}(\text{\textgreek{t}}-\text{\textgreek{t}}_{*},\text{\textgreek{t}}+\text{\textgreek{t}}_{*})\cap\{r\le R\}}\big(J_{\text{\textgreek{m}}}^{N}(\text{\textgreek{f}})N^{\text{\textgreek{m}}}+|\text{\textgreek{f}}|^{2}\big)+\label{eq:ZeroPartBound}\\
 & \hphantom{\le C\Big(\text{\textgreek{w}}_{0}^{2}}+C_{a}\Big(\text{\textgreek{w}}_{0}^{2}\big(\log(2+\text{\textgreek{t}})\big)^{4}+(1+\text{\textgreek{w}}_{0}\text{\textgreek{t}})^{-1}\Big)(1+\text{\textgreek{w}}_{0}^{-1-a})R^{2}\mathcal{E}_{log}[\text{\textgreek{f}}]\le\nonumber \\
 & \le C_{a}\Big\{\text{\textgreek{w}}_{0}^{2}R^{2}\text{\textgreek{t}}_{*}\big(\log(2+\text{\textgreek{t}})\big)^{4}+\text{\textgreek{w}}_{0}^{2}\big(\log(2+\text{\textgreek{t}})\big)^{4}(1+\text{\textgreek{w}}_{0}^{-1-a})R^{2}+\nonumber \\
 & \hphantom{\le C_{a}\Big\{\text{\textgreek{w}}_{0}^{2}R^{2}T_{*}\big(\log(2+\text{\textgreek{t}})\big)^{4}++}+(1+\text{\textgreek{w}}_{0}\text{\textgreek{t}})^{-1}(1+\text{\textgreek{w}}_{0}^{-1-a})R^{2}\Big\}\mathcal{E}_{log}[\text{\textgreek{f}}].
\end{align}
 From Lemmas \ref{lem:DtToOmegaInequalities} and \ref{lem:HighFrequencies},
we obtain for any $\text{\textgreek{t}}>0$ and any $\text{\textgreek{d}}>0$:
\begin{align}
\sum_{j=0}^{\lfloor\frac{\text{\textgreek{d}}^{-1}-1}{2}\rfloor}\int_{\mathcal{R}(\text{\textgreek{t}}+2j\text{\textgreek{t}}_{*},\text{\textgreek{t}}+2(j+1)\text{\textgreek{t}}_{*})\cap\{r\le R\}}\big(J_{\text{\textgreek{m}}}^{N}(\text{\textgreek{y}}_{\ge\text{\textgreek{w}}_{+}})N^{\text{\textgreek{m}}}+|\text{\textgreek{y}}_{\ge\text{\textgreek{w}}_{+}}|^{2}\big) & \le\int_{\mathcal{R}(\text{\textgreek{t}},\text{\textgreek{t}}+\text{\textgreek{d}}^{-1}\text{\textgreek{t}}_{*})\cap\{r\le R\}}\big(J_{\text{\textgreek{m}}}^{N}(\text{\textgreek{y}}_{\ge\text{\textgreek{w}}_{+}})N^{\text{\textgreek{m}}}+|\text{\textgreek{y}}_{\ge\text{\textgreek{w}}_{+}}|^{2}\big)\le\label{eq:HighFrequencyPidgeonholePrinciple}\\
 & \le C\text{\textgreek{d}}^{-1}\text{\textgreek{t}}_{*}\text{\textgreek{w}}_{+}^{-2}\big(\mathcal{E}_{log}[\text{\textgreek{y}}]+\mathcal{E}[T\text{\textgreek{y}}]\big)+C\mathcal{E}_{log}[\text{\textgreek{y}}].\nonumber 
\end{align}

We will assume that $\text{\textgreek{w}}_{+}$ is sufficiently large
in terms of $\text{\textgreek{e}},\text{\textgreek{t}}_{*},\mathcal{E}_{log}[\text{\textgreek{y}}],\mathcal{E}[T\text{\textgreek{y}}]$
so that 
\begin{equation}
\text{\textgreek{t}}_{*}\text{\textgreek{w}}_{+}^{-2}\big(\mathcal{E}_{log}[\text{\textgreek{y}}]+\mathcal{E}[T\text{\textgreek{y}}]\big)\ll\text{\textgreek{e}}.\label{eq:SmallnessFromHighFrequency}
\end{equation}
Let us also use the ansatz 
\begin{equation}
\text{\textgreek{w}}_{0}=\frac{\bar{\text{\textgreek{w}}}_{0}}{\big(\log(2+\text{\textgreek{t}}_{1})\big)^{8}},\label{eq:OmegaZero}
\end{equation}
and let us assume that $\bar{\text{\textgreek{w}}}_{0}$ is sufficiently
small in terms of $\text{\textgreek{e}},R,\text{\textgreek{t}}_{*},\mathcal{E}_{log}[\text{\textgreek{f}}]$,
and $\text{\textgreek{t}}_{1}$ is sufficiently large in terms of
$\text{\textgreek{e}},R,\mathcal{E}_{log}[\text{\textgreek{f}}],\bar{\text{\textgreek{w}}}_{0}$,
so that for any $\text{\textgreek{t}}_{1}\le\text{\textgreek{t}}\le100\text{\textgreek{t}}_{1}$
(having fixed an $a\in(0,1)$: 
\begin{equation}
\big(\text{\textgreek{w}}_{0}^{2}R^{2}\text{\textgreek{t}}_{*}\big(\log(2+\text{\textgreek{t}})\big)^{4}+\text{\textgreek{w}}_{0}^{2}\big(\log(2+\text{\textgreek{t}})\big)^{4}(1+\text{\textgreek{w}}_{0}^{-1-a})R^{2}+(1+\text{\textgreek{w}}_{0}\text{\textgreek{t}})^{-1}(1+\text{\textgreek{w}}_{0}^{-1-a})R^{2}\big)\mathcal{E}_{log}[\text{\textgreek{f}}]\ll\text{\textgreek{e}}\label{eq:SmallnessFromLowFrequency}
\end{equation}
(later, we will also need to assume that $\text{\textgreek{t}}_{1}$
is also sufficiently large in terms of $\text{\textgreek{w}}_{+}$).
Then, (\ref{eq:ZeroPartBound}), (\ref{eq:HighFrequencyPidgeonholePrinciple}),
(\ref{eq:SmallnessFromLowFrequency}) and (\ref{eq:SmallnessFromHighFrequency})
imply that for any $\text{\textgreek{t}}\ge0$: 
\begin{equation}
\sum_{l=0}^{\lfloor\frac{\text{\textgreek{d}}^{-1}-1}{2}\rfloor}\int_{\mathcal{R}(\text{\textgreek{t}}+2l\text{\textgreek{t}}_{*},\text{\textgreek{t}}+2(l+1)\text{\textgreek{t}}_{*})\cap\{r\le R\}}\big(J_{\text{\textgreek{m}}}^{N}(\text{\textgreek{y}}_{0})N^{\text{\textgreek{m}}}+J_{\text{\textgreek{m}}}^{N}(\text{\textgreek{y}}_{\ge\text{\textgreek{w}}_{+}})N^{\text{\textgreek{m}}}+|\text{\textgreek{y}}_{0}|^{2}+|\text{\textgreek{y}}_{\ge\text{\textgreek{w}}_{+}}|^{2}\big)\le\frac{1}{20}\text{\textgreek{e}}\text{\textgreek{d}}^{-1}+C\mathcal{E}_{log}[\text{\textgreek{y}}].\label{eq:EndBounds}
\end{equation}

Repeating the same procedure for $T\text{\textgreek{y}}$ in place
of $\text{\textgreek{y}}$ and adding the result to (\ref{eq:EndBounds}),
we obtain for any $\text{\textgreek{d}}>0$ (provided $\bar{\text{\textgreek{w}}}_{0}$
is fixed sufficiently small in terms of $\text{\textgreek{e}},R,\text{\textgreek{t}}_{*},\mathcal{E}_{log}[\text{\textgreek{f}}],\mathcal{E}_{log}[\text{\textgreek{y}}]$,
$\text{\textgreek{t}}_{1}$ is fixed sufficiently large in terms of
$\text{\textgreek{e}},R,\mathcal{E}_{log}[\text{\textgreek{f}}],\mathcal{E}_{log}[\text{\textgreek{y}}]\bar{\text{\textgreek{w}}}_{0}$
and $\text{\textgreek{w}}_{+}$ is fixed sufficiently large in terms
of $\text{\textgreek{e}},\text{\textgreek{t}}_{*},\mathcal{E}_{log}[\text{\textgreek{y}}],\mathcal{E}_{log}[T\text{\textgreek{y}}],\mathcal{E}[T^{2}\text{\textgreek{y}}]$):
\begin{align}
\sum_{l=0}^{\lfloor\frac{\text{\textgreek{d}}^{-1}-1}{2}\rfloor}\sum_{j=0}^{1}\int_{\mathcal{R}(\text{\textgreek{t}}+2l\text{\textgreek{t}}_{*},\text{\textgreek{t}}+2(l+1)\text{\textgreek{t}}_{*})\cap\{r\le R\}}\big(J_{\text{\textgreek{m}}}^{N}((T^{j}\text{\textgreek{y}})_{0})N^{\text{\textgreek{m}}} & +J_{\text{\textgreek{m}}}^{N}((T^{J}\text{\textgreek{y}})_{\ge\text{\textgreek{w}}_{+}})N^{\text{\textgreek{m}}}+|(T^{j}\text{\textgreek{y}})_{0}|^{2}+|(T^{j}\text{\textgreek{y}})_{\ge\text{\textgreek{w}}_{+}}|^{2}\big)\le\label{eq:EndBounds-1}\\
 & \le\frac{1}{10}\text{\textgreek{e}}\text{\textgreek{d}}^{-1}+C\sum_{j=0}^{1}\mathcal{E}_{log}[T^{j}\text{\textgreek{y}}].\nonumber 
\end{align}

In view of Corollary \ref{cor:CarlemanForPsik}, for any $1\le k\le n$,
any $0<\text{\textgreek{d}}_{1},\text{\textgreek{d}}_{2}<1$, any
$0<\text{\textgreek{e}}_{0}<1$ and any $\bar{\text{\textgreek{t}}}\ge\text{\textgreek{t}}_{1}$,
we can bound: 
\begin{equation}
\begin{split}\int_{(\mathcal{R}(\text{\textgreek{t}}_{1},\bar{\text{\textgreek{t}}})\backslash\mathscr{E}_{\text{\textgreek{d}}_{1}})\cap\{r\le R\}}\Big(J_{\text{\textgreek{m}}}^{N}(\text{\textgreek{y}}_{k})N^{\text{\textgreek{m}}} & +|\text{\textgreek{y}}_{k}|^{2}\Big)\le\\
\le C_{R} & \text{\textgreek{d}}_{2}\int_{\mathcal{R}(\text{\textgreek{t}}_{1},\bar{\text{\textgreek{t}}})\cap\mathscr{E}_{\text{\textgreek{d}}_{1}/2}}\Big(J_{\text{\textgreek{m}}}^{N}(\text{\textgreek{y}}_{k})N^{\text{\textgreek{m}}}+|\text{\textgreek{y}}_{k}|^{2}\Big)+\\
 & +C_{\text{\textgreek{e}}_{0}\text{\textgreek{d}}_{1}R}\text{\textgreek{w}}_{0}^{-10}\big(\log(2+\bar{\text{\textgreek{t}}})\big)^{4}e^{C_{\text{\textgreek{e}}_{0}\text{\textgreek{d}}_{1}\text{\textgreek{w}}_{+}}\max\{\text{\textgreek{w}}_{0}^{-\text{\textgreek{e}}_{0}},-\log\text{\textgreek{d}}_{2}\}}\mathcal{E}_{log}[\text{\textgreek{y}}].
\end{split}
\label{eq:CarlemanEstimate-1}
\end{equation}
Let us set 
\begin{equation}
\text{\textgreek{d}}_{2}=\text{\textgreek{w}}_{0}^{3}\text{\textgreek{w}}_{+}^{-1}\bar{\text{\textgreek{d}}}_{2},
\end{equation}
where $\bar{\text{\textgreek{d}}}_{2}$ is sufficiently small in terms
of $\text{\textgreek{e}},\text{\textgreek{e}}_{0},R,\text{\textgreek{t}}_{*},\mathcal{E}_{log}[\text{\textgreek{y}}]$.
Assuming also that $\bar{\lyxmathsym{\textgreek{w}}}_{0}$ in (\ref{eq:OmegaZero})
has been fixed sufficiently small in terms of $\text{\textgreek{e}},\text{\textgreek{e}}_{0},R,\text{\textgreek{t}}_{*},\mathcal{E}_{log}[\text{\textgreek{y}}]$,
from (\ref{eq:CarlemanEstimate-1}), (\ref{eq:EnergyClassAPrioriBoundLocal}),
(\ref{BoundednessEnergyPsi}) and the Poincare inequality
\begin{equation}
\int_{\mathcal{R}(\text{\textgreek{t}}_{1},\bar{\text{\textgreek{t}}})\cap\mathscr{E}_{\text{\textgreek{d}}_{1}}}|\text{\textgreek{y}}_{k}|^{2}\le C\int_{\mathcal{R}(\text{\textgreek{t}}_{1},\bar{\text{\textgreek{t}}})\cap\{r\le R\}}J_{\text{\textgreek{m}}}^{N}(\text{\textgreek{y}}_{k})N^{\text{\textgreek{m}}}+C\int_{(\mathcal{R}(\text{\textgreek{t}}_{1},\bar{\text{\textgreek{t}}})\backslash\mathscr{E}_{\text{\textgreek{d}}_{1}})\cap\{r\le R\}}|\text{\textgreek{y}}_{k}|^{2},
\end{equation}
we obtain after summing over all $k\in\{1,\ldots,n\}$ provided $\bar{\text{\textgreek{d}}}_{2}$
is sufficiently small in terms of $\text{\textgreek{e}},\text{\textgreek{e}}_{0},R,\text{\textgreek{t}}_{*},\mathcal{E}_{log}[\text{\textgreek{y}}]$
(recall that $n\sim\log(\text{\textgreek{w}}_{0}^{-1}\text{\textgreek{w}}_{+})$):
\begin{align}
\sum_{k=1}^{n} & \int_{(\mathcal{R}(\text{\textgreek{t}}_{1},\bar{\text{\textgreek{t}}})\backslash\mathscr{E}_{\text{\textgreek{d}}_{1}})\cap\{r\le R\}}\Big(J_{\text{\textgreek{m}}}^{N}(\text{\textgreek{y}}_{k})N^{\text{\textgreek{m}}}+|\text{\textgreek{y}}_{k}|^{2}\Big)\le\label{eq:BoundMiddle}\\
 & \le C_{\text{\textgreek{e}}_{0}R}(\bar{\text{\textgreek{t}}}-\text{\textgreek{t}}_{1})\text{\textgreek{d}}_{2}\text{\textgreek{w}}_{0}^{-2}\log(\text{\textgreek{w}}_{0}^{-1}\text{\textgreek{w}}_{+})\mathcal{E}_{log}[\text{\textgreek{y}}]+C_{\text{\textgreek{e}}_{0}\text{\textgreek{d}}_{1}R}\text{\textgreek{w}}_{0}^{-10}\big(\log(2+\bar{\text{\textgreek{t}}})\big)^{4}e^{C_{\text{\textgreek{e}}_{0}\text{\textgreek{d}}_{1}\text{\textgreek{w}}_{+}}\max\{\text{\textgreek{w}}_{0}^{-\text{\textgreek{e}}_{0}},-\log\text{\textgreek{d}}_{2}\}}\mathcal{E}_{log}[\text{\textgreek{y}}]\le\nonumber \\
 & \le\frac{1}{40\text{\textgreek{t}}_{*}}\text{\textgreek{e}}\bar{\text{\textgreek{t}}}+C_{1}\big(\log(2+\bar{\text{\textgreek{t}}})\big)^{14}e^{C_{1}\big(\log(2+\text{\textgreek{t}}_{1})\big)^{8\text{\textgreek{e}}_{0}}},\nonumber 
\end{align}
where $C_{1}$ depends on $\text{\textgreek{e}},\text{\textgreek{e}}_{0},\text{\textgreek{d}}_{1},R,\text{\textgreek{t}}_{*},\mathcal{E}_{log}[\text{\textgreek{y}}],\text{\textgreek{w}}_{+}$. 

Repeating the same procedure for $T\text{\textgreek{y}}$ in place
of $\text{\textgreek{y}}$, we obtain the following analogue of (\ref{eq:BoundMiddle}):
\begin{equation}
\sum_{k=1}^{n}\int_{(\mathcal{R}(\text{\textgreek{t}}_{1},\bar{\text{\textgreek{t}}})\backslash\mathscr{E}_{\text{\textgreek{d}}_{1}})\cap\{r\le R\}}\Big(J_{\text{\textgreek{m}}}^{N}((T\text{\textgreek{y}})_{k})N^{\text{\textgreek{m}}}+|(T\text{\textgreek{y}})_{k}|^{2}\Big)\le\frac{1}{40\text{\textgreek{t}}_{*}}\text{\textgreek{e}}\bar{\text{\textgreek{t}}}+C_{2}\big(\log(2+\bar{\text{\textgreek{t}}})\big)^{14}e^{C_{2}\big(\log(2+\text{\textgreek{t}}_{1})\big)^{8\text{\textgreek{e}}_{0}}},\label{eq:BoundMiddle-2}
\end{equation}
where $C_{2}$ depends on$\text{\textgreek{e}},\text{\textgreek{e}}_{0},\text{\textgreek{d}}_{1},R,\text{\textgreek{t}}_{*},\mathcal{E}_{log}[T\text{\textgreek{y}}],\text{\textgreek{w}}_{+}$.

From (\ref{eq:BoundMiddle}) and (\ref{eq:BoundMiddle-2}) we obtain
for any $\text{\textgreek{d}}>0$ (setting $\bar{\text{\textgreek{t}}}=\text{\textgreek{t}}_{1}+\text{\textgreek{d}}^{-1}\text{\textgreek{t}}_{*}$)
\begin{align}
\sum_{l=0}^{\lfloor\frac{\text{\textgreek{d}}^{-1}-1}{2}\rfloor}\Bigg\{\sum_{k=1}^{n}\sum_{j=0}^{1} & \int_{\mathcal{R}(\text{\textgreek{t}}_{1}+2l\text{\textgreek{t}}_{*},\text{\textgreek{t}}_{1}+2(l+1)\text{\textgreek{t}}_{*})\backslash\mathscr{E}_{\text{\textgreek{d}}_{1}})\cap\{r\le R\}}\Big(J_{\text{\textgreek{m}}}^{N}((T^{j}\text{\textgreek{y}})_{k})N^{\text{\textgreek{m}}}+|(T^{j}\text{\textgreek{y}})_{k}|^{2}\Big)\Bigg\}\le\label{eq:BoundMiddle-1}\\
 & \hphantom{\int_{\mathcal{R}(\text{\textgreek{t}}_{1}+2lT_{*},\text{\textgreek{t}}_{1}+2(l+1)}}\le\frac{1}{20\text{\textgreek{t}}_{*}}\text{\textgreek{e}}(\text{\textgreek{t}}_{1}+\text{\textgreek{d}}^{-1}\text{\textgreek{t}}_{*})+C_{3}\big(\log(\text{\textgreek{t}}_{1}+\text{\textgreek{d}}^{-1}\text{\textgreek{t}}_{*})\big)^{14}e^{C_{3}\big(\log(2+\text{\textgreek{t}}_{1})\big)^{8\text{\textgreek{e}}_{0}}},\nonumber 
\end{align}
where $C_{3}=C_{1}+C_{2}$. Adding (\ref{eq:EndBounds-1}) (for $\text{\textgreek{t}}=\text{\textgreek{t}}_{1}$)
and (\ref{eq:BoundMiddle-1}), we therefore obtain for any $\text{\textgreek{d}}>0$:
\begin{align}
\sum_{l=0}^{\lfloor\frac{\text{\textgreek{d}}^{-1}-1}{2}\rfloor}\Bigg\{\sum_{j=0}^{1} & \int_{\mathcal{R}(\text{\textgreek{t}}_{1}+2l\text{\textgreek{t}}_{*},\text{\textgreek{t}}_{1}+2(l+1)\text{\textgreek{t}}_{*})\backslash\mathscr{E}_{\text{\textgreek{d}}_{1}})\cap\{r\le R\}}\Big(J_{\text{\textgreek{m}}}^{N}(T^{j}\text{\textgreek{y}})N^{\text{\textgreek{m}}}+|T^{j}\text{\textgreek{y}}|^{2}\Big)\Bigg\}\le\label{eq:BoundBeforeFinalPidgeonhole}\\
 & \le\frac{1}{10}\text{\textgreek{e}}\text{\textgreek{d}}^{-1}+C\sum_{j=0}^{1}\mathcal{E}_{log}[T^{j}\text{\textgreek{y}}]+\frac{1}{20\text{\textgreek{t}}_{*}}\text{\textgreek{e}}(\text{\textgreek{t}}_{1}+\text{\textgreek{d}}^{-1}\text{\textgreek{t}}_{*})+C_{3}\big(\log(\text{\textgreek{t}}_{1}+\text{\textgreek{d}}^{-1}\text{\textgreek{t}}_{*})\big)^{14}e^{C_{3}\big(\log(2+\text{\textgreek{t}}_{1})\big)^{8\text{\textgreek{e}}_{0}}}.\nonumber 
\end{align}

Applying the pidgeonhole principle on (\ref{eq:BoundBeforeFinalPidgeonhole})
(assuming that $\text{\textgreek{d}}\ll1$), we infer that there exists
some $l_{0}\in\{0,\ldots,\lfloor\frac{\text{\textgreek{d}}^{-1}-1}{2}\rfloor\}$
such that 
\begin{align}
\sum_{j=0}^{1} & \int_{\mathcal{R}(\text{\textgreek{t}}_{1}+2l_{0}\text{\textgreek{t}}_{*},\text{\textgreek{t}}_{1}+2(l_{0}+1)\text{\textgreek{t}}_{*})\backslash\mathscr{E}_{\text{\textgreek{d}}_{1}})\cap\{r\le R\}}\Big(J_{\text{\textgreek{m}}}^{N}(T^{j}\text{\textgreek{y}})N^{\text{\textgreek{m}}}+|T^{j}\text{\textgreek{y}}|^{2}\Big)\le\label{eq:AlmostDoneWithCarleman}\\
 & \le\lfloor\frac{\text{\textgreek{d}}^{-1}-1}{2}\rfloor^{-1}\big(\frac{1}{5}\text{\textgreek{e}}\text{\textgreek{d}}^{-1}+\frac{1}{20}\text{\textgreek{e}}\text{\textgreek{t}}_{*}^{-1}\text{\textgreek{t}}_{1}+C\sum_{j=0}^{1}\mathcal{E}_{log}[T^{j}\text{\textgreek{y}}]+C_{3}\big(\log(\text{\textgreek{t}}_{1}+\text{\textgreek{d}}^{-1}\text{\textgreek{t}}_{*})\big)^{14}e^{C_{3}\big(\log(2+\text{\textgreek{t}}_{1})\big)^{8\text{\textgreek{e}}_{0}}}\big)\le\nonumber \\
 & \le\frac{\text{\textgreek{e}}}{2}(1+\text{\textgreek{t}}_{*}^{-1}\text{\textgreek{d}}\text{\textgreek{t}}_{1})+C\text{\textgreek{d}}\sum_{j=0}^{1}\mathcal{E}_{log}[T^{j}\text{\textgreek{y}}]+C_{3}\text{\textgreek{d}}\big(\log(\text{\textgreek{t}}_{1}+\text{\textgreek{d}}^{-1}\text{\textgreek{t}}_{*})\big)^{14}e^{C_{3}\big(\log(2+\text{\textgreek{t}}_{1})\big)^{8\text{\textgreek{e}}_{0}}}.\nonumber 
\end{align}
Thus, provided 
\begin{equation}
\text{\textgreek{d}}=\frac{\bar{\text{\textgreek{d}}}}{\text{\textgreek{t}}_{1}},
\end{equation}
where $\bar{\text{\textgreek{d}}}$ is small in terms of $\text{\textgreek{e}},\text{\textgreek{t}}_{*},\mathcal{E}_{log}[\text{\textgreek{y}}],\mathcal{E}_{log}[T\text{\textgreek{y}}]$
and the precise choice of the constants $C_{1},C_{2}$, and that $\text{\textgreek{t}}_{1}$
is chosen sufficiently large in terms of of $\text{\textgreek{e}},\text{\textgreek{e}}_{0},\text{\textgreek{t}}_{*}$
and the precise choice of $C_{1},C_{2}$ (assuming also that $\text{\textgreek{e}}_{0}$
has been fixed so that $0<\text{\textgreek{e}}_{0}<\frac{1}{8}$),
from (\ref{eq:AlmostDoneWithCarleman}) we infer: 
\begin{equation}
\sum_{j=0}^{1}\int_{\mathcal{R}(\text{\textgreek{t}}_{1}+2l_{0}\text{\textgreek{t}}_{*},\text{\textgreek{t}}_{1}+2(l_{0}+1)\text{\textgreek{t}}_{*})\backslash\mathscr{E}_{\text{\textgreek{d}}_{1}})\cap\{r\le R\}}\Big(J_{\text{\textgreek{m}}}^{N}(T^{j}\text{\textgreek{y}})N^{\text{\textgreek{m}}}+|T^{j}\text{\textgreek{y}}|^{2}\Big)<\text{\textgreek{e}}.\label{eq:AlmostDoneWithCarleman-1}
\end{equation}

Setting $\text{\textgreek{t}}_{\natural}=\text{\textgreek{t}}_{1}+(2l_{0}+1)\text{\textgreek{t}}_{*}$
(and thus $\text{\textgreek{t}}_{1}+2l_{0}\text{\textgreek{t}}_{*}=\text{\textgreek{t}}_{\natural}-\text{\textgreek{t}}_{*}$
and $\text{\textgreek{t}}_{1}+2(l_{0}+1)\text{\textgreek{t}}_{*}=\text{\textgreek{t}}_{\natural}+\text{\textgreek{t}}_{*}$),
(\ref{eq:AlmostDoneWithCarleman-1}) yields (\ref{eq:ForLocalConvergence}).
\qed

\section{\label{sec:Proof-of-Corollary}Proof of Corollary \ref{cor:VortexMain}}

The proof of Corollary \ref{cor:VortexMain} follows immediately from
Theorem \ref{thm:FriedmanInstability} applied to the quotient of
$(\mathbb{R}\times\mathcal{V}_{hyd,\text{\textgreek{d}}},g_{hyd})$
by the translations in the $z$-direction, i.\,e.~the $2+1$ dimensional
spacetime $(\mathbb{R}\times\bar{\mathcal{V}}_{hyd,\text{\textgreek{d}}},\bar{g}_{hyd})$,
where $\bar{\mathcal{V}}_{hyd,\text{\textgreek{d}}}=\mathbb{R}^{2}\backslash\{\bar{r}\le\text{\textgreek{d}}\}$
(in the polar $(\bar{r},\text{\textgreek{j}})$ coordinate system)
and 
\begin{equation}
\bar{g}_{hyd}=-\big(1-\frac{C^{2}}{\bar{r}^{2}}\big)dt^{2}+d\bar{r}^{2}-2Cdtd\text{\textgreek{j}}+\bar{r}^{2}d\text{\textgreek{j}}^{2}\label{eq:HydrodynamicVortex-1}
\end{equation}
(see also the remark below Theorem \ref{thm:FriedmanInstability},
as well as Section \ref{sub:ProofWithBoundaryConditions}, regarding
the Dirichlet or Neumann boundary conditions on $\{\bar{r}=\text{\textgreek{d}}\}$). 

In particular, in the language of Section \ref{sub:ProofWithBoundaryConditions},
$(\mathbb{R}\times\bar{\mathcal{V}}_{hyd,\text{\textgreek{d}}},\bar{g}_{hyd})$
is a smooth Lorentzian manifold with smooth timelike boundary 
\begin{equation}
\partial_{tim}\big(\mathbb{R}\times\bar{\mathcal{V}}_{hyd,\text{\textgreek{d}}}\big)=\{\bar{r}=\text{\textgreek{d}}\}.
\end{equation}
The double $(\widetilde{\mathbb{R}\times\bar{\mathcal{V}}}_{hyd,\text{\textgreek{d}}},\tilde{g}_{hyd})$
of $(\mathbb{R}\times\bar{\mathcal{V}}_{hyd,\text{\textgreek{d}}},\bar{g}_{hyd})$
across the boundary $\partial_{tim}\big(\mathbb{R}\times\bar{\mathcal{V}}_{hyd,\text{\textgreek{d}}}\big)$
is diffeomorphic to $\mathbb{R}\times\mathbb{R}\times\mathbb{S}^{1}$,
with the metric $\tilde{g}_{hyd}$ in the $(t,\bar{r},\text{\textgreek{j}})$
coordinate chart of $\mathbb{R}\times\mathbb{R}\times\mathbb{S}^{1}$
having the form: 
\begin{equation}
\tilde{g}_{hyd}=-\big(1-\frac{C^{2}}{(|\bar{r}-\text{\textgreek{d}}|+\text{\textgreek{d}})^{2}}\big)dt^{2}+d\bar{r}^{2}-2Cdtd\text{\textgreek{j}}+(|\bar{r}-\text{\textgreek{d}}|+\text{\textgreek{d}})^{2}d\text{\textgreek{j}}^{2}\label{eq:DoubleMetric}
\end{equation}
Notice that $(\widetilde{\mathbb{R}\times\mathcal{V}}_{hyd,\text{\textgreek{d}}},\tilde{g}_{hyd})$
is a globally hyperbolic spacetime without boundary, with Cauchy hypersurface
$\{t=0\}$. Let $i_{1},i_{2}:(\mathbb{R}\times\bar{\mathcal{V}}_{hyd,\text{\textgreek{d}}},\bar{g}_{hyd})\rightarrow(\widetilde{\mathbb{R}\times\bar{\mathcal{V}}}_{hyd,\text{\textgreek{d}}},\tilde{g}_{hyd})$
be the two natural inclusions (see Section \ref{sub:ProofWithBoundaryConditions}).
Then, in the coordinate charts $(t,\bar{r},\text{\textgreek{j}})$
on $\mathbb{R}\times[\text{\textgreek{d}},+\infty)\times\mathbb{S}^{1}\simeq\mathbb{R}\times\bar{\mathcal{V}}_{hyd,\text{\textgreek{d}}}$
and $\mathbb{R}\times\mathbb{R}\times\mathbb{S}^{1}\simeq\widetilde{\mathbb{R}\times\bar{\mathcal{V}}}_{hyd,\text{\textgreek{d}}}$,
we have $i_{1}\big((t,\bar{r},\text{\textgreek{j}})\big)=(t,\bar{r},\text{\textgreek{j}})$
and $i_{2}\big((t,\bar{r},\text{\textgreek{j}})\big)=(t,\text{\textgreek{d}}-\bar{r},\text{\textgreek{j}})$.

Note that $\tilde{g}_{hyd}$ is smooth everywhere except on $i_{1}\big(\partial_{tim}\big(\mathbb{R}\times\bar{\mathcal{V}}_{hyd,\text{\textgreek{d}}}\big)\big)=\{\bar{r}=\text{\textgreek{d}}\}$.
Notice also that $(\widetilde{\mathbb{R}\times\bar{\mathcal{V}}}_{hyd,\text{\textgreek{d}}},\tilde{g}_{hyd})$
has no event horizon $\mathcal{H}$ (and thus, trivially, $\mathcal{H}\cap i_{1}\big(\partial_{tim}\big(\mathbb{R}\times\bar{\mathcal{V}}_{hyd,\text{\textgreek{d}}}\big)\big)=\emptyset$),
and $i_{1}\big(\partial_{tim}\big(\mathbb{R}\times\bar{\mathcal{V}}_{hyd,\text{\textgreek{d}}}\big)\big)\cap\{t=0\}$
is compact. Thus, in view of the remark below Theorem \ref{thm:FriedmanInstability}
on spacetimes with timelike boundary, it only remains to verify that
$(\widetilde{\mathbb{R}\times\bar{\mathcal{V}}}_{hyd,\text{\textgreek{d}}},\tilde{g}_{hyd})$
satisfies Assumptions \hyperref[Assumption 1]{G1}--\hyperref[Assumption 3]{G3}
and \hyperref[Assumption 4]{A1}, and that $i_{1}\big(\partial_{tim}\big(\mathbb{R}\times\bar{\mathcal{V}}_{hyd,\text{\textgreek{d}}}\big)\big)$
is invariant with respect to the stationary Killing field of $(\widetilde{\mathbb{R}\times\bar{\mathcal{V}}}_{hyd,\text{\textgreek{d}}},\tilde{g}_{hyd})$.

\begin{enumerate}

\item The vector field $\partial_{t}$ (in the $(t,\bar{r},\text{\textgreek{j}})$
coordinate system for $(\widetilde{\mathbb{R}\times\bar{\mathcal{V}}}_{hyd,\text{\textgreek{d}}},\tilde{g}_{hyd})$
is Killing, and the metric (\ref{eq:HydrodynamicVortex-1}) is asymptotically
flat (with the asymptotically flat region $\tilde{\mathcal{I}}_{as}=\{\bar{r}\ge R_{0}\gg1\}$
consisting of two connected components) and satisfies Assumption \hyperref[Assumption 1]{G1}.
Furthermore, $\partial_{tim}\big(\mathbb{R}\times\bar{\mathcal{V}}_{hyd,\text{\textgreek{d}}}\big)$
is $\partial_{t}$-invariant. 

\item The spacetime $(\widetilde{\mathbb{R}\times\bar{\mathcal{V}}}_{hyd,\text{\textgreek{d}}},\tilde{g}_{hyd})$
has no event horizon $\mathcal{H}$, and thus Assumption \hyperref[Assumption 2]{G2}
is trivially satisfied.

\item The spacetime $(\widetilde{\mathbb{R}\times\bar{\mathcal{V}}}_{hyd,\text{\textgreek{d}}},\tilde{g}_{hyd})$
has a non empty ergoregion $\tilde{\mathscr{E}}=\{2\text{\textgreek{d}}-C<\bar{r}\le C\}$.
The boundary $\partial\tilde{\mathscr{E}}=\{\bar{r}=2\text{\textgreek{d}}-C\}\cup\{\bar{r}=C\}$
of $\tilde{\mathscr{E}}$ is a smooth hypersurface of $\widetilde{\mathbb{R}\times\bar{\mathcal{V}}}_{hyd,\text{\textgreek{d}}}$,
and $\widetilde{\mathbb{R}\times\bar{\mathcal{V}}}_{hyd,\text{\textgreek{d}}}\backslash\tilde{\mathscr{E}}$
consists of two connected components, each containing one asymptotically
flat end of $\widetilde{\mathbb{R}\times\bar{\mathcal{V}}}_{hyd,\text{\textgreek{d}}}$
(and, thus, $\tilde{\mathscr{E}}_{ext}=\tilde{\mathscr{E}}$). In
particular, Assumption \hyperref[Assumption 3]{G3} is satisfied.

\item Assumption \hyperref[Assumption 4]{A1} is readily satisfied
in view of the fact that $(\widetilde{\mathbb{R}\times\bar{\mathcal{V}}}_{hyd,\text{\textgreek{d}}},\tilde{g}_{hyd})$
is also axisymmetric, with axisymmetric Killing field $\partial_{\text{\textgreek{j}}}$
such that $[\partial_{\text{\textgreek{j}}},\partial_{t}]=0$ and
the span of $\partial_{\text{\textgreek{j}}},\partial_{t}$ contains
a timelike direction (see the discussion in Section \ref{sub:Discussion On the unique continuation assumption}).

\end{enumerate}

Thus, the proof of Corollary \ref{cor:VortexMain} is complete.

\section{\label{sec:DiscussionFriedmanHeuristics}Aside: Discussion on Friedman's
heuristic argument}

In this Section, we will briefly sketch the heuristic arguments developed
by Friedman in \cite{Friedman1978}, and we will discuss their connections
with the methods used in this paper.

\subsection{Friedman's argument}

As we already explained in the introduction, on any globally hyperbolic,
stationary and asymptotically flat spacetime $(\mathcal{M},g)$ with
a non-empty ergoregion $\mathscr{E}$ and no future event horizon
$\mathcal{H}^{+}$, Friedman constructed, in \cite{Friedman1978},
a class of smooth solutions $\text{\textgreek{y}}$ to the wave quation
(\ref{eq:WaveEquation}) satisfying 
\begin{equation}
\int_{\text{\textgreek{S}}}J_{\text{\textgreek{m}}}^{T}(\text{\textgreek{y}})n^{\text{\textgreek{m}}}=-1,\label{eq:NegativeInitialEnergyFriedman}
\end{equation}
where $\text{\textgreek{S}}$ is a Cauchy hypersurface of $(\mathcal{M},g)$,
$T$ is the stationary Killing field of $(\mathcal{M},g)$ and $n$
is the future directed unit normal to $\text{\textgreek{S}}$. In
view of the conservation of the $T$-energy flux for solutions to
(\ref{eq:WaveEquation}) on $(\mathcal{M},g)$ and the fact that $J_{\text{\textgreek{m}}}^{T}(\text{\textgreek{y}})n^{\text{\textgreek{m}}}\ge0$
on $\mathcal{M}\backslash\mathscr{E}$, from (\ref{eq:NegativeInitialEnergyFriedman})
Friedman inferred that for any $\text{\textgreek{t}}\ge0$: 
\begin{equation}
\int_{\text{\textgreek{S}}_{\text{\textgreek{t}}}\cap\mathscr{E}}J_{\text{\textgreek{m}}}^{T}(\text{\textgreek{y}})n^{\text{\textgreek{m}}}\le-1,\label{eq:NonDecayErgoregionFriedman}
\end{equation}
where $\text{\textgreek{S}}_{\text{\textgreek{t}}}$ is defined as
in Section \ref{sec:Notational-conventions} (i.\,e.~the image of
$\text{\textgreek{S}}$ under the flow of $T$ for time $\text{\textgreek{t}}$).

Proceeding to study the consequences of the bound (\ref{eq:NonDecayErgoregionFriedman})
on the (in)stability properties of equation (\ref{eq:WaveEquation}),
Friedman first noted the following dichotomy for the energy flux through
the future null infinity $\mathcal{I}^{+}$ of any solution $\text{\textgreek{y}}$
to (\ref{eq:WaveEquation}), satisfying (\ref{eq:NegativeInitialEnergyFriedman}):%
\footnote{See \cite{Moschidisc} for the definition of the Friedlander radiation
field and the energy flux of $\text{\textgreek{f}}$ through $\mathcal{I}^{+}$
on general asymptotically flat spacetimes.%
} Either 
\begin{equation}
\int_{\mathcal{I}^{+}}J_{\text{\textgreek{m}}}^{T}(\text{\textgreek{y}})n_{\mathcal{I}^{+}}^{\text{\textgreek{m}}}=+\infty,
\end{equation}
 in which case (in view of (\ref{eq:NegativeInitialEnergyFriedman})
and the conservation of the $J^{T}$-flux) there exists a sequence
of hyperboloidal hypersurcases $\mathcal{S}_{\text{\textgreek{t}}_{n}}$
terminating at $\mathcal{I}^{+}$ such that 
\begin{equation}
\limsup_{n\rightarrow+\infty}\int_{\mathcal{S}_{\text{\textgreek{t}}_{n}}}J_{\text{\textgreek{m}}}^{T}(\text{\textgreek{y}})n_{\mathcal{S}_{\text{\textgreek{t}}_{n}}}^{\text{\textgreek{m}}}=+\infty,\label{eq:InstabilityOnHyperboloids}
\end{equation}
or 
\begin{equation}
\int_{\mathcal{I}^{+}}J_{\text{\textgreek{m}}}^{T}(\text{\textgreek{y}})n_{\mathcal{I}^{+}}^{\text{\textgreek{m}}}<+\infty.\label{eq:FiniteEnergyFluxNullInfinity}
\end{equation}
In case the first scenario (\ref{eq:InstabilityOnHyperboloids}) holds,
one immediately obtains an energy instability statement for equation
(\ref{eq:WaveEquation}). In case the second scenario (\ref{eq:FiniteEnergyFluxNullInfinity}),
Friedman argued (see \cite{Friedman1978}) that $\text{\textgreek{y}}$
``settles down'' to a ``non-radiative state'' $\tilde{\text{\textgreek{y}}}$,
which is to be interpreted as a solution to (\ref{eq:WaveEquation})
such that 
\begin{equation}
\int_{\mathcal{I}^{+}}J_{\text{\textgreek{m}}}^{T}(\tilde{\text{\textgreek{y}}})n_{\mathcal{I}^{+}}^{\text{\textgreek{m}}}=0.\label{eq:ZeroEnergyFluxTildePhi}
\end{equation}
Furthermore, in view of (\ref{eq:NonDecayErgoregionFriedman}), Friedman
argued that $\tilde{\text{\textgreek{y}}}$ should also satisfy for
all $\text{\textgreek{t}}\ge0$: 
\begin{equation}
\int_{\text{\textgreek{S}}_{\text{\textgreek{t}}}\cap\mathscr{E}}J_{\text{\textgreek{m}}}^{T}(\tilde{\text{\textgreek{y}}})n^{\text{\textgreek{m}}}\le-1.\label{eq:NegativeEnergyTildePhi}
\end{equation}

Assuming that $(\mathcal{M},g)$ is globally real analytic and that
the metric $g$ has a proper asymptotic expansion in powers of $r^{-1}$
in a neighborhood of $\mathcal{I}^{+}$, Friedman inferred from (\ref{eq:ZeroEnergyFluxTildePhi})
(using an adaptation of Holmgren's uniqueness theorem for analytic
linear partial differential equations, see \cite{Holmgren1901}) that
\begin{equation}
\tilde{\text{\textgreek{y}}}\equiv0\label{eq:IdenticallyZeroFinalState}
\end{equation}
on $(\mathcal{M},g)$. Thus, (\ref{eq:NegativeEnergyTildePhi}) and
(\ref{eq:IdenticallyZeroFinalState}) yield a contradiction, implying
that the scenario (\ref{eq:FiniteEnergyFluxNullInfinity}) should
not occur on such spacetimes.

\subsection{Comparison with the proof of Theorem \ref{thm:FriedmanInstability}}

In general terms, the proof of Theorem \ref{thm:FriedmanInstability}
(see Section \ref{sec:Proof-of-Theorem}) follows the roadmap of the
heuristic arguments of Friedman. In particular, our proof proceeds
by contradiction, assuming the energy bound (\ref{eq:BoundForContradiction})
on the $\{t=\text{\textgreek{t}}\}$ hypersurfaces, which is a slightly
stronger assumption than the energy bound (\ref{eq:FiniteEnergyFluxNullInfinity})
on $\mathcal{I}^{+}$ in the second scenario considered by Friedman. 

In Lemma \ref{lem:DecayToATrappedSolution}, we show that, under the
assumption (\ref{eq:BoundForContradiction}), a function $\text{\textgreek{y}}$
solving (\ref{eq:WaveEquation}) with compactly supported initial
data indeed ``settles down'' to a function $\tilde{\text{\textgreek{y}}}$
(in a well defined way), such that $\tilde{\text{\textgreek{y}}}$
vanishes identically outside the extended ergoregion $\mathscr{E}_{ext}$.
This result makes use (through Proposition \ref{prop:QuantitativeDecayOutsideErgoregion})
of the Carleman-type estimates of Section \ref{sec:Carleman}, as
well as the bound (\ref{eq:BoundForContradiction}). Here, assuming
merely the bound (\ref{eq:FiniteEnergyFluxNullInfinity}) on $\mathcal{I}^{+}$
would not be enough. Note that, in the argument of \cite{Friedman1978},
no justification is provided (even at the heuristic level) of why
a function $\text{\textgreek{y}}$ solving (\ref{eq:WaveEquation})
and satisfying (\ref{eq:FiniteEnergyFluxNullInfinity}) is expected
to ``settle down'' to a non-radiating solution $\tilde{\text{\textgreek{y}}}$
of (\ref{eq:WaveEquation}).

The fact that $\tilde{\text{\textgreek{y}}}$ vanishes outside $\mathscr{E}_{ext}$
follows from the estimates of Section \ref{sec:Carleman}, without
any need to impose a real analyticity assumption on $(\mathcal{M},g)$
or a complete asymptotic expansion for $g$ on $\mathcal{I}^{+}$.
In general, however, it can not be inferred that $\tilde{\text{\textgreek{y}}}$
vanishes also on $\mathscr{E}$.%
\footnote{We can in fact construct spacetimes $(\mathcal{M}^{d+1},g)$, $d\ge3$,
with a smooth solution $\tilde{\text{\textgreek{y}}}$ to an equation
of the form $\square_{g}\tilde{\text{\textgreek{y}}}+V\tilde{\text{\textgreek{y}}}=0$,
such that $T(V)=0$, $\tilde{\text{\textgreek{y}}}\equiv0$ on $\mathcal{M}\backslash\mathscr{E}$
and $\tilde{\text{\textgreek{y}}}$ not identically $0$ in $\mathscr{E}$.%
} Thus, a contradiction can not be reached following the argument of
Friedman in this setting. Instead, after restricting ourselves to
spacetimes $(\mathcal{M},g)$ satisfying the unique continuation assumption
\hyperref[Assumption 4]{A1}, which guarantees that $\tilde{\text{\textgreek{y}}}$
vanishes on $(\mathcal{M}\backslash\mathscr{E}_{ext})\cup\mathcal{U}$,
we reach the desired contradiction by exploiting our freedom to choose
the initial data for $\text{\textgreek{y}}$ appropriately: We choose
$(\text{\textgreek{y}},T\text{\textgreek{y}})|_{\text{\textgreek{S}}}$
to be supported in $\text{\textgreek{S}}\cap\mathcal{U}$, so that
the support of $(\text{\textgreek{y}},T\text{\textgreek{y}})|_{\text{\textgreek{S}}}$
will be disjoint from $\mathcal{M}\backslash\mathcal{U}$, where the
support of all the time translates of $(\tilde{\text{\textgreek{y}}},T\tilde{\text{\textgreek{y}}})|_{\text{\textgreek{S}}}$
is contained. Therefore, $\text{\textgreek{y}}$ and all the time
translates of $\tilde{\text{\textgreek{y}}}$ are orthogonal with
respect to the (indefinite) $T$-inner product (\ref{eq:TInnerProductSymplecic}).
This fact leads to relation (\ref{eq:ZeroForContradiction}), from
which a contradiction follows readily in view of (\ref{eq:NegativeEnergyTrappedSolution}).

\section{Acknowledgements}

I would like to thank my advisor Mihalis Dafermos for suggesting this
problem to me and providing me with valuable assistance while this
paper was being written. I would also like to thank Igor Rodnianski
for many insightful suggestions.

\bibliographystyle{plain}
\bibliography{DatabaseExample}

\end{document}